\documentclass[11pt,a4paper,reqno]{amsart}
\usepackage{dutchcal}
\usepackage[T1]{fontenc}
\usepackage{amsmath}
\usepackage{amsthm}
\usepackage{amsfonts}
\usepackage{amssymb}
\usepackage{graphicx}
\usepackage{color}
\usepackage{amsbsy}
\usepackage{mathrsfs}
\usepackage{bbm}
\usepackage{verbatim}
\usepackage[colorlinks=true,linkcolor=black,citecolor=blue]{hyperref}
\newtheorem{theorem}{Theorem}[section]
\newtheorem{lemma}[theorem]{Lemma}

\newtheorem{proposition}[theorem]{Proposition}

\theoremstyle{remark}
\newtheorem{remark}[theorem]{\it \bf{Remark}\/}

\numberwithin{equation}{section}
\catcode`@=11
\def\section{\@startsection{section}{1}%
  \z@{1.5\linespacing\@plus\linespacing}{.5\linespacing}%
  {\normalfont\bfseries\large\centering}}
\catcode`@=12
\newcommand{\be}{\begin{equation}}
\newcommand{\ee}{\end{equation}}
\newcommand{\bea}{\begin{eqnarray}}
\newcommand{\eea}{\end{eqnarray}}
\newcommand{\bee}{\begin{eqnarray*}}
\newcommand{\eee}{\end{eqnarray*}}

\def\pa{\partial}
\def\pr{\partial}

\def\RR{\mathbb{R}}

\def\sigmat{\tilde{\sigma}}

\def\de{\delta}

\def\ep{\varepsilon}

\catcode`@=11
\def\supess{\mathop{\operator@font Sup\,ess}}
\catcode`@=12

\def\RR{\mathbb{R}}

\def\a{\alpha}

\def\R2+{\RR ^2_+}

\def\pa{\partial}

\def\lim{\mathop{\rm lim}}

\def\sup{\mathop{\rm sup}}

\def\l{\lambda}

\def\log{{\rm log}}

\def\rhoh{\hat{\rho}}

\def\T{\Theta}

\def\pa{\partial}

\def\rhot{\tilde{\rho}}
\def\Psit{\Psi}

\def\Lamdba{\Lambda}

\def\pa{\partial}
\def\la{\langle}
\def\matchal{\mathcal}
\def\ra{\rangle}

\def\NL{\textrm{NL}}

\def\Psipx{\Psi_{P}}
\def\rhopx{\rho_{P}}
\def\rhox{\rho}
\def\Psix{\Psi}
\def\qx{Q}

\def\rhox{\rho}
\def\Psix{\Psi}
\def\Phix{\Phi}
\def\rhoh{\hat{\rho}}

\def\Et{\tilde{\mathcal E}}

\def\nut{\tilde{\nu}}
\newcommand\blue{{\color{\blue}}}
\begin{document}

\title[]{On blow up for the energy super critical defocusing {nonlinear Schr\"odinger equations} }
{\author[F. Merle]{Frank Merle}
\address{AGM, Universit\'e de Cergy Pontoise and IHES, France}
\email{merle@math.u-cergy.fr}
\author[P. Rapha\"el]{Pierre Rapha\"el}
\address{Department of Pure Mathematics and Mathematical Statistics, Cambridge, UK}
\email{praphael@maths.cam.ac.uk}
\author[I. Rodnianski]{Igor Rodnianski}
\address{Department of Mathematics, Princeton University, Princeton, NJ, USA}
\email{irod@math.princeton.edu}
\author[J. Szeftel]{Jeremie Szeftel}
\address{CNRS $\&$ Laboratoire Jacques Louis Lions, Sorbonne Universit\'e, Paris, France}
\email{jeremie.szeftel@upmc.fr}}

\begin{abstract} 
We consider the energy supercritical {\em defocusing} nonlinear Schr\"odinger equation $$i\pa_tu+\Delta u-u|u|^{p-1}=0$$ in dimension $d\ge 5$. In a suitable range of energy supercritical parameters $(d,p)$, we prove the existence of $\mathcal C^\infty$ well localized spherically symmetric initial data such that the corresponding unique strong solution blows up in finite time. Unlike other known blow up mechanisms, the singularity formation does not occur by concentration of a soliton or through a self similar solution, which are unknown in the defocusing case, but via a {\em front mechanism}. Blow up is achieved by {\em compression} in the associated  hydrodynamical flow which in turn produces a highly oscillatory singularity. The front blow up profile is chosen among the countable family of {\em $\mathcal C^\infty$} spherically symmetric self similar solutions to the compressible Euler equation whose existence and properties in a suitable range of parameters are established 
in the companion paper \cite{MRRSprofile}. 
\end{abstract}

\maketitle


\section{Introduction}


We consider the defocusing nonlinear Schr\"odinger equation
\be
\label{nls}
(\text{NLS}) \ \ \left|\begin{array}{ll}i\pa_tu+\Delta u-u|u|^{p-1}=0\\ u_{|t=0}=u_0
\end{array}\right., \ \ (t,x)\in [0,T)\times \Bbb R^d, \ \ u(t,x)\in \Bbb C.
\ee
in dimension $d\ge 3$ for an integer nonlinearity $p\in 2\Bbb N^*+1$ and address the problem of its global dynamics.
We begin by giving a quick introduction to the problem and its development.

\subsection{Cauchy theory and scaling} 
It is a very classical statement that smooth well localized initial data $u_0$ yield {\it local in time}, unique, smooth, strong solutions. For the global dynamics,
 two quantities conserved along the flow \eqref{nls} are of the utmost importance:
$${\textit {mass:}}\qquad  M(u)=\int_{\Bbb R^d}|u(t,x)|^2=\int_{\Bbb R^d} |u_0(x)|^2 $$ 
\be\label{energy}{\textit {energy:}} \qquad E(u)=\frac12\int_{\Bbb R^d}|\nabla u(t,x)|^{2}+\frac1{p+1}\int_{\Bbb R^d}|u(t,x)|^{p+1}dx=E(u_0).\ee The scaling symmetry group $$u_\l(t,x)=\l^{\frac 2{p-1}}u(\l^2t,\l x), \ \ \l>0$$ acts on the space of solutions by leaving the critical norm invariant $$\int_{\Bbb R^d}|\nabla^{s_c}u_\l(t,x)|^2=\int_{\Bbb R^d}|\nabla^{s_c}u(t,x)|^2\ \ \mbox{for}\ \ s_c=\frac d2-\frac{2}{p-1}.$$
Accordingly, the problem \eqref{nls} can be classified as energy subcritical, critical or supercritical depending on whether
the critical Sobolev exponent $s_c$ lies below, equal or above the energy exponent $s=1$. This classification also reflects the 
(in)/ability for the kinetic term in \eqref{energy} to control the potential one via the Sobolev embedding  $H^1\hookrightarrow L^q$.

{\subsection{Classification of the dynamics} We review the main known dynamical results which rely on the scaling classification.}\\

\noindent{\bf Energy subcritical case}. In the energy subcritical case $s_c<1$, the pioneering work of Ginibre-Velo \cite{GV} showed that for all $u_0\in H^1$, there exists a unique strong solution $u\in \mathcal C^0([0,T),H^1)$ to \eqref{nls} and identified the blow up criterion 
\be
\label{blowupcrtierion}
T<+\infty \Longrightarrow \lim_{t\uparrow  T}\|u(t)\|_{H^1}=+\infty.
\ee 
Conservation of energy, which is positive definite and thus controls the energy norm $H^1$, then immediately implies 
that the solution is global, $T=+\infty$. In fact, it can be shown in addition that these solutions scatter  as $t\to \pm\infty$, \cite{GV1}.\\

\noindent{\bf Energy critical problem}. In the energy critical case $s_c=1$, the criterion \eqref{blowupcrtierion} fails and the energy density could {\em concentrate}. For the data with a small critical norm, Strichartz estimates allow one to rule out such a scenario, \cite{CaW}. The large data critical problem has been an arena of an intensive and remarkable work in the 
last 20 years.

For large spherically symmetric data in dimensions $d=3,4$, the energy concentration mechanism was ruled out by Bourgain \cite{Bo} 
and Grillakis \cite{gril} via a localized Morawetz estimate. In Bourgain's work, a new {\it induction on energy} argument led to the statements of both the global existence and scattering. These results were extended to higher dimensions by Tao, \cite{Tao1}. 

The {\it interaction Morawetz estimate},
introduced in \cite{ckstt}, led to a breakthrough on the global existence and scattering for general solutions without symmetry, first in $d=3$, \cite{ckstt}, then in $d=4$, \cite{rvis}, 
and $d\ge 5$, \cite{vis}. 

{A new approach was introduced in Kenig-Merle  \cite{KM} in which, if there exists one global non-scattering solution, then using the {\it concentration compactness} profile decomposition \cite{BaGe,MeVe}, one extracts a minimal blow up solution and proves that up to renormalization, such a minimal element must behave like a soliton. The existence of such objects is ruled out using the defocusing nature of the nonlinearity, which is directly related to the non existence of solitons for defocusing models.}

In all of these large data arguments, the a priori bound on the critical norm provided by the conservation of energy played a fundamental role. Let us note that in the energy critical {\it focusing} setting, the concentration of the critical norm is known to be possible via type II (non self similar) blow up with soliton profile, see e.g \cite{KST,MRRnlsmap,RSch, RodSter,venoiveien,Perlem}.\\

\noindent{\bf Energy supercritical problem}. In the energy supercritical range $s_c>1$, {\it local in time} unique strong solutions can be constructed in the critical Sobolev space $H^{s_c}$, \cite{CaW,kvis}. Kenig-Merle's  approach, \cite{KMsuper,kvis}, gives a blow up criterion $$T<+\infty\Longrightarrow \limsup_{t\uparrow T}\|u(t,\cdot)\|_{H^{s_c}}=+\infty,$$  but the question of whether this actually happens for {\it any} solution remained completely open. On the other hand, the main difficulty in proving that $T=\infty$ for {\it all} solutions is that there are  {\em no a priori} bounds at the scaling level of regularity $H^{s_c}$.

 \subsection{Qualitative behavior for supercritical models} The question of global existence or blow up for energy supercritical models is a fundamental open problem in many nonlinear settings, both focusing and defocusing. For focusing problems, the existence of finite energy type I (self similar) blow up solutions is known in various instances, see e.g \cite{DS, MaMecpam, lepin, CRStypeI}, and solitons have been proved to be admissible blow up profiles in certain type II (non self-similar) blow up regimes in all three settings of heat, wave and Schr\"odinger equations, see e.g. \cite{HV, MRRnls,Charles,Mizo2,MatanoMe}. There are also several examples of supercritical problems with positive definite energy (wave maps, Yang-Mills) which admit smooth self-similar profiles and thus provide explicit blow up solutions, \cite{shatah,bizon,don}. \\
 
 On the other hand, for defocusing problems, soliton-like solutions are known not to exist and  admissible self similar solutions are expected not to exist.  For a simple defocusing model like the scalar nonlinear defocusing heat equation, a direct application of the maximum principle ensures that bounded data yield uniformly bounded solutions which are global in time and in fact dissipate. {We recall again that for the {\it energy critical} problems, blow up occurs in the focusing case, where solitons exist, and it does not in the defocusing case where solitons are known not to exist.}\\
 
{This collection of facts led to the belief, as explicitly conjectured by Bourgain in  \cite{bo}, 
 that global existence and scattering should hold for the energy supercritical defocusing Schr\"odinger and wave equations.} Indications of 
 various qualitative behaviors supporting different conclusions have been 
 provided (we give a highly incomplete list) in numerical
 simulations e.g. \cite{css,mz}, in model  problems showing blow up e.g. \cite{tao2,tao}, in examples of global solutions e.g.  \cite{KSClahg,bs}, in logarithmically supercritical problems e.g. \cite{tao3,struwe,colombo}, and in ill-posedness and norm inflation type results e.g. \cite{grenier,lebau,alazard,thoman}. 
 
 The behavior of solutions in other supercritical models such as the ones arising in fluid and gas dynamics is extremely 
 interesting and not yet well understood. We will not discuss it here.


\subsection{Statement of the result}


We assert that in dimensions $5\le d\le 9$ the defocusing (NLS) model \eqref{nls} admits finite time type II (non self similar) blow up solutions arising from $\mathcal C^\infty$ well localized initial data. {The singularity formation is based neither on soliton concentration nor self similar profiles, but on a new {\em front scenario} producing a {\em highly oscillatory} blow up profile.}

\begin{theorem}[Existence of energy supercritical type II defocusing blow up]
\label{thmmain}
Let 
\be
\label{assumptiondimension}
(d,p)\in\{(5,9),(6,5),(8,3),(9,3)\},
\ee
and let the critical blow up speed be 
\be
\label{vneoneenneo}
r^*(d,\ell)=\frac{\ell+d}{\ell+\sqrt{d}}, \ \ \ell=\frac{4}{p-1}.
\ee Then there exists a discrete sequence of blow up speeds $(r_k)_{k\ge 1}$ with $$2<r_k< r^*(d,\ell), \ \ \lim_{k\to+\infty} r_k=r^*(d,\ell)$$ such that any all $k\ge 1$,  there exists a finite co-dimensional manifold of smooth initial data $u_0\in \cap_{m\ge 0} H^m(\Bbb R^d,\Bbb C)$ with spherical symmetry such that the corresponding solution to \eqref{nls} blows up in finite time $0<T<+\infty$ at the center of symmetry with
\be
\label{vmempemempevp}\|u(t,\cdot)\|_{L^\infty}=\frac{c_{p,r,d}(1+o_{t\to T}(1))}{(T-t)^{\frac 1{p-1}\left(1+\frac{r_k-2}{r_k}\right)}}, \ \ c_{p,r,d}>0.
\ee
\end{theorem}

\noindent{\bf Comments on the result.}\\

\noindent{\em 1. Hydrodynamical formulation}. The heart of the proof of Theorem \ref{thmmain} is a study of \eqref{nls} 
in its hydrodynamical formulation, i.e. with respect to its phase and modulus variables. The key to our analysis is the identification 
of an underlying {\it compressible Euler dynamics}. The latter arises as a leading order approximation of a "front" like renormalization of the original equation. In this process, the Laplace term applied to the modulus\footnote{But not to the phase!} of the solution is treated {\em perturbatively} in the blow up regime. This is one of the key insights of the paper. The approximate 
Euler dynamics furnishes us with a self-similar solution, which requires very special properties and is constructed in the companion
paper \cite{MRRSprofile} and  which, in turn, acts as a blow up profile for the original equation. The existence of these blow up profiles is directly related to the restriction on the parameters \eqref{assumptiondimension} which we discuss in comment 3 below. Let us recall that there is a long history of trying to use the hydrodynamical variables in (NLS) problems and exploit a connection with fluid mechanics, going back to Madelung's original formulation of quantum mechanics in hydrodynamical  variables, \cite{ma}. Geometric optics and the hydrodynamical formulation were used to address  ill-posedness and norm inflation in the defocussing Schr\"odinger equations, \cite{grenier,alazard}. There is also a recent study of vortex filaments in \cite{BaVe} and its dynamical use of the Hasimoto transform. The scheme of proof of Theorem \ref{thmmain} will directly apply to produce the first complete description of singularity formation for the three dimensional compressible Navier-Stokes equation in the companion paper \cite{MRRSfluid}.\\

\noindent{\em 2. Blow up profile}. The blow up profile of Theorem \ref{thmmain} is more easily described in terms of the hydrodynamical variables: 
\be
\label{vneoenvneovnie}
u(t,x)=\rho_T(t,x)e^{i\phi(t,x)}.
\ee 
More precisely, we establish the decomposition 
\be
\label{nvivnineoneinov}
\left|\begin{array}{ll}\rho_T(t,x)=\frac{1}{(T-t)^{\frac{1}{p-1}\left(1+\frac{r-2}{r}\right)}}(\rho_P+\rho)(Z)\\\phi(t,x)=\frac{1}{(T-t)^{\frac{r-2}{r}}}(\Psi_P+\Psi)(Z)\end{array}\right., \ \ Z=\frac{x}{(T-t)^{\frac{1}{r}}}
\ee and prove the local asymptotic stability $$\lim_{t\to T}\|\Psi\|_{L^\infty(Z\le 1)}+\|\rho\|_{L^\infty(Z\le 1)}=0.$$ Here, the blow up profile $(\rho_P,\Psi_P)$ is, after a suitable transformation, picked among the family of spherically symmetric, smooth and decaying as $Z\to +\infty$ self-similar solutions to the compressible Euler equations. The interest in self-similar solutions for the equations of gas dynamics goes back to the pioneering works of Guderley \cite{guderley} and Sedov \cite{sedov} (and references therein,) who in particular considered converging motion of a compressible gas towards the center
of symmetry.  However, the rich amount of literature produced since then  is concerned with {\it non-smooth} self-similar 
solutions. This is partly due to the physical motivations, e.g. interests in solutions modeling implosion or detonation waves, where 
self-similar rarefaction or compression is followed by a shock wave (these are self-similar solutions which contain shock discontinuities 
already present in the data), and, partly due to the fact that, as it turns out, global solutions with the desired behavior at infinity and at the
center of symmetry are {\it generically} not $\mathcal C^\infty$. This appears to be a fundamental feature of the self-similar Euler dynamics
and, in the language of underlying acoustic geometry, means that {\it generically} such solutions are not smooth across the 
backward light (acoustic) cone with the vertex at the singularity. The key of our analysis is to find those non-generic $\matchal C^\infty$ solutions and to discover that this regularity is {\em an essential} element in controlling suitable {\it repulsivity} properties of the associated linearized operator. 
This is at the heart of the control of the full blow up. A novel contribution of the companion paper \cite{MRRSprofile} is the
 construction of $\mathcal C^\infty$ spherically symmetric self-similar solutions to the compressible Euler equations with suitable behavior at infinity and at the center of symmetry for {\em discrete values of the blow up speed parameter $r$  in the vicinity of the limiting blow up speed $r^*(d,\ell)$ given by \eqref{vneoneenneo}}.\\

\noindent{\em 3. Restriction on the parameters}. There is nothing specific with the choice of parameters \eqref{assumptiondimension}, and clearly the proof provides a full range of parameters. Two main constraints govern these restrictions. First of all, a {\em fundamental} restriction in order to make the Eulerian regime dominant is the constraint
\be
\label{neineigieieogig}
r^*(d,\ell)>2\Leftrightarrow \ell<\ell_2(d)=d-2\sqrt{d}
\ee which provides a non empty set of nonlinearities iff $$\ell_2(d)>0\Leftrightarrow d\ge 5.$$ As a result, the case of dimensions $d=3,4$ is not amenable to our analysis at this point, and the existence of blow up solutions for $d=3,4$ remains open. The second restriction concerns the existence of $\mathcal C^\infty$ smooth blow up profiles with suitable repulsivity properties of the associated linearized operator, as addressed in \cite{MRRSprofile}, see section \ref{bloeuoje} and remark \ref{renamrkaofjeof} for detailed statements. In particular, a non degeneracy condition $S_\infty(d,\ell)\neq 0$ for 
an explicit convergent series required. An elementary numerical computation is performed in \cite{MRRSprofile} to check the condition in the range \eqref{assumptiondimension}.\\

\noindent{\em 4. Behavior of Sobolev norms}. The conservation of mass and  energy imply a uniform $H^1$ bound on the solution. This can also be checked directly on the leading order representation formulas \eqref{vneoenvneovnie}, \eqref{nvivnineoneinov}. For higher Sobolev norms, a computation, see Appendix \ref{formal}, shows that the blow up solutions of Theorem \ref{thmmain} break scaling, i.e., we can find 
$$1<\sigma<s_c=\frac d2-\frac{2}{p-1}$$ such that 
$$\lim_{t\to T}\|u(t)\|_{H^{\sigma}}=+\infty,$$ and the critical Sobolev norm $\|u(t,\cdot)\|_{H^{s_c}}$ blows up polynomially.\\

\noindent{\em 5. Stability of blow up}. The blow up profiles of Theorem \ref{thmmain} have a finite number of instability directions, possibly none. Local asymptotic stability in the {\it interior of the backward light cone} (of the acoustical metric 
associated to the Euler profile) from the singularity relies on an abstract spectral argument for compact perturbations of maximal accretive operators. Related arguments have been used in the literature for the study of self-similar solutions both in focusing and defocusing regimes, for example \cite{BK, GW, MZduke, MRSheat, CMRcylindrique} for parabolic  and \cite{DS} for hyperbolic problems. The key to the control of the nonlinear flow 
in the {\it exterior} of the light cone is the propagation of  certain {\em weighted scale invariant} norms. This generalizes a Lyapunov functional based approach developed in \cite{MRRnls}. Counting the precise number of instability directions  is an independent problem, disconnected to the nonlinear analysis of the blow up, and remains to be addressed.\\

\noindent{\em 6. {Oscillatory behavior}}. The constructed solutions are smooth at the blow up time 
away from $x=0$:
\be
\label{neneioneonoennev}
\forall R>0, \ \ \lim_{t\to T}u(t,x)= u^*(x) \ \ \mbox{in}\ \ H^k(|x|>R), \ \ k\in \Bbb N.
\ee 
As in the cases for blow up problems in the focusing setting, see e.g. \cite{MRCMP}, the profile outside the blow up point has a {\em universal} behavior when approaching the singularity
\be
\label{cneineneonoev}
u^*(x)=c_P(1+o_{|x|\to 0}(1))\frac{e^{i\frac{c_\Psi}{|x|^{r-2}}}}{|x|^{\frac{2(r-1)}{p-1}}}, \ \ c_P\ne 0.
\ee 
What is unusual, and together with potential non-genericity perhaps responsible for difficulties in numerical detection of the blow up phenomena, is the {\em highly oscillatory} behavior.  This appears to be a deep consequence of the structure of the self-similar solution to the compressible Euler equation and the coupling of phase and modulus variable in the blow up regime, 
generating an {\em anomalous Euler scaling}. {The heart of our analysis is to show that after passing to the suitable renormalized variables provided by the front, the highly oscillatory behavior \eqref{cneineneonoev} becomes regular near the singularity and can be controlled with the {\em monotonicity} estimates of energy type, without appealing to Fourier analysis.}\\

\noindent{\em 7. Type I blow up}. The existence of self similar solutions to the defocusing energy supercritical (NLS) decaying at infinity is an open problem. Such solutions are easily ruled out for the heat equation using the maximum principle, and we refer to \cite{KSClahg} for further discussion in the case of the wave equation.\\

The paper is organized as follows. In section \ref{sectiontwo}, we present the ``front'' renormalization of the flow which makes the Euler dynamics dominant, and recall all necessary facts about the corresponding self similar profile built in \cite{MRRSprofile}. Theorem \ref{thmmain} reduces to building a global in time non vanishing solution to the renormalized flow \eqref{exactliearizedflow} written in hydrodynamical variables. In section \ref{stragegyoftheproof} we detail the strategy of the proof. In section \ref{sectionlinear}, we introduce the functional setting related to maximal accretivity (modulo a compact perturbation) of the corresponding linear operator which leads to a statement of exponential decay in a neighborhood of the light cone  for the space of solutions (modulo an a priori control of a finite dimensional manifold corresponding to the unstable directions.) In section \ref{sectionbootstrap}, we describe our set of initial data and the set of bootstrap assumptions 
which govern the analysis. In sections \ref{sectionsobolev}, \ref{sec:point}, \ref{sec:high}, we close the control of weighted Sobolev norms and the associated pointwise bounds. In section \ref{low}, we close the exponential decay of low Sobolev norms by relying on spectral estimates and finite speed of propagation arguments. 

\subsection*{Acknowledgements}  The authors would like  to thank N. Burq (Orsay), V. Georgescu (Cergy-Pontoise) and 
L. Vega (Bilbao) for stimulating discussions at the early stages of this work.\\

P.R. is supported by the ERC-2014-CoG 646650 SingWave. P.R would like to thank the Universit\'e de la C\^ote d'Azur where part of this work was done for its kind hospitality. I.R. is partially supported by the NSF 
grant DMS \#1709270 and a Simons Investigator Award. J.S  is supported by the ERC grant  ERC-2016 
CoG 725589 EPGR. 


\subsection*{Notations} 


The bracket $$\la r\ra=\sqrt{1+r^2}.$$ The weighted scalar product for a given measure $g$:
\be
\label{scalarltwo}
(u,v)_g=\int_{\Bbb R^d} u\overline{v}g dx.
\ee
The integer part of $x\in \Bbb R$
$$x\leq [x]<x+1, \ \ [x]\in \Bbb Z.$$
The infinitesimal generator of dilations $$\Lambda =y\cdot\nabla.$$


\section{Front renormalization, blow up profile and strategy of the proof}
\label{sectiontwo}


In this section we introduce the hydrodynamical variables to study \eqref{nls} and the associated renormalization procedure which makes the compressible Euler structure dominant. We collect from \cite{MRRSprofile} the main facts about the existence of smooth spherically symmetric self-similar solutions to the compressible Euler equations which will serve as blow up profiles.


\subsection{Hydrodynamical formulation and front renormalization}


We renormalize the flow and, for {\em non vanishing solutions}, write  the equivalent hydrodynamical formulation in phase and modulus variables.\\
We begin with the standard {\em self-similar} renormalization  $$u(t,x)=\frac{1}{\l(t)^{\frac 2{p-1}}}v(s,y)e^{i\gamma}, \ \ y=\frac{x}{\l}$$ where we freeze the scaling parameter at the self-similar scale
$$\frac{d\tau}{dt}=\frac1{\l^2}, \ \ y=\frac x{\l(t)}, \ \ -\frac{\l_\tau}{\l}=\frac 12,$$ then \eqref{nls} becomes 
\be
\label{nbeivbeibei}
 i\pa_\tau v+\Delta v-\gamma_\tau v-i\frac{\l_\tau}{\l}\left(\frac 2{p-1}v+ \Lambda v\right)-v|v|^{p-1}=0.
 \ee 
In the defocusing case, \eqref{nbeivbeibei}  has no obvious type I self similar stationary solution, or type II soliton like solutions, \cite{MRRnls}, but, it turns out, that it admits approximate {\em front like} solutions. Their existence relies on a 
specific phase and modulus coupling and {\it anomalous} scaling. We introduce the parameters
\be
\label{vneionveoneonv}
\left|\begin{array}{l}
r=\frac{2}{1-\mathcal e}, \ \ 0<\mathcal e <1\\
\mu=\frac{1}{r}=\frac{1-\mathcal e}2\\
\ell=\frac{4}{p-1}
\end{array}\right.
\ee and claim:

\begin{lemma}[Front renormalization of the self similar flow]
\label{newequationlemma}
Define geometric parameters
\be
\label{defparameters}
-\frac{\l_\tau}{\l}=\frac 12, \ \ \frac{b_\tau}{b}=-\mathcal e, \ \ \gamma_\tau=-\frac 1b, \ \ \frac{d\tau}{dt}=\frac{1}{\l^2}
\ee
and introduce the renormalization
$$u(t,x)=\frac{1}{\l(t)^{\frac 2{p-1}}}v(s,y)e^{i\gamma}, \ \ y=\frac{x}{\l}$$ with the phase and modulus 
$$
\left|\begin{array}{l}
v=we^{i\phi}\\
 w(\tau,y)=\frac{1}{(\sqrt{b})^{\frac 2{p-1}}}\rho_T(\tau,Z)\in \Bbb R^*_+\\
 \phi(\tau,y)=\frac{1}{b}\Psi_T(\tau,Z)\\
 Z=y\sqrt{b}
 \end{array}\right.
 $$
In these variables \eqref{nls} becomes, on $[\tau_0,+\infty)$:
\be
\label{newequation}
\left|\begin{array}{ll}\pa_\tau \rho_T=-\rho_T\Delta \Psi_T-\frac{\mu\ell(r-1)}{2}\rho_T-\left(2\pa_Z\Psi_T+\mu Z\right)\pa_Z\rho_T\\
\rho_T\pa_\tau \Psi_T=b^2\Delta \rho_T-\left[|\nabla \Psi_T|^2+\mu(r-2)\Psi_T-1+\mu\Lambda \Psi_T+\rho_T^{p-1}\right]\rho_T.
\end{array}\right.
\ee
\end{lemma}
 
\begin{proof} Starting from \eqref{nbeivbeibei}, we define a polar decomposition 
$$v=we^{i\phi}$$ so that
$$v'=(w'+i\phi'w)e^{i\phi}, \ \ v''=w''-|\phi'|^2w+2i\phi'w'+i\phi'' w$$
and 
\bea
\label{equationw}
 0&=&i\pa_\tau w+\Delta w+\left(-\pa_\tau\phi-|\nabla\phi|^2-\gamma_\tau+\frac{\lambda_\tau}{\lambda}y\cdot\nabla \phi \right)w\\
\nonumber & +& i\left(\Delta\phi-\frac{2}{p-1}\frac{\l_\tau}{\l}\right)w+i\left(2\nabla \phi-\frac{\l_\tau}{\l} y\right)\cdot\nabla w-w|w|^{p-1}.
\eea
Separating the real and imaginary parts yields the self-similar equations \eqref{nbeivbeibei}:
\be
\label{selfsimphasemodulus}
\left|\begin{array}{ll}
\pa_\tau w=-\left(\Delta\phi+\frac{1}{p-1}\right)w-\left(2\frac{\pa_y\phi}{y}+\frac 12\right)\Lambda w\\
w\pa_\tau \phi=\Delta w+\left(-|\nabla\phi|^2-\gamma_\tau-\frac 12\Lambda \phi\right)w -w|w|^{p-1}
\end{array}\right.
\ee
We now renormalize according to
$$w(\tau,y)=\frac{1}{(\sqrt{b})^{\frac 2{p-1}}}\rho_T(\tau,Z)\in \Bbb R^*_+, \ \ \phi(\tau,y)=\frac{1}{b}\Psi_T(\tau,Z) \ \ Z=y\sqrt{b}$$
with a {\it fixed} choice of parameters in the modulation equations
$$
\frac{b_\tau}{b}=-{\mathcal e}, \ \ \gamma_\tau=-\frac{1}{b}, \ \ 0<{\mathcal e}<1
$$
which transforms \eqref{selfsimphasemodulus} into 
$$
\left|\begin{array}{ll}\pa_\tau \rho_T=-\rho_T\Delta \Psi_T-\frac{{\mathcal e}+1}{p-1}\rho_T-\left(2\pa_Z\Psi_T+\frac{1-{\mathcal e}}{2}Z\right)\pa_Z\rho_T\\
\rho_T\pa_\tau \Psi_T=b^2\Delta \rho_T-\left[|\nabla \Psi_T|^2+{\mathcal e}\Psi_T-1+\frac 12(1-{\mathcal e})\Lambda \Psi_T+\rho_T^{p-1}\right]\rho_T
\end{array}\right.
$$
We now compute from \eqref{vneionveoneonv}:
$$\left|\begin{array}{l}
\frac{\mu\ell(r-1)}{2}=\frac{2}{p-1}(1-\mu)=\frac{1+{\mathcal e}}{p-1}\\
\mu(r-2)=1-(1-{\mathcal e})={\mathcal e}
\end{array}\right.
$$ and \eqref{newequation} is proved.
\end{proof}


\subsection{Blow up profile and Emden transform}
\label{bloeuoje}


A stationary solution $(\rho_P,\Psi_P)$ to \eqref{newequation} in the limiting Eulerian regime $b=0$ satisfies the profile equation
\be
\label{profileequation}
\left|\begin{array}{ll}
|\nabla \Psi_P|^2+\rho_P^{p-1}+\mu(r-2)\Psi_P+\mu\Lambda \Psi_P=1\\
\Delta \Psi_P+\frac{\mu\ell(r-1)}{2}+\left(2\pa_Z\Psi_P+\mu Z\right)\frac{\pa_Z\rho_P}{\rho_P}=0
\end{array}\right.
\ee
{ We supplement it with the boundary conditions:}
\be
\left|\begin{array}{ll}
\label{boundarydata}
\rho_P(0)=1, \ \ \Psi_P(0)=0,\\ \rho_P(x)\to 0,\ \Psi_P(x)\to \frac 1{\mathcal e} \ \ {\text as}\ \  x\to \infty
\end{array}\right.
\ee
We now show that the system \eqref{profileequation}, \eqref{boundarydata} is equivalent to the corresponding system of equations
describing self-similar solutions of the Euler equations. We define the Emden variables:
 \be
 \label{relationsprofileemden}
\left|\begin{array}{lll}
 \phi=\frac{\mu}{2}\sqrt{\ell}, \ \ p-1=\frac{4}{\ell}\\
 Q=\rho_P^{p-1}=\frac{1}{M^2}, \ \ \frac{1}{M}=\phi Z \sigma\\
 \frac{\Psi'_P}{Z}=-\frac{\mu}{2}w
 \end{array}\right., \ \ x=\log Z,
\ee
then \eqref{profileequation} is mapped onto 
\be
\label{systemedefoc}
\left|\begin{array}{ll}
(w-1)w'+\ell \sigma\sigma'+(w^2-rw+\ell \sigma^2)=0\\
\frac{\sigma}{\ell}w'+(w-1)\sigma'+\sigma\left[w\left(\frac{d}{\ell}+1\right)-r\right]=0
\end{array}
\right.
\ee
or equivalently $$\left|\begin{array}{ll}a_1w' +b_1\sigma'+d_1=0\\
a_2 w'+b_2\sigma'+d_2=0
\end{array}\right.$$
with 
\be
\label{defvaluesboinedone}
\left|\begin{array}{ll}
a_1=w-1, \ \ b_1=\ell\sigma, \ \ d_1=w^2-rw+\ell\sigma^2\\
a_2=\frac{\sigma}{\ell}, \ \ b_2=w-1, \ \ d_2=\sigma\left[\left(1+\frac{d}{\ell}\right)w-r\right].
\end{array}\right.
\ee
The system \eqref{systemedefoc} is exactly the one describing spherically symmetric self-similar solutions to the compressible Euler equation, \cite{sedov} (and the references therein). For an explicit derivation see Appendix A. It is analyzed in \cite{MRRSprofile}, following pioneering work of Guderley, Sedov and others.

Let 
\be
\label{definitionwe}
w_e=\frac{\ell(r-1)}{d}
\ee
and the determinants
\be
\label{calculdeterm}
\left|\begin{array}{lll}
\Delta=a_1b_2-b_1a_2=(w-1)^2-\sigma^2\\
\Delta_1=-b_1d_2+b_2d_1=w(w-1)(w-r)-d(w-w_e)\sigma^2\\
\Delta_2=d_2a_1-d_1a_2=\frac{\sigma}{\ell}\left[(\ell+d-1)w^2-w(\ell+d+\ell r-r)+\ell r-\ell \sigma^2\right]
\end{array}\right.
\ee
then
\be\label{ws}
w'=-\frac{\Delta_1}{\Delta}, \ \ \sigma'=-\frac{\Delta_2}{\Delta}, \ \ 
\frac{dw}{d\sigma}=\frac{\Delta_1}{\Delta_2}.
\ee
Solution curves $w=w(\sigma)$  of the above system can be examined through its {\it phase portrait} in the $(\sigma,w)$ plane. The shape of the phase portrait depends {\em crucially} on the polynomials $\Delta$,  $\Delta_{1}$, $\Delta_2$ and the parameters $(r,d,\ell)$. It is not hard to see that there is a unique solution with the normalization 
\be
\label{normlazationprofile}
\rho_P(0)=1, \ \ \Psi_P(0)=0,
\ee
at $x=0$, which is also ${\mathcal C}^\infty$ in the vicinity of $x=0$,
but the heart of the matter is the global behavior of this unique solution. 

In particular, any such solution with the required asymptotics as $x\to+\infty$ needs to pass through the point $P_2$ which 
lies on the so called {\it sonic line\footnote{Any point $Z_0$ on the sonic line corresponds to the acoustic cone described the 
equation $(\tau, Z=Z_0)$ for the acoustic metric defined by the profile passing through $Z_0$.}}: $\Delta=0$ but where also
\be
\label{neovneneoneo}
\Delta_1(P_2)=\Delta_2(P_2)
\ee 
(there are potentially two such points).  
It turns out that at $P_2$ the solution experiences an unavoidable discontinuity of high derivatives, {\em except for discrete values of the speed $r$}. The following structural proposition on the blow up profile is proved in the companion paper \cite{MRRSprofile}.

\begin{theorem}[Existence and asymptotics of a $\matchal C^\infty$ profile, \cite{MRRSprofile}]
\label{propexistenceprofile}
Let $$(d,p)\in\{(5,9),(6,5),(8,3),(9,3)\}
$$
and recall \eqref{vneoneenneo}.
Then there exists a sequence $(r_k)_{k\ge 1}$ with 
\be
\label{limimtirmt}
\lim_{k\to \infty} r_k=r^*(d,\ell), \ \ r_k<r^*(d,\ell)
\ee 
such that for all $k\ge 1$, the following holds:\\
\noindent{\em 1. Existence of a smooth profile at the origin}: the unique radially symmetric solution to \eqref{profileequation} with Cauchy data at the origin \eqref{boundarydata} reaches in finite time $Z_2>0$ the point $P_2$.\\
\noindent{\em 2. Passing through $P_2$}: the solution passes through $P_2$ with $\mathcal C^\infty$ regularity.\\
\noindent{\em 3. Large $Z$ asymptotic}: the solution admits the asymptotics as $Z\to +\infty$:
\be
\label{limitprofilesbsi}
\left|\begin{array}{l}
w(Z)=\frac{c_w}{Z^r}\left(1+O\left(\frac{1}{Z^{ r}}\right)\right)\\
\sigma(Z)=\frac{c_\sigma}{Z^r}\left(1+O\left(\frac{1}{Z^{r}}\right)\right)\\
\end{array}\right.
\ee
or equivalently
\be
\label{decayprofile}
\left|\begin{array}{ll}
Q(Z)=\rho_P^{p-1}(Z)=\frac{c_P^{p-1}}{Z^{2(r-1)}}\left(1+O\left(\frac{1}{Z^r}\right)\right), \\ \Psi_P(Z)=\frac 1{\mathcal e}+\frac{c_\Psi}{Z^{r-2}}\left(1+O\left(\frac{1}{Z^r}\right)\right)\end{array}\right.
\ee 
with non zero constants $c_\sigma,c_P$. Similar asymptotics hold for all higher order derivatives.\\
\noindent{\em 4. Non vanishing}:  there holds $$\forall Z\ge 0, \ \ \rho_P>0.$$
\noindent{\em 5. Repulsivity inside the light cone}: let 
\be
\label{definitionF}
F=\sigma_P+\Lambda\sigma_P,
\ee 
then there exists $c=c(d,\ell,r)>0$ such that
\be
\label{coercivityquadrcouplinginside}
\forall 0\le Z\le Z_2, \ \ \left|\begin{array}{l}
(1-w-\Lambda w)^2-F^2>c\\
1-w- \Lambda  w-\frac{(1-w)F}{\sigma}> c.
\end{array}\right.
\ee
\noindent{\em 6. Repulsivity outside the light cone}: \be
\label{P}
\exists c=c_{d,\ell,r}>0, \ \ \forall Z\ge Z_2, \ \ \left|\begin{array}{l} (1-w-\Lambda w)^2-F^2>c\\ 1-w-\Lambda w>c
\end{array}\right.
\ee
\end{theorem}

\begin{remark}[Restriction on the parameters]
\label{renamrkaofjeof}
The proof of Theorem \ref{propexistenceprofile} requires the non degeneracy of an explicit series $S_\infty(d,\ell)\neq 0$ which is {\em numerically} checked in \cite{MRRSprofile} in the range \eqref{assumptiondimension}. The positivity properties \eqref{coercivityquadrcouplinginside}, \eqref{P} are checked analytically in \cite{MRRSprofile} and will be fundamental for the well-posedness of the linearized flow {\em inside the light cone}, and the control of global Sobolev norms outside the light cone. Let us insist that the restriction on parameters relies on the intersection of the conditions \eqref{neineigieieogig}, $S_\infty(d,\ell)\ne 0$ and \eqref{coercivityquadrcouplinginside}, \eqref{P} hold. The range \eqref{assumptiondimension} is just an example where this holds, but a larger range of parameters can be directly extracted from \cite{MRRSprofile}, and the conclusion of Theorem \ref{thmmain} would follow. In particular and since we are working with non vanishing solutions, the fact that the non linearity is an odd integer can be relaxed as in \cite{MRRSfluid}.
\end{remark}

From now on and for the rest of this paper, we assume \eqref{assumptiondimension}. We observe from direct check that there holds: 
$$r^*(\ell)=\frac{d+\ell}{\ell+\sqrt{d}}>2\Leftrightarrow \ell<d-2\sqrt{d}=\ell_2(d).$$ Recalling \eqref{vneionveoneonv}, we may therefore assume from \eqref{limimtirmt} that the blow speed $r=r_k$ satisfies  $$r>2\Leftrightarrow {\mathcal e}=\frac{r-2}{r}>0.$$


\subsection{Linearization of the renormalized flow}


We look for $u$ solution to \eqref{nls} and proceed to the decomposition of Lemma \ref{newequationlemma}. We are left with finding a global, in self similar time $\tau\in[\tau_0,+\infty)$, solution to \eqref{newequation}:
\be
\label{fullflowrenormalized}
\left|\begin{array}{ll}\pa_\tau \rho_T=-\rho_T\Delta \Psi_T-\frac{\mu\ell(r-1)}{2}\rho_T-\left(2\pa_Z\Psi_T+\mu Z\right)\pa_Z\rho_T\\
\rho_T\pa_\tau \Psi_T=b^2\Delta \rho_T-\left[|\nabla \Psi_T|^2+\mu(r-2)\Psi_T-1+\mu\Lambda \Psi_T+\rho_T^{p-1}\right]\rho_T
\end{array}\right.
\ee
with non vanishing density $\rho_T>0$. We define 
\be
\label{defhtwohunbis}
\left|\begin{array}{ll}
H_2=\mu+2\frac{\Psipx'}{Z}=\mu(1-w)\\
H_1=-\left(\Delta \Psipx+\frac{\mu\ell(r-1)}2\right)=H_2\frac{\Lambda \rhopx}{\rhopx}=\frac{\mu\ell}{2}(1-w)\left[1+\frac{\Lambda \sigma}{\sigma}\right]
\end{array}\right.
\ee

We linearize $$\rho_T=\rho_P+\rho, \ \ \Psi_T=\Psi_P+\Psi$$ and compute, using the profile equation \eqref{profileequation}, for the first equation:
\bee
\pa_\tau \rho&=&-(\rho_P+\rho)\Delta(\Psi_P+\Psi)-\frac{\mu\ell(r-1)}{2}(\rho_P+\rho)-(2\pa_Z\Psi_P+\mu Z+2\pa_Z\Psi)(\pa_Z\rho_P+\pa_Z\rho)\\
& = & -\rho_T\Delta \Psi-2\nabla\rho_T\cdot\nabla \Psi+H_1\rho-H_2\Lambda \rho
\eee
and for the second one:
\bee
&&\rho_T\pa_\tau \Psi=b^2\Delta \rho_T-\rho_T\Big\{|\nabla \Psi_P|^2+2\nabla \Psi_P\cdot\nabla \Psi+|\nabla \Psi|^2\\
&-& 1+\mu(r-2)\Psi_P+\mu(r-2)\Psi+\mu(\Lambda\Psi_P+\Lambda\Psi)+(\rho_P+\rho)^{p-1}\Big\}\\
& = & b^2\Delta \rho_T-\rho_T\left\{2\nabla \Psi_P\cdot\nabla \Psi+\mu\Lambda \Psi+\mu(r-2)\Psi+|\nabla\Psi|^2+(\rho_P+\rho)^{p-1}-\rho_P^{p-1}\right\}\\
& = & b^2\Delta \rho_T-\rho_T\left\{H_2\Lambda \Psi+\mu(r-2)\Psi+|\nabla \Psi|^2+(p-1)\rho_P^{p-2}\rho +\NL(\rho)\right\}
\eee
with $$\NL(\rho)=(\rho_P+\rho)^{p-1}-\rho_P^{p-1}-(p-1)\rho_P^{p-2}\rho.$$
We arrive at the exact (nonlinear) linearized flow
\be
\label{exactliearizedflow}
\left|\begin{array}{ll} \pa_\tau \rho=H_1\rho-H_2\Lambda \rho-\rho_T\Delta \Psi-2\nabla\rho_T\cdot\nabla \Psi\\
\pa_\tau \Psi=b^2\frac{\Delta \rho_T}{\rho_T}-\left\{H_2\Lambda \Psi+\mu(r-2)\Psi+|\nabla \Psi|^2+(p-1)\rho_P^{p-2}\rho +\NL(\rho)\right\}.
\end{array}\right.
\ee
Theorem \ref{thmmain} is therefore equivalent to exhibiting a finite co-dimensional manifold of smooth well localized initial data leading to global, in renormalized $\tau$-time, solutions to \eqref{exactliearizedflow}.


\subsection{Strategy of the proof}
\label{stragegyoftheproof}


We now explain the strategy of the proof of Theorem \ref{thmmain}.\\

\noindent{\bf step 1} Wave equation and propagator estimate. After the change of variables $\Phi=\rho_P\Psi$, we may schematically rewrite the linearized flow \eqref{exactliearizedflow} in the form 
\be
\label{nkenvenvnnvoenn}
\pa_\tau X=\mathcal M X+\NL(X)-b^2\left|\begin{array}{ll}0\\\Delta (\rho_P+\rho)\end{array}\right.
\ee 
with 
\be
\label{neonvenoeno}
X= \left|\begin{array}{ll}\rho\\ \Phi\end{array}\right., \ \ \mathcal M=\left(\begin{array}{ll}H_1-H_2\Lambda\ & -\Delta +H_3\\-(p-1)Q-H_2\Lambda&H_1-\mu(r-2)\end{array}\right)
\ee where $Q,H_1,H_2,H_3$ are explicit potentials generated by the profile $\rho_P,\Psi_p$. During the first step the 
$b^2\Delta$ term is treated perturbatively. We commute the equation with the powers of the laplacian $\Delta^k$ and obtain for $X_k=\Delta^kX$  
\be
\label{nevneenonone}
\pa_\tau X_k=\mathcal M_k X+\NL_k(X).
\ee We then show that, provided $k$ is large enough, $\mathcal M_k$ is a {\rm finite rank} perturbation of a {\it maximally dissipative} operator with a spectral gap $\delta>0$. The topology in which maximal accretivity is established depends  
on the properties of the wave equation\footnote{Reminiscent of the wave equation arising in a linearization of the compressible Euler equations.} encoded in \eqref{nevneenonone} and is based on weighted Sobolev norms with weights vanishing on 
the light cone corresponding to the point $P_2$ of the profile. Indeed, the principal part of the wave equation is roughly of the form $$\pa_\tau^2\rho-D(Z)\pa_Z^2\rho$$ where the weight $D(Z)$ vanishes on the light cone $Z=Z_2$ corresponding to the ${P_2}$ point. The corresponding propagation estimates for the wave equation produce  an priori control of the solution in the 
interior of the light cone $Z<Z_2$, modulo an a priori control of a finite number of directions corresponding to non positive eigenvalues of $\mathcal M_k$. An essential structural fact of this step is the $\mathcal C^\infty$ regularity of the profile. Indeed, we claim that for a generic non $\matchal C^\infty$ solution at $P_2$, the number of derivatives required to show accretivity of the linearized operator is {\em always} strictly greater than the regularity of the profile at ${P_2}$. As a result such profiles may be completely unstable and are not amenable to our analysis. The $\mathcal C^\infty$ regularity obtained in \cite{MRRSprofile} is therefore absolutely fundamental. The analytic properties leading to the  maximality of the linearized operator will be consequences of \eqref{coercivityquadrcouplinginside}, \eqref{P}. We note that the coercivity constant in \eqref{coercivityquadrcouplinginside} {\em degenerates} as $r\to r^*$, and the number of derivatives needed for accretivity is inversely proportional to this constant. This is a manifestation of a completely new nonlinear effect: the problem sees a scaling which depends on the chosen self similar profile.\\

\noindent{\bf step 2} Extension slightly beyond the light cone. Exponential decay estimates provided in the first step yield control in the interior of the light cone $Z< Z_2$ only. It turns out that the analysis of the first step can be made more robust and 
extended\footnote{Reminiscent of a non-characteristic energy estimate.}  
slightly beyond the light cone, all the way to a spacelike hypersurface $Z=Z_2+a$, $0<a\ll 1$, even though it is complicated by
the dependence of the underlying wave equation on {\it variable} coefficients or, equivalently, on non constancy of the $Q(Z)$ term in \eqref{neonvenoeno}. We can revisit the first step by producing a new maximal accretivity structure for a norm which does not generate in the zone $Z<Z_2+a$, $0<a\ll 1$. The argument relies on a new generalized monotonicity formula. The corresponding propagation estimates recovers exponential decay in the extended zone $Z<Z_2+a$. Once decay has been obtained {\em strictly beyond} the light cone, a simple finite speed of propagation argument allows us to propagate decay to any compact set $Z<Z_0$, $Z_0\gg1$.\\

\noindent{\bf step 3} Loss of derivatives. The decay obtained in step 2 relies on energy estimates compatible with the wave 
propagation and the Eulerian structure of approximation.  The full evolution however is that of the Schr\"odinger equation 
and contains the $b^2\Delta$ term on the right hand side of \eqref{nkenvenvnnvoenn}. Such a term leads to an unavoidable loss of one derivative. However, this loss comes with a $b^2$ smallness in front. We then argue as follows. We  
pick a large enough regularity level $k_m=k_m(r,d)\gg 2k_0$, where $k_0$ is the power of the laplacian used for commutation 
in step 2,
 and derive a {\it global} Schr\"odinger like energy identity on the {\em full} flow \eqref{exactliearizedflow}. 
 The choice of phase and modulus as basic variables turns the equation quasilinear and 
 makes this identity rather complicated and unfamiliar.
 An essential difficulty, which is deeply related to step 2, is that at the highest level of derivatives, the non trivial space dependence of the profile measured by $Q(Z)=\rho_P^{p-1}(Z)$ in \eqref{neonvenoeno} produces a coupling term and a non trivial quadratic form. The condition \eqref{P} implies that the corresponding quadratic form is definite positive for $k_m$ large enough.\\
 
 \noindent{\bf step 4} Closing estimates. As explained above, we work with a linearized nonlinear equation, i.e., obtained after
 subtracting off the profile, written in terms of the phase and modulus unknowns $(\Psi, \rho)$, in renormalized self-similar 
 variables $(\tau, Z)$, where the singularity corresponds to $(\tau=\infty, Z=0)$, a special light cone is $(\tau, Z=Z_2)$ and
 where in the original variables $(t,r)$ the region $r\ge 1$ corresponds to $Z\ge e^{\mu\tau}$.     
 
 First, outside the singularity $r\ge 1$, we modify the profile by strengthening its decay to make it rapidly decaying and of  finite energy. 
 Relative to the self-similar variables this modification happens at $Z\sim e^{\mu\tau}$, far from the singularity, and as a result is harmless. Then, we run two sets of estimates. First, we employ
 wave propagation like estimates which go initially just slightly beyond the special light cone and then extend 
 to any compact set
 in $Z$. These estimates are carried out at a sufficiently high level of regularity with $\sim 2k_0$ derivatives. The number 
 $k_0$ emerges from the linear theory and is determined by the (conditional) positivity of a certain quadratic form responsible 
 for maximal accretivity. 
 
 Then, we couple these estimates to global Schr\"odinger like estimates which take into account previously ignored $b^2\Delta$ and take care of global control. These estimates are carried out at all levels of regularity up to $k_m$ derivatives with $k_m\gg k_0$. They are carefully designed weighted $L^2$ type estimates. The weights depend on the number
 of derivatives $k$: at first, their strength grows with $k$ but by the time we reach the highest level of regularity $k_m$ the 
 weight function is identically $=1$. The latter has to do with a well-known fact that even for a linear Schr\"odinger equation
 estimates, use of weights leads to a derivative loss  
 ($\Delta$ is not self-adjoint on a weighted $L^2$ space.) Therefore, our
 highest derivative norm should correspond to an {\it unweighted} $L^2$ estimate. Of course, this last estimate also sees a 
 positivity condition \eqref{P} responsible for the coercivity of an appearing quadratic form.
 
 These global weighted $L^2$ bounds then allow us to prove pointwise bounds for the solution and its derivatives which,
 in turn, allow us to control nonlinear terms.  The obtained sets of weighted $L^\infty$ bounds on derivatives recover in particular the {\it non vanishing} assumption required of the solution.
 We should note that while all the local (in $Z$) norms decay exponentially in 
 $\tau$, the global norms are merely bounded. In the original $(t,r)$ variables this means that the perturbation decays inside
 and slightly beyond the backward light cone from the singular point but does not decay away from the singularity. This is, 
 of course, entirely consistent with the global conservation of energy for NLS. 
\\
 
 The whole proof proceeds  via a bootstrap argument which also involves a Brouwer type argument to deal with unstable modes, if any, arising in linear theory of step 1. This is what produces a finite co-dimension manifold of admissible data.


\section{Linear theory slightly beyond the light cone}
\label{sectionlinear}

Our aim in this section is to study the linearized problem \eqref{exactliearizedflow} for the exact Euler problem $b=0$. We in particular aim at setting up the suitable functional framework in order to apply classical propagator estimates which will yield exponential decay on compact sets $Z\lesssim 1$ modulo the control of a finite number of unstable directions.


\subsection{Linearized equations}


Recall the exact linearized flow \eqref{exactliearizedflow} which we rewrite:
$$
\left|\begin{array}{ll} \pa_\tau \rho=H_1\rho-H_2\Lambda \rho-\rho_P\Delta \Psi-2\nabla\rho_P\cdot\nabla \Psi-\rho\Delta \Psi-2\nabla \rho\cdot\nabla \Psi\\
\pa_\tau \Psi=b^2\frac{\Delta \rho_T}{\rho_T}-\left\{H_2\Lambda \Psi+\mu(r-2)\Psi+(p-1)\rho_P^{p-2}\rho +|\nabla \Psi|^2+\NL(\rho)\right\}\end{array}\right.
$$
We introduce the new unknown 
\be
\label{defnewvariablephi}
\Phix=\rhopx \Psi
\ee and obtain equivalently using \eqref{defhtwohunbis}:
\be
\label{nekoneneon}
\left|\begin{array}{ll} \pa_\tau \rhox=H_1\rhox-H_2\Lambda \rhox-\Delta \Phix+H_3\Phix+G_\rho\\
\pa_\tau \Phix=-(p-1)\qx\rhox-H_2\Lambda \Phix+(H_1-\mu(r-2))\Phix+G_\Phi
 \end{array}\right.
\ee
with 
\bea
\qx=\rhopx^{p-1}, \ \ H_3=\frac{\Delta\rhopx}{\rhopx}
\eea 
and the nonlinear terms:
\be
\label{defgrho}
\left|\begin{array}{ll}G_\rho=-\rho\Delta \Psi-2\nabla\rho\cdot\nabla \Psix\\
G_\Phi=-\rho_P(|\nabla \Psi|^2+\NL(\rho))+\frac{b^2\rho_P}{\rho_T}\Delta \rho_T.
\end{array}\right.
\ee
 We transform \eqref{nekoneneon} into a wave equation for $\Phi$ and compute:
  \bee
&&\pa_\tau^2\Phix=-(p-1)\qx(H_1\rhox-H_2\Lambda \rhox-\Delta \Phix+H_3\Phix+G_\rho)+\pa_\tau G_\Phi\\
& & -H_2\Lambda\pa_\tau\Phix+  (H_1-\mu(r-2))\pa_\tau\Phix\\
& = & -(p-1)\qx(H_1\rhox-\Delta \Phix+H_3\Phix) -H_2\Lambda\pa_\tau\Phix+  (H_1-\mu(r-2))\pa_\tau\Phix\\
& + & (p-1)\qx H_2\Lambda\left\{\frac{1}{(p-1)\qx}\left[-\pa_\tau \Phi-H_2\Lambda \Phix+(H_1-\mu(r-2))\Phix+G_\Phi\right]\right\}\\
& + & \pa_\tau G_\Phi-(p-1)QG_\rho\\
& = & (p-1)Q\Delta \Phi-H_2^2\Lambda^2\Phi-2H_2\Lambda\pa_\tau \Phix+  A_1\Lambda \Phi+A_2\pa_\tau\Phix+A_3 \Phix\\
& + & \pa_\tau G_\Phi-\left(H_1+H_2\frac{\Lambda Q}{Q}\right)G_\Phi+H_2\Lambda G_\Phi-(p-1)QG_\rho
\eee
with
$$\left|\begin{array}{llll}
A_1=H_2H_1-H_2\Lambda H_2+H_2(H_1-\mu(r-2))+H_2^2\frac{\Lambda Q}{Q}\\
A_2=2H_1-\mu(r-2)+H_2\frac{\Lambda Q}{Q}\\
A_3=-(H_1-\mathcal e)H_1+H_2\Lambda H_1-H_2(H_1-\mu(r-2))\frac{\Lambda Q}{Q} -  (p-1)QH_3
\end{array}\right.
$$
In this section we focus  on deriving decay estimates for \eqref{nekoneneon}.

\begin{remark}[Null coordinates and red shift]
 We note that the principal symbol of the above wave equation is given by the second order operator
$$
\Box_Q:=\pa_\tau^2 - ((p-1)Q-H_2^2 Z^2)\pa_Z^2+2H_2Z \pa_Z\pa_\tau
$$
In the variables of Emden transform this can be written equivalently as 
$$
\Box_Q=\pa_\tau^2 - \mu^2\left[\sigma^2-(1-w)^2\right]\pa_x^2+2\mu(1-w)\pa_x\pa_\tau
$$
The two principal null direction associated with the above equation are 
$$
L=\pa_\tau+\mu \left[(1-w)-\sigma\right]\pa_x,\qquad \underline{L}=\pa_\tau+\mu \left[(1-w)+\sigma\right]\pa_x,
$$
so that 
$$
\Box_Q=L\underline{L}
$$
We observe that at $P_2$, we have $L=\pa_\tau$ and the surface $Z=Z_2$ is a null cone. Moreover, the associated acoustical metric is
$$
g_Q=\mu^2\Delta d\tau^2 -2\mu(1-w) d\tau dx +dx^2,\qquad \Delta=(1-w)^2-\sigma^2
$$
for which $\pa_\tau$ is a Killing field (generator of translation symmetry). Therefore, $Z=Z_2$ is a {\it Killing horizon} (generated by a null Killing field.) We can make it even more precise by transforming the metric $g_Q$ into a slightly different form by defining the coordinate $s$:
$$
s=\mu\tau-f(x),\qquad f'=\frac {1-w}{\Delta},
$$
so that 
$$
g_Q=\Delta (ds)^2-\frac {\sigma^2}{\Delta} dx^2
$$
and then the coordinate $x^*$:
$$
x^*=\int \frac{\sigma}{\Delta} dx, 
$$
so that
$$
g_Q=\Delta\, d(s+x^*)\, d(s-x^*)
$$
and $s+x^*$ and $s-x^*$ are the null coordinates of $g_Q$.
The Killing horizon $Z=Z_2$ corresponds to $x^*=-\infty$ and $\Delta\sim e^{Cx^*}$ for some positive constant $C$.
In this form, near $Z_2$ the metric $g_Q$ resembles the $1+1$-quotient Schwarzschild metric near the black hole horizon.

The associated {\it surface gravity} $\kappa$ which can be computed according to
\begin{align*}
\kappa&=\frac {\pa_{x^*} \Delta}{2\Delta}|_{P_2} =\frac {\pa_{x} \Delta}{2\sigma}|_{P_2}=
\frac {-w'(1-w)-\sigma'\sigma}{\sigma}|_{P_2}\\ &=(-w'-\sigma')|_{P_2}=1-w-\Lambda w -\frac{(1-w)F}{\sigma}|_{P_2}>0
\end{align*}
This is precisely the repulsive condition \eqref{coercivityquadrcouplinginside}  (at $P_2$).
The positivity of surface gravity implies the presence of the {\it red shift} effect along $Z=Z_2$ both as an optical phenomenon for the acoustical metric
$g_Q$ and also as an indicator of local monotonicity estimates for solutions of the wave equation 
$\Box_Q \varphi=0$, \cite{dr}. The complication in the analysis below is the presence of lower order terms in the wave equation as well as the need for global in space estimates.
\end{remark}


\subsection{The linearized operator}


Pick a small enough parameter $$0<a\ll 1$$ and consider the new variable
\be
\label{defintionT}
T=\pa_\tau\Phi+aH_2\Lambda \Phi,
\ee then 
\bee
\pa_\tau T&=&\pa^2_\tau\Phi+aH_2\Lambda\pa_\tau \Phi= \pa^2_\tau\Phi+aH_2\Lambda(T-aH_2\Lambda \Phi)\\
&=& \pa^2_\tau\Phi+aH_2\Lambda T-a^2H_2\Lambda H_2\Lambda \Phi-a^2H_2^2\Lambda^2\Phi
\eee
which yields the $(T,\Phi)$ equation $$\pa_\tau \Phi=T-aH_2\Lambda \Phi$$ and 
\bee
\pa_\tau T& = & (p-1)Q\Delta \Phi-H_2^2\Lambda^2\Phi-2H_2\Lambda(T-aH_2\Lambda \Phi)+  A_1\Lambda \Phi+A_2(T-aH_2\Lambda \Phi)+A_3 \Phix\\
& + & aH_2\Lambda T-a^2H_2\Lambda H_2\Lambda \Phi-a^2H_2^2\Lambda^2\Phi+  G_T\\
&= & (p-1)Q\Delta \Phi-(1-a)^2H_2^2\Lambda^2\Phi+\tilde{A_2}\Lambda \Phi+A_3\Phi-(2- a)H_2\Lambda T +A_2T\\
&+&G_T
\eee
with 
\be
\label{defgt}
G_T=\pa_\tau G_\Phi-\left(H_1+H_2\frac{\Lambda Q}{Q}\right)G_\Phi+H_2\Lambda G_\Phi-(p-1)QG_\rho
\ee 
and $$\tilde{A}_2=A_1+(2a-a^2)H_2\Lambda H_2-a A_2H_2.$$
We rewrite these equations in vectorial form
\be
\label{newlinearflow}
\pa_\tau X=\mathcal M X +G, \ \ X=\left|\begin{array}{ll}\Phi\\ T\end{array}\right., \ \ G=\left|\begin{array}{ll}0\\ G_T\end{array}\right.
\ee
with 
\be
\label{defiiotinm}
\mathcal M=\left(\begin{array}{ll} -aH_2\Lambda & 1\\ (p-1)Q\Delta -(1-a)^2H_2^2\Lambda^2+\tilde{A_2}\Lambda +A_3& -(2-a)H_2\Lambda +A_2\end{array}\right).
\ee


\subsection{Shifted measure}


The fine structure of the operator \eqref{defiiotinm} involves the understanding of the associated light cone.

\begin{lemma}[Shifted measure]
\label{shiftoight}
Let 
\be
\label{defdacnoneo}
D_a=(1-a)^2(w-1)^2-\sigma^2
\ee
then for $0<a<a^*$ small enough, there exists a $\mathcal C^1$ map $a\mapsto Z_a$ with $$Z_{a=0}=Z_2, \ \ \frac{\pa Z_a}{\pa a}>0$$ such that 
\be
\label{eineneoneonoe}
\left|\begin{array}{l}
D_a(Z_a)=0\\
-D_a(Z)>0\ \ \mbox{on}\ \ 0\le Z<Z_a\\
\lim_{Z\to 0}Z^2(-D_a)>0.
\end{array}\right.
\ee
\end{lemma}

\begin{proof}[Proof of Lemma \ref{shiftoight}]

We recall the notations of the Emden transform: 
\be
\label{enoenoenvonev}
\left|\begin{array}{l}
x=\log Z\\
\mu=\frac{1-{\mathcal e}}{2}\\
F=\sigma+\Lambda \sigma\\
(p-1)Q=\mu^2Z^2\sigma^2, \ \ \frac{\Lambda Q}{Q}=2+2\frac{\Lambda \sigma}{\sigma}=\frac{2F}{\sigma}\\
(p-1)\pa_ZQ=(p-1)\frac{\Lambda Q}{Q}\frac{Q}{Z}=2\mu^2Z\sigma^2\left(1+\frac{\Lambda\sigma}{\sigma}\right)=2\mu^2Z\sigma F\\
H_2=\frac{1-{\mathcal e}}{2}+2\frac{\pa_Z\Psi_P}{Z}=\mu(1-w)\\
H_1=H_2\frac{\Lambda \rho_P}{\rho_P}=\frac{H_2}{p-1}\frac{\Lambda Q}{Q}=\frac{2\mu F(1-w)}{(p-1)\sigma}\\
D=(w-1)^2-\sigma^2.
\end{array}\right.
\ee

\noindent{\bf step 1} Values of derivatives at $P_2$. Let $$\Delta=(w-1)^2-\sigma^2.$$ Let the variables $$w=w_2+W, \ \ \sigma=\sigma_2+\Sigma,$$ then near $P_2$: $$W=c_-\Sigma+O(\Sigma^2).$$ 
Let
\be
\label{defvalueci}
\left|\begin{array}{llll}
c_1=\pa_W\Delta_1(P_2)\\
c_2=\pa_W\Delta_2(P_2)\\
c_3=\pa_\Sigma \Delta_1(P_2)\\
c_4=\pa_\Sigma\Delta_2(P_2)=-2\sigma_2^2
\end{array}\right.
\ee
Then, in our range of parameters, 
\bea
c_1<0, \ \ c_2<0, \ \  c_3<0,\ \  c_4<0, 
\eea
and we have
\be
\label{realtionslopeweignefuncitons}
\left|\begin{array}{ll}
c_2c_-+c_4=\l_+\\
c_2c_++c_4=\l_-\\
c_{\pm}=\frac{c_1c_\pm+c_3}{c_2c_\pm+c_4}
\end{array}\right.
\ee
which imply 
$$c_1c_-+c_3=c_-(c_2c_-+c_4)=c_-\l_+$$
as well as
\bea
-1<c_-<0<c_+, \ \ \l_-<\l_+<0,
\eea
see Lemma 2.8 and Lemma 2.9 in \cite{MRRSprofile}. 

We compute
$$
\left|\begin{array}{l}
\Delta_1=c_1W+c_3\Sigma+O(W^2+\Sigma^2)=(c_1c_-+c_3)\Sigma+O(\Sigma^2)\\
\Delta_2=c_2W+c_4\Sigma+O(W^2+\Sigma^2)=(c_2c_-+c_4)\Sigma+O(\Sigma^2)\\
\Delta=(1-w_2-W)^2-(\sigma_2+\Sigma)^2=(\sigma_2-W)^2-(\sigma_2+\Sigma)^2=-2\sigma_2(c_-+1)\Sigma+O(\Sigma^2)
\end{array}\right.
$$
This yields 
\be
\label{venvneneoevnen}
\left|\begin{array}{l}
\frac{dw}{dx}=-\frac{\Delta_1}{\Delta}=-\frac{c_1c_-+c_3+O(\Sigma)}{-2\sigma_2(1+c_-)+O(\Sigma)}=\frac{|c_-||\l_+|}{2\sigma_2(1+c_-)}+O(\Sigma)\\
\frac{d\sigma}{dx}=-\frac{\Delta_2}{\Delta}=-\frac{c_2c_-+c_4+O(\Sigma)}{-2\sigma_2(1+c_-)+O(\Sigma)}=-\frac{|\l_+|}{2\sigma_2(1+c_-)}+O(\Sigma)
\end{array}\right.
\ee
and hence
\bea
\label{siggaddoadne}
\nonumber &&Z_2\frac{d\Delta}{dZ}(Z_2)=  \frac{d\Delta}{dx}(P_2)=-2(1-w_2)\frac{dw}{dx}(P_2)-2\sigma_2\frac{d\sigma}{dx}(P_2)\\
\nonumber & = & -2\sigma_2\frac{|c_-||\l_+|}{2\sigma_2(1+c_-)}-2\sigma_2\left[-\frac{|\l_+|}{2\sigma_2(1+c_-)}\right]=  \frac{|\l_+|}{1+c_-}(1-|c_-|)\\
&=& |\l_+|>0
\eea
\noindent{\bf step 2} Computation of $Z_a$. Let $D_0(Z)=\Delta(Z)$, we have $D_0'(Z_2)>0$ from \eqref{siggaddoadne} and hence by the implicit function theorem applied to the function $F(a, Z)=D_a(Z)$ at $(a, Z)=(0, Z_2)$ where $D_0(Z_2)=0$, we infer for all $a$ small enough the existence of a locally unique solution $Z_a$ to 
\be
\label{defitionza}
D_a(Z_a)=0
\ee
Furthermore, $Z_a$ is $\mathcal{C}^1$ in a neighborhood of $a=0$ and its derivative is given by
$$
 \frac{\pa Z_a}{\pa a}_{|a=0} = -\left(\frac{\frac{\pr D_a(Z)}{\pr a}}{\frac{\pr D_a(Z)}{\pr Z}}\right)_{a=0, \, Z=Z_2}= \frac{2\sigma_2^2}{D_0'(Z_2)}>0$$ Thus
\be
\label{vneiovnneo}
 \frac{\pa Z_a}{\pa a}_{|a=0}>0,\ \ Z_a>Z_2\ \ \mbox{for}\ \ 0<a\ll1, \ \ D_a'(Z_a)>0.
 \ee
We now observe
$$D_a(Z)= ((1-a)(1-w)+\sigma)((1-a)(1-w)-\sigma)$$
so that $D_a(Z)$ is of the sign of $(1-a)(1-w)-\sigma$ since $w<1$ and $\sigma>0$. Now from \eqref{venvneneoevnen}:
\bee
\frac{d}{dx}\Big((1-a)(1-w)-\sigma\Big) &=&-(1-a)\frac{|c_-||\l_+|}{2\sigma_2(1+c_-)}+\frac{|\l_+|}{2\sigma_2(1+c_-)}\\
&=&\frac{|\l_+|}{2\sigma_2(1+c_-)}[1-(1-a)|c_-|]>0.
\eee
Thus, $(1-a)(1-w)-\sigma$ is increasing on $(0,Z_a]$ and vanishes at $Z=Z_a$ so that 
\bee
D_a(Z)<0\textrm{ on }(0, Z_a).
\eee
Moreover, we have in view of the behavior of $\sigma$ and $w$ as $Z\to 0_+$, see Lemma 3.1 in  \cite{MRRSprofile}, 
\bee
\lim_{Z\to 0_+}Z^2(-D_a(Z)) &=& \lim_{Z\to 0_+}Z^2\sigma^2=1>0.
\eee
This concludes the proof of \eqref{eineneoneonoe}.
\end{proof}


\subsection{Commuting with derivatives}


We define $$T_k=\Delta^kT, \ \ \Phi_k=\Delta^k\Phi.$$ 

\begin{lemma}[Commuting with derivatives]
\label{lemmaderivative}
Let $k\in \Bbb N$. There exists a smooth measure $g$ defined for $Z\in [0,Z_a]$ such that the following holds. Let the elliptic operator
$$\mathcal L_g\Phi_k=\frac{\mu^2}{gZ^{d-1}}\pa_Z\left(Z^{d-1}Z^2g(-D_a) \pa_Z\Phi_k\right),$$ 
then there holds 
\be
\label{commutationderivative}
\Delta^k(\mathcal MX)=\mathcal M_k\left|\begin{array}{ll}\Phi_k\\T_k\end{array}\right.+\widetilde{\mathcal M}_k X
\ee
with
$$
\mathcal M_k\left|\begin{array}{ll}\Phi_k\\T_k\end{array}\right.=\left|\begin{array}{ll}-aH_2\Lambda \Phi_k-2ak(H_2+\Lambda H_2)\Phi_k+ T_k\\ \matchal L_g \Phi_k-(2-a)H_2\Lambda T_k-2k(2-a)(H_2+\Lambda H_2) T_k +A_2T_k\end{array}\right.
$$
where $\widetilde{\mathcal M}_k$ satisfies the following pointwise bound
\be
\label{contlmak}
|\widetilde{\mathcal M}_kX|\lesssim_k \left|\begin{array}{ll} \displaystyle\sum_{j=0}^{2k-1}|\pa^{j}_Z\Phi|,\\
 \displaystyle\sum_{j=0}^{2k}|\pa^{j}_Z\Phi|+\sum_{j=0}^{2k-1}|\pa_Z^jT|.
\end{array}\right.
\ee
Moreover, $g>0$ in $[0,Z_a)$ and admits the asymptotics:
\be
\label{asymptoticsg}
\left|\begin{array}{ll}g(Z)=1+O(Z^2)\ \ \mbox{as}\ \ Z\to 0\\
g(Z)=c_{a,d,r,\ell}(Z_a-Z)^{c_g}\left[1+O(Z-Z_a)\right]\ \mbox{as}\ \ Z\uparrow Z_a, 
\end{array}\right.
\ee
with 
\be
\label{estcg}
c_g>0
\ee for all $k\ge k_1$ large enough and $0<a<a^*$ small enough.
\end{lemma}

\begin{proof} This  is a direct computation.\\

\noindent{\bf step 1} Proof of \eqref{commutationderivative}, \eqref{contlmak}. We recall \eqref{estimatecommutatorvlkeveln}:
$$[\Delta^k,V]\Phi-2k\nabla V\cdot\nabla \Delta^{k-1}\Phi=\sum_{|\alpha|+|\beta|=2k,|\beta|\le 2k-2}c_{k,\alpha,\beta}\pa^\alpha V\pa^\beta\Phi$$ which together with the commutator formulas 
\be
\label{cbjeibvejbegi}
\left|\begin{array}{l}
[\Delta^k,\Lambda]=2k\Delta^k,\ \ [\pa_Z,\Lambda]=\pa_Z\\
\Lambda^2=Z^2\Delta-(d-2)\Lambda\\
\pa_Z\Lambda=\frac{\Lambda^2}{Z}=Z\Delta-(d-2)\pa_Z
\end{array}\right.
\ee 
yields
\bee
\Delta^k(V\Lambda \Phi)&=&V\Delta^k(\Lambda \Phi)+2k\nabla V\cdot\nabla \Delta^{k-1}\Lambda\Phi+\sum_{|\alpha|+|\beta|=2k,|\beta|\le 2k-2}c_{k,\alpha,\beta}\pa^\alpha V\pa^\beta\Lambda \Phi.
\eee
and
\bee
&&2k\nabla V\cdot\nabla \Delta^{k-1}\Lambda\Phi=  2k\pa_ZV\pa_Z\left[\Lambda \Delta^{k-1}\Phi+2(k-1)\Phi_{k-1}\right]\\
& = & 2k\pa_ZV\left[(Z\Delta-(d-2)\pa_Z)\Phi_{k-1}+2(k-1)\pa_Z\Phi_{k-1}\right]\\
& = & 2k\Lambda V\Phi_k+2k(2k-2-d+2)\pa_ZV\pa_Z\Phi_{k-1}
\eee
from which for $0\le Z\le Z_a$:
$$\Delta^k(V\Lambda \Phi)=V(2k+\Lambda) \Phi_k+2k\Lambda V\Phi_k+O_{k,a}\left(\sum_{j=0}^{2k-1}|\pa^{j}_Z\Phi|\right).$$ We then use
\bee
[\Delta^k,\Lambda^2]&=&\Delta^k\Lambda^2-\Lambda^2\Delta^k=[\Delta^k,\Lambda]\Lambda+\Lambda\Delta^k\Lambda-\Lambda(-[\Delta^k,\Lambda]+\Delta^k\Lambda)\\
& = & 2k\Delta^k\Lambda+2k\Lambda\Delta^k=4k^2\Delta^k+4k\Lambda \Delta^k
\eee
to compute similarly:
\bee
&&\Delta^k(V\Lambda^2 \Phi)=V\Delta^k(\Lambda^2 \Phi)+2k\nabla V\cdot\nabla \Delta^{k-1}\Lambda^2\Phi+\sum_{|\alpha|+|\beta|=2k,|\beta|\le 2k-2}c_{k,\alpha,\beta}\pa^\alpha V\pa^\beta\Lambda \Phi\\
& = & V\left[\Lambda^2\Phi_k+4k^2\Phi_k+4k\Lambda \Phi_k\right]+2k\pa_ZV\pa_Z\Delta^{k-1}\Lambda^2\Phi+O\left(\sum_{j=0}^{2k}|\pa^{j}_Z\Phi|\right)
\eee
and 
\bee
&&\pa_Z\Delta^{k-1}\Lambda^2\Phi=\pa_Z\left[\Lambda^2\Phi_{k-1}+4(k-1)^2\Phi_{k-1}+4(k-1)\Lambda \Phi_{k-1}\right]\\
& = & \pa_Z(Z^2\Phi_k-(d-2)\Lambda \Phi_{k-1})+O\left(\sum_{j=0}^{2k}|\pa^{j}_Z\Phi|\right)=Z\Lambda \Phi_k+O\left(\sum_{j=0}^{2k}|\pa^{j}_Z\Phi|\right)
\eee
and hence
\bee
&&\Delta^k(V\Lambda^2 \Phi)=V\left[\Lambda^2\Phi_k+4k^2\Phi_k+4k\Lambda \Phi_k\right]+2k\Lambda V\Lambda \Phi_k+O\left(\sum_{j=0}^{2k}|\pa^{j}_Z\Phi|\right).
\eee
Recalling the definition of the operator \eqref{defiiotinm}, we obtain \eqref{commutationderivative}, \eqref{contlmak} with
$$\mathcal M_k\left|\begin{array}{ll}\Phi_k\\T_k\end{array}\right.=\left|\begin{array}{ll}-aH_2\Lambda \Phi_k -2ak(H_2+\Lambda H_2)\Phi_k+ T_k\\ (p-1)Q\Delta \Phi_k-(1-a)^2H_2^2\Lambda^2\Phi_k+A_k\Lambda \Phi_k\\
-(2-a)H_2\Lambda T_k-2k(2-a)(H_2+\Lambda H_2) T_k +A_2T_k
\end{array}\right.$$
and
$$A_k=2k(p-1)\frac{\pa_ZQ}{Z}-(1-a)^24kH_2\left[H_2+\Lambda H_2\right]+\tilde{A}_2.
$$

\noindent{\bf step 2} Equation for the measure. We compute using \eqref{enoenoenvonev}, \eqref{defdacnoneo}:
\bee
&&(p-1)Q\Delta \Phi_k-(1-a)^2H_2^2\Lambda^2\Phi_k\\
&=&\mu^2Z^2\sigma^2\left(\pa_Z^2\Phi_k+\frac{d-1}{Z}\pa_Z\Phi_k\right)-\mu^2(1-w)^2(1-a)^2\left(Z^2\pa_Z^2\Phi_k+\Lambda \Phi_k\right)\\
& =  & \mu^2\left\{-Z^2D_a\pa_Z^2\Phi_k+Z\pa_Z\Phi_k\left[(d-1)\sigma^2-(1-a)^2(1-w)^2\right]\right\}
\eee
and 
\bee
A_k& = & 2k(p-1)\frac{\pa_ZQ}{Z}-(1-a)^24kH_2\left[H_2+\Lambda H_2\right]+\tilde{A}_2\\
&= & 4k\mu^2\left[\sigma F-(1-a)^2(1-w)+(1-a)^2(1-w)(w+\Lambda w)\right]+\tilde{A}_2
\eee
and hence:
\bee
&&(p-1)Q\Delta \Phi_k-(1-a)^2H_2^2\Lambda^2\Phi_k+A_k\Lambda \Phi_k=  -\mu^2Z^2D_a\pa_Z^2\Phi_k\\
&+& \Lambda \Phi_k\Bigg[\mu^2\left((d-1)\sigma^2-(1-a)^2(1-w)^2\right)\\
&&+4k\mu^2\left[\sigma F-(1-a)^2(1-w)+(1-a)^2(1-w)(w+\Lambda w)\right]+\tilde{A}_2\Bigg].
\eee 
We compute the measure
 $$\frac{\mu^2}{gZ^{d-1}}\pa_Z\left(Z^{d-1}Z^2g(-D_a) \Phi_k'\right)=\mu^2Z^2(-D_a)\pa_Z^2\Phi_k-\mu^2\Lambda \Phi_k\left((d+1)D_a+\frac{g'}{g}ZD_a+ZD_a'\right)$$ and hence the relation:
\bee
&&-\mu^2\left((d+1)D_a+\frac{g'}{g}ZD_a+ZD_a'\right)\\
&=& \mu^2\left((d-1)\sigma^2-(1-w)^2\right)+4k\mu^2\left[\sigma F-(1-a)^2(1-w)+(1-a)^2(1-w)(w+\Lambda w)\right]+\tilde{A}_2.
\eee
Equivalently:
\be
\label{equationforthemeasure}
(-D_a)\frac{\Lambda g}{g}=- \mathcal G
\ee
with
\bea
\label{formulag}
\matchal G&= &-(d-1)\sigma^2+(1-w)^2-(d+1)D_a-\Lambda D_a\\
\nonumber && +4k\left[(1-a)^2(1-w) -\sigma F-(1-a)^2(1-w)(w+\Lambda w)\right]-\frac{\tilde{A}_2}{\mu^2}.
\eea

\noindent{\bf step 3} Asymptotics of the measure. We now solve \eqref{equationforthemeasure}. 
 Near the origin, the normalization \eqref{normlazationprofile} and \eqref{enoenoenvonev} yield $$\sigma=\frac{\sqrt{p-1}}{\mu Z}\left[1+O(Z^2)\right], \ \ F=\sigma+\Lambda \sigma=O(Z), \ \ -D_a=\sigma^2+O(1).$$

We compute 
$$\frac{\tilde{A_2}}{\mu^2}=O\left(\frac{|F|}{\sigma}+|\Lambda w|+|a|\right)=O(1)$$ and hence
\bee
\matchal G &=& -(d-1)\sigma^2 -(d+1)(-\sigma^2)-\Lambda(-\sigma^2)+O\left(1+|a|+\sigma|F|+|w|+|\Lambda w|\right)\\
&=& 2\sigma F+O(1) = O(1)
\eee
which, recalling \eqref{eineneoneonoe}, yields:
$$-\frac{\mathcal G}{(-D_a)}=\frac{O(1)}{\sigma^2+O(1)}=O(Z^2)$$
and we may therefore choose explicitly:
$$g=e^{\int_0^Z \left[-\frac{\mathcal G}{(-D_a)}\right]\frac{d\tau}{\tau}}.$$ 
To compute the behavior near $Z_a$, recall from \eqref{defitionza} \eqref{vneiovnneo} that we have 
\bee
D_a(Z_a)=0, \ \ D_a'(Z_a)>0.
\eee
We infer in the neighborhood of $Z=Z_a$
\bea
\label{vneioenevvnn}
\nonumber\frac{\pa_Zg}{g}& =& \frac{\matchal G}{ZD_a}= \frac{\matchal G(Z_a)}{\Lambda D_a(Z_a)}\frac{1+O(Z-Z_a)}{Z-Z_a}\\
&=& \left(\frac{\matchal G(Z_2)}{\Lambda D_0(Z_2)}+O(|a|)\right)\frac{1+O(Z-Z_a)}{Z-Z_a}.
\eea
The fundamental computation is then at $P_2$ using \eqref{venvneneoevnen}:
\bee
&&\left[(1-w) -\sigma F-(1-w)(w+\Lambda w)\right]=(1-w_2)(1-w_2-\Lambda w)-\sigma_2(\sigma_2+\Lambda \sigma)\\
& = & \sigma_2(\sigma_2-\Lambda w)-\sigma_2(\sigma_2+\Lambda \sigma)=\sigma_2(-\Lambda w-\Lambda \sigma)\\
& = & \sigma_2\left[-\frac{|c_-||\l_+|}{2\sigma_2(1+c_-)}+\frac{|\l_+|}{2\sigma_2(1+c_-)}\right]=\frac{|\l_+|}{2}>0
\eee
Hence from \eqref{formulag} $$\mathcal G(Z_2)=2k(|\l_+|+O(a))+O(1).$$
and from \eqref{siggaddoadne}
$$\frac{\matchal G(Z_2)}{\Lambda D_0(Z_2)}=\frac{2k\left(|\l_+|+O(a)\right)+O(1)}{|\l_+|}>k$$ for $0<a<a^*$ small enough and $k\ge k_1$ large enough\footnote{In particular, we need $k_1>|\lambda_+|^{-1}$.}. Inserting this into \eqref{vneioenevvnn} yields
 \eqref{asymptoticsg}.
\end{proof}


\subsection{Hardy inequality and compactness}


We let $k\ge k_1$ large enough so that \eqref{estcg} holds and extend the measure $g$ by zero for $Z\ge Z_a$. We let $\chi$ be a smooth cut off function  supported strictly inside the light cone $|Z|<Z_2$ with $$g\geq \frac 12\ \ \mbox{on}\ \ {\rm Supp}\chi.$$ Let 
$$\mathcal D_\Phi=\left\{\Phi\in \matchal C^{\infty}([0,Z_a],\Bbb C)\ \ \mbox{with spherical symmetry}\right\},$$ 
be the space of test functions and 
\bea
\label{defscalarproductbis}
\la\la \Phi,\tilde{\Phi}\ra\ra&=&-(\mathcal L_g\Phi_k,\tilde{\Phi}_{k})_g+\int \chi \Phi\overline{\tilde{\Phi}}gZ^{d-1}dZ
\eea
be a Hermitian scalar product,
where we recall the notation \eqref{scalarltwo}. We let $\Bbb H_{\Phi}$ be the completion of $\mathcal{D}_\Phi$ for the norm associated to \eqref{defscalarproductbis}. 
We claim the following compactness subcoercivity estimate:

\begin{lemma}[Subcoercivity estimate] 
For $0<\nu\ll1$:
\be
\label{lowerboundhardy}
 \la\la \Phi,\Phi\ra\ra\gtrsim \int \frac{|\Phi_k|^2}{Z_a-Z}gZ^{d-1}dZ+\sum_{m=0}^{2k} \int  |\pa_Z^{m}\Phi(Z)|^2\frac{g}{(Z_a-Z)^{1-\nu}}Z^{d-1}dZ.
\ee
Furthermore, there exists $c>0$ and a sequence $\mu_n>0$ with $\lim_{n\to +\infty}\mu_n=+\infty$ and $\Pi_n\in \Bbb H_{\Phi}$, $c_n>0$ such that $\forall n\geq 0$, $\forall \Phi\in \Bbb H_{\Phi}$,
\bea
\label{coercivityformula}
\nonumber \la\la \Phi,\Phi\ra\ra &\geq & c\int \frac{|\Phi_k|^2}{Z_a-Z}gZ^{d-1}dZ+\mu_n \sum_{m=0}^{2k} \int  |\pa_Z^{m}\Phi(Z)|^2\frac{g}{(Z_a-Z)^{1-\nu}}Z^{d-1}dZ\\
&- & c_n\sum_{i=1}^n(\Phi,\Pi_i)_g^2.
\eea
\end{lemma}

\begin{proof} This is a classical Hardy and Sobolev based argument.\\

\noindent{\bf step 1} Interior estimate. Let $Z_0<Z_a$ which will be chosen close enough to $Z_a$ in step 2. Then, we have 
\bee
&&\int_0^{Z_0} \frac{|\Phi_k|^2}{Z_a-Z}gZ^{d-1}dZ+\sum_{m=0}^{2k} \int_0^{Z_0}  |\pa_Z^{m}\Phi(Z)|^2\frac{g}{(Z_a-Z)^{1-\nu}}Z^{d-1}dZ\\
&\leq& C_{Z_0}\|\Phi\|_{H^{2k}(0, Z_0)}^2\leq C_{Z_0}\left[ \int_0^{Z_0} |\pa_Z\Phi_k|^2Z^{d-1}dZ+ \int_0^{Z_0} \chi|\Phi(Z)|^2Z^{d-1}dZ\right].
\eee
Since $-Z^2D_a$ and $g$ are smooth and satisfy $-Z^2D_a>0$ and $g>0$ on $[0,Z_0]$, we infer
\bee
\la\la \Phi,\Phi\ra\ra\geq c_{Z_0}\left[ \int_0^{Z_0} \frac{|\Phi_k|^2}{Z_a-Z}gZ^{d-1}dZ+\sum_{m=0}^{2k} \int_0^{Z_0}  |\pa_Z^{m}\Phi(Z)|^2\frac{g}{(Z_a-Z)^{1-\nu}}Z^{d-1}dZ\right]
\eee
for some $c_{Z_0}>0$. Thus, to prove \eqref{lowerboundhardy}, it remains to consider the region $(Z_0,Z_a)$. This will be done in step 2 and step 3.\\

\noindent{\bf step 2} Hardy inequality with loss. Let $0<\nu \ll1$, we claim the lossy Hardy bound for all $\Phi\in \matchal D_\Phi$:
\be
\label{kenkenonevneovn}
\sum_{m=0}^{2k} \int_{Z_0}^{Z_a} |\pa_Z^{m}\Phi(Z)|^2\frac{g}{(Z_a-Z)^{1-\nu}}Z^{d-1}dZ\leq c_{\nu} \la\la \Phi,\Phi\ra\ra.
\ee
Indeed, let $Z_0=Z_a-\de$ with $\de>0$ small enough, we estimate by Taylor expansion for $Z_0\leq Z< Z_a$ for $0\leq m\leq 2k$:
\bee
&&|\pa_Z^{m}\Phi(Z)|^2\leq C\left[  \sum_{j=m}^{2k}\left|\pa_Z^{j}\Phi\left(Z_0\right)\right|^2+\left(\int_{Z_0}^Z\left|\pa_Z^{2k+1}\Phi(\tau)\right|d\tau\right)^2\right]
\eee
From Sobolev, 
\bee
\sum_{j=m}^{2k}\left|\pa_Z^{j}\Phi\left(Z_0\right)\right|^2\leq C_{Z_0}\|\Phi\|_{H^{2k+1}(0, Z_0)}^2\leq  C_{Z_0}\la\la \Phi,\Phi\ra\ra
\eee
and hence
\bee
&&|\pa_Z^{m}\Phi(Z)|^2 \leq C_{Z_0} \la\la \Phi,\Phi\ra\ra+C\left(\int_{Z_0}^Z\left|\pa_Z\Delta^k\Phi(\tau)\right|d\tau\right)^2+C\left(\int\sum_{j=1}^{2k}|\pa_Z^j \Phi|d\tau\right)^2\\
& \leq & C_{Z_0}\la\la \Phi,\Phi\ra\ra+C\left(\int_{Z_0}^Z|\pa_Z\Phi_k|^2(Z_a-Z)^{1-\nu}Z^{d-1}dZ\right)\left(\int_{Z_0}^Z\frac{d\tau}{(Z_a-\tau)^{1-\nu}}\right)\\
& + &C\sum_{j=1}^{2k}\left(\int_{Z_0}^{Z}\frac{|\pa_Z^j\Phi|^2}{(Z_a-Z)^{1-\nu}}dZ\right)\left(\int_{Z_0}^Z(Z_a-\tau)^{1-\nu}d\tau\right)\\
& \leq & C_{Z_0}\la\la \Phi,\Phi\ra\ra+C\de^\nu\int_{Z_0}^Z|\pa_Z\Phi_k|^2(Z_a-\tau)^{1-\nu}d\tau+C\de\sum_{j=m}^{2k}\left(\int_{Z_0}^Z\frac{|\pa_Z^j\Phi|^2}{(Z_a-\tau)^{1-\nu}}d\tau\right)
\eee
where we used the fact that $0<\nu<1$ and $Z_a-Z_0=\de$. Using again $0<\nu<1$ and $Z_a-Z_0=\de$, as well as Fubini and  the fact that $\pa_Zg<0$ on $(Z_0, Z_a)$ for $Z_0$ close enough to $Z_a$ so that $g$ is decreasing on $(Z_0, Z_a)$, we infer
\bee
&&\sum_{m=0}^{2k} \int_{Z_0}^{Z_a} |\pa_Z^{m}\Phi(Z)|^2\frac{g}{(Z_a-Z)^{1-\nu}}Z^{d-1}dZ\\
&\leq& C_{Z_0}\la\la \Phi,\Phi\ra\ra+C\de^\nu\int_{Z_0}^{Z_a}\frac{g}{(Z_a-Z)^{1-\nu}}\left(\int_{Z_0}^Z|\pa_Z\Phi_k|^2(Z_a-\tau)^{1-\nu}d\tau\right)Z^{d-1}dZ\\
&&+C\de\sum_{j=0}^{2k}\int_{Z_0}^{Z_a}\frac{g}{(Z_a-Z)^{1-\nu}}\left(\int_{Z_0}^Z\frac{|\pa_Z^j\Phi|^2}{(Z_a-\tau)^{1-\nu}}d\tau\right)Z^{d-1}dZ\\
&\leq&C_{Z_0}\la\la \Phi,\Phi\ra\ra+C\de^\nu\int_{Z_0}^{Z_a}|\pa_Z\Phi_k|^2g(\tau)(Z_a-\tau)^{1-\nu}\left(\int_\tau^{Z_a}\frac{dZ}{(Z_a-Z)^{1-\nu}}\right)\tau^{d-1}d\tau\\
&&+C\de\sum_{j=0}^{2k}\int_{Z_0}^{Z_a}\frac{|\pa_Z^j\Phi|^2g(\tau)}{(Z_a-\tau)^{1-\nu}}\left(\int_\tau^{Z_a}\frac{dZ}{(Z_a-Z)^{1-\nu}}\right)\tau^{d-1}d\tau\\
&\leq& C_{Z_0}\la\la \Phi,\Phi\ra\ra+C\de^\nu\int_{Z_0}^{Z_a}|\pa_Z\Phi_k|^2g(\tau)(Z_a-\tau)\tau^{d-1}d\tau+C\de\sum_{j=0}^{2k}\int_{Z_0}^{Z_a}\frac{|\pa_Z^j\Phi|^2g(\tau)}{(Z_a-\tau)^{1-\nu}}\tau^{d-1}d\tau.
\eee
Letting $\delta=\delta(\nu)$ small enough and estimating from \eqref{vneiovnneo}
\be
\label{lwoneenovnne}
-(D_a)(Z)\ge c(Z_a-Z)
\ee
 yields \eqref{kenkenonevneovn}.\\

\noindent{\bf step 3} Sharp Hardy. We now claim the sharp Hardy inequality for $f\in \mathcal D_\Phi$:
\be
\label{cneneonveoneo}
\la\la \Phi,\Phi\ra\ra \gtrsim \int_{Z_0}^{Z_a} \frac{|\Phi_k|^2}{(Z_a-Z)}gZ^{d-1}dZ.
\ee
Indeed, recall \eqref{asymptoticsg}, \eqref{estcg} near $Z_a$:
$$g(Z)=c(Z_a-Z)^{c_g}\left[1+O(Z-Z_a)\right],$$
then integrating by parts:
\bee
&& \int_{Z_0}^{Z_a} \frac{|\Phi_k|^2}{(Z_a-Z)}gZ^{d-1}dZ\lesssim \int_{Z_0}^{Z_a} |\Phi_k|^2(Z_a-Z)^{c_g-1}dZ\\
&=&-\frac{1}{c_g}[|\Phi_k|^2(Z_a-Z)^{c_g}]_{Z_0}^{Z_a}+\frac{1}{c_g}\int_{Z_0}^{Z_a}2\Phi_k\pa_Z\Phi_k(Z_a-Z)^{c_g}dZ\\
&\lesssim & |\Phi_k|^2(Z_0)+\left(\int |\Phi_k|^2(Z_a-Z)^{c_g-1}Z^{d-1}dZ\right)^{\frac 12}\left(\int |\pa_Z\Phi_k|^2(Z_a-Z)^{c_g+1}Z^{d-1}dZ\right)^{\frac 12}\\
&\lesssim &\la\la \Phi,\Phi\ra\ra+\left(\int_{Z_0}^{Z_a} \frac{|\Phi_k|^2}{(Z_a-Z)}gZ^{d-1}dZ\right)^{\frac 12}\left(\int |\pa_Z\Phi_k|^2g(-D_a)Z^{d-1}dZ\right)^{\frac 12}
\eee
where we used \eqref{lwoneenovnne}. The bound \eqref{cneneonveoneo} now follows using H\"older. Together with step 1 and step 2, this concludes the proof of  \eqref{lowerboundhardy}.\\

\noindent{\bf step 4} Compactness. We now turn to the proof of \eqref{coercivityformula} which follows from a standard compactness argument. Let us consider $T\in L^2_g$. Then from \eqref{lowerboundhardy}, the antilinear form $$h \mapsto ( T,h)_g,$$  is continuous on $\Bbb H_\Phi$, and hence by Riesz, there exists a unique $L(T)\in \Bbb H_\Phi$ such that 
\be
\label{vneioghenvenoene}
\forall h\in \Bbb H_\Phi, \ \ \la\la L(T), h\ra\ra =(T,h)_g,
\ee and the linear map $L$ is bounded from $L^2_g$ to $\Bbb H_\Phi$. For any $0<\de<Z_a$, we have in view of \eqref{lowerboundhardy}
$$
\|h\|_{L^2_g} \leq \de^{\frac{1-\nu}{2}}\left\|\frac{h}{(Z_a-Z)^{\frac{1-\nu}{2}}}\right\|_{L^2_g}+\|h\|_{L^2_g(Z\leq Z_a-\de)}\lesssim \de^{\frac{1-\nu}{2}}\left\|h\right\|_{\Bbb H_\Phi}+\|h\|_{L^2_g(Z\leq Z_a-\de)}.
$$
Relying on the smallness of $\de^{\frac{1-\nu}{2}}$ for the first term, and Rellich for the second one, we easily infer that 
\bea\label{eq:compactembeddingforHPsiinL2g}
\Bbb H_\Phi\textrm{ embeds compactly in }L^2_g.
\eea
Since $L$ is bounded from $L^2_g$ to $\Bbb H_\Phi$, we infer that the map $$L: L^2_g\mapsto L^2_g$$ 
is compact. Moreover, if $\Phi_1=L(T_1),$ $\Phi_2=L(T_2)$:
$$(L(T_1),T_2)_g=(\Phi_1, T_2)_g=\overline{(T_2,\Phi_1)_g}=\overline{\la\la LT_2,\Phi_1\ra\ra }=\la\la\Phi_1,\Phi_2\ra\ra $$ and hence interchanging the roles of $T_1,T_2$:
$$ (T_1,L(T_2))_g=\overline{(L(T_2),T_1)_g}=\overline{\la\la\Phi_2,\Phi_1\ra\ra }=\la\la\Phi_1,\Phi_2\ra\ra =(L(T_1),T_2)_g$$ 
and $L$ is selfadjoint on $L^2_g$. Since $L>0$ from \eqref{vneioghenvenoene}, we conclude that 
$L$ is a diagonalizable with a non increasing sequences of eigenvalues $\l_n>0$,
 $\lim_{n\to +_\infty}\l_n=0$, and let $(\Pi_{n,i})_{1\leq i\leq I(n)}$ be an $L^2_g$ orthonormal basis for the eigenvalue $\l_n$. The eigenvalue equation implies $\Pi_{n,i}\in \Bbb H_\Phi$.\\
 Let then $$\mathcal{A}_n=\left\{\Phi\in \Bbb H_\Phi, \ \ \int |\Phi|^2gZ^{d-1}dZ=1, \ \ (\Phi,\Pi_{j,i})_{g}=0, \ \ 1\leq i\leq I(j),  \ \ 1\leq j\leq n\right\}$$ and the minimization problem 
 $$I_n=\inf_{\Phi\in \mathcal{A}_n} \la\la \Phi,\Phi\ra\ra ,$$ 
 then the infimum is attained in view of  \eqref{eq:compactembeddingforHPsiinL2g} at $\Phi\in \mathcal A_n$ and, by a standard Lagrange multiplier argument:
$$\forall h\in \Bbb H_\Phi, \ \  \la\la \Phi,h\ra\ra =\sum_{j=1}^n\sum_{i=1}^{I(j)}\alpha_{i,j}(\Pi_{j,i},h)_g+\alpha (\Phi, h)_g.$$ Letting $h=\Pi_{i,j}$ implies $\alpha_{i,j}=0$ and hence from \eqref{vneioghenvenoene}: $$L(\Phi)=\frac{1}{\alpha} \Phi$$ which together with our orthogonality conditions implies $$\frac{1}{\alpha}\leq \l_{n+1}$$ and hence 
\be
\label{neiovneonevneo}
I_n=\la\la \Phi,\Phi\ra\ra =\alpha \la\la L(\Phi),\Phi\ra\ra =\alpha (\Phi,\Phi)_g=\alpha\geq \frac{1}{\l_{n+1}}.
\ee 
Also, for $Z_0=Z_a-\de$ with $\de>0$ small enough, we estimate from \eqref{kenkenonevneovn}
$$
\sum_{m=0}^{2k} \int_{Z_0}^{Z_a} |\pa_Z^{m}\Phi(Z)|^2\frac{g}{(Z_a-Z)^{1-\nu}}Z^{d-1}dZ \leq c_\nu\de^{\frac{\nu}{2}}\la\la \Phi,\Phi\ra\ra.
$$
On the other hand, from Rellich and an elementary compactness argument, for all $Z_a>0$, $\delta>0$, $\epsilon>0$, $k\geq 1$, there exists $c_{Z_a,\delta,\epsilon,k}>0$ such that 
$$\sum_{m=0}^{2k}\int_{Z\leq Z_a-\de}|\pa_Z^m\Phi|^2Z^{d-1}dZ\leq \epsilon \int_{Z\leq Z_a-\de}|\pa_Z\Delta^k\Phi|^2Z^{d-1}dZ+c_{Z_a,\delta,\epsilon,k} \int_{Z\leq Z_a-\de}|\Phi|^2Z^{d-1}dZ.$$
Summing the two inequalities yields for all $\de>0$ small and $\epsilon$ smaller still:
$$
\sum_{m=0}^{2k} \int_0^{Z_a} |\pa_Z^{m}\Phi(Z)|^2\frac{g}{(Z_a-Z)^{1-\nu}}Z^{d-1}dZ \le c_{\nu}\de^{\frac{\nu}{2}}\la\la \Phi,\Phi\ra\ra+\tilde c_{Z_a,\delta,k} \int_0^{Z_a}|\Phi|^2gZ^{d-1}dZ.
$$
Together with \eqref{neiovneonevneo}, this implies for any $\Phi$ satisfying the orthogonality conditions $(\Phi,\Pi_{j,i})_{g}=0, \ \ 1\leq i\leq I(j),  \ \ 1\leq j\leq n$, and for any $\de>0$
$$
\sum_{m=0}^{2k} \int_0^{Z_a} |\pa_Z^{m}\Phi(Z)|^2\frac{g}{(Z_a-Z)^{1-\nu}}Z^{d-1}dZ \leq \Big(c_\nu\de^{\frac{\nu}{2}}+\tilde c_{\delta,Z_a,k}\l_{n+1}\Big)\la\la\Phi, \Phi\ra\ra
$$
which yields \eqref{coercivityformula}.
\end{proof}


\subsection{Accretivity}


We now turn to the proof of the accretivity of the operator $\mathcal M$.\\

\noindent\underline{Hilbert space}. Recall \eqref{defscalarproductbis}. We define the space of test functions
$$\mathcal D_0=\mathcal D_\Phi\times C_{\rm radial}^\infty([0,Z_a],\Bbb C),$$
and let ${\Bbb H}_{2k}$
 be the completion of $\mathcal D_0$ for the scalar product:
\be
\label{defscalarproduct}
 \la X,\tilde{X}\ra = \la\la\Phi, \Phi\ra\ra+(T_k, \tilde{T}_k)_g+\int \chi T\overline{\tilde{T}}Z^{d-1}dZ
\ee
which is a coercive Hermitian form from \eqref{lowerboundhardy}.\\

\noindent\underline{Unbounded operator}. Following \eqref{defiiotinm}, we define the operator  $$\ \ \mathcal M=\left(\begin{array}{ll} -aH_2\Lambda & 1\\ (p-1)Q\Delta -(1-a)^2H_2^2\Lambda^2+\tilde{A_2}\Lambda +A_3& -(2-a)H_2\Lambda +A_2\end{array}\right)$$
  with domain 
\be
\label{defintiondoamin}
D(\mathcal M)=\{X\in \Bbb H_{2k}, \ \ \mathcal M X\in \Bbb H_{2k}\}
\ee 
equipped with the domain norm. We then pick suitable directions $(X_i)_{1\leq i\leq N}\in \Bbb H_{2k}$ and consider the finite rank projection operator
$$\mathcal A=\sum_{i=1}^N\la \cdot,X_i\ra X_i.$$
The aim of this section is to prove the following accretivity property:

\begin{proposition}[Maximal accretivity/dissipativity]
\label{propaccretif}
Let $$\mu, r>0.$$
There exist $k^*\gg 1$ and $0<c^*,a^*\ll 1$ such that for all $k\ge k^*$,  $\forall 0<a<a^*$ small enough, there exist $N=N(k,a)$ directions $(X_i)_{1\leq i\leq N}\in \Bbb H_{2k}$ such that the modified unbounded operator 
$$ \ \ \tilde{\mathcal M}=\mathcal M - \mathcal A$$
is dissipative\footnote{Equivalently, $-\tilde{\mathcal M}$ is accretive.}:
\be
\label{accretivityalmost}
\forall X\in \matchal D(\matchal M), \ \ \Re\la -\tilde{\matchal M} X,X\ra \geq c^*ak \la X,X\ra
\ee
and maximal:
\be
\label{estmaximal}
\forall R>0, \ \ \forall F\in \Bbb H_{2k}, \ \ \exists X\in \matchal D(\matchal M)\ \ \mbox{such that} \ \ (-\tilde{\matchal M}+R)X=F.
\ee
\end{proposition}

\begin{remark} We recall that maximal dissipative operators are closed.
\end{remark}

\begin{proof}[Proof of Proposition \ref{propaccretif}] given $R>R^*(k)$ large enough, we define the space of test functions 
\bea
\label{def:moreregulardomainDARdenseinDM}
\nonumber\mathcal{D}_{R} &:=& \Big\{X=(\Phi, T), \ \ X\in C^{\frac{\sqrt{R}}{2}}([0,Z_a])\times C^{\frac{\sqrt{R}}{2}}([0,Z_a])\Big\}\\
&&\cap\Big\{X \ \ / \ \ (-\mathcal{M}+R)X\in C^\infty([0,Z_a])\times C^\infty([0,Z_a])\Big\}.
\eea
In steps 1 to 3 below, we prove \eqref{accretivityalmost} for $X\in \mathcal D_R$  so that all integrations by parts in steps 1 to 3 are justified, and all boundary terms at $Z=Z_a$ vanish due to the vanishing of $g$ at $Z=Z_a$. In steps 4 and 5, for any smooth $F$ on $[0,Z_a]$, we show existence and uniqueness of a solution $X\in\mathbb{H}_{2k}$ to $(-\mathcal{M}+R)X=F$ for $R>R^*(k)$ large enough. In step 6, we prove that $\mathcal{D}_R$ is dense in $D(\mathcal M)$. In step 7, we conclude the proof of \eqref{accretivityalmost} and \eqref{estmaximal}.

\vspace{0.3cm}

\noindent{\bf step 1} Main integration by parts. Let $X\in \mathcal D_R$ for $R>R^*(k)$ large enough. We aim at proving \eqref{accretivityalmost} and split the computation in two:
$$
\left|\begin{array}{lll}
\displaystyle\la X,\tilde{X}\ra_1=-(\mathcal L_g\Phi_k,\Phi_{k})_g+(T_k,\tilde{T}_k)_g,\\[2mm]
\displaystyle\la X,\tilde{X}\ra_3=\int \chi \Phi\overline{ \tilde{\Phi}} +\int \chi T\overline{\tilde{T}}.
\end{array}\right.
$$
In step 1, we consider the principal part. We compute from \eqref{commutationderivative}:
\bee
&&-\Re\la \mathcal MX,X\ra_1=\Re( \mathcal L_g\Delta ^k(\mathcal MX)_\Phi,\Phi_k)_g-\Re(\Delta ^k(\mathcal MX)_T,T_k)_g\\
&= & -\Re\left\{\int \nabla\left[-aH_2\Lambda \Phi_k-2ak(H_2+\Lambda H_2)\Phi_k+ T_k+(\widetilde{\matchal M_k}X)_\Phi\right]\cdot \overline{\nabla \Phi_k} Z^2(-D_a)Z^{d-1}\mu^2gdZ \right\}\\
& -& \Re\left\{\int \left[\matchal L_g \Phi_k-(2-a)H_2\Lambda T_k-2k(2-a)(H_2+\Lambda H_2) T_k +A_2T_k +(\widetilde{\matchal M_k}X)_T\right]\overline{T_k}gdZ\right\}\\
& = & -\Re\left\{\int \nabla\left[-aH_2\Lambda \Phi_k-2ak(H_2+\Lambda H_2)\Phi_k+(\widetilde{\matchal M_k}X)_\Phi\right]\cdot \overline{\nabla \Phi_k} Z^2(-D_a)Z^{d-1}g\mu^2dZ \right\}\\
& -& \Re\left\{\int \left[-(2-a)H_2\Lambda T_k-2k(2-a)(H_2+\Lambda H_2) T_k +A_2T_k +(\widetilde{\matchal M_k}X)_T\right]\overline{T_k}gdZ\right\}.
\eee
\noindent\underline{\em $T_k$ terms}. We use 
$$\Re\left(\int fh \overline{\Lambda h}\right)=-\frac 12\int |h|^2 f\left(d+\frac{\Lambda f}f\right)$$ to compute
\bee
-\Re\left\{\int \left[-(2-a)H_2\Lambda T_k\right]\overline{T_k}gdZ\right\}=-\frac{2-a}{2}\int |T_k|^2gH_2\left(d+\frac{\Lambda g}{g}+\frac{\Lambda H_2}{H_2}\right)
\eee
and hence
\bee
- \Re\left\{\int \left[-(2-a)H_2\Lambda T_k-2k(2-a)(H_2+\Lambda H_2) T_k +A_2T_k\right]\overline{T_k}gdZ\right\}=(2-a)\int A_5 H_2|T_k|^2gdZ
\eee
with
\be
\label{formulaa5}
A_5:=  -\frac{1}{2}\left[d+\frac{\Lambda g}{g}+\frac{\Lambda H_2}{H_2}\right]+2k\left(1+\frac{\Lambda H_2}{H_2}\right) -\frac{A_2}{(2-a)H_2}.
\ee

\noindent\underline{\em $\Phi_k$ terms}. We first compute:
\bee
&&-\Re\left\{\int \nabla\left[-2ak(H_2+\Lambda H_2)\Phi_k\right]\cdot \overline{\nabla \Phi_k} Z^2(-D_a)Z^{d-1}gdZ \right\}\\
& =&  2ak\int (H_2+\Lamdba H_2)|\nabla \Phi_k|^2Z^2(-D_a)Z^{d-1}gdZ\\
&+& 2ak\Re\left\{\int \Phi_k\nabla(H_2+\Lambda H_2)\cdot \overline{\nabla \Phi_k} Z^2(-D_a)Z^{d-1}gdZ \right\}\\
& = & 2ak\int (H_2+\Lamdba H_2)|\nabla \Phi_k|^2Z^2(-D_a)Z^{d-1}gdZ\\
& - & ak\int |\Phi_k|^2\nabla \cdot \left( Z^2(-D_a)\nabla(H_2+\Lambda H_2) g\right)Z^{d-1}dZ.
\eee

For the second term:
\bee
 &&-\Re\left\{\int \nabla\left[-aH_2\Lambda \Phi_k\right]\cdot \overline{\nabla \Phi_k} Z^2(-D_a)Z^{d-1}gdZ\right\}= -a\Re\left\{\int \pa_Z(H_2\Lambda \Phi_k)H_2\overline{\Lambda \Phi_k} \frac{D_aZ^{d}}{H_2} g dZ\right\}\\
 & = & \frac a2\int |H_2\Lambda \Phi_k|^2\frac{D_aZ^{d}g}{H_2}\left(\frac{\pa_ZD_a}{D_a}+\frac{d}{Z}-\frac{\pa_ZH_2}{H_2}+\frac{\pa_Z g}{g}\right)dZ\\
 & = & -\frac a2\int |\pa_Z\Phi_k|^2H_2\left(\frac{\Lambda D_a}{D_a}+d-\frac{\Lambda H_2}{H_2}+\frac{\Lambda g}{g}\right)(-D_a) gZ^2Z^{d-1}dZ.
 \eee
  We have therefore obtained the formula: 
 \bea
 \label{cnevneoenoevonoe}
\nonumber&& -\Re\la \mathcal MX,X\ra_1=   (2-a)\int A_5H_2|T_k|^2 g+\mu^2a\int |\nabla \Phi_k|^2A_6Z^2(-D_a)Z^{d-1}gdZ\\
& - &  \mu^2ak\int |\Phi_k|^2\nabla \cdot \left( Z^2(-D_a)\nabla(H_2+\Lambda H_2) g \right)Z^{d-1}dZ\\
\nonumber&- &  \mu^2\Re\left\{\int \nabla(\widetilde{\matchal M_k}X)_\Phi\cdot \overline{\nabla \Phi_k} Z^2(-D_a)Z^{d-1}gdZ\right\}-  \Re\left\{\int (\widetilde{\matchal M_k}X)_T\overline{T_k}gdZ\right\}
 \eea
 where we have defined
 \bee
 A_6 = 2k(H_2+\Lambda H_2)-\frac{H_2}{2}\left(\frac{\Lambda D_a}{D_a}+d-\frac{\Lambda H_2}{H_2}+\frac{\Lambda g}{g}\right).
 \eee
 We now claim the following lower bounds on $A_5,A_6$: there exist universal constants $k^*\gg 1$, $0<c^*,a^*\ll 1$ such that for all $k\ge k^*$ and $0<a<a^*$, 
\be
\label{lowerbounda5asix}
\forall 0\le Z\le Z_1, \ \ \left|\begin{array}{l}
A_5 \ge\frac{c^*k}{Z_a-Z}\\
A_6 \ge \frac{c^*k}{Z_a-Z} 
\end{array}\right.
\ee
\noindent{\em Proof of \eqref{lowerbounda5asix}}. Recall \eqref{equationforthemeasure}, \eqref{formulag}:
\bee
&&-\frac{\Lambda g}{g}=\frac{1}{(-D_a)}\Big\{-(d-1)\sigma^2+(1-w)^2-(d+1)D_a-\Lambda D_a\\
\nonumber && +4k\left[(1-a)^2(1-w) -\sigma F-(1-a)^2(1-w)(w+\Lambda w)\right]-\frac{\tilde{A}_2}{\mu^2}\Big\}\\
& = & \frac{4k}{(-D_a)}\left[(1-w)-\sigma F-(1-w)(w+\Lambda w)+O\left(a+\frac{1}{k}\right)\right]
\eee

and hence from \eqref{formulaa5}:
\bee
A_5&=&  -\frac{1}{2}\left[d+\frac{\Lambda g}{g}+\frac{\Lambda H_2}{H_2}\right]+2k\left(1+\frac{\Lambda H_2}{H_2}\right) -\frac{A_2}{(2-a)H_2}\\
& = & \frac{2k}{(-D_a)}\left[(1-w)-\sigma F-(1-w)(w+\Lambda w)+O\left(a+\frac{1}{k}\right)\right]+2k\left(1-\frac{\Lamdba w}{1-w}+O\left(\frac{1}{k}\right)\right)\\
& = & \frac{2k}{(-D_a)}\left[(1-w)-\sigma F-(1-w)(w+\Lambda w)+(-D_a)\left(1-\frac{\Lamdba w}{1-w}\right)+O\left(a+\frac{1}{k}\right)\right]\\
& = & \frac{2k}{(-D_a)}\left[(1-w)-\sigma F-(1-w)(w+\Lambda w)+(-\Delta)\left(1-\frac{\Lamdba w}{1-w}\right)+O\left(a+\frac{1}{k}\right)\right].
\eee
We now compute for $Z\le Z_2$
\bee
&&(1-w)-\sigma F-(1-w)(w+\Lambda w)+(-\Delta)\left(1-\frac{\Lamdba w}{1-w}\right)\\
& = & (1-w)(1-w-\Lambda w)-\sigma F+(\sigma^2-(1-w)^2)\frac{1-w-\Lambda w}{1-w}\\
& = & \frac{\sigma^2(1-w-\Lambda w)}{1-w}-\sigma F=  \frac{\sigma^2}{1-w}\left[1-w-\Lambda w-\frac{1-w}{\sigma}F\right]\\
& \ge & c\sigma^2
\eee
from the fundamental coercivity bound \eqref{coercivityquadrcouplinginside}, and hence for $Z\le Z_a$ and $a<a^*$ small enough:
$$A_5\ge \frac{kc\sigma^2}{(-D_a)}\ge\frac{kc^*}{Z_a-Z}$$
for some $c^*$ independent of $k,a$. Similarly:
\bee
A_6& =& \frac{2kH_2}{-D_a}\left\{\left[1+\frac{\Lambda H_2}{H_2}+O\left(\frac 1k\right)\right](-D_a)+(1-w)-\sigma F-(1-w)(w+\Lambda w)+O\left(a+\frac{1}{k}\right)\right\}\\
& = & \frac{2k\mu(1-w)}{(-D_a)}\left[(1-w)-\sigma F-(1-w)(w+\Lambda w)+(-\Delta)\left(1-\frac{\Lamdba w}{1-w}\right)+O\left(a+\frac{1}{k}\right)\right]\\
& \ge & \frac{kc^*}{Z_a-Z}
\eee
arguing as above. This concludes the proof of \eqref{lowerbounda5asix}.\\
 
\noindent{\bf step 2} No derivatives term. We compute
\bea
\label{poeutinieog}
\nonumber &&-\Re\la \mathcal MX,X\ra_3=-\Re\left\{\int\chi(\mathcal MX)_\Phi\overline{\Phi}Z^{d-1}dZ+\int \chi (\mathcal MX)_T\overline{T}Z^{d-1}dZ\right\}\\
\nonumber& = & -\Re\left\{\int \chi\left[-aH_2\Lambda \Phi+T\right]\overline{\Phi}Z^{d-1}dZ\right\}\\
\nonumber&-&\Re\left\{\int\chi\left[(p-1)Q\Delta \Phi-(1-a)^2H_2^2\Lambda^2\Phi+\tilde{A_2}\Lambda \Phi +A_3\Phi -(2-a)H_2\Lambda T +A_2T\right]\overline{T}\right\}\\
& =& O\left(\int (\chi+|\Lambda \chi|)\Big(|\Phi|^2+|\pr_Z\Phi|^2+|\pr_Z^2\Phi|^2+|T|^2\Big)\right).
\eea

\noindent{\bf step 3} Accretivity in $\mathcal D_0$. We compute from \eqref{poeutinieog}, \eqref{cnevneoenoevonoe}:
\bee
&&-\Re\la \matchal MX,X\ra = -\Re\la \matchal MX,X\ra_1 -\Re\la \matchal MX,X\ra_3\\
& = & (2-a)\int A_5H_2|T_k|^2 g+\mu^2a\int |\nabla \Phi_k|^2A_6Z^2(-D_a)Z^{d-1}gdZ\\
& - &  \mu^2ak\int |\Phi_k|^2\nabla \cdot \left( Z^2(-D_a)\nabla(H_2+\Lambda H_2) g \right)Z^{d-1}dZ\\
\nonumber&- &  \mu^2\Re\left\{\int \nabla(\widetilde{\matchal M_k}X)_\Phi\cdot \overline{\nabla \Phi_k} Z^2(-D_a)Z^{d-1}gdZ\right\}-  \Re\left\{\int (\widetilde{\matchal M_k}X)_T\overline{T_k}gdZ\right\}\\
& + & O\left(\int (\chi+|\Lambda \chi|)\Big(|\Phi|^2+|\pr_Z\Phi|^2+|\pr_Z^2\Phi|^2+|T|^2\Big)\right).
\eee
We lower bound from \eqref{lowerbounda5asix}:
\bee
&&(2-a)\int A_5H_2|T_k|^2 g+\mu^2a\int |\nabla \Phi_k|^2A_6Z^2(-D_a)Z^{d-1}gdZ\\
&\ge & c^*ak\left[\int \frac{|T_k|^2}{Z_a-Z}+|\nabla \Phi_k|^2\frac{Z^2(-D_a)}{Z_a-Z}gZ^{d-1}dZ\right]
\eee
The smoothness and boundedness of the profile together with \eqref{equationforthemeasure}, \eqref{formulag} ensure
that
$$
\left|\nabla \cdot \left[ Z^2(-D_a)\nabla(H_2+\Lambda H_2)g \right]\right| \le C_k\frac{Z^2(-D_a)g}{Z_a-Z}\le C_k g$$
and in view of  \eqref{contlmak},
\bee
\nonumber&&\Bigg| -  \Re\left\{\int \nabla(\widetilde{\matchal M_k}X)_\Phi\cdot \overline{\nabla \Phi_k} Z^2(-D_a)Z^{d-1}gdZ\right\}-  \Re\left\{\int (\widetilde{\matchal M_k}X)_T\overline{T_k}gdZ\right\}
\\
&\leq& C_k\left(\int |\nabla \Phi_k|^2Z^2(-D_a)gZ^{d-1}dZ\right)^\frac{1}{2}\left(\sum_{j=0}^{2k}\int |\pa_Z^j\Phi|^2gZ^{d-1}dZ\right)^{\frac{1}{2}}\\
&+&C_k\left(\int |T_k|^2gZ^{d-1}dZ\right)^\frac{1}{2}\left[\left(\sum_{j=0}^{2k-1}\int |\pa^j_ZT|^2gZ^{d-1}dZ\right)^{\frac{1}{2}}+\left(\sum_{j=0}^{2k}\int |\pa_Z^j\Phi|^2gZ^{d-1}dZ\right)^{\frac{1}{2}}\right]
\eee
The collection of above bounds yields:
\bee
&&-\Re\la \matchal MX,X\ra\ge   c^*ak\left[\int \frac{|T_k|^2}{Z_a-Z}gZ^{d-1}dZ+\int |\nabla \Phi_k|^2\frac{Z^2(-D_a)}{Z_a-Z}gZ^{d-1}dZ\right]\\
& - & C_k\left[\sum_{j=0}^{2k}\int |\pa_Z^j\Phi|^2gZ^{d-1}dZ+\sum_{j=0}^{2k-1}\int |\pa^j_ZT|^2gZ^{d-1}dZ\right].
\eee
We conclude using \eqref{coercivityformula} with $N=N(a,k)$ large enough and its analogue for $T$:
$$
-\Re\la \mathcal MX,X\ra\geq c^*ak\la X,X\ra-C_{a,k}\sum_{i=1}^N\Big((\Phi,\Pi_i)^2_g+(T,\mathcal T_i)_g^2\Big).$$ 
Therefore,
$$-\Re\la ({\mathcal{M}}-\mathcal A)X,X\ra\geq c^*ak\la X,X\ra+ \sum_{i=1}^{N}\Big(\la X,X_{i,1}\ra^2+\la X,X_{i,2}\ra^2-C_{a,k}\left[(\Phi,\Pi_i)^2_g+(T,\mathcal T_i)_g^2\right]\Big).
 $$
 The linear from $$X=(\Phi,T)\mapsto \sqrt{C_{a,k}}(\Phi,\Pi_i)_g$$ from $(\Bbb H_{2k},\la \cdot\ra)$ into $\Bbb C$ is continuous from Cauchy-Schwarz and \eqref{lowerboundhardy}, and hence by Riesz theorem, there exists $X_i\in \Bbb H_{2k}$ such that $$\forall X\in \Bbb H_{2k}, \ \ \la X,X_{i}\ra=(\Phi,\Pi_i)_g,$$ and similarly for $\mathcal T_i$, and \eqref{accretivityalmost} follows for $X\in \mathcal D_R$.\\
 
\noindent{\bf step 4} ODE formulation of maximality. Our goal, in steps 4 to step 6, is to prove that forall $R>0$ large enough,
\be\label{nnneone:0}
\forall F\in C^{\infty}([0,Z_a]), \ \ \exists!\, X\in\mathbb{H}_{2k}\ \ \mbox{such that} \ \ (-\matchal M+ R)X=F.
\ee
\eqref{nnneone:0} corresponds to solving 
$$
\left|\begin{array}{ll}
-\left[-aH_2\Lambda\right]\Phi-T+R\Phi=F_\Phi,\\ [2mm]
-\left\{\left[(p-1)Q\Delta -(1-a)^2H_2^2\Lambda^2+\tilde{A_2}\Lambda +A_3\right]\Phi-(2-a)H_2\Lambda T +A_2T\right\}+RT=F_T.
\end{array}\right.
$$
Solving for $T$: 
\bea\label{vnenovenneo}
T &=& (aH_2\Lamdba +R)\Phi-F_\Phi,
\eea
we look for $\Phi$ -- solution to  the second order elliptic equation:
\bee
&&\left[(p-1)Q\Delta -(1-a)^2H_2^2\Lambda^2+\tilde{A_2}\Lambda +A_3\right]\Phi +\Big[-(2-a)H_2\Lambda +A_2\Big]\left[aH_2\Lambda\Phi+R\Phi-F_\Phi\right]\\
&=& -F_T+R\left(aH_2\Lambda\Phi+R\Phi-F_\Phi\right)
\eee
i.e.
\bee
&&(p-1)Q\Delta \Phi -H_2^2\Lambda^2\Phi+\Lambda \Phi\left[\tilde{A}_2 +aH_2A_2-2RH_2-a(2-a)H_2\Lambda H_2\right]+(A_3 +RA_2 -R^2)\Phi\\
& = & -F_T-RF_\Phi +\Big[-(2-a)H_2\Lambda +A_2\Big]F_\Phi.
\eee
Now, we have
\bee
(p-1)Q\Delta \Phi -H_2^2\Lambda^2\Phi &=& \Big((p-1)Q -H_2^2Z^2\Big)\pr_Z^2\Phi + \left(\frac{(d-1)(p-1)Q}{Z} -H_2^2Z\right)\pr_Z\Phi
\eee
and hence
\bee
&& \Big((p-1)Q -H_2^2Z^2\Big)\pr_Z^2\Phi+\Bigg\{\frac{(d-1)(p-1)Q}{Z} -H_2^2Z\\
&&+Z\left[\tilde{A}_2+aH_2A_2- 2RH_2-a(2-a)H_2\Lambda H_2\right]\Bigg\}\pr_Z\Phi+(A_3 +RA_2 -R^2)\Phi\\
& = & -F_T-RF_\Phi +\Big[-(2-a)H_2\Lambda +A_2\Big]F_\Phi.
\eee
Since $(p-1)Q=\mu^2Z^2\sigma^2$, we have
\bee
&& \Big((p-1)Q -H_2^2Z^2\Big)\pr_Z^2\Phi+\Bigg\{\frac{(d-1)(p-1)Q}{Z} -H_2^2Z\\
&&+Z\left[\tilde{A}_2+aH_2A_2- 2RH_2-a(2-a)H_2\Lambda H_2\right]\Bigg\}\pr_Z\Phi\\
&=& \Big(\mu^2\sigma^2 -H_2^2\Big)Z^2\pr_Z^2\Phi+\Bigg\{(d-1)\mu^2Z\sigma^2 -H_2^2Z\\
&&+Z\left[\tilde{A}_2+aH_2A_2- 2RH_2-a(2-a)H_2\Lambda H_2\right]\Bigg\}\pr_Z\Phi\\
&=& \frac{1}{Z^{d-1}\varpi}\pr_Z\left(Z^{d-1}\varpi \Big(\mu^2\sigma^2 -H_2^2\Big)Z^2\pr_Z\Phi\right)
\eee
with
\bee
&&\left(\frac{\pr_Z\varpi}{\varpi}+\frac{d-1}{Z}\right)\Big(\mu^2\sigma^2 -H_2^2\Big)Z^2+2Z\Big(\mu^2\sigma^2 -H_2^2\Big)+\Big(2\mu^2\sigma\pr_Z\sigma -2H_2\pr_ZH_2\Big)Z^2\\
&=& (d-1)\mu^2Z\sigma^2 -H_2^2Z+Z\left[\tilde{A}_2+aH_2A_2- 2RH_2-a(2-a)H_2\Lambda H_2\right],
\eee
i.e.
\bee
\frac{\pr_Z\varpi}{\varpi} &=& -\frac{2}{Z}-\frac{2\mu^2\sigma\pr_Z\sigma -2H_2\pr_ZH_2}{\mu^2\sigma^2 -H_2^2}\\
&-& \frac{2RH_2-(d-2)H_2^2 -\tilde{A}_2-aH_2A_2+a(2-a)H_2\Lambda H_2}{\Big(\mu^2\sigma^2 -H_2^2\Big)Z}.
\eee
Recalling $H_2=\mu(1-w)$ yields
\bee
\frac{\pr_Z\varpi}{\varpi} &=& -\frac{2}{Z}-\frac{\pr_Z[\sigma^2 -(1-w)^2]}{\sigma^2 -(1-w)^2}\\
&-& \frac{\frac{2(1-w)}{\mu}R -(d-2)(1-w)^2 -\frac{\tilde{A}_2}{\mu^2}- a(1-w)\frac{A_2}{\mu} -a(2-a)(1-w)\Lambda w}{Z\Big(\sigma^2 -(1-w)^2\Big)}.
\eee
We therefore define
\be
\label{deftho}
\varpi(Z)=\left|\begin{array}{ll}\displaystyle \frac{e^{-F_-(Z)}}{Z^2(\sigma^2 -(1-w)^2)}, \ \ \ \ \mbox{for}\  \ 0\leq Z\leq Z_2,\\[4mm]
\displaystyle\frac{e^{-F_+(Z)}}{Z^2(\sigma^2 -(1-w)^2)}, \ \ \ \ \mbox{for}\  \ Z>Z_2.
\end{array}\right.
\ee
where\footnote{The choice of the lower limits $\frac{Z_2}{2}$ and $2Z_2$ in the definition of $F_\pm$ is arbitrary but
dictate the choice of the constants $C_\pm$ in such a way as to ensure that
$\lim_{Z\uparrow Z_2}F_-(Z)-\lim_{Z\downarrow Z_2} F_+(Z)=0$. The additional degree of freedom in the choice of $C_\pm$ is used to fix an overall normalization  of $\varpi$.}
\bee
F_-(Z) &=& \int_{\frac{Z_2}{2}}^Z\frac{\frac{2(1-w)}{\mu}R -(d-2)(1-w)^2 -\frac{\tilde{A}_2}{\mu^2}- a(1-w)\frac{A_2}{\mu} -a(2-a)(1-w)\Lambda w}{Z'\Big(\sigma^2 -(1-w)^2\Big)}dZ' + C_-,\\
F_+(Z) &=& \int_{2Z_2}^Z\frac{\frac{2(1-w)}{\mu}R -(d-2)(1-w)^2 -\frac{\tilde{A}_2}{\mu^2}- a(1-w)\frac{A_2}{\mu} -a(2-a)(1-w)\Lambda w}{Z'\Big(\sigma^2 -(1-w)^2\Big)}dZ'+C_+.
\eee
In view of the above, we have obtained the elliptic equation:
\be
\label{ellipticproblem}
\left|\begin{array}{ll}
-\frac{1}{Z^{d-1}\varpi}\pr_Z\left(Z^{d-1}\varpi \Big(\sigma^2 -(1-w)^2\Big)Z^2\pr_Z\Phi\right)+\frac{1}{\mu^2}(R^2 -A_2R -A_3)\Phi=H,\\
H= \frac{1}{\mu^2}\left\{F_T+RF_\Phi +\Big[(2-a)H_2\Lambda -A_2\Big]F_\Phi\right\},
\end{array}\right.
\ee
with $T$ recovered by \eqref{vnenovenneo}. As $Z\to Z_2$, we have from \eqref{siggaddoadne}: $$\Delta(Z)=\frac{|\l_+|}{Z_2}(Z-Z_2)+O((Z-Z_2)^2)$$ and hence $$Z(\sigma^2-(1-w)^2)=|\l_+|(Z_2-Z)\left[1+O(Z-Z_2)\right]$$ and hence
\bee
&&\frac{\frac{2(1-w)}{\mu}R -(d-2)(1-w)^2 -\frac{\tilde{A}_2}{\mu^2}- a(1-w)\frac{A_2}{\mu} -a(2-a)(1-w)\Lambda w}{Z\Big(\sigma^2 -(1-w)^2\Big)}\\
&=&\frac{\frac{2\sigma_2}{{\mu}|\l_+|}R\left[1+O\left(\frac{1}{R}\right)\right]}{(Z_2-Z)\left[1+O(Z-Z_2)\right]}
\eee
Since the profile passes through $P_2$ in a $\mathcal C^\infty$ way,
we obtain the development of the measure at $P_2$: for any $M\ge 1$,
\be
\label{DLrhorogin}
\varpi(Z)=|Z_2-Z|^{c_{\varpi}}\left[1+\sum_{m=0}^{M}d_{\sigma,m,R}(Z_2-Z)^{m}+O_M\Big(|Z_2-Z|^{M+1})\right],
\ee 
where 
\be
\label{estiamteerrurr}
c_{\varpi}=\frac{2\sigma_2}{{\mu}|\l_+|}R\left[1+O\left(\frac 1R\right)\right]\ge c^*R>0
\ee
for $R>R^*$ large enough.  Note that the above choice of $C_\pm$ is made to fix the normalization constant 
in front of $|Z_2-Z|^{c_{\varpi}}$ to be equal to $1$.\\

\noindent{\bf step 5} Solving \eqref{ellipticproblem}. We analyze the singularity of \eqref{ellipticproblem} at $P_2$ using a change of variables.\\
\noindent\underline{$0\le Z<Z_2$}. We let $$\Phi(Z)=\Psi(Y), \ \ Y=h(Z), \ \  h(Z)=\int_{\frac{Z_2}{2}}^Z\frac{dZ'}{{Z'}^{d-1}\varpi {Z'}^2(\sigma^2-(1-w)^2)}$$ which maps \eqref{ellipticproblem} onto:
\be
\label{mainproblem}
\left|\begin{array}{ll}-\pa_Y^2\Psi +\frac{1}{\mu^2}(R^2 -A_2R -A_3)Z^{2d}\varpi^2(\sigma^2-(1-w)^2)\Psi=\tilde{H},\\
\tilde{H}=Z^{2d}\varpi^2(\sigma^2-(1-w)^2)H.
\end{array}\right.
\ee
From \eqref{DLrhorogin}, 
\bea
\label{cneineneonmmoruirt}
\nonumber &&Y=h(Z)=\int_{\frac{Z_2}{2}}^Z\frac{dz}{{z}^{d-1}\varpi {z}^2(\sigma^2-(1-w)^2)}\\
\nonumber & = & \int_{\frac{Z_2}{2}}^Z\frac{dz}{{z}^{d-1}{z}|\l_+|(Z_2-z)(Z_2-z)^{c_{\varpi}}\left[1+\sum_{m=0}^{M}d_{\sigma,m,R}(Z_2-z)^{m}+O_M\Big(|Z_2-z|^{M+1})\right]}
\\
&  =&\frac{C}{R\varpi}[1+\Gamma(Z)]
\eea
where from \eqref{estiamteerrurr}  constant $C>0$ is independent of $R$ and, choosing $M=\sqrt R$,
\be
\label{gegenric}
\Gamma(Z)=\sum_{m=1}^{\sqrt{R}}\tilde d_{\sigma,m,R}(Z_2-Z)^{m}+O\Big((Z_2-Z)^{\sqrt{R}+1}\Big)
\ee 
with similar estimates for derivatives. Hence the potential term in \eqref{mainproblem} can be expanded in $Y$ and estimated as $Y\to +\infty$ for $R$ large enough:
\be\label{eq:potexp}
\frac{1}{\mu^2}(R^2 -A_2R -A_3)Z^{2d}\varpi^2(\sigma^2-(1-w)^2)=\frac{C_R}{Y^{2+c_R}}\left[1+\sum_{j=1}^{\sqrt{R}}\frac{\tilde{\tilde {d_{j}}}}{Y^{jc_R}}+O\left(\frac{1}{Y^{c_R(\sqrt{R}+1)}}\right)\right]
\ee
for some universal constants $\tilde{\tilde {d_{j}}}$,
$$C_R=C+O\left(\frac{1}{R}\right), \ \  0<c_R=\frac{1}{c_\varpi}\lesssim \frac 1R$$ 
where $C>0$ is independent of R. 
Therefore, by an elementary fixed point argument, \eqref{mainproblem} with $\tilde{H}=0$ admits a basis of solutions $\Psi^-_1$ and $\Psi^-_2$ with the following behavior as $Y\to +\infty$
\be
\label{veninevoneoneovn}
\left|\begin{array}{l}
\Psi^-_1=1+\sum_{j=1}^{\sqrt{R}}\frac{c_{j,1}}{Y^{jc_R}}+O\left(\frac{1}{Y^{(\sqrt{R}+1)c_R}}\right)\\
\Psi^-_2=Y\left[1+\sum_{j=1}^{\sqrt{R}}\frac{c_{j,2}}{Y^{jc_R}}+O\left(\frac{1}{Y^{(\sqrt{R}+1)c_R}}\right)\right]
\end{array}\right.
\ee
with similar estimates for derivatives. The sequences $(c_{j,1})_{j=1,2}$ are uniquely determined inductively 
from \eqref{mainproblem} with $\tilde{H}=0$ using the expansion of the potential \eqref{eq:potexp}.\\
\noindent\underline{$Z_2<Z\le Z_a$}. To the right of $P_2$, we let$$\Phi(Z)=\Psi(Y), \ \ Y=h(Z), \ \  h(Z)=\int_{2Z_2}^{Z}\frac{dz}{{z}^{d-1}\varpi{z}^2(\sigma^2-(1-w)^2)}+\tilde C_+,$$
which sends\footnote{We add constant $\tilde C_+$ to match the asymptotic expansion of $Y$ in terms of $(Z_2-Z)$. In principle, it is unnecessary as it influences the terms of order $R$ and higher while we only need the universality of the expansion up to the order $\sqrt R$.} $Y\to +\infty$ as $Z\downarrow Z_2$. We construct a similar basis of homogenous solutions $\Psi^+_1$ and $\Psi^+_2$ as $Y\to +\infty$ with asymptotics given by:
$$\Psi^+_1=1+\sum_{j=1}^{\sqrt{R}}\frac{c_{j,1}}{Y^{jc_R}}+O\left(\frac{1}{Y^{(1+\sqrt{R})c_R}}\right), \ \ \Psi^+_2=Y\left[1+\sum_{j=1}^{\sqrt{R}}\frac{c_{j,2}}{Y^{jc_R}}+O\left(\frac{1}{Y^{(1+\sqrt{R})c_R}}\right)\right]$$
with the sequences $c_{j,1}$, $c_{j,2}$ the same as in \eqref{veninevoneoneovn}.

\noindent\underline{Basis of fundamental solutions}. The function $\Phi_1(Z)=\Psi^-_1(Y)$ for $Z<Z_2$ and 
$\Phi_1(Z)=\Psi^+_1(Y)$ for $Z>Z_2$,obtained by gluing $\Psi^\pm_1(Y)$ belongs to $\mathcal C^{\sqrt{R}}([0,Z_a])$ and is a solution to the homogeneous equation \eqref{mainproblem}.  Let now $\Phi_{\rm rad}(Z)$ be the radial solution to the homogeneous problem associated to \eqref{ellipticproblem} with $\Phi_{\rm rad}(0)=1$. Then the wronskian is given by 
$$W=\pa_Z\Phi_1\Phi_{\rm rad}-\pa_Z\Phi_{\rm rad}\Phi_1=\frac{W_0}{Z^{d-1}\varpi Z^2(\sigma^2-(1-w)^2)}$$
where $W_0$ is a constant. We claim $W_0\neq 0$. Indeed, otherwise $\Phi_{\rm rad}$ is proportionate to $\Phi_1$ and hence is $C^{\sqrt{R}}$ on $[0,Z_a]$. In particular, if $T_{\rm rad}$ is given by  \eqref{vnenovenneo} with $F_\Phi=0$, then $X_{\rm rad}=(\Phi_{\rm rad}, T_{\rm rad})$ satisfies $$(-\matchal M+ R)X_{\rm rad}=0\textrm{ on }(0,Z_a).$$
Since $X_{rad}$ is $C^{\sqrt{R}-1}[[0,Z_2])$, we may apply the analysis in steps 1 to 4 for $R>R^*(k)$ large enough and \eqref{accretivityalmost} holds for $X_{\rm rad}$, i.e.
\bee
0 &=& \Re\la (-\mathcal M+ R )X_{rad},X_{rad}\ra\\
&=&\Re\la (-\mathcal{M} +\mathcal{A})X_{rad},X_{rad}\ra -\Re\la \mathcal A X_{rad},X_{rad}\ra +R\|X_{rad}\|_{\mathbb{H}_{2k}}^2\\
&\geq& R\|X_{rad}\|_{\mathbb{H}_{2k}}^2 -\la \mathcal A X_{rad},X_{rad}\ra
\eee
so that for $R>R^*(k)$ sufficiently large 
\bee
\frac R2\|X_{rad}\|_{\mathbb{H}_{2k}}^2\leq 0
\eee
and hence $X_{rad}=0$ a contradiction. This concludes the proof of $W_0\neq 0$.\\
\noindent\underline{Inner solution of  the inhomogeneous problem}.  $(\Phi_{\rm rad}, \Phi_1)$ is then a basis for the homogeneous problem corresponding to \eqref{ellipticproblem}. As a consequence, the only solution to \eqref{ellipticproblem} which is $o((Z_2-Z)^{-\frac{1}{c_R}})$ at $Z=Z_2$ is given by\footnote{Note that $\Phi_{rad}(Z)\sim (Z_2-Z)^{-\frac{1}{c_R}}$ as $Z\to Z_2$ in view of the behavior of $\Psi_2$ as $Y\to +\infty$.}
\bee
\Phi(Z)=-\Phi_1(Z)\int_{0}^{Z}\frac{H(\tau)\Phi_{\rm rad}(\tau)}{W(\tau)}d\tau-\Phi_{\rm rad}(Z)\int_{Z}^{Z_2}\frac{H(\tau)\Phi_1(\tau)}{W(\tau)}d\tau.
\eee
For a smooth $H$, $\Phi$ is smooth on $[0,Z_2)$ and we study its regularity at $Z_2$. In $Y$ variables we obtain for some $Y_0$ large enough:
\be
\label{cenonneo}
\Psi(Y)=c_{Y_0,H}\Psi^-_1(Y)-\Psi^-_1(Y)\int_{Y_0}^{Y}\tilde{H}(\tau)\Psi^-_2(\tau)d\tau-\Psi^-_2(Y)\int_{Y}^{+\infty}\tilde{H}(\tau)\Psi^-_1(\tau)d\tau.
\ee 
We have from \eqref{DLrhorogin}, \eqref{cneineneonmmoruirt}:
$$
(RY)^{c_R}= \frac {1}{Z_2-Z} \left(\sum_{m=0}^{\sqrt{R}} \beta_m (Z_2-Z)^m + O(|Z_2-Z|^{\sqrt{R}+1})\right),
$$
and hence
$$
Z_2-Z = \sum_{m=1}^{\sqrt{R}} \frac {y_m} {Y^{mc_R}}+O\left(\frac{1}{Y^{(\sqrt{R}+1)c_R}}\right)
$$
with similar estimates for derivatives. In particular, a smooth function $H(Z)$ yields expansion for $\tilde H(Z)$:
\bee
\tilde H &=& (Z_2-Z)^{1+2c^{-1}_R}\left(\sum_{m=0}^{\sqrt{R}} h_m (Z_2-Z)^m+O\Big((Z_2-Z)^{\sqrt{R}+1}\Big)\right)\\
&=&\sum_{m=1}^{\sqrt{R}} \frac {q_m}{Y^{2+mc_R}}+O\left(\frac{1}{Y^{2+(\sqrt{R}+1)}c_R}\right).
\eee
Conversely, an expansion of the form 
$$
G=\sum_{m=0}^{\sqrt{R}-1} \frac {b_m}{Y^{mc_R}}+O\left(\frac{1}{Y^{\sqrt{R}c_R}}\right)
$$
defines a $C^{\sqrt{R}}$ function $G(Z)$ at $Z=Z_2$. Plugging in the asymptotic expansion for $\Psi^-_1, \Psi^-_2$ and $\tilde H$ in \eqref{cenonneo} yields
\bee
\Psi(Y)&=& c_{Y_0,H}\left(\sum_{m=0}^{\sqrt{R}}\frac {c_{m,1}}{Y^{mc_R}}+O\left(\frac{1}{Y^{(\sqrt{R}+1)c_R}}\right)\right)- \left(\sum_{m=0}^{\sqrt{R}}\frac {c_{m,1}} {Y^{mc_R}}+O\left(\frac{1}{Y^{(\sqrt{R}+1)c_R}}\right)\right)\\
&&\times\int_{Y_0}^Y 
\left(\sum_{m=0}^{\sqrt{R}} \frac {c_{m,2}} {\tau^{mc_R}}+O\left(\frac{1}{Y^{(\sqrt{R}+1)c_R}}\right)\right)\left(\sum_{j=1}^{\sqrt{R}} \frac {q_j}{\tau^{1+jc_R}}+O\left(\frac{1}{\tau^{1+(\sqrt{R}+1)c_R}}\right)\right)d\tau\\ 
&-& \left(\sum_{m=0}^{\sqrt{R}}\frac{c_{m,2}} {Y^{mc_R}} Y+O\left(\frac{1}{Y^{(\sqrt{R}+1)c_R}}\right)\right)\int_Y^\infty 
\left(\sum_{m=0}^{\sqrt{R}}\frac {c_{m,1}} {\tau^{mc_R}}+O\left(\frac{\log(\tau)}{\tau^{(\sqrt{R}+1)c_R}}\right)\right)\\
&&\times\left(\sum_{j=1}^{\sqrt{R}} \frac {q_j}{\tau^{2+jc_R}}+O\left(\frac{1}{\tau^{2+(\sqrt{R}+1) c_R}}\right)\right)d\tau=\sum_{m=0}^{\sqrt{R}-1} \frac {b_m}{Y^{mc_R}}+O\left(\frac{1}{Y^{\sqrt{R}c_R}}\right).
\eee
We therefore have proved that for $H\in \mathcal C^\infty([0,Z_2])$, there exists a unique solution $\Phi$ to \eqref{ellipticproblem} on $[0,Z_2]$ which is  $o((Z_2-Z)^{-\frac{1}{c_R}})$ at $Z=Z_2$. Furthermore, this solution is smooth on $[0,Z_2)$, and is $C^{\sqrt{R}}$ at $Z=Z_2$ where it admits an asymptotic expansion
\bea\label{eq:asymptoticexapnsionofsolutionMplusRXequalFinnersolutionZ2}
\Phi(Z) &=& \sum_{j=0}^{\sqrt{R} -1}c_{j,\Phi}(Z_2-Z)^j+O\Big((Z_2-Z)^{\sqrt{R}}\Big).
\eea
\noindent\underline{Outer solution of the inhomogeneous problem}. We argue similarly, considering  the basis 
$\Phi_1(Z)$ and $\Phi_{rad}^+(Z)$ with $\Phi_{rad}^+(Z_a)=1$, for $Z_2<Z\le Z_a$ and construct $\Phi$ solution to \eqref{ellipticproblem} on $[Z_2,Z_a]$ which is smooth on $(Z_2, Z_a]$, $o((Z_2-Z)^{-\frac{1}{c_R}})$ at $Z=Z_2$ and $C^{\sqrt{R}}$ at $Z=Z_2$. Furthermore, $\Phi$ admits at $Z=Z_2$ the following asymptotic expansion analogous to \eqref{eq:asymptoticexapnsionofsolutionMplusRXequalFinnersolutionZ2}
$$ \Phi(Z)=\sum_{j=0}^{\sqrt{R} -1}\tilde{c}_{j,\Phi}(Z_2-Z)^j+O\Big((Z_2-Z)^{\sqrt R}\Big).$$ The asymptotic expansion is uniquely determined from the equation \eqref{ellipticproblem} and the first coefficient
$\tilde{c}_{0,\Phi}$. We now recall that the function $\Phi_1$ belongs to $\mathcal C^{\sqrt R}[0,Z_a]$ and 
$\Phi_1(Z_2)=1$. By adding $\Phi_1$ to the above expansion, we obtain another solution in which we can force the condition
$$\tilde{c}_{0,\Phi}=c_{0,\Phi}$$ with $c_{0,\Phi}$ appearing in \eqref{eq:asymptoticexapnsionofsolutionMplusRXequalFinnersolutionZ2}. As a result, the asymptotic expansions of the inner and outer solutions are matched to order $\sqrt R$, so that the constructed solution is $\matchal C^{\sqrt{R}}$ at $Z_2$. Finally, we have shown that given any smooth function $H$ on $[0,Z_a]$, there exists a unique solution $\Phi$ to \eqref{ellipticproblem} on $[0,Z_a]$ which is $o((Z_2-Z)^{-\frac{1}{c_R}})$ at $Z=Z_2$. Furthermore, this solution is smooth for $Z\neq Z_2$ and $C^{\sqrt{R}}$ at $Z=Z_2$. In particular, with $T$ recovered by \eqref{vnenovenneo} and smooth for $Z\neq Z_2$ and $C^{\sqrt{R} -1}$ at $Z=Z_2$, we have that $(\Phi, T)\in \Bbb H_{2k}$ for $R>R(k)$ large enough. Also, since $(\Phi, T)$ with $\Phi\sim (Z_2-Z)^{-\frac{1}{c_R}}$ near $Z=Z_2$ does not belong to $\Bbb H_{2k}$\footnote{Recall that $(c_R)^{-1}\gtrsim R\gg1$}, we have now proved that, in fact, there exists  a unique solution $X=(\Phi, T)$ to $(-\mathcal{M}+R)X=F$ on $[0,Z_a]$ in $\Bbb H_{2k}$, which concludes the proof of \eqref{nnneone:0}.\\

\noindent{\bf step 6} Density of $\mathcal{D}_{R}$. We now prove that $\mathcal{D}_{R}$ given by \eqref{def:moreregulardomainDARdenseinDM} is dense in $D(\mathcal M)$. Indeed,  
if $X\in D(\mathcal{M})$, then $X\in\mathbb{H}_{2k}$ and $\mathcal{M}X\in\mathbb{H}_{2k}$ so that there exists a sequence  $(Y_n)_{n\in\mathbb{N}}\in C^\infty([0,Z_a],\Bbb{C}^2)$ with
$$\lim_{n\to +\infty}Y_n\to (-\mathcal{M}+R)X\textrm{ in }\mathbb{H}_{2k}.$$ 
From step 5, for each integer $n$, there exist a unique $Z_n\in\mathcal{D}_{R}$ solution to
\bee
(-\mathcal{M}+R)Z_n=Y_n, \ \ Z_n\in \Bbb H_{2k},
\eee 
and hence 
$$(-\mathcal{M}+R)Z_n\to (-\mathcal{M}+R)X\textrm{ in }\mathbb{H}_{2k}.$$
Thus, to conclude, it remains to check that $Z_n$ converges to $X$ in $\mathbb{H}_{2k}$. To this end, since $Z_n\in\mathcal{D}_{R}$,    \eqref{accretivityalmost} holds for $Z_n-Z_q$ and thus:
\bee
\Re\la Y_n-Y_q, Z_n-Z_q\ra &=& \Re\la (-\mathcal M+ R )(Z_n-Z_q), Z_n-Z_q\ra\\
&=& \Re\la (-\mathcal{M} +\mathcal{A})(Z_n-Z_q),Z_n-Z_q\ra -\Re\la \mathcal A(Z_n-Z_q),Z_n-Z_q\ra \\
&&+R\|Z_n-Z_q\|_{\mathbb{H}_{2k}}^2\\
&\geq& R\|Z_n-Z_q\|_{\mathbb{H}_{2k}}^2 -\Re\la \mathcal A(Z_n-Z_q),Z_n-Z_q\ra
\eee
so that, since $\mathcal A$ is a bounded operator, we infer for $R$ sufficiently large 
\bee
\frac R2\|Z_n-Z_q\|_{\mathbb{H}_{2k}} &\leq& \|Y_n-Y_q\|_{\mathbb{H}_{2k}}.
\eee
In view of the convergence of $(Y_n)$ in $\mathbb{H}_{2k}$, we deduce that $Z_n$ is a Cauchy sequence in $\mathbb{H}_{2k}$ and hence converges, i.e. 
\bee
\lim_{n\to +\infty}Z_n\to Z\textrm{ in }\mathbb{H}_{2k}, \ \ \ \ Z\in \mathbb{H}_{2k}.
\eee
Since $(-\mathcal{M}+R)Z_n$ converges to $(-\mathcal{M}+R)X$ in $\mathbb{H}_{2k}$, we infer
\bee
(-\mathcal{M}+R)(Z-X)=0\textrm{ in }\mathcal{D}'(0,Z_a), \ \ Z-X\in\mathbb{H}_{2k}.
\eee
The uniqueness statement in \eqref{nnneone:0} applied for $F=0$ yields
$Z=X$.
Thus $Z_n\to X$ and $(-\mathcal{M}+R)Z_n\to (-\mathcal{M}+R)X$ in $\mathbb{H}_{2k}$. Finally, we have obtained a sequence $Z_n\in \mathcal{D}_{R}$ such that $Z_n\to X$ in $D(\mathcal{M})$, and hence $\mathcal{D}_{R}$ is dense in $D(\mathcal{M})$ as claimed.\\

\noindent{\bf step 7} Maximal accretivity. We have proved in steps 1 to 3 that \eqref{accretivityalmost} holds for $X\in \mathcal{D}_{R}$, i.e.
\bee
\forall X\in\mathcal{D}_{R}, \ \ \Re\la (-\matchal M+\matchal A) X,X\ra \geq c^*ak \la X,X\ra.
\eee
Since $\mathcal{D}_{R}$ is dense in $D(\mathcal{M})$, in view of step 6, we have
\bee
\forall X\in \matchal D(\matchal M), \ \ \Re\la (-\matchal M+\matchal A) X,X\ra \geq c^*ak\la X,X\ra
\eee
which concludes the proof of the accretivity property \eqref{accretivityalmost}.\\
We now claim: 
\be
\label{nnneone}\forall F\in \Bbb H_{2k}, \ \ \exists X\in D(\matchal M)\ \ \mbox{such that} \ \ (-\matchal M+ R)X=F.
\ee
Indeed, since $F\in \Bbb H_{2k}$, by density, there exists  
\bee
\lim_{n\to +\infty}F_n\to F\textrm{ in }\mathbb{H}_{2k}, \ \ \ \ F_n\in C^\infty([0,Z_a]).
\eee
Since $F_n\in C^\infty([0,Z_a])$, by \eqref{nnneone:0}, there exists $X_n\in \mathbb{H}_{2k}$ -- solution to 
\bee
(-\matchal M+ R)X_n=F_n.
\eee
Using \eqref{accretivityalmost} and arguing as in step 6, we have for $R$ sufficiently large 
\bee
\frac R2\|X_n-X_q\|_{\mathbb{H}_{2k}} &\leq& \|F_n-F_q\|_{\mathbb{H}_{2k}}.
\eee
In view of the convergence of $(F_n)$ in $\mathbb{H}_{2k}$, we deduce that $X_n$ is a Cauchy sequence in $\mathbb{H}_{2k}$ and hence converges, i.e. 
\bee
\lim_{n\to +\infty}X_n\to X\textrm{ in }\mathbb{H}_{2k}, \ \ \ \ X\in \mathbb{H}_{2k}.
\eee
On the other hand, since $(-\mathcal{M}+R)X_n=F_n$ convergence to $F$ in $\mathbb{H}_{2k}$, we infer
\bee
(-\matchal M+ R)X=F, \ \ X\in D(\matchal M)
\eee
which concludes the proof of \eqref{nnneone}.\\
Finally, \eqref{accretivityalmost} and a classical and elementary induction argument
ensures that the maximality property \eqref{estmaximal} is implied by: 
$$\exists R>0, \ \ \forall F\in \Bbb H_{2k}, \ \ \exists X\in \matchal D(\matchal M)\ \ \mbox{such that} \ \ (-\tilde{\matchal M}+R)X=F.$$
Indeed, let  $R>0$ large enough and $F\in \Bbb H_{2k}$. Since $\mathcal A$ is a bounded operator, for $R$ large enough, from \eqref{nnneone} and  \eqref{accretivityalmost},
$$\Re\la F,X\ra=\Re\la (-\mathcal M+ R )X,X\ra=\Re\la (-\widetilde{\mathcal M} - \mathcal A+R)X,X\ra\geq \frac R2\|X\|_{\Bbb H_{2k}}^2.$$ 
Therefore, for any $F\in \Bbb H_{2k}$, solution $X$ to \eqref{nnneone} is unique.  Therefore, $(-\mathcal M+ R)^{-1}$ is well defined on $\Bbb H_{2k}$ with the bound 
$$\|(-\mathcal M+R)^{-1}\|_{\matchal L(\Bbb H_{2k},\Bbb H_{2k})}\lesssim \frac 1{R}.$$ 
Hence 
$$-\widetilde{\matchal M}+R=-\matchal M + \matchal A+ R=(-\mathcal M+ R)\left[{\rm Id} + (-\mathcal M+ R)^{-1}\matchal A\right]$$ 
is invertible on $\Bbb H_{2k}$ for $R$ large enough, which yields \eqref{estmaximal}. This concludes the proof of Proposition \ref{propaccretif}.
\end{proof}


\subsection{Growth bounds for dissipative operators} 


We conclude this section by recalling classical facts about unbounded operators and their semigroups. Let $(H,\la\cdot,\cdot\ra)$  be a hermitian Hilbert space and $A$ be a closed operator with a dense domain $D(A)$. We recall the definition of the adjoint operator $A^*$: let
$$D(A^*)=\{X\in H, \ \ \tilde{X}\in D(A)\mapsto \la X, A\tilde{X} \ra\ \ \mbox{extends as a bounded functional on}\ \  H\},$$ then $A^*X$ is given by the Riesz theorem as the unique element of $H$ such that 
\be
\label{defnfeoeone}
\forall \tilde{X}\in D(A), \ \ \la A^*X,\tilde{X}\ra=\la X,A\tilde{X}\ra.
\ee

We recall the following classical lemma.

\begin{lemma}[Properties of maximal dissipative operators, \cite{bookinternet} p.49]
\label{lemamcnneoe}
\label{lemmdnone} Let $A$ be a maximal dissipative operator on a Hilbert space $H$ with domain $D(A)$, then:\\
(i) A is closed;\\
(ii) $A^*$ is maximal dissipative;\\
(iii) $\sigma(A)\subset\{\l\in \Bbb C, \ \ \Re(\lambda)\le 0\}$;\\
(iv) $\|(A-\l)^{-1}\|\leq |\Re(\l)|^{-1}$ for $\Re(\l)>0$.
\end{lemma}

We now recall from Hille-Yoshida's theorem that a maximally dissipative operator $A_0$ generates a strongly continuous semigroup $T_0$ on $H$, and so does $A_0+K$ for any bounded perturbation $K$. Let us now recall the following classical properties of strongly continuous semigroup $T(t)$. Let $\sigma(A)$ denote the spectrum of $A$, i.e., the complement of the resolvent set.

\begin{proposition}[Growth bound, \cite{engelnagel} Cor 2.11 p.258]
\label{lemmagrowthbound}
Let the growth bound of the semigroup be defined as $$w_0=\inf\{w\in \Bbb R, \exists M_w\ \ \mbox{such that}\ \ \forall t\geq 0, \ \ \|T(t)\|\leq M_we^{ wt}\}.$$ Let $w_{\rm ess}$ denote the essential growth bound of the semigroup: 
$$w_{\rm ess}=\inf\{w\in \Bbb R, \exists M_w\ \ \mbox{such that}\ \ \forall t\geq 0, \ \ \|T(t)\|_{\rm ess}\leq M_we^{ wt}\}$$
with 
$$
\|T(t)\|_{\rm ess}=\inf_{K\in\mathcal K(H)} \|T(t)-K\|_{H\to H}
$$
and $\mathcal K(H)$ is the ideal of compact operators on $H$; 
and let
$$s(A)=\sup\{\Re(\l), \ \ \l\in \sigma(A)\}.$$ Then $$w_0=\max\{w_{\rm ess},s(A)\}$$ and 
\be
\label{eignevlauefintie}
\forall w>w_{\rm ess}, \ \ \mbox{the set}\ \ \Lambda_w(A):=\sigma(A)\cap \{\Re(\l)>w\}\ \ \mbox{is finite}.
\ee
Moreover, each eigenvalue $\l\in \Lambda_w(A)$ has finite algebraic multiplicity $m^a_\l$: $\exists k_\l\in \Bbb Z$ such that
$$
ker (A-\l I)^{k_\l}\ne \emptyset,\qquad ker (A-\l I)^{k_\l+1}=\emptyset,\qquad m^a_\l:=dim ker (A-\l I)^{k_\l}
$$
\end{proposition}
We note that the subspaces $V_w(A)=\cup_{\l\in  \Lambda_w(A)} ker (A-\l I)^{k_\l}$ and $V_w^\perp(A^*)$ are invariant for $A$.
In particular, $A\left(D(A)\cap V_w^\perp(A^*)\right)\subset V_w^\perp(A^*)$. The invariance $V_w(A)$ is immediate.
To show that $A\left(D(A)\cap V_w^\perp(A^*)\right)\subset V_w^\perp(A^*)$ we let $X\in D(A)\cap V_w^\perp(A^*)$,
$Y\in V_w(A^*)$ and consider 
$
\la AX,Y\ra$. Since $Y\in D(A^*)$ and $V_w(A^*)$ is invariant for $A^*$, 
$$
\la AX,Y\ra=\la X,A^*Y\ra =0.
$$

We claim the following corollary.

\begin{lemma}[Perturbative exponential decay] 
\label{cneoneoneobis}
Let $T_0$ be the strongly continuous semigroup generated by a maximal dissipative operator $A_0$, and $T$ be the strongly continuous semi group generated by $A=A_0+K$ where $K$ is a compact operator on $H$. Then for any $\delta>0$, the following holds:\\
(i) the set $\Lambda_\delta(A)=\sigma(A)\cap \{\l\in \Bbb C, \ \ \Re(\l)> \delta\}$ is finite, 
each eigenvalue $\l\in \Lambda_\delta(A)$ has finite algebraic multiplicity $k_\l$. In particular, the subspace 
$V_\delta(A)$ is finite dimensional;\\
(ii) 
We have $\Lambda_\delta(A)=\overline{ \Lambda_\delta(A^*)}$
and $dim V_\delta(A^*)=dim V_\delta(A)$.
The direct sum decomposition 
\be\label{eq:direct}
H=V_\delta(A)\bigoplus V^\perp_\delta(A^*)
\ee
is preserved by $T(t)$ and there holds:
\be
\label{stabiliteexpo}
\forall X\in V^\perp_\delta(A^*), \ \ \|T(t)X\|\leq M_\delta e^{\delta t}\|X\|.
\ee
(iii) The restriction of $A$ to $V_\delta(A)$ is given by a direct sum of $(m_\l\times m_\l)_{\l\in \Lambda_\delta(A)}$
matrices each of which is the Jordan block associated to the eigenvalue $\l$ and the number of Jordan blocks corresponding 
to $\l$ is equal to the geometric multiplicity of $\l$ -- $m^g_\l=dim ker (A-\l I)$. In particular, $m^a_\l\le m^g_\l k_\l$.
Each block corresponds to an invariant subspace $J_\l$ and the semigroup $T$ restricted to $J_\l$ is given by the nilpotent 
matrix 
$$
T(t)|_{J_\l}=\begin{pmatrix} e^{\l t} & te^{\l t}&...&t^{m_\l-1} e^{\l t}\\
0& e^{\l t}&...&t^{m_\l-2} e^{\l t}\\
...\\
0&0&...& e^{\l t}
\end{pmatrix}
$$
\end{lemma}

\begin{proof} This is a simple consequence of Proposition \ref{lemmagrowthbound}.\\

\noindent{\bf step 1} Perturbative bound. First, since $A_0$ is maximally dissipative, $$\forall t\ge 0, \ \ \|T_0(t)\|\lesssim 1$$ implies $w_0(A_0)\leq 0$. By Proposition  \ref{lemmagrowthbound}, $s(A_0)\leq 0$ and $$w_{ess}(T_0)\leq 0.$$ On the other hand, from \cite{engelnagel} Prop 2.12 p.258, compactness of $K$ implies $$w_{\rm ess}(T)=w_{\rm ess}(T_0)\leq 0.$$ 
Let now $\l\in \sigma(A)$ with $\Re(\l)>0$, then the formula $$A-\l=A_0+K-\l=(A_0-\l)({\rm Id}+(A_0-\l)^{-1}K)$$ and
invertibility of  $(A_0-\l)$ imply that $\l$ belongs to the spectrum of the Fredholm operator ${\rm Id} +(A_0-\l)^{-1}K$. 
Therefore, $\l $ is an eigenvalue  of $A$. On the other hand, $\Re(\l)>\delta$ implies $\Re(\l)>\delta>0\geq w_{\rm ess}(T)$, and hence, by \eqref{eignevlauefintie}, there are finitely many eigenvalues with $\Re(\l)>\delta$. In fact, Proposition \ref{lemmagrowthbound} also directly shows that each some $\l$ is an eigenvalue and implies the rest of (i).

Since $A^*=A_0^*+K^*$ and $A_0^*$ is maximally dissipative from Lemma \ref{lemamcnneoe}, we can run the same argument as above for $A^*$.  Moreover, $\sigma(A)=\overline{\sigma(A^*)}$ (\cite{bookinternet}, prop. 2.7), (i) is proved.\\

The argument above, in fact, shows that $\{\l\in \Bbb C, \ \ \Re(\l)> \delta\}\cap \{\l\in \sigma(A)\}$ is finite, since for 
every $\Re(\l)>0$ and $\l\in\sigma(A)$, $\l$ is an eigenvalue of $A$.\\

\noindent{\bf step 2} The first statement of (ii) is standard. We already explained that the subspaces $V_\delta(A)$
and $D(A)\cap V_\delta^\perp(A^*)$ are invariant for $A$. To prove the direct decomposition we recall that the subspace 
$V_\delta(A)$ is the image of $H$ under the spectral projection $P_\delta(A)$ associated to the set $\Lambda_\delta(A)$: 
$$
P_\delta(A)=\frac 1{2\pi i} \int_\Gamma \frac {d\lambda}{\lambda I -A},
$$
where $\Gamma$ is an arbitrary contour containing the set $\Lambda_\delta(A)$. There is a direct decomposition 
$$
H=Im P_\delta(A)\bigoplus ker P_\delta(A).
$$
On the other hand, the adjoint 
$$
P^*_\delta(A)=\frac 1{2\pi i} \int_{\overline \Gamma} \frac {d\lambda}{\lambda I -A^*}=P_\delta(A^*)
$$
is the spectral projection of $A^*$ associated to the set ${\overline{\Lambda_\delta(A)}}$. The result is now immediate.\\

\noindent{\bf step 3} Semigroups generated by restriction and conclusion. Let $V=V_\delta(A)$, $U=V^\perp_\delta(A^*)$ and $P$ denote the projection on $V_\delta^\perp(A^*)$ in the 
direct decomposition \eqref{eq:direct}. 
 Let $\tilde{A}$ denote the restriction of $A$ to $U$ with the domain $D(\tilde{A})=U\cap D(A)$. By invariance
$$\forall X\in U\cap D(A), \ \ \tilde{A}X=AX.$$

Let $\tilde{T}$ be the semigroup on $U$ generated by $\tilde {A}=A$. Then for all $X\in D(A)\cap U$, $\tilde{T}(t)X\in \mathcal C^1([0,+\infty),D(\tilde{A}))$ is the unique strong solution to the ode
$$\frac{dX(t)}{dt}={A}X(t), \ \ X(0)=X.$$ This implies that $\tilde{T}(t)X=T(t)X$ 
for all $X\in D(A)\cap U$ and thus for all $X\in U$ by continuity of the semigroup. 
By Proposition \ref{lemmagrowthbound} the growth bound of $\tilde{T}$ satisfies $$w_0(\tilde{T})\leq \max\{w_{\rm ess}(\tilde T),s(\tilde{A})\}.$$ 
We first argue that 
$$
w_{\rm ess}(\tilde T)\le 0.
$$
To prove that we note that we already established that $w_{\rm ess}(T)\le 0$. We then fix $\ep>0$ and,
for any $t\ge 0$ 
choose a compact operator 
$K(t)\in \mathcal K(H)$ on $H$ such that,
$$
\log\|T(t)-K(t)\|_{H\to H}<\ep t + \log M
$$
for some constant $M$ which may depend on $\ep$.
The restriction $\tilde K(t)=PK(t)$ of $K(t)$ to $U$ is a compact operator on $U$. Then, for any $t\ge 0$
\bee
\log\|\tilde T(t)-\tilde K(t)\|_{U\to U}&=&\log\|P(T(t)-K(t))\|_{U\to U}\\ &\le& \log \{C _P\|T(t)-K(t)\|_{H\to H}\}<\log {C _P}+
\log M+ \ep t,
\eee
where $C_P$ denotes the norm of the projector $P$. The desired conclusion follows.

To show that $s(\tilde A)\le \de$
we assume that  $\l\in \sigma(\tilde{A})$ with $\Re(\l)>\delta$, then $\l$ is an eigenvalue of $\tilde{A}$ and, by invariance of $U$, ${\l}$ is an eigenvalue of $A$ with a non-trivial eigenvector $\psi\in U$. However, by construction, all such $\psi$
belong to the subspace $V=V_\delta(A)$, contradiction. Hence $s(\tilde{A})\leq \delta$ and Proposition \ref{lemmagrowthbound} yields \eqref{stabiliteexpo}.

Finally, part (iii) is completely standard.
\end{proof}

 We will use Lemma \ref{cneoneoneobis} in the following form.
 
 \begin{lemma}[Exponential decay modulo finitely many instabilities]
 \label{elmsnjennnw}
  Let $\delta>0$ and let $T_0$ be the strongly continuous semigroup generated by a maximal dissipative operator $A_0$, and  $T$ be the strongly continuous semigroup generated by $A=A_0-\delta+K$ where $K$ is a compact operator on $H$. Let the (possibly empty) finite set $$\Lambda=\{\l \in \Bbb C, \ \ \Re(\l)\ge 0\} \cap \{\l\ \ \mbox{is an eigenvalue of}\ \ A\}=(\l_i)_{1\le i\le N}$$ and let 
  $$
  H=U\bigoplus V,
  $$
where $U$ and $V$ are invariant subspaces for $A$ and  $V$ is the image of the spectral projection of $A$ associated to the set $\Lambda$. Then there exist $C,\delta_g>0$ such that 
 \be
\label{stabiliteexpobis}
\forall X\in U, \ \ \|T(t)X\|\leq C e^{-\frac{\delta_g}{2} t}\|X\|.
\ee
\end{lemma}

\begin{proof} We apply Lemma \ref{cneoneoneobis} to $\tilde{A}=A+\delta=A_0+K$ with generates the semi group $\tilde{T}$. Hence the set  $$\Lambda_{\frac{\delta}{4}}(\tilde{A})=\{\l\in \Bbb C, \ \ \Re(\l)> \frac{\delta}{4}\}\cap \{\l\ \ \mbox{is an eigenvalue of}\ \ \tilde{A}\}$$ is finite. Moreover $$AX=\l X\Leftrightarrow \tilde{A}X=(\l+\delta)X$$ and hence $$\Lambda \subset \Lambda_{\frac{\delta}{4}}$$ 
Let 
$$
H=U_\delta\bigoplus V_\delta,
$$
be the invariant decomposition of $\tilde A$ (and of $A$) associated to the set $\Lambda_{\frac{\delta}{4}}$. Clearly, 
$U_\delta\subset U$ and 
$$U=U_\delta\bigoplus O_\delta,$$ where $O_\delta$ is the image of the spectral projection of $A$ associated with the set
$\Lambda_{\frac{\delta}{4}}\setminus\Lambda$. By Lemma \ref{cneoneoneobis},
$$
\forall X\in U_\delta, \ \ \|\tilde{T}(t)X\|\leq M_\delta e^{\frac{\delta}{4} t}\|X\|
$$ 
which implies 
\be
\label{veineionevonee}
\forall X\in U_\delta, \ \ \|T(t)X\|=e^{-\delta t} \|\tilde{T}(t)X\|\leq M_\delta e^{-\frac{3\delta}{4} t}\|X\|.
\ee
Let now $X\in U$. Since $U_\delta$ is invariant by $T$ and \eqref{veineionevonee} yields exponential decay on $U_\delta$, we assume $X\in O_\delta$. $O_\delta$ is an invariant subspace of $A$ generated by the eigenvalues $\l$ with the property that 
$-\frac 34\delta\le \Re (\l) <0$. Let $\delta_g>0$ be defined as 
$$
-{\delta_g}:=\sup\{\Re(\l): -\frac 34\delta\le \Re (\l) <0\}
$$
From part (iii) of Lemma \ref{cneoneoneobis},
$$
\|T(t)X\|_{O_\delta}\le C \sup_{\Re (\l) < 0} e^{\l t} t^{m_\l -1}\|X\|\le C^{-\frac{\delta_g}2t}\|X\|
$$
This concludes the proof of Lemma \ref{elmsnjennnw}.
\end{proof}
Our final result in this section is to set up a Brouwer type argument for the evolution of unstable modes.
\begin{lemma}
\label{browerset}
Let $A,\delta_g$ as in Lemma \ref{elmsnjennnw} with the decomposition 
$$
H=U\bigoplus V
$$
into stable and unstable subspaces Fix sufficiently large $t_0>0$ (dependent on $A$). 
Let $F(t)$ such that, $\forall t\ge t_0$, $F(t)\in V$ and 
$$
\|F(t)\|\le e^{-\frac {2\delta_g}3 t}
$$
 be given. Let $X(t)$ denote  the solution to the ode
$$
\left|\begin{array}{l}
\frac {dX} {dt} = A X + F(t)\\
 X(t_0)=x\in V.
 \end{array}\right. 
$$
Then, for any $x$ in the ball
$$
\|x\|\le e^{-\frac {3\delta_g} 5t_0},
$$
we have
\be\label{eq:growth}
\|X(t)\|\le e^{-\frac{\delta_g} 2 t},\qquad t_0\le t\le t_0+\Gamma
\ee
for some large constant $\Gamma$ (which only depends on $A$ and $t_0$.)
Moreover, there exists $x^*\in V$ in the same ball as a above such that $\forall t\ge t_0$,
$$
\|X(t)\|\le e^{-\frac {3\delta_g}5 t}
$$

\end{lemma}
\begin{proof}
According to Lemma \ref {cneoneoneobis} the subspace $V$ can be further decomposed into invariant subspaces 
on which $A$ is represented by Jordan blocks. We may therefore assume that $V$ is irreducible and corresponds to 
a Jordan block of $A$ of length $m_\l$ associated with an eigenvalue $\l$ with ${\Re(\l)\ge 0}$ and restrict $A$ to $V$. 
We decompose $A$ as 
$$
A=\l I + N,
$$ 
where $N$ has the property that $N^{m_\l-1}=0$, and
$$
e^{tN}=\begin{pmatrix} 1 & t&...&t^{m_\l-1} \\
0& 1&...&t^{m_\l-2} \\
...\\
0&0&...& 1
\end{pmatrix}
$$
The claim \eqref{eq:growth} follows from the growth on the Jordan block:
\bee
&&\|X(t)\|=\left\|e^{(t-t_0)A}x+\int_{t_0}^te^{(t-\tau)A}F(\tau)d\tau\right\|\\
&\leq &C\Gamma^{m_\l-1}e^{\Re(\l)\Gamma}e^{-\frac{3\delta_g t_0}{5}}+\int_{t_0}^{t}C|\tau-t_0| ^{m_\l-1}e^{\Re(\l)(t-\tau)}e^{-\frac{2}{3}\delta_g\tau}d\tau\le C\Gamma^{m_\l-1}e^{\Re(\l)\Gamma}e^{-\frac{3\delta_g t_0}{5}}
\eee
and hence the size of constant $\Gamma$ is determined
from the inequality
$$
C\Gamma^{m_\l-1}e^{\Re(\l)\Gamma}e^{-\frac{3\delta_g t_0}{5}}\le e^{-\frac \delta 2(t_0+\Gamma)},
$$
a sufficient condition being 
$$
\Gamma \le \frac {t_0}2\left[\frac {\delta_g}{10\Re(\l)+5\delta_g} \right]
$$
which can be made arbitrarily large by a choice of $t_0$.\\
We now define a new variable 
$$
Y(t)=e^{-tN}e^{\frac {19\delta_g}{30} t}X(t).
$$
Since $N$ and $A$ commute,
$$
\frac {dY} {dt} = \left(\l+\frac {19\delta_g}{30}\right) Y +  \tilde F(t), \qquad Y(t_0)=y
$$
where $\tilde F(t)=e^{-tN}e^{\frac {19\delta_g}{30} t}F(t)$ and 
$$
\|\tilde F(t)\|\lesssim e^{-\frac{\delta_g}{31} t}.
$$
Since $t_0$ was chosen to be sufficiently large, we can assume that $\forall t\ge t_0$
$$
\|\tilde F(t)\|\lesssim \epsilon e^{-\frac{\delta_g}{60} t}
$$
and $\epsilon < \Re (\l)+\frac {19\delta_g}{60}$.
We now run a standard Brouwer type argument for $Y$. For any $y$ such that $\|y\|\le 1$ we define the exit time 
$t^*$ to be the first time such that $\|Y(t^*)\|=1$. If for some $y$, $t^*=\infty$, we are done. Otherwise, assume 
that for all $\|y\|\le 1$, $t^*<\infty$ and define the map 
$\Phi: B\to S$ as $\Phi(y)=Y(t^*)$ mapping the unit ball to the unit sphere. Note that $\Phi$ is the identity map on the
boundary of $B$. To prove continuity of $\Phi$ we compute 
$$
\frac {d\|Y\|^2}{dt}(t^*)=2\Re(\l) +\frac{19\delta_g}{15}+ 2 \Re \la \tilde F(t^*),Y(t^*)\ra\ge \frac{19\delta_g}{30}>0.
$$
This is the outgoing condition which implies continuity. The Brouwer argument applies and shows that such $\Phi$ 
does not exist. We now reinterpret the result in terms of $X$. We have shown existence of $x$ such that the corresponding 
solution $X(t)$ has the property that $\forall t\ge t_0$,
$$
\|e^{-tN} X(t)\|\le  e^{-\frac {19\delta_g}{30} t}.
$$
Now $e^{-tN}$ is an invertible operator with the inverse given by $e^{tN}$ and its norm bounded by 
$C t^{m_\l-1}$. The result follows immediately. We note that the resulting solution $X(t)$ has initial data $X(t_0)$ in the ball
$
\|X(t_0)\|\le e^{-\frac {3\delta_g}5t_0}.
$
\end{proof}


\section{Set up and the bootstrap}
\label{sectionbootstrap}


In this section we describe a set of smooth well localized initial data which lead to the conclusions of 
Theorem \ref{thmmain}. The heart of the proof is a bootstrap argument coupled to the classical Brouwer topological argument of Lemma \ref{browerset} to avoid finitely many unstable directions of the corresponding linear flow. 
Since our analysis relies {\em essentially} on the phase-modulus decomposition of solutions of the Schr\"odinger equation,
our chosen data needs to give rise to {\em nowhere vanishing} solutions to \eqref{nls} (at least for a sufficiently small time.)


\subsection{Renormalized variables} 


Let $u(t,x)\in \mathcal C([0,T),\cap_{k\ge 0} H^k)$ be a solution to \eqref{nls} 
{such that} $u(t,x)$ does not vanish at any $(t,x)\in [0,T)\times{\Bbb R}^d$. This will be a consequence of our choice of initial data and suitable bootstrap assumptions. We introduce for such a solution the decomposition of Lemma \ref{newequationlemma} 
\be
\label{renorlianoineo}
u(t,x)=\frac{1}{(\l\sqrt{b})^{\frac 2{p-1}}}w(s,y)e^{i\gamma}, \ \ w(\tau,y)=\rho_T(\tau,Z)e^{i\frac{\Psi_T}{b}}
\ee
with the renormalized space and times
\be
\label{renormalization}
\left|\begin{array}{lll}
Z=y\sqrt{b}=Z^*x, \ \ Z^*=e^{\mu\tau}\\
\l(\tau)=e^{-\frac{\tau}{2}}, \ \ b(\tau)=e^{-{\mathcal e}\tau}, \ \ \gamma_\tau=-\frac 1b=-e^{{\mathcal e}\tau}\\
\tau=-\log(T-t), \ \ \tau_0=-\log T.
\end{array}\right.
\ee
Here, $0<{\mathcal e}<1$ is the {\it fixed} front speed such that $$r=\frac{2}{1-{\mathcal e}}>2.$$ Up to a constant the phase can more explicitly be written in the form
\be
\label{formulaphase}
\gamma(\tau)=-\frac{1}{{\mathcal e}b}.
\ee
Our claim is that given $$\tau_0=-\log T$$
large enough, we can construct a finite co-dimensional manifold of smooth well localized initial data $u_0$ such that the corresponding solution to the renormalized flow \eqref{fullflowrenormalized} is global in renormalized time $\tau\in[\tau_0,+\infty)$, bounded in a suitable topology and nowhere vanishing. Upon unfolding \eqref{renorlianoineo}, this produces a solution to \eqref{nls} blowing up at $T$ in the regime described by Theorem \ref{thmmain}.\\


\subsection{Stabilization and regularization of the profile outside the singularity}


The spherically symmetric profile solution $(\rho_P,\Psi_P)$ has an intrinsic slow decay as $Z\to +\infty$  $$\rho_P(Z)=\frac{c_P}{\la Z\ra^{\frac{2(r-1)}{p-1}}}\left(1+O\left(\frac{1}{\la Z\ra^r}\right)\right),$$
which need to be regularized in order to produce finite energy non vanishing initial data.\\

\noindent{\em 1. Stabilization of the profile}. Recall the asymptotics \eqref{decayprofile} and the choice of parameters \eqref{formulaphase}, \eqref{renormalization} which yield $$\l^{2(r-2)}=b^r, \ \ r=\frac{2}{1-{\mathcal e}}, \ \ {\mathcal e}=\frac{r-2}{r},\ \
\mu=\frac{1-{\mathcal e}}2.$$ For $Z=\frac{\sqrt{b}}{\l}x\gg 1$, i.e., outside the singularity:
\bea
\label{outerprofile}
\nonumber u_P(t,x)&=&\frac{e^{i\gamma(\tau)}}{(\l\sqrt{b})^{\frac{2}{p-1}}}\rho_P(Z)e^{i\frac{\Psi_P}{b}}=\frac{c_Pe^{-\frac{i}{{\mathcal e}b}}}{(\l\sqrt{b})^{\frac{2}{p-1}}\left(\frac{\sqrt{b}}{\l}x\right)^{\frac{2(r-1)}{p-1}}}e^{i\left[\frac{1}{eb}+\frac{c_{\Psi}}{b\left(\frac{\sqrt{b}}{\l}x\right)^{r-2}}\right]}\\
& = & \frac{c_P}{x^{\frac{2(r-1)}{p-1}}}e^{i\frac{c_\Psi}{x^{r-2}}\left[1+O\left(\frac{1}{Z^r}\right)\right]}\left[1+O\left(\frac{1}{Z^r}\right)\right].
\eea
We see that far away from the singularity the profile $u_P$ is {\it stationary}. It is precisely this property that will allow us 
to dampen the tail of the profile below and construct solutions arising from rapidly decaying (in particular, finite energy) initial
data.\\

\noindent{\em 2. Dampening of the tail}.  We dampen the tail outside the singularity $x\ge 1$, i.e., $Z\ge Z^*$ as follows. Let 
\be
\label{vnbeineoneoneov}
R_P(t,x)=\frac{1}{(\l\sqrt{b})^{\frac 2{p-1}}}\rho_P(Z), \ \ x=Ze^{-\mu\tau},
\ee then the asympotics \eqref{outerprofile} imply the existence of a limiting profile for $x\ge 1$:
$$R_P(t,x)=\frac{c_P}{x^{\frac{2(r-1)}{p-1}}}\left(1+O(e^{-\mu r\tau})\right)$$
We then pick once and for all a large integer $n_P\gg 1$ 
and define a smooth non decreasing connection $\matchal K(x)$ \be
\label{defconnexion}
\mathcal K(x)=\left|\begin{array}{ll} 0\ \  \mbox{for}\ \ |x|\leq 5\\
n_P-\frac{2(r-1)}{p-1}\ \ \mbox{for}\ \ |x|\geq 10
\end{array}\right., 
\ee
for some large enough universal constant $$n_P=n_P(d)\gg 1.$$
We then define the dampened tail profile in original variables
\be
\label{dampenedprofile}R_D(t,x)=R_P(t,x)e^{-\int_0^x\frac{\matchal K(x')}{x'}dx'}=\left|\begin{array}{ll} R_P(t,x)\ \ \mbox{for}\ \ |x|\leq 5\\  
\frac{c_{n,\delta}}{x^{n_P}}\left[1+O\left(e^{-c\delta\tau})\right)\right]\ \ \mbox{for}\ \ |x|\ge 10\end{array}\right., 
\ee
and hence in renormalized variables: 
\be
\label{definitionprofilewithtailchange}
\left|\begin{array}{l}
\rho_D(\tau,Z)=(\l\sqrt{b})^{\frac 2{p-1}} R_D(t,x)\\
x=\frac{Z}{Z^*}, \ \ Z^*=e^{\mu \tau}.
\end{array}\right.
\ee 
Let $$\zeta(x)=e^{-\int_0^x\frac{\matchal K(x')}{x'}dx'}, $$ we have the equivalent representation:
\be
\label{fromularhod}
\rho_D(Z)=(\l\sqrt{b})^{\frac 2{p-1}} R_D(\tau, x)=(\l\sqrt{b})^{\frac 2{p-1}}R_P(t,x)\zeta(x)=\zeta\left(\frac Z{Z^*}\right)\rho_P(Z)
\ee
Note that by construction for  $j\in \Bbb N^*$:
\be
\label{fundmaetnaldecay}
-\frac{Z^j\pa^j_Z\rho_D}{\rho_D}=\left|\begin{array}{ll}\ \ (-1)^{j-1}\Pi_{m=0}^{j-1}\left(\frac{2(r-1)}{p-1}+m\right)+O\left(\frac{1}{\la Z\ra^r}\right)\ \ \mbox{for}\ \ Z\leq 5Z^*\\
(-1)^{j-1} \Pi_{m=0}^{j-1}\left(n_P+m\right)+O\left(\frac{1}{\la Z\ra^r}\right)\mbox{for}\ \ Z\geq 10Z^*
\end{array}\right.
\ee
and $$ \left|\frac{Z^j\pa_j\rho_D}{\rho_D}\right|_{L^\infty}\lesssim 1.$$

The obtained dampened profile for $Z\geq Z^*$ will be denoted $$(\rho_D,\Psi_P), \ \ Q_D=\rho_D^{p-1}.$$


\subsection{Initial data}


We now describe explicitly an open set of initial data which will be considered as perturbations of the profile $(\rho_D,\Psi_P)$ in a suitable topology. The conclusions of Theorem \ref{thmmain} will hold for a finite co-dimension set of such data. \\

We pick universal constants $0<a\ll 1$, $Z_0\gg 1$ which will be adjusted along the proof and depend only on $(d,\ell)$.
 We define two levels of regularity $$\frac d2\ll k_0\ll k_m$$ where $k_m$ denotes the maximum level of regularity required for the solution and $k_0$ is the level of regularity required for the linear spectral theory on the compact set $[0,Z_a]$.\\

\noindent{\em 0}. Variables and notations for derivatives. We define the variables
 \be
 \label{definitionvariables}
 \left|\begin{array}{l}
 \rho_T=\rho_P+\rho=\rho_D+\rhot\\
  \Psi_T=\Psi_P+\Psi \\
  \Phi=\rho_P\Psi
  \end{array}\right.
 \ee
and specify the data in the $(\rhot,\Psi)$ variables. We will use the following notations for derivatives. Given $k\in \Bbb N$, we note $$\pa^k=(\pa_1^k,...,\pa_d^k), \ \ f^{(k)}:=\pa^kf$$ the vector of $k$-th derivatives in each direction. The notation $\pa_Z^k f$ is the $k$-th radial derivative. We  let $$\rhot_k=\Delta^k \rhot,\qquad \Psi_k=\Delta^k\Psi.$$
Given a multiindex $\alpha=(\alpha_1,\alpha_d)\in \Bbb N^d$, we note $$\pa^\alpha=\pa^{\alpha_1}_1\dots\pa^{\alpha_d}_d, \ \ |\alpha|=\alpha_1+\dots+\alpha_d.$$
 
\noindent{\em 1. Initializing the Brouwer argument}.  We define the variables adapted to the spectral analysis according to \eqref{defnewvariablephi},  \eqref{defintionT}: 
  \be
  \label{notatinotphi}
  \left|\begin{array}{ll}
  \Phi=\rho_P\Psi \\
   T=\pa_\tau\Phi+aH_2\Lambda \Phi
   \end{array}\right., \ \ X=\left|\begin{array}{ll}T\\ \Phi \end{array}\right.
  \ee
and recall the scalar product \eqref{defscalarproduct}. For $0<c_g,a\ll1$ small enough, we choose $k_0\gg 1$ such that Proposition \ref{propaccretif}  applies in the Hilbert space $\Bbb H_{2k_0}$
with the spectral gap 
\be
\label{accretivityalmostbis}
\forall X\in \matchal D(\matchal M), \ \ \Re\la (-\matchal M+\matchal A) X,X\ra \geq c_g\la X,X\ra.
\ee
Hence $$\mathcal M=(\mathcal M-\matchal A+c_g)-c_g+\mathcal A$$ and we may apply Lemma \ref{elmsnjennnw}: 
\be
\label{nvknneknengno}
\Lambda_0=\{\l \in \Bbb C, \ \ \Re(\l)\ge 0\} \cap \{\l\ \ \mbox{is an eigenvalue of}\ \ \mathcal M\}=(\l_i)_{1\le i\le N}
\ee 
is a finite set corresponding to unstable eigenvalues, $V$ is an associated (unstable) finite dimensional invariant set, $U$ is the complementary 
(stable) invariant set  
\be\label{eq:decomp}
  \Bbb H_{2k_0}=U\bigoplus V
  \ee
  and $P$ is the associated projection on $V$. We denote by ${\mathcal N}$ the nilpotent part of the matrix representing 
  ${\mathcal M}$ on V:
  \be\label{eq:nilp}
  {\mathcal M}|_V={\mathcal N} + {\text {diag}}
  \ee 
Then there exist $C, \delta_g>0$ such that \eqref{stabiliteexpobis} holds: 
$$
\forall X\in U, \ \ \|e^{\tau\mathcal M}X\|_{\Bbb H_{2k_0}}\leq C e^{-\frac{\delta_g}{2} (\tau-\tau_0)}\|X\|_{\Bbb H_{2k_0}},\qquad \forall \tau\ge\tau_0.
$$
We now choose the data at $\tau_0$ such that 
$$
\|(I-P) X(\tau_0)\|_{\Bbb H_{2k_0}}\le e^{-\frac{\delta_g}{2} \tau_0},\qquad \|PX(\tau_0)\|_{\Bbb H_{2k_0}}\le e^{-\frac{3\delta_g}5 \tau_0}.
$$

   \noindent{\em 2. Bounds on local low Sobolev norms}. Let $0\le m\leq 2k_0$ and 
   \be
   \label{venovnoenneneo}
   \nu_0=-\frac{2(r-1)}{p-1}+\frac{\delta_g}{2\mu},
   \ee 
   let the weight function 
  \be
 \label{defchimunu}
 \chi_{\nu_0,m}=\frac{1}{\la Z\ra^{d-2(r-1)+2(\nu_0-m)}}\zeta\left(\frac{Z}{Z^*}\right), \ \ \zeta(Z)=\left|\begin{array}{ll}1\ \ \mbox{for}\ \ Z\leq 2\\ 0\ \ \mbox{for}\ \ Z\ge 3.
 \end{array}\right.
 \ee
Then:
 \be
 \label{improvedsobolevlowinit}
 \sum_{m=0}^{2k_0}\int (p-1)Q(\pa^m\rho(\tau_0))^2\chi_{\nu_0,m}+|\nabla \pa^m\Phi(\tau_0)|^2\chi_{\nu_0,m}\leq e^{-\delta_g\tau_0}.
 \ee
    
\noindent{\em 4. Pointwise assumptions}. We assume the following interior pointwise bounds 
\be
\label{smallnessoutsideinitbis}
\forall 0\leq k\leq 2k_m, \ \ \|\frac{\la Z\ra^k\pa_Z^k\rhot(\tau_0)}{\rho_D}\|_{L^\infty(Z\le Z^*_0)}+\|\la Z\ra^{r-2}\la Z\ra^k\pa_Z^k\Psi(\tau_0)\|_{L^\infty(Z\le Z^*_0)}\le b_0^{c_0}
\ee
for some small enough universal constant $c_0$, and the exterior bounds:
\be
\label{smallnessoutsideinit}
\forall 0\leq k\leq 2k_m, \ \ \|\frac{Z^{k{+1}}\pa_Z^k\rhot(\tau_0)}{\rho_D}\|_{L^\infty(Z\ge Z^*_0)}+\frac{\|Z^{k{+1}}\pa_Z^k\Psi(\tau_0)\|_{L^\infty(Z\ge Z^*_0)}}{b_0}\le b_0^{C_0}
\ee
for some large enough universal $C_0(d,r,p)$. Note in particular that \eqref{smallnessoutsideinitbis}, \eqref{smallnessoutsideinit} ensure for $0<b_0<b_0^*\ll1$ small enough:
\be
\label{smallenenoendata}
\left\|\frac{\rhot(\tau_0)}{\rho_D}\right\|_{L^\infty}\leq \delta_0\ll 1
\ee and hence the data does not vanish.\\

\noindent {\em 5. Global rough bound for large Sobolev norms}.  We pick a large enough constant $k_m(d,r,\ell)$ and consider the global Sobolev norm 
\be
\label{globalsobolevnorm}
\|\rhot,\Psit\|_{k_m}^2:=\sum_{j=0}^{k_m}\sum_{|\alpha|=j}\int \frac{b^2|\nabla\pa^\alpha\rhot|^2+(p-1)\rho_D^{p-2}\rho_T(\pa^\alpha\rhot)^2+\rho_T^2|\nabla \pa^\alpha\Psi|^2}{\la Z\ra^{2(k_m-j)}},
\ee
then we require:
\be
\label{sobolevinit}
 \|\rhot(\tau_0),\Psit(\tau_0)\|_{k_m}\le 1
\ee
The bound above is actually implied by the pointwise assumptions.

\begin{remark}{Note that we may without loss of generality assume $u_0\in \cap_{k\ge 0}H^k.$}
\end{remark}


\subsection{Bootstrap bounds}


 We make the following bootstrap assumptions on the maximal interval $[\tau_0,\tau^*)$.\\
\vskip .5pc

\noindent{\em 0. Non vanishing and hydrodynamical variables.} From standard Cauchy theory and the smoothness of the nonlinearity since $p\in 2\Bbb N^*+1$, the smooth data $u_0 \in\cap_{k\ge 0}H^k$ generates a unique local solution $u\in \mathcal C([0,T),\cap_{k\ge 0}H^k)$ with the blow up criterion
\be
\label{blowupcrtiertrion}
T<+\infty\Rightarrow \lim_{t\to T}\|u(t,\cdot)\|_{H^{k_c}}=+\infty
\ee
for some large enough $k_c(d,p)$. To ensure non vanishing, we first note that since $\inf_{|x|\le 10}|u_0(x)|>0$, the continuity of $u$ in time ensures $\inf_{|x|\le 10}|u(t,x)|>0$ for $t\in [0,T]$, $T$ small enough. For $|x|\ge 10$, we estimate from the flow
 $$\left|{r^{n_P}}{|u(t,x)|}-{r^{n_P}}{|u_0|}\right|\leq \left|\int_0^{t}{r^{n_P}}(\Delta u-u|u|^{p-1})dt\right|$$ and hence from our choice of initial data, the non vanishing of $u(t,x)$ follows on a time interval where 
\be
\label{bootnonvanishing}
T\left\|{r^{n_P}}(|\Delta u|+|u|^{p})\right\|_{L^\infty([0,T),|x|\ge 10)}\le \delta
\ee
for some sufficiently small universal constant $0<\delta\ll1$. Using spherical symmetry we can replace the above by
$$
T\left(\|\la x\ra^{n_P+1-\frac d2+\epsilon}\Delta u\|_{L^\infty([0,T); H^1)}+
\|r^{2\epsilon} u\|^{p-1}_{L^\infty([0,T),|x|\ge 10)}\|\la x\ra^{n_P+1-\frac d2-\epsilon}u\|_{L^\infty([0,T); H^1)}\right)\le \delta
$$
for an arbitrarily small $\epsilon>0$. Our initial data $u_0$ belongs to the space 
$$
\cap_{k\ge 0}H^{k}\cap \|{\la x\ra^{n_P+1-\frac d2-\epsilon}}u\|_{L^2)}\cap \|{\la x\ra^{n_P+3-\frac d2-\epsilon}}\Delta u\|_{L^2)}.
$$
Existence of the desired time interval $[0,T)$ now follows from a local well-posedness for NLS in weighted Sobolev spaces which is (essentially)
in \cite{hnt}.

We may therefore introduce the hydrodynamical variables \eqref{renorlianoineo} on such a small enough time interval and will bootstrap the smallness bound which ensures non vanishing:
\be
\label{nonvavnishig}
\left\|\frac{\rhot}{\rho_T}\right\|_{L^\infty}\le \delta
\ee
for some sufficiently small $0<\delta=\delta(k_m)\ll1.$\\

\noindent {\em 1. Global weighted Sobolev norms}.  Pick a small enough universal constant $0<\nut<\nut^*(k_m)\ll 1$, we define 
\be
\label{defsigmnu}
\left|\begin{array}{l}
\sigma_\nu=\nu+\frac d2-(r-1)\\ 
m_0=\frac{4k_m}{9}+1\\
\nut=\nu+\frac{2(r-1)}{p-1}
\end{array}\right.
\ee
and let the continuous function:
\be
\label{defsimgma}
\sigma(m)=\left|\begin{array}{ll} \sigma_\nu-m\ \ \mbox{for}\ \ 0\leq m\leq m_0\\ -\alpha(k_m-m)\ \ \mbox{for}\ \  m_0\leq m\leq k_m\end{array}\right.
\ee
with the continuity requirement at $m_0$: 
\be
\label{defalphacomnot}
\alpha(k_m-m_0)=m_0-\sigma_\nu, \ \ \alpha=\frac{m_0-\sigma_\nu}{k_m-m_0}=\frac 45+O\left(\frac{1}{k_m}\right).
\ee
In particular, $\alpha<1$.
We note that for all $1\le m\le k_m$
\be\label{eq:sigma}
\sigma(m-1)\ge \sigma(m)-\alpha.
\ee
We also define the function
\bea
\label{def-sigma}
\nonumber \tilde\sigma(k)&=&\left|\begin{array}{l} n_P-\frac{2(r-1)}{p-1}-(r-2)+{2\nut}\ \  \mbox{for}\ \ 0\le k\le \frac {2k_m}3+1\\ \beta(k_m-k)\ \  \mbox{for}\ \ \frac {2k_m}3+1\leq m\leq k_m\end{array}\right.\\
&\le& n_P-\frac{2(r-1)}{p-1}-(r-2)+{2\nut}
\eea
where $\beta$ is computed through the continuity requirement at $\frac{2k_m}3$:
$$
\frac {k_m}3\beta= n_P-\frac{2(r-1)}{p-1}-(r-2)+2\nut\Leftrightarrow  \beta=3\frac {n_P-\frac{2(r-1)}{p-1}-(r-2)+2\nut}{k_m}.
$$
We will choose $n_P\ll k_m$, e.g. $n_P=\frac {k_m}{30}$, so that in particular, 
$$
\beta<\frac 1{10},\qquad \alpha+\beta\le 1.
$$
We also note that 
$$
{\sigmat(m-1)\le \sigmat(m)+\beta.}
$$
We then define the weighted Sobolev norm:
\be
\label{einivneineoneone}
\left|\begin{array}{ll}
\|\rhot,\Psit\|_{m,\sigma}^2=\sum_{k=0}^m\int \chi_{m,k,\sigma}\left[b^2|\nabla\rhot^{(k)}|^2+(p-1)\rho_D^{p-2}\rho_T(\rhot^{(k)})^2+\rho_T^2|\nabla \Psit^{(k)}|^2\right]\\
\chi_{m,k,\sigma}(Z)=\frac{1}{\la Z\ra^{2(m-k+\sigma)}} \, \xi_m\left(\frac Z{Z^*}\right),
\end{array}\right.
\ee
where the function 
$$\xi_m(x)=\left|\begin{array}{l}
1 \ \ \mbox{for}\ \ x\le 1\\
x^{2\tilde\sigma(m)}\ \  \mbox{for}\ \ x>1
\end{array}\right.
$$
We assume the bootstrap bound:
 \be
\label{weightedl2boot}
\|\rhot,\Psit\|^2_{m,\sigma(m)}\leq 1,  \ \ 0\le  m\leq k_m-1.
 \ee
 
 \begin{remark}[Equivalence of norms]
 \label{delta} The spherical symmetry assumption ensures the equivalence of norms
  \be\label{eq:psirho}
 \|\rhot,\Psit\|^2_{m,\sigma(m)}\approx \sum_{k=0}^m\sum_{|\alpha|=k} \int \chi_{m,k,\sigma(m)}\left[b^2|\nabla\nabla^\alpha\rhot|^2+\rho_D^{p-1}{|\nabla^\a\rhot|^2}+\rho_D^2{|\nabla \nabla^\alpha\Psit|^2}\right]
 \ee
 and for even $m$
 \be\label{eq:psirho'}
 \|\rhot,\Psit\|^2_{m,\sigma(m)}\approx \sum_{k=0}^{\frac m2}\int \chi_{m,k,\sigma(m)}\left[b^2|\nabla\rhot_k|^2+\rho_D^{p-1}{|\rhot_k|^2}+\rho_D^2{|\nabla \Psit_k|^2}\right].
 \ee
 Let us briefly sketch the proof. For $d=2$, we compute

$$
\pa_1\pa_2 f(r)= \frac {z_1z_2}{r^2} \pa_r^2 f(r) - \frac{z_1z_2}{r^3} \pa_r f\
$$
and $$
\pa_1^2f(r)+\pa_2^2 f(r)= \frac 1r\pa_r(r \pa_r f(r)),
$$
so that for $f$ regular at the origin,
$$
\frac 1r \pa_r f(r))= \frac 1{r^2}\int_0^r r'(\pa_1^2f(r')+\pa_2^2 f(r')) dr'.
$$
As a result we obtain
$$
|\nabla^2 f(Z)|\lesssim |\pa^2 f(Z)|+\frac {|\pa f(Z)|}{\la Z\ra} + \int_{|Z|\le 1}  |\pa^2 f|.
$$
The desired claim for the term
$$
\int\chi_{2,m,\sigma}\rho_D^{p-2}\rho_T|\nabla^2\rhot|^2\lesssim  \|\rhot,\Psi\|_{2,\sigma}^2
$$
then follows from the above inequality and the condition \eqref{boundarydata} $\rho_P(0)=\rho_D(0)=1$
together with the regularity of the profile $\rho_P$ and the non vanishing bound \eqref{nonvavnishig}.
The statement for other terms and higher derivatives follows by iteration.
 \end{remark}
 \noindent {\em 2. Global control of the highest Sobolev norm}:
\be
\label{sobolevinitboot}
  \|\rhot,\Psit\|^2_{k_m}= \|\rhot,\Psit\|^2_{k_m,\sigma(k_m)}\le 1.
\ee
\noindent {\em 3. Local decay of low Sobolev norms}: for any $0\le k\le 2k_0$, any large $\hat Z\le Z^*$ 
and universal constant $C=C(k_0)$:
\be\label{eq:bootdecay}
 \|(\rhot,\Psit)\|_{H^k(Z\le \hat{Z})}\le \hat Z^C  e^{-\frac{3\delta_g}8\tau}
\ee

\noindent {\em 4. Pointwise bounds}: 
\be
\label{smallglobalboot}
\left|\begin{array}{l}
\forall 0\le k\le \frac{2k_m}{3}, \ \ \|\frac{Z^{n(k)}\pa_Z^k\rhot}{\rho_D}\|_{L^\infty}\leq 1\\
\forall 1\le k\le \frac{2k_m}{3},\ \ \|Z^{n(k)}\la Z\ra^{r-2}\pa_Z^k\Psi\|_{L^\infty(Z \le Z^*)}+\frac{\|Z^{n(k)}\pa_Z^k\Psi\|_{L^\infty(Z\geq Z^*)}}{b}\le 1
\end{array}\right.
\ee
with
\be
\label{defnm}
n(k)=\left|\begin{array}{ll} k\ \ \mbox{for}\ \ k\leq \frac{4k_m}{9},\\  \frac{k_m}{4}\ \ \mbox{for}\ \ \frac{4k_m}{9}<k\leq\frac{2k_m}{3}.
\end{array}\right.
\ee
\begin{remark}
\label{rem:data}
Since $b=e^{-\mu(r-2)\tau}$, \eqref{smallnessoutsideinitbis} and \eqref{smallnessoutsideinit} imply that 
the initial data verify the bootstrap inequalities  \eqref{weightedl2boot}, \eqref{sobolevinitboot}, \eqref{smallglobalboot} with the bound $e^{-c\tau_0}$ for some small universal constant $c$.
\end{remark}

The heart of the proof of Theorem \ref{thmmain} is the following:

\begin{proposition}[Bootstrap]
\label{propboot}
Assume (see \eqref{eq:nilp}) that
\be\label{eq:unstboot}
\|e^{t{\mathcal N}} PX(\tau)\|_{\Bbb H_{2k_0}}\le e^{-\frac {19\delta_g}{30}\tau} 
\ee
for all  $\tau\in [\tau_0,\tau^*)$ and that the bounds \eqref{bootnonvanishing}, \eqref{weightedl2boot}, \eqref{sobolevinitboot}, \eqref{eq:bootdecay}, \eqref{smallglobalboot}, \eqref{nonvavnishig} 
hold on $[\tau_0,\tau^*]$ with ${\delta}^{-1},\tau_0$ large enough. Then the following holds:\\
\noindent{\em 1. Exit criterion}. The bounds  \eqref{bootnonvanishing}, \eqref{weightedl2boot}, \eqref{sobolevinitboot}, \eqref{eq:bootdecay}, \eqref{smallglobalboot}, \eqref{nonvavnishig}  can be strictly improved on $[\tau_0,\tau^*)$. Equivalently, $\tau^*<+\infty$ implies 
\be
\label{eibovbeiboebeo}
\|e^{t{\mathcal N}} PX(\tau^*)\|_{\Bbb H_{2k_0}} e^{\frac{19\delta_g}{30}\tau^*}=1.
 \ee
 \noindent{\em 2. Linear evolution}. The right hand side $G$ of the equation for $X(\tau)$ 
 $$
 \pa_\tau X = {\mathcal M} X + G 
 $$
 satisfies 
 \be\label{eq:Gest}
 \|G(\tau)\|_{\Bbb H_{2k_0}} \le e^{-\frac{2\delta_g}3\tau},\qquad \forall\tau\in [\tau_0,\tau^*]
 \ee
\end{proposition}

We will show in section \ref{proofthmmamin} that Proposition \ref{propboot} immediately implies Theorem \ref{thmmain}.

\begin{remark}
\label{weak}
We note that the assumption \eqref{eq:unstboot}
implies that 
\be\label{eq:unstX}
\|PX(\tau)\|_{\Bbb H_{2k_0}}\le e^{-\frac {\delta_g}{2}\tau}, \qquad \forall\tau\in [\tau_0,\tau^*)
\ee
We will prove the bootstrap proposition \ref{propboot} under the weaker assumption \eqref{eq:unstX}.
Specifically, we will define $[\tau_0,\tau^*]$ to be the maximal time interval on which \eqref{eq:unstX} holds 
and will show that both the bounds  \eqref{bootnonvanishing}, \eqref{weightedl2boot}, \eqref{sobolevinitboot}, \eqref{smallglobalboot}, \eqref{nonvavnishig} can be improved and that $G$ satisfies \eqref{eq:Gest}.
\end{remark}

We now focus on the proof of Proposition \ref{propboot} and work on a time interval $[0,\tau^*]$, $\tau_0<\tau^*\le+\infty$ on which  \eqref{bootnonvanishing}, \eqref{weightedl2boot}, \eqref{sobolevinitboot}, \eqref{eq:bootdecay}, \eqref{smallglobalboot}, \eqref{nonvavnishig}  and \eqref{eq:unstX} hold.


\section{Control of high Sobolev norms}
\label{sectionsobolev}


We first turn to the {\rm global in space} control of high Sobolev norms. This is an essential step to control the $b$ dependence of the flow and the dissipative structure which can neither be treated by spectral analysis nor perturbatively.\\ 

We claim an improvement of the bound  \eqref{weightedl2boot}, controlling all but the highest 
weighted Sobolev norm.

\begin{proposition}
\label{prophighsobloev}
The exists a universal constant $c_{k_m}>0$ such that for all $0\le m\le k_m-1$
\be
\label{eq:weightedl2boot}
\|\rhot,\Psit\|_{m,\sigma(m)}\leq e^{-c_{k_m}\tau}.
 \ee
 \end{proposition}
 
 The rest of this section is devoted to the proof of Proposition \ref{prophighsobloev}.
 

\subsection{Algebraic energy identity} 


We derive the energy identity for high Sobolev norms which in the hydrodynamical formulation has a quasilinear structure.\\

\noindent{\bf step 1} Equation for $\tilde{\rho}, {\Psi}$. Recall \eqref{fullflowrenormalized}:
$$\left|\begin{array}{ll}\pa_\tau \rho_T=-\rho_T\Delta \Psi_T-\frac{\mu\ell(r-1)}{2}\rho_T-\left(2\pa_Z\Psi_T+\mu Z\right)\pa_Z\rho_T\\
\rho_T\pa_\tau \Psi_T=b^2\Delta \rho_T-\left[|\nabla \Psi_T|^2+\mu(r-2)\Psi_T-1+\mu\Lambda \Psi_T+\rho_T^{p-1}\right]\rho_T
\end{array}\right.
$$
By construction
\be
\label{profileequationtilde}
\left|\begin{array}{ll}
|\nabla \Psit_P|^2+\rho_D^{p-1}+\mu(r-2)\Psit_P+\mu\Lambda \Psit_P-1=\tilde{\mathcal E}_{P,\Psi}\\
\pa_\tau \rho_D+\rho_D\left[\Delta \Psit_P+\frac{\mu\ell(r-1)}{2}+\left(2\pa_Z\Psit_P+\mu Z\right)\frac{\pa_Z\rho_D}{\rho_D}\right]=\tilde{\mathcal E}_{P,\rho}
\end{array}\right.
\ee
with $\Et$ supported in $Z\geq 3Z^*$. 
The linearized flow
\be
\label{exactliearizedflowtilde}
\left|\begin{array}{ll} \pa_\tau \rhot=-\rho_T\Delta \Psit-2\nabla\rho_T\cdot\nabla \Psit+H_1\rhot-H_2\Lambda \rhot-\tilde{\mathcal E}_{P,\rho}\\
\pa_\tau \Psit=b^2\frac{\Delta \rho_T}{\rho_T}-\left\{H_2\Lambda \Psit+\mu(r-2)\Psit+|\nabla \Psit|^2+(p-1)\rho_D^{p-2}\rhot +\NL(\rhot)\right\}-\Et_{P,\Psi}.
\end{array}\right.
\ee
with the nonlinear term $$\NL(\tilde{\rho})=(\rho_D+\rhot)^{p-1}-\rho_D^{p-1}-(p-1)\rho_D^{p-2}\rhot.$$
Note that the potentials  $$H_2=\mu+2\frac{\Psit'_P}{Z}, \ \ H_1=-\left(\Delta \Psit_P+\frac{\mu\ell(r-1)}{2}\right)$$ 
remain the same in these equations: they are not affected by the profile localization introduced by passing from $\rho_P$ 
to $\rho_D$.
We recall the Emden transform formulas \eqref{defhtwohunbis}:
\be
\label{eq:autroformluehonhtwo}
\left|\begin{array}{ll}
H_2=\mu(1-w)\\
H_1=\frac{\mu\ell}{2}(1-w)\left[1+\frac{\Lambda \sigma}{\sigma}\right]\\
\end{array}\right.
\ee
which, using \eqref{limitprofilesbsi}, \eqref{decayprofile}, yield the bounds:
 \be
\label{eq:esterrorpotentials}
\left|\begin{array}{llll}
H_2=\mu+O\left(\frac{1}{\la Z\ra^r}\right), \ \ H_1=-\frac{2\mu(r-1)}{p-1}+O\left(\frac{1}{\la Z\ra^r}\right)\\
 |\la Z\ra^j\pa_Z^j H_1|+||\la Z\ra^j\pa_Z^jH_2|\lesssim \frac 1{\la Z\ra^{r}}, \ \ j\ge 1
 \end{array}\right.
 \ee

Our main task is now to produce an energy identity for \eqref{exactliearizedflowtilde} which respects the quasilinear nature of \eqref{exactliearizedflowtilde} and does not loose derivatives.\\

\noindent{\bf step 2} Equation for derivatives. We recall the notation for the vector $\pa^k$:
$$\left|\begin{array}{l}
\pa^k:=(\pa_1^k,...,\pa_d^k)\\
 \rhot^{(k)}=\pa^k\rhot, \ \ \Psi^{(k)}=\pa^k\Psi.
 \end{array}\right.$$ 
 We use $$[\pa^k,\Lambda]=k\pa^k$$ to compute from \eqref{exactliearizedflowtilde}:
\bea
\label{estqthohrk}
\nonumber \pa_\tau \rhot^{(k)}&=&(H_1-kH_2)\rhot^{(k)}-H_2\Lambda \rhot^{(k)}-(\pa^k\rho_T)\Delta \Psit-k\pa\rho_T\pa^{k-1}\Delta \Psit-\rho_T\Delta\Psi^{(k)}\\
\nonumber &-& 2\nabla(\pa^k\rho_T)\cdot\nabla \Psit-2\nabla \rho_T\cdot\nabla \Psi^{(k)}\\
& + & F_1
\eea
with
\bea
\label{formluafone}
F_1&=&-\pa^k\tilde{\mathcal E}_{P,\rho}+[\pa^k,H_1]\rhot-[\pa^k,H_2]\Lambda \rhot\\
\nonumber &-& \sum_{\left|\begin{array}{ll} j_1+j_2=k\\ j_1\geq 2, j_2\geq 1\end{array}\right.}c_{j_1,j_2}\pa^{j_1}\rho_T\pa^{j_2}\Delta\Psit-\sum_{\left|\begin{array}{ll}j_1+j_2=k\\
j_1,j_2\geq 1\end{array}\right.}c_{j_1,j_2}\pa^{j_1}\nabla\rho_T\cdot\pa^{j_2}\nabla\Psit.
\eea
For the second equation:
\bea
\label{nkenononenon}
\nonumber \pa_\tau\Psi^{(k)}&=&b^2\left(\frac{\pa^k\Delta \rho_T}{\rho_T}-\frac{k\pa^{k-1}\Delta\rho_T\pa\rho_T}{\rho_T^2}\right)-kH_2\Psi^{(k)}-H_2\Lambda \Psi^{(k)}-\mu(r-2)\Psi^{(k)}-2\nabla \Psit\cdot\nabla \Psi^{(k)}\\
&-& \left[(p-1)\rho_{D}^{p-2}\rhot^{(k)}+k(p-1)(p-2)\rho_D^{p-3}\pa\rho_D\pa^{k-1}\rhot\right]+F_2
\eea
with
\bea
\label{estqthohrkbis}
\nonumber F_2&=& -\pa^k\tilde{\mathcal E}_{P,\Psi}+b^2\left[\pa^k\left(\frac{\Delta \rho_T}{\rho_T}\right)-\frac{\pa^k\Delta \rho_T}{\rho_T}+\frac{k\pa^{k-1}\Delta \rho_T\pa\rho_T}{\rho_T^2}\right]\\
\nonumber &-& [\pa^k,H_2]\Lambda \Psit-(p-1)\left([\pa^k,\rho_D^{p-2}]\rhot-k(p-2)\rho_D^{p-3}\pa\rho_D\pa^{k-1}\rhot\right)\\
&-& \sum_{j_1+j_2=k,j_1,j_2\geq 1}\pa^{j_1}\nabla\Psit\cdot\pa^{j_2}\nabla\Psit-\pa^k\NL(\rhot).
\eea

\noindent{\bf step 3} Algebraic energy identity. Let $\chi$ be a smooth function.  
We compute:
\bee
&&\frac 12\frac{d}{d\tau}\left\{\int b^2\chi|\nabla\rhot^{(k)}|^2+(p-1)\int \chi\rho_D^{p-2}\rho_T(\rhot^{(k)})^2+\int\chi\rho_T^2|\nabla \Psi^{(k)}|^2\right\}\\
\nonumber & = &\frac 12\int \pa_\tau \chi \left\{b^2|\nabla\rhot^{(k)}|^2+(p-1)\rho_D^{p-2}\rho_T(\rhot^{(k)})^2+\rho_T^2|\nabla \Psi^{(k)}|^2\right\}\\
& = & -eb^2\int\chi|\nabla\rhot^{(k)}|^2+\int \pa_\tau\rhot^{(k)}\left[-b^2\chi\Delta \rhot^{(k)}-b^2\nabla \chi\cdot\nabla \rhot^{(k)}+(p-1)\chi\rho_D^{p-2}\rho_T\rhot^{(k)}\right]\\
&+& \frac{p-1}{2}\int \chi(p-2)\pa_\tau\rho_D\rho_D^{p-3}\rho_T(\rhot^{(k)})^2+ \int\chi\pa_\tau\rho_T\left[\frac{p-1}{2}\rho_D^{p-2}(\rhot^{(k)})^2+\rho_T|\nabla \Psi^{(k)}|^2\right]\\
&-&  \int\pa_\tau\Psi^{(k)}\left[2\chi\rho_T\nabla\rho_T\cdot\nabla \Psi^{(k)}+\chi\rho_T^2\Delta \Psi^{(k)}+\rho_T^2\nabla \chi\cdot\nabla \Psi^{(k)}\right].
\eee
We compute:
\bee
&&\int\pa_\tau \rhot^{(k)}\left[-b^2\chi\Delta \rhot^{(k)}-b^2\nabla \chi\cdot\nabla \rhot^{(k)}+(p-1)\chi\rho_D^{p-2}\rho_T\rhot^{(k)}\right]\\
& = & \int F_1\left[-b^2\nabla \cdot(\chi\nabla \rhot^{(k)})+(p-1)\chi\rho_D^{p-2}\rho_T\rhot^{(k)}\right]\\
& + & \int \left[(H_1-kH_2)\rhot^{(k)}-H_2\Lambda \rhot^{(k)}-(\pa^k\rho_T)\Delta \Psit-2\nabla(\pa^k\rho_T)\cdot\nabla \Psit\right]\\
&\times&\left[-b^2\chi\Delta \rhot^{(k)}-b^2\nabla \chi\cdot\nabla \rhot^{(k)}+(p-1)\chi\rho_D^{p-2}\rho_T\rhot^{(k)}\right]\\
& - & \int k\pa\rho_T\pa^{k-1}\Delta \Psit\left[-b^2\chi\Delta \rhot^{(k)}-b^2\nabla \chi\cdot\nabla \rhot^{(k)}+(p-1)\chi\rho_D^{p-2}\rho_T\rhot^{(k)}\right]\\
&- & \int (\rho_T\Delta\Psi^{(k)}+2\nabla \rho_T\cdot\nabla \Psi^{(k)})[\left[-b^2\chi\Delta \rhot^{(k)}-b^2\nabla \chi\cdot\nabla \rhot^{(k)}+(p-1)\chi\rho_D^{p-2}\rho_T\rhot^{(k)}\right]\\
& = & b^2\int \chi \nabla F_1\cdot\nabla \rhot^{(k)}+(p-1)\int \chi F_1\rho_D^{p-2}\rho_T\rhot^{(k)}\\
& + & \int \left[(H_1-kH_2)\rhot^{(k)}-H_2\Lambda \rhot^{(k)}-(\pa^k\rho_T)\Delta \Psit-2\nabla(\pa^k\rho_T)\cdot\nabla \Psit\right]\\
&\times&  \left[-b^2\nabla \cdot(\chi \nabla \rhot^{(k)})+(p-1)\chi\rho_D^{p-2}\rho_T\rhot^{(k)}\right]\\
& - & \int k\pa\rho_T\pa^{k-1}\Delta \Psit\left[-b^2\nabla \cdot(\chi\nabla \rhot^{(k)})+(p-1)\chi\rho_D^{p-2}\rho_T\rhot^{(k)}\right]\\
& + & b^2\int\nabla \chi\cdot\nabla \rhot^{(k)}\left(\rho_T\Delta \Psi^{(k)}+2\nabla\rho_T\cdot\nabla \Psi^{(k)}\right)\\
& - & \int \chi(\rho_T\Delta\Psi^{(k)}+2\nabla \rho_T\cdot\nabla \Psi^{(k)})\left[-b^2\Delta \rhot^{(k)}+(p-1)\rho_D^{p-2}\rho_T\rhot^{(k)}\right]
\eee
Similarly:
\bee
&-&  \int\pa_\tau\Psi^{(k)}\left[2\chi\rho_T\nabla\rho_T\cdot\nabla \Psi^{(k)}+\chi\rho_T^2\Delta \Psi^{(k)}+\rho_T^2\nabla \chi\cdot\nabla \Psi^{(k)}\right]=  -\int F_2\nabla\cdot(\chi\rho_T^2\nabla \Psi^{(k)})\\
& - & \int\left\{b^2\left(\frac{\pa^k\Delta \rho_T}{\rho_T}-\frac{k\pa^{k-1}\Delta\rho_T\pa\rho_T}{\rho_T^2}\right)\right\}\left[2\chi\rho_T\nabla\rho_T\cdot\nabla \Psi^{(k)}+\chi\rho_T^2\Delta \Psi^{(k)}+\rho_T^2\nabla \chi\cdot\nabla \Psi^{(k)}\right]\\
&-& \int\Big\{-kH_2\Psi^{(k)}-H_2\Lambda \Psi^{(k)}-\mu(r-2)\Psi^{(k)}-2\nabla \Psit\cdot\nabla \Psi^{(k)}\\
&-& \left[(p-1)\rho_{D}^{p-2}\rhot^{(k)}+k(p-1)(p-2)\rho_D^{p-3}\pa\rho_D\pa^{k-1}\rhot\right]\Big\}\\
&\times& \left[2\chi\rho_T\nabla\rho_T\cdot\nabla \Psi^{(k)}+\chi\rho_T^2\Delta \Psi^{(k)}+\rho_T^2\nabla \chi\cdot\nabla \Psi^{(k)}\right]\\
& = & \int \chi\rho^2_T\nabla \Psi^{(k)}\cdot\nabla F_2\\
& - & b^2\int(\pa^k\Delta \rho_{D}+\Delta \rhot^{(k)})\left[2\chi\nabla\rho_T\cdot\nabla \Psi^{(k)}+\chi\rho_T\Delta \Psi^{(k)}+\rho_T\nabla \chi\cdot\nabla \Psi^{(k)}\right]\\
& + &b^2\int  \frac{k\pa^{k-1}\Delta\rho_T\pa\rho_T}{\rho_T}\left[2\chi\nabla\rho_T\cdot\nabla \Psi^{(k)}+\chi\rho_T\Delta \Psi^{(k)}+\rho_T\nabla \chi\cdot\nabla \Psi^{(k)}\right]\\
&-& \int\left[-kH_2\Psi^{(k)}-H_2\Lambda \Psi^{(k)}-\mu(r-2)\Psi^{(k)}-2\nabla \Psit\cdot\nabla \Psi^{(k)}\right]\nabla\cdot(\chi\rho_T^2\nabla \Psi^{(k)})\\
& +& \int (p-1)\rho_{D}^{p-2}\rhot^{(k)}\left[2\chi\rho_T\nabla\rho_T\cdot\nabla \Psi^{(k)}+\chi\rho_T^2\Delta \Psi^{(k)}+\rho_T^2\nabla \chi\cdot\nabla \Psi^{(k)}\right]\\
& + & \int k(p-1)(p-2)\rho_D^{p-3}\pa\rho_D\pa^{k-1}\rhot\nabla\cdot(\chi\rho_T^2\nabla \Psi^{(k)})\\
& = & \int \chi\rho^2_T\nabla \Psi^{(k)}\cdot\nabla F_2-  b^2\int(\pa^k\Delta \rho_{D})\nabla\cdot(\chi\rho_T^2\nabla \Psi^{(k)})\\
& +&\int(-b^2\Delta \rhot^{(k)}+(p-1)\rho_D^{p-2}{\rho_T}\rhot^{(k)})\left[2\chi\nabla\rho_T\cdot\nabla \Psi^{(k)}+\chi\rho_T\Delta \Psi^{(k)}+\rho_T\nabla \chi\cdot\nabla \Psi^{(k)}\right]\\
& + &b^2\int  \frac{k\pa^{k-1}\Delta\rho_T\pa\rho_T}{\rho^2_T}\nabla\cdot(\chi\rho_T^2\nabla \Psi^{(k)})\\
&-& \int\left[-kH_2\Psi^{(k)}-H_2\Lambda \Psi^{(k)}-\mu(r-2)\Psi^{(k)}-2\nabla \Psit\cdot\nabla \Psi^{(k)}\right]\nabla\cdot(\chi\rho_T^2\nabla \Psi^{(k)})\\
& + & \int k(p-1)(p-2)\rho_D^{p-3}\pa\rho_D\pa^{k-1}\rhot\nabla\cdot(\chi\rho_T^2\nabla \Psi^{(k)}).
\eee
This yields the algebraic energy identity:
\bea
\label{algebracienergyidnentiy}
\nonumber&&\frac 12\frac{d}{d\tau}\left\{\int b^2\chi|\nabla\rhot^{(k)}|^2+(p-1)\int \chi\rho_D^{p-2}\rho_T(\rhot^{(k)})^2+\int\chi\rho_T^2|\nabla \Psi^{(k)}|^2\right\}\\
\nonumber & = &{\frac 12\int \pa_\tau \chi \left\{b^2|\nabla\rhot^{(k)}|^2+(p-1)\rho_D^{p-2}\rho_T(\rhot^{(k)})^2+\rho_T^2|\nabla \Psi^{(k)}|^2\right\}}\\
 \nonumber &-&  b^2\int(\pa^k\Delta \rho_{D})\nabla\cdot(\chi\rho_T^2\nabla \Psi^{(k)})\\
\nonumber &-&eb^2\int\chi|\nabla\rhot^{(k)}|^2+\int\chi\frac{\pa_\tau\rho_T}{\rho_T}\left[\frac{p-1}{2}\rho_D^{p-2}\rho_T(\rhot^{(k)})^2+\rho^2_T|\nabla \Psi^{(k)}|^2\right]\\
\nonumber &+& \frac{p-1}{2}\int \chi(p-2)\frac{\pa_\tau\rho_D}{\rho_D}\rho_D^{p-2}\rho_T(\rhot^{(k)})^2\\
\nonumber & + & \int F_1\chi(p-1)\rho_D^{p-2}\rho_T\rhot^{(k)}+b^2\int \chi\nabla F_1\cdot\nabla \rhot^{(k)}+\int \chi\rho^2_T\nabla F_2\cdot\nabla \Psi^{(k)}\\
\nonumber  & + & \int \left[(H_1-kH_2)\rhot^{(k)}-H_2\Lambda \rhot^{(k)}-(\pa^k\rho_T)\Delta \Psit-2\nabla(\pa^k\rho_T)\cdot\nabla \Psit\right]\\
\nonumber &\times &  \left[-b^2\nabla \cdot(\chi \nabla\rhot^{(k)})+(p-1)\chi\rho_D^{p-2}\rho_T\rhot^{(k)}\right]\\
\nonumber&-& \int\left[-kH_2\Psi^{(k)}-H_2\Lambda \Psi^{(k)}-\mu(r-2)\Psi^{(k)}-2\nabla \Psit\cdot\nabla \Psi^{(k)}\right]\nabla\cdot(\chi\rho_T^2\nabla \Psi^{(k)})\\
\nonumber& - & \int k\pa\rho_T\pa^{k-1}\Delta \Psit\left[-b^2\nabla \cdot(\chi \nabla \rhot^{(k)})+(p-1)\chi\rho_D^{p-2}\rho_T\rhot^{(k)}\right]\\
\nonumber& + &b^2\int  \frac{k\pa^{k-1}\Delta\rho_T\pa\rho_T}{\rho^2_T}\nabla\cdot(\chi\rho_T^2\nabla \Psi^{(k)})\\
\nonumber& + & \int k(p-1)(p-2)\rho_D^{p-3}\pa\rho_D\pa^{k-1}\rhot\nabla\cdot(\chi\rho_T^2\nabla \Psi^{(k)})\\
\nonumber& + & b^2\int\nabla \chi\cdot\nabla \rhot^{(k)}\left(\rho_T\Delta \Psi^{(k)}+2\nabla\rho_T\cdot\nabla \Psi^{(k)}\right)\\
&+& \int(-b^2\Delta \rhot^{(k)}+(p-1)\rho_{D}^{p-2}{\rho_T}\rhot^{(k)})\left[\rho_T\nabla \chi\cdot\nabla \Psi^{(k)}\right].
\eea


\subsection{Weighted $L^2$ bound for $m\le k_m-1$}


 Given $\sigma\in \Bbb R$, we recall the notation
$$
\left|\begin{array}{ll}
\|\rhot,\Psit\|_{k,\sigma}^2=\sum_{m=0}^k\sum_{i=1}^d\int \chi_{k,m,\sigma}\left[b^2|\nabla\rhot_m|^2+(p-1)\rho_D^{p-2}\rho_T\rhot_m^2+\rho_T^2|\nabla \Psit_m|^2\right]\\
\chi_{k,m,\sigma}(Z)=\frac{1}{\la Z\ra^{2(k-m+\sigma)}} {\xi_{{k}}\left(\frac Z{Z^*}\right)}
\end{array}\right.
$$
We let \be
\label{defimsigma}
I_{k,\sigma}=\int \frac{\xi_k\left(\frac Z{Z^*}\right)}{\la Z\ra^{2\sigma}}\left[b^2|\nabla\rhot^{(k)}|^2+(p-1)\rho_D^{p-2}\rho_T(\rhot^{(k)})^2+\rho_T^2|\nabla \Psi^{(k)}|^2\right].
\ee

\begin{lemma}[Weighted $L^2$ bound]
\label{propinduction}
Recall the definition \eqref{defsigmnu}, \eqref{defsimgma} of $\sigma(m)$ and let
\be
\label{conditionsigma}
\left|\begin{array}{l}
\sigma= \sigma(k)\\
\nu+\frac{2(r-1)}{p-1}=\nut
\end{array}\right. 
\ee
then there exists $c_{k_m}>0$ such that for all $0<\nut<\nut(k_m)\ll1$ and $b_0<b_0(k_m)\ll 1$, for all $1\le k\le k_m-1$, $I_k:=I_{k,\sigma(k)}$ given by \eqref{defimsigma} satisfies the differential inequality
\be
\label{estnienonneoinduction}
 \frac{dI_k}{d\tau}+2\mu\nut I_k\le e^{-c_{k_m}\tau}.
\ee
\end{lemma}

We claim that Lemma \ref{propinduction} implies Proposition \ref{prophighsobloev}.

\begin{proof}[Proof of Proposition \ref{prophighsobloev}]

Integrating \eqref{estnienonneoinduction} on the interval $[\tau_0,\tau],$ with initial data prescribed at $\tau_0$, we obtain 
$$
I_k(\tau) \le e^{-2\mu\nut(\tau-\tau_0)} I_k(\tau_0) + \left(e^{-2\mu\nut(\tau-\tau_0)-c_{k_m}\tau_0}-e^{-c_{k_m}\tau}\right) 
$$
We now recall, see Remark \ref{rem:data}, that $I_k(\tau_0)\le e^{-c\tau_0}$. Choosing  $2\mu\nut\le\min\{c,c_{k_m}\}$ 
we obtain that
\be\label{eq:Im}
I_k(\tau) \le 2 e^{-2\mu\nut\tau}.
\ee
We now recall from \eqref{eq:psirho} and \eqref{eq:psirho'} for even $m$ that $\|\rhot,\Psit\|_{m,\sigma}$ controls all the corresponding
 Sobolev norms: let a multi-index $\alpha=(\alpha_1,\dots,\alpha_d)$ with $$\alpha_1+\dots+\alpha_d=|\alpha|,\ \ \nabla^\alpha:=\pa_1^{\alpha_1}\dots\pa^{\alpha_d}_d,$$  then for all $|\alpha|=k, \ \ 0\leq k\leq m$,
\bea
\label{interpolationnorms}
\nonumber &&b^2\int \chi_{k,m,\sigma}|\nabla\nabla^\alpha \rhot|^2+(p-1)\int \chi_{k,m,\sigma}\rho_D^{p-2}\rho_T|\nabla^\alpha\rhot|^2+\int\chi_{k,m.\sigma}\rho_T^2|\nabla\nabla^\alpha \Psit|^2\\
&\lesssim& \|\rhot,\Psi\|_{k,\sigma}^2,
\eea
and similarly the norm $\|\rhot,\Psi\|_{k,\sigma}^2$ (with even $k$) is equivalent to the one 
where $\pa^m$ with $1\le m\le k$ derivatives are replaced by $\Delta^m$ with $1\le m \le \frac k2$.\\
We now claim 
\be\label{eq:km}
\|\rho,\Psi\|_{m,\sigma(m)}^2\le \sum_{k=0}^m I_{k,\sigma(k)}.
\ee
Combining this with \eqref{eq:Im} concludes the proof of \eqref{eq:weightedl2boot} (with $c_{k_m}=2\mu\nut$).\\
\noindent{\em Proof of \eqref{eq:km}}. {Indeed,
\bee
&&\|\rho,\Psi\|_{m,\sigma(m)}^2=\sum_{k=0}^m\int \chi_{m,k,\sigma(m)}\left[b^2|\nabla\rhot^{(k)}|^2+(p-1)\rho_D^{p-2}\rho_T(\rhot^{(k)})^2+\rho_T^2|\nabla \Psit^{(k)}|^2\right]\\
& = & \frac{1}{\la Z\ra^{2(m+\sigma(m))}}\sum_{k=0}^m\int \la Z\ra^{2k}\xi_m(x)\left[b^2|\nabla\rhot^{(k)}|^2+(p-1)\rho_D^{p-2}\rho_T(\rhot^{(k)})^2+\rho_T^2|\nabla \Psit^{(k)}|^2\right]
\eee
and
\bee
&&\sum_{k=0}^m I_{k,\sigma(k)}=\sum_{k=0}^m\int \frac{\xi_k(x)}{\la Z\ra^{2\sigma(k)}}\left[b^2|\nabla\rhot^{(k)}|^2+(p-1)\rho_D^{p-2}\rho_T(\rhot^{(k)})^2+\rho_T^2|\nabla \Psi^{(k)}|^2\right]\\
& = & \sum_{k=0}^m\int \frac{\la Z\ra^{2k}\xi_k(x)}{\la Z\ra^{2(\sigma(k)+k)}}\left[b^2|\nabla\rhot^{(k)}|^2+(p-1)\rho_D^{p-2}\rho_T(\rhot^{(k)})^2+\rho_T^2|\nabla \Psi^{(k)}|^2\right]
\eee
and hence \eqref{eq:km} follows from $\sigma(k)+k\leq \sigma(m)+m$ and $\xi_k(x)\ge \xi_m(x)$ for  $0\le k\le m$.
}
\end{proof}


\subsection{Proof of Lemma \ref{propinduction}} This follows from the energy identity \eqref{algebracienergyidnentiy} coupled with the pointwise bound \eqref{smallglobalboot} to control the nonlinear term.\\


 \noindent{\bf step 1} Interpolation bounds. In what follows we use the convention $\lesssim$ to denote any dependence on the universal constants, including $k_m$.
Constants $c, c_{k_m}$ will stand for generic, universal small constant.\\
 Our main technical tool below will be the following interpolation bound: for any $0\le m\le k_m-1$ and $\delta>0$, there exists $c_{\delta,k_m}>0$ such that
\be
\label{interpolatedbound}
\|\rhot,\Psit\|^2_{m,\sigma(m)+\delta}\le e^{-c_{\delta,k_m}\tau}.
\ee
Indeed, the claim follows by interpolating the local decay bootstrap bound \eqref{eq:bootdecay} and the 
bound \eqref{sobolevinitboot} for the highest Sobolev norm for $Z\le Z^*_c:=(Z^*)^c$ 
and using the global weighted Sobolev bound for \eqref{weightedl2boot} for $Z\ge Z^*_c$
\be
\label{inerpo;soationcrucial}
\|\rhot,\Psit\|^2_{m,\sigma(m)+\delta}\le (Z^*)^Ce^{-c_{k_m}\tau}+\frac{1}{(Z^*_c)^{2\delta}}\|\rhot,\Psit\|^2_{m,\sigma(m)}\le e^{-c_{\delta,k_m}\tau}
\ee
We will also use the bound for the damped profile from \eqref{dampenedprofile}, \eqref{definitionprofilewithtailchange} and \eqref{fromularhod}: 
\be
\label{estiamtionnprofile}
|Z^k\pa_Z^k\rho_D|\lesssim \frac{1}{\la Z\ra^{\frac{2(r-1)}{p-1}}}{\bf 1}_{Z\le Z^*}+\frac{1}{(Z^*)^{\frac{2(r-1)}{p-1}}}\frac{1}{\left(\frac{Z}{Z^*}\right)^{n_P}}{\bf 1}_{Z\ge Z^*}.
\ee
We will also use the bound
\be\label{eq:chik}
 \chi_{k-1,k-1,\sigma(k-1)}\le  \la Z\ra^{2(\alpha+\beta)}\chi_{k,k,\sigma(k)}\le  \la Z\ra^{2}\chi_{k,k,\sigma(k)},
\ee
which follows from
\be\label{sigmalpha}
\sigma(k-1)+\alpha \ge \sigma(k),\qquad \tilde\sigma(k-1)\le \tilde\sigma(k)+\beta
\ee       
and $\alpha+\beta\le 1$.\\

\noindent{\bf step 2} Energy identity. We run \eqref{algebracienergyidnentiy} with $$\chi=\frac{1}{\la Z\ra^{2\sigma}}
\xi_k\left(\frac{Z}{Z^*}\right), \ \  \sigma=\sigma(k), \ \ 1\leq k\leq k_m-1$$
with $\xi_k(x)=1$ for $x\le 1$ and $\xi_k(x)=x^{2\tilde\sigma(k)}$ for $x>1$,
 and estimate all terms. In our notations
 $$
 \chi=\chi_{k,k,\sigma(k)}.
 $$

From \eqref{defsigmnu}, \eqref{defsimgma} and recalling $m_0=\frac{4k_m}{9}+1$:
\bea
\label{defsigmnavrelations}
\nonumber \sigma(k)+k&=&\left|\begin{array}{l}\sigma_\nu\ \  \mbox{for}\ \ 0\le k\le m_0\\-\alpha(k_m-k)+k=(\alpha+1)(k-m_0)+\sigma_\nu\ \  \mbox{for}\ \ m_0\leq m\leq k_m\end{array}\right.\\
&\ge& \sigma_\nu
\eea
and
\bea
\label{defsigmnavrelations-sigma}
\nonumber \tilde\sigma(k)&=&\left|\begin{array}{l} n_P-\frac{2(r-1)}{p-1}-(r-2)+2\nut\ \  \mbox{for}\ \ 0\le k\le \frac {2k_m}3+1\\ \beta(k_m-k)\ \  \mbox{for}\ \  \frac {2k_m}3+1\leq m\leq k_m\end{array}\right.\\
&\le& n_P-\frac{2(r-1)}{p-1}-(r-2)+2\nut
\eea
which implies
\bea
\label{uppperboundchi}
\chi&=&\frac{1}{\la Z\ra^{2\sigma(k)}}\xi_k\left(\frac{Z}{Z^*}\right)\\
\nonumber &\lesssim& \frac{1}{\la Z\ra^{2\sigma(k)}}
\left[1+ \left (\frac Z{Z^*}\right)^{ 2n_P-\frac{4(r-1)}{p-1}-2(r-2)+4\nut}{\bf 1}_{Z\ge Z^*}\right]\\ &\lesssim &\frac{1}{\la Z\ra^{-2k+2\left(\frac{d}{2}+\tilde{\nu}-\frac{2(r-1)}{p-1}-(r-1)\right)}} +\frac{\left (\frac Z{Z^*}\right)^{ 2n_P-\frac{4(r-1)}{p-1}-2(r-2)+4\nut}}{\la Z\ra^{-2k+2\left(\frac{d}{2}+\tilde{\nu}-\frac{2(r-1)}{p-1}-(r-1)\right)}}{\bf 1}_{Z\ge Z^*}\nonumber.
\eea
which we will use below. The following additional inequality will be of particular significance ($b=(Z^*)^{2-r}$):
\bea
\label{eq:Dsigma}
\nonumber \rho_D^2 \chi &\lesssim& \frac{1}{\la Z\ra^{-2k+2\left(\frac{d}{2}+\tilde{\nu}-(r-1)\right)}} \left[{\bf 1}_{Z\le Z^*}+
{\left (\frac Z{Z^*}\right)^{4\nut-2(r-2)}} {\bf 1}_{Z\ge Z^*}\right]\\&=&\frac{1}{\la Z\ra^{-2k+2\left(\frac{d}{2}+\tilde{\nu}-(r-1)\right)}} {\bf 1}_{Z\le Z^*}+
\frac{1}{b^{2-\frac{4\nut}{r-2}} \la Z\ra^{-2k+2\left(\frac{d}{2}-\tilde{\nu}-1\right)}} {\bf 1}_{Z\ge Z^*}\label{eq:Dsigma'}
\eea
\\

\noindent{\bf step 3} Leading order terms. In what follows, we will systematically use the standard Pohozhaev identity:
\bea
\label{pohozaevbispouet}
\nonumber \int\Delta g F\cdot\nabla gdx&=&\sum_{i,j=1}^d \int \pa_i^2 g F_j\pa_jgdx=-\sum_{i,j=1}^d \int\pa_ig(\pa_iF_j\pa_jg+F_j\pa^2_{i,j}g)\\
&=&-\sum_{i,j=1}^d \int\pa_iF_j\pa_ig\pa_jg+\frac 12\int |\nabla g|^2\nabla \cdot F
\eea
which becomes in the case of spherically symmetric functions
\bee
\nonumber \int_{\Bbb R^d} f\Delta g \pa_rgdx=c_d\int_{\Bbb R^+}\frac{f}{r^{d-1}} \pa_r(r^{d-1}\pa_rg)r^{d-1}\pa_rgdr=-\frac 12\int_{\Bbb R^d}|\pa_rg|^2\left[f'-\frac{d-1}{r}f\right]dx
\eee

\noindent\underline{\em Cross terms}. We consider
\bee
A_1=b^2k\left[\int \pa\rho_T\pa^{k-1}\Delta \Psit \nabla \cdot(\chi \nabla \rhot^{(k)})+\frac{\pa^{k-1}\Delta\rho_T\pa\rho_T}{\rho^2_T}\nabla\cdot(\chi\rho_T^2\nabla \Psi^{(k)})\right].
\eee
We compute:
\bee
&&\pa\rho_T\pa^{k-1}\Delta \Psit \nabla \cdot(\chi \nabla \rhot^{(k)})+\frac{\pa^{k-1}\Delta\rho_T\pa\rho_T}{\rho^2_T}\nabla\cdot(\chi\rho_T^2\nabla \Psi^{(k)})\\
& = & \pa\rho_T\pa^{k-1}\Delta \Psit\left[\nabla\chi\cdot\nabla \rhot^{(k)}+\chi\Delta \rhot^{(k)}\right]\\
& + & \pa^{k-1}\Delta\rho_T\pa\rho_T\left[\nabla\chi \cdot\nabla \Psi^{(k)}+2\chi\frac{\nabla \rho_T}{\rho_T}\cdot\nabla \Psi^{(k)}+\chi\Delta \Psi^{(k)}\right]\\
& = & \pa\rho_T\pa^{k-1}\Delta \Psi\nabla\chi\cdot\nabla \rhot^{(k)}+\pa^{k-1}\Delta\rho_T\pa\rho_T\nabla\chi \cdot\nabla \Psi^{(k)}+2\pa^{k-1}\Delta\rho_T\pa\rho_T\chi\frac{\nabla \rho_T}{\rho_T}\cdot\nabla \Psi^{(k)}\\
& +  & \chi\pa\rho_T\pa^{k-1}\Delta \Psit\Delta \rhot^{(k)}+\chi\pa^{k-1}\Delta\rho_T\pa\rho_T\Delta \Psi^{(k)}.
\eee
The last 2 terms require an integration by parts:
\bee
\nonumber &&b^2k\left|\int \left[\chi\pa\rho_T\pa^{k-1}\Delta \Psit\Delta \rhot^{(k)}+\chi\pa^{k-1}\Delta\rho_T\pa\rho_T\Delta \Psi^{(k)}\right]\right|\\
\nonumber &=&b^2k\left| \int \left[-(\Delta \pa^{k-1}\rhot)\pa(\chi\pa\rho_T\pa^{k-1}\Delta \Psit)+\chi\pa^{k-1}\Delta\rho_T\pa\rho_T\Delta\Psi^{(k)}\right]\right|\\
\nonumber & = & b^2k\left|\int\left[-\pa^{k-1}\Delta \rhot\left[\chi\pa^2\rho_T\pa^{k-1}\Delta \Psit+\pa\chi\pa\rho_T\pa^{k-1}\Delta \Psit)\right]+\chi\pa^{k-1}\Delta\rho_D\pa\rho_T\Delta\Psi^{(k)}\right]\right|\\
\nonumber &\lesssim &  b^2k\int\chi|\pa^{k-1}\Delta \Psit|\left[|\pa^{k-1}\Delta \rhot\pa^2\rho_T|+\frac{|\pa^{k-1}\Delta \rhot\pa\rho_T|}{\la Z\ra}+|\pa(\pa^{k-1}\Delta\rho_D\pa\rho_T)|\right]\\
\nonumber & \lesssim &  C_kb^2\int\chi\rho_T|\pa^{k-1}\Delta \Psit|\left[\frac{\rho_D}{\la Z\ra^{k+2}}+\frac{|\pa^{k-1}\Delta \rhot|}{\la Z\ra}\right]\\
& \lesssim & \sum_{|\alpha|=k}\int\frac{\chi\rho_T^2}{\la Z\ra}|\nabla \pa^\alpha\Psit|^2+c_kb^4\int \frac{\chi}{\la Z\ra}|\nabla \pa^\alpha\rhot|^2+b^4\int\chi\frac{\rho_{D}^2}{\la Z\ra^{2k+3}},
\eee
where in penultimate inequality we used the pointwise bound \eqref{smallglobalboot}.
 
We now estimate the source term from\eqref{eq:Dsigma}:
\bea
\label{nivneinveoneinv}
\nonumber 
b^4\int\chi\frac{\rho_D^2}{\la Z\ra^{2k+3}}&\lesssim& b^4\int_{Z\leq Z^*}\frac{Z^{d-1}dZ}{\la Z \ra^{2k+3-2k+2\left(\frac{d}{2}+\tilde{\nu}-(r-1)\right)}}\\
\nonumber & + & b^2\int_{Z\ge Z^*}\left(\frac Z{Z^*}\right)^{4\nut}\frac{Z^{d-1}dZ}{\la Z \ra^{2k+3-2k+2\left(\frac{d}{2}+\tilde{\nu}-1\right)}}\\
\nonumber &\lesssim& b^4\int_{Z\le Z^*}\la Z\ra^{2(r-2)-2-2\nut}dZ+b^2\int_{Z\ge Z^*}\left(\frac Z{Z^*}\right)^{4\nut}{\la Z\ra^{-2-2\nut}} dZ\\
& \lesssim& b^4(Z^*)^{2(r-2)-1-2\nut}\lesssim e^{-c\tau}
\eea
and hence, using \eqref{inerpo;soationcrucial}, 
\bee
&&\left|b^2k\left|\int \left[\chi\pa\rho_T\pa^{k-1}\Delta \Psit\Delta \rhot^{(k)}+\chi\pa^{k-1}\Delta\rho_T\pa\rho_T\Delta \Psi^{(k)}\right]\right|\right|\lesssim e^{-c\tau}+\|\rhot,\Psit\|^2_{k,\sigma+\frac 12}\\
&\lesssim & e^{-c_{k_m}\tau}.
\eee
We estimate similarly,
\bee
&&kb^2\left| \pa\rho_T\pa^{k-1}\Delta \Psi\nabla\chi\cdot\nabla \rhot^{(k)}+\pa^{k-1}\Delta\rho_T\pa\rho_T\nabla\chi \cdot\nabla \Psi^{(k)}+2\pa^{k-1}\Delta\rho_T\pa\rho_T\chi\frac{\nabla \rho_T}{\rho_T}\cdot\nabla \Psi^{(k)}\right|\\
&\lesssim &  \sum_{|\alpha|=k}\left[b^4\int\frac{\chi}{\la Z\ra}|\nabla\pa^\alpha\rhot|^2+\int\frac{\chi}{\la Z\ra^{3}}\rho_T^2|\nabla \pa^\alpha\Psit|^2\right]+b^4\int\chi\frac{\rho_{D}^2}{\la Z\ra^{2k+3}}\\
&\lesssim &  e^{-c_{k_m}\tau}+\|\rhot,\Psit\|^2_{k,\sigma+\frac 12}\lesssim e^{-c_{k_m}\tau}.
\eee
The remaining cross terms are estimated as follows.
\bee
&&k(p-1)\left|\int\chi\pa\rho_T\pa^{k-1}\Delta\Psit\rho_D^{p-2}\rho_T\rhot^{(k)}\right|\lesssim c_k\int \chi\frac{\rho_T^{p-1}}{\la Z\ra}|\rho_T\pa^{k-1}\Delta \Psi|\rhot^{(k)}|\\
&\lesssim &\int \frac{\chi}{\la Z\ra}\rho_D^{p-1}(\rhot^{(k)})^2+\int\frac{\chi}{\la Z\ra}\rho_T^2|\nabla \pa^\alpha\Psit|^2\le\|\rhot,\Psit\|^2_{k,\sigma+\frac 12}\lesssim e^{-c_{k_m}\tau},
\eee
where we used that $p\ge 1$ and a trivial bound $|\rho_D|\lesssim 1$.
Similarly,
$$\int\left|(p-1)\rho_D^{p-2}\rhot^{(k)}\rhot^{2}_T\nabla \chi\cdot\nabla \Psi^{(k)}\right|
\int \frac{\chi}{\la Z\ra}\rho_D^{p-1}(\rhot^{(k)})^2+\int\frac{\chi}{\la Z\ra}\rho_T^2|\nabla \pa^\alpha\Psit|^2\le\|\rhot,\Psit\|^2_{k,\sigma+\frac 12}\lesssim e^{-c_{k_m}\tau}.
$$
The other remaining cross term is estimated using an integration by parts:
\bee
&& k(p-1)(p-2)\left|\int \nabla \cdot(\rho_T^2\nabla \Psi^{(k)})\chi\rho_D^{p-3}\pa\rho_D\pa^{k-1}\rhot\right|\\
&\lesssim &  \int \frac{\chi}{\la Z\ra}\rho_D^{p-1}|\nabla \rhot_{k-1}|^2+ \int \frac{\chi}{\la Z\ra^3}\rho_D^{p-1}\rhot_{k-1}^2+\int\frac{\chi}{\la Z\ra}\rho_T^2|\nabla \pa^\alpha\Psit|^2\le
\|\rho,\Psit\|_{k,\sigma+\frac 12}^2\\
&\lesssim& e^{-c_{k_m}\tau}.
\eee

\noindent\underline{\em $\rho_k$ terms}. We compute using \eqref{eq:esterrorpotentials}:
\bee
&&\int\chi(H_1-kH_2)\rhot^{(k)}(-b^2\Delta \rhot^{(k)}+(p-1)\rho_D^{p-2}\rho_T\rhot^{(k)})-b^2\int\left[H_1-kH_2\right]\rhot^{(k)}\nabla\chi\cdot\nabla \rhot^{(k)}\\
&=& \int\chi(H_1-kH_2)\left[b^2|\nabla \rhot^{(k)}|^2+(p-1)\rho_D^{p-2}\rho_T\rhot^2_k\right]- \frac{b^2}{2}\int(\rhot^{(k)})^2\nabla\cdot\left[ \chi \nabla(H_1-kH_2)\right]\\
\nonumber&=&  -\int \mu\chi \left(k+\frac{2(r-1)}{p-1}+O\left(\frac{1}{\la Z\ra^{r}}\right)\right)\left(b^2|\nabla \rhot^{(k)}|^2+(p-1)\rho_D^{p-2}\rho_T\rhot^2_k\right)\\
&-& \frac{b^2}{2}\int(\rhot^{(k)})^2\nabla\cdot\left[ \chi \nabla(H_1-kH_2)\right]\\
\nonumber& \lesssim & e^{-c_{k_m}\tau}-\int \mu \chi\left(k+\frac{2(r-1)}{p-1})\right)\left(b^2|\nabla \rhot^{(k)}|^2+(p-1)\rho_D^{p-2}\rho_T\rhot^2_k\right)\\
&-& \frac{b^2}{2}\int\chi(\rhot^{(k)})^2\left[\frac{\Lambda\chi\Lambda (H_1-kH_2)}{\chi Z^2}+\Delta(H_1-kH_2) \right],
\eee
where we used the interpolation bound \eqref{inerpo;soationcrucial}.
Similarly, using that $\chi_{k,k,\sigma}=\la Z\ra \chi_{k,k-1,\sigma}$ and $|\rho_k|\le |\nabla \rho_{k-1}|$ as well as
\eqref{eq:esterrorpotentials}, \eqref{inerpo;soationcrucial} gives 
$$\frac{b^2}{2}\int\chi(\rhot^{(k)})^2\left[\frac{\Lambda\chi\Lambda (H_1-kH_2)}{Z^2}+\Delta(H_1-kH_2) \right]
\lesssim  \|\rhot,\Psit\|^2_{k,\sigma(k)+\frac 12(1+r)}
\lesssim e^{-c_{k_m}\tau}.$$Next using $$|\pa^k\rho_D|\lesssim \frac{\rho_D}{\la Z\ra^{k}},$$
we estimate after an integration by parts:
\bee
&&b^2\left|\int (\chi\Delta \rhot^{(k)}+\nabla \chi\cdot\nabla \rhot^{(k)})\left[(\pa^k\rho_D)\Delta \Psit+2\nabla(\pa^k\rho_D)\cdot\nabla \Psit\right]\right|\\
&\lesssim & b^2\int\chi |\nabla \rhot^{(k)}|\left[|\nabla(\pa^k\rho_D\Delta \Psit)|+|\nabla(\nabla\pa^k\rho_D\cdot\nabla \Psit)|+\frac{|(\pa^k\rho_D)\Delta \Psit+2\nabla(\pa^k\rho_D)\cdot\nabla \Psit|}{\la Z\ra}\right]\\
& \leq & b^2\int \chi\frac{|\nabla \rhot^{(k)}|^2}{\la Z\ra}+b^2\sum_{j=1}^3\int \chi\frac{\la Z\ra \rho_D^2}{\la Z\ra^{2k}}\left(\frac{|\pa^j\Psit|}{\la Z\ra^{3-j}}\right)^2
\eee
We use the pointwise bootstrap bound \eqref{smallglobalboot}
\be
\label{venovneoienneonen}
|\la Z\ra^j\pa^j\Psi|\le C_K\left[\frac{{\bf 1}_{Z\le Z^*}}{\la Z\ra^{r-2}}+b\right]\lesssim\left[\frac{{\bf 1}_{Z\le Z^*}}{\la Z\ra^{r-2}}+\frac{{\bf 1}_{Z\ge Z^*}}{\la Z^*\ra^{r-2}}\right], \ \ 1\le j\le 3
\ee
to estimate from \eqref{eq:Dsigma}:
\bee
&&b^2\sum_{j=1}^3\int \chi\frac{\la Z\ra \rho_D^2}{\la Z\ra^{2k}}\left(\frac{|\pa^j\Psit|}{\la Z\ra^{3-j}}\right)^2\\
&\leq& b^2C_K\int\frac{Z^{d-1}dZ}{\la Z\ra^{2(\tilde{\nu}+\frac d2-(r-1))}}\frac{1}{\la Z\ra^{5}}\left(\left[\frac{{\bf 1}_{Z\le Z^*}}{\la Z\ra^{2(r-2)}}+\left(\frac Z{Z^*}\right)^{4\nut}\frac{{\bf 1}_{Z\ge Z^*}}{\la Z\ra^{2(r-2)}}\right]\right)\\
& \lesssim & e^{-c_{k_m}\tau}
\eee
and hence 
$$
b^2\left|\int  (\chi\Delta \rhot^{(k)}+\nabla \chi\cdot\nabla \rhot^{(k)})\left[(\pa^k\rho_D)\Delta \Psit+2\nabla(\pa^k\rho_D)\cdot\nabla \Psit\right]\right|\lesssim e^{-c_{k_m}\tau}.
$$
 For the nonlinear term, we use the Pohozhaev identity \eqref{pohozaevbispouet} and the pointwise bound \eqref{venovneoienneonen} 
 $$
 |Z^j \pa^j \Psi|\lesssim \la Z\ra^{-(r-2)}, \qquad j=1,...,3
 $$to estimate by the interpolation bound \eqref{inerpo;soationcrucial}
\bee
&&b^2\left|\int (\chi\Delta \rhot^{(k)}+\nabla \chi\cdot\nabla \rhot^{(k)})\left[\rhot^{(k)}\Delta \Psit+2\nabla\rhot^{(k)}\cdot\nabla \Psit\right]\right|\lesssim b^2\left[\int \chi\frac{|\nabla \rhot^{(k)}|^2}{\la Z\ra^{r}}+\int \chi\frac{(\rhot^{(k)})^2}{\la Z\ra^{r+2}}\right]\\
&\lesssim& e^{-c_{k_m}\tau}.
\eee
Note that the last term in the case $k=0$ should be treated with the help of the bound $\rhot\lesssim \rho_D$ and 
the estimate \eqref{nivneinveoneinv}. For $k\ne 0$, we simply use $|\rho_k|\le |\nabla\rho_{k-1}|$.
We recall that by definition of the norm:
\bee
\|\rhot,\Psit\|_{k,\sigma}^2\gtrsim \sum_{m=0}^k\int \chi_{k,k,\sigma}\frac {\rho_T^2|\pa^m\nabla \Psi|^2}{\la Z\ra^{2(k-m)}}\gtrsim\sum_{m=1}^{k+1}\int\chi \frac{\rho_T^2|\pa^m\Psi|^2}{\la Z\ra^{2(k+1-m)}}.
\eee
Hence, by the interpolation bound,
\bee
&&\left|\int \chi\left[(\pa^k\rho_D)\Delta \Psit+2\nabla(\pa^k\rho_D)\cdot\nabla \Psit\right](p-1)\rho_D^{p-2}\rho_T\rhot^{(k)}\right|\\
& \lesssim & \int \chi\frac{\rho_D^{p-2}\rho_T(\rhot^{(k)})^2}{\la Z\ra}+\int\chi\rho_T^{p-2}\rho_T^2\left[\frac{|\pa^2\Psit|^2}{\la Z\ra^{2k-1}}+\frac{|\pa\Psit|^2}{\la Z\ra^{2(k+1)-1}}\right]\\
& \leq & \|\rhot,\Psit\|_{k,\sigma+\frac 12}^2\lesssim e^{-c_{k_m}\tau}.
\eee
For the nonlinear term, we integrate by parts and use \eqref{venovneoienneonen}:
$$\left|\int \chi\left[\rhot^{(k)}\Delta \Psit+2\nabla\rhot^{(k)}\cdot\nabla \Psit\right](p-1)\rho_D^{p-2}\rho_T\rhot^{(k)}\right|\lesssim \int \chi\frac{\rho_D^{p-2}\rho_T(\rhot^{(k)})^2}{\la Z\ra}\lesssim e^{-c_{k_m}\tau}.
$$
From Pohozhaev \eqref{pohozaevbispouet} and \eqref{eq:esterrorpotentials}:
\bea
\label{vneineonoenonenv} &&-\int H_2\chi\Lambda \rhot^{(k)}(-b^2\Delta \rhot^{(k)})+b^2\int H_2\Lambda \rho_k\nabla\chi\cdot\nabla \rhot^{(k)}\\
\nonumber& = & b^2\left[\int \Delta \rhot^{(k)}(Z\chi H_2)\cdot\nabla \rhot^{(k)}+\int H_2\Lambda \rho_k\nabla\chi\cdot\nabla \rhot^{(k)}\right]\\
\nonumber& = & b^2\left[-\sum_{i,j=1}^d\int\pa_i(Z_j\chi H_2)\pa_i\rhot^{(k)}\pa_j\rhot^{(k)}+\frac12\int|\nabla \rhot^{(k)}|^2\nabla \cdot(Z\chi H_2)+\sum_{i,j=1}^d\int H_2Z_j\pa_j\rhot^{(k)}\pa_i\chi\pa_i \rhot^{(k)}\right]\\
\nonumber& = & b^2\Big\{-\sum_{i,j=1}^d\int \pa_i\rhot^{(k)}\pa_j\rhot^{(k)}\left[\delta_{ij}\chi H_2+Z_j\pa_i\chi H_2+Z_j\chi\pa_iH_2-H_2Z_j\pa_i\chi\right]\\
\nonumber&+&\frac 12\int|\nabla \rhot^{(k)}|^2\chi H_2\left[d+\frac{\Lambda\chi}{\chi}+\frac{\Lambda H_2}{H_2}\right]\Big\}=   \frac{\mu}{2}b^2\int \chi|\nabla \rhot^{(k)}|^2\left[d-2+\frac{\Lambda \chi}{\chi}+O\left(\frac{1}{\la Z\ra^{r-1}}\right)\right].
\eea
Integrating by parts and using \eqref{eq:esterrorpotentials}, \eqref{globalbeahvoiru}:
\bee
&&-\int \chi H_2\Lambda \rhot^{(k)}\left[(p-1)\rho_D^{p-2}\rho_T\rhot^{(k)}\right]+\frac{p-1}{2}\int \chi(p-2)\pa_\tau\rho_D\rho_D^{p-3}\rho_T(\rhot^{(k)})^2\\
&+& \frac{p-1}{2}\int \chi\pa_\tau\rho_T\rho_D^{p-2}(\rhot^{(k)})^2\\
&=& \frac {p-1}2\int (\rhot^{(k)})^2\left[\nabla \cdot(Z\chi H_2\rho_D^{p-2}\rho_T)+\chi\rho_T\pa_\tau(\rho_D^{p-2})+\chi\pa_\tau\rho_T\rho_D^{p-2}\right]\\
& = & \frac{p-1}{2}\int \chi\rho_D^{p-2}\rho_T(\rhot^{(k)})^2\left[\mu d+\mu \frac{\Lambda \chi}{\chi}+(p-2)\left(\frac{\pa_\tau\rho_D+\mu\Lambda \rho_D}{\rho_D}\right)+\frac{\pa_\tau\rho_T+\mu\Lambda \rho_T}{\rho_T}+O\left(\frac{1}{\la Z\ra^{r-1}}\right)\right].
\eee
We now claim the fundamental behavior
\be
\label{globalbeahvoiru}
\frac{\pa_\tau\rho_D+\mu\Lambda\rho_D}{\rho_D}=-\frac{2\mu(r-1)}{p-1}+O\left(\frac{1}{\la Z\ra^r}\right)
\ee
and
\be
\label{globalbeahvoirubis}
\frac{\pa_\tau\rho_T+\mu\Lambda \rho_T}{\rho_T}=-\frac{2\mu(r-1)}{p-1}+O\left(\frac{1}{\la Z\ra^r}\right).
\ee
Assume \eqref{globalbeahvoiru}, \eqref{globalbeahvoirubis}, we obtain
\bee
&&-\int \chi H_2\Lambda \rhot^{(k)}\left[(p-1)\rho_D^{p-2}\rho_T\rhot^{(k)}\right]+\frac{p-1}{2}\int \chi(p-2)\pa_\tau\rho_D\rho_D^{p-3}\rho_T(\rhot^{(k)})^2\\
&+& \frac{p-1}{2}\int \pa_\tau\rho_T\rhot^{p-2}(\rhot^{(k)})^2\\
& = &  \mu\frac{p-1}{2}\int \chi\rho_D^{p-2}\rho_T(\rhot^{(k)})^2\left[d+\frac{\Lambda \chi}{\chi}-2(r-1)+O\left(\frac{1}{\la Z\ra^r}\right)\right]\\
& = & \mu\frac{p-1}{2}\int \chi\rho_D^{p-2}\rho_T(\rhot^{(k)})^2\left[d+\frac{\Lambda \chi}{\chi}-2(r-1)\right]+O\left(e^{-c_{k_m}\tau}\right).
\eee
\noindent{\em Proof of \eqref{globalbeahvoiru}}. From \eqref{fromularhod}: 
\bee
\pa_\tau\rho_D+\mu\Lambda \rho_D=-\mu\Lambda \zeta\left(\frac Z{Z^*}\right)\rho_P(Z)+\mu\Lambda  \zeta\left(\frac Z{Z^*}\right)\rho_P(Z)+\mu \zeta\left(\frac Z{Z^*}\right)\Lambda \rho_P=\mu  \zeta\left(\frac Z{Z^*}\right)\Lambda \rho_P
\eee
$$\frac{\pa_\tau\rho_D+\mu\Lambda\rho_D}{\rho_D}=\mu\frac{\Lambda \rho_P}{\rho_P}=-\frac{2\mu(r-1)}{p-1}+O\left(\frac{1}{\la Z\ra^r}\right)
$$
and \eqref{globalbeahvoiru} is proved.\\
\noindent{\em Proof of \eqref{globalbeahvoirubis}}. Recall \eqref{fullflowrenormalized}
$$\pa_\tau \rho_T=-\rho_T\Delta \Psi_T-\frac{\mu\ell(r-1)}{2}\rho_T-\left(2\pa_Z\Psi_T+\mu Z\right)\pa_Z\rho_T$$ which yields:
\bee
\left|\frac{\pa_\tau \rho_T+\mu\Lambda \rho_T}{\rho_T}+\frac{\mu\ell(r-1)}{2}\right|=\left|-\Delta \Psi_T-2\frac{\pa_Z\Psi_T\pa_Z\rho_T}{\rho_T}\right|
\eee
and \eqref{globalbeahvoirubis} follows from \eqref{venovneoienneonen}.\\

\noindent\underline{\em $\Psi^{(k)}$ terms}. Integrating by parts:
\bee
&&\left|b^2\int(\pa^k\Delta \rho_{D)}\nabla\cdot(\chi\rho_T^2\nabla \Psi^{(k)})\right|\lesssim  b^2\int\chi \rho^2_T\frac{|\nabla\Psi^{(k)}|}{\la Z\ra^{k+3}}\\
&\lesssim &\int\chi \frac{\rho_T^2|\nabla \Psi^{(k)}|^2}{\la Z\ra}+b^4\int \chi\frac{\rho_T^2}{\la Z\ra^{2(k+3)-1}}\lesssim e^{-c_{k_m}\tau}
\eee
where we used \eqref{nivneinveoneinv}.

Next:
$$\mu(r-2)\int\Psi^{(k)}\nabla\cdot(\chi\rho_T^2\nabla \Psi^{(k)})=-\mu(r-2)\int \chi\rho_T^2|\nabla \Psi^{(k)}|^2.
$$
Similarly, using $\pa_ZH_2=O\left(\frac{1}{\la Z\ra^r}\right)$:
\bee
&&k\int H_2\Psi^{(k)}\nabla \cdot(\chi\rho_T^2\nabla \Psi^{(k)})=   -k\left[\int \chi \mu\left[1+O\left(\frac{1}{\la Z\ra^{\frac 12}}\right)\right]\rho_T^2|\nabla \Psi^{(k)}|^2+O\left(\int \chi\rho_T^2\frac{|\Psi^{(k)}|^2}{\la Z\ra^{2r-\frac 12}}\right)\right]\\
& = & -k\mu\int \chi \rho_T^2|\nabla \Psi^{(k)}|^2+O\left(e^{-c_{k_m}\tau}\right),
\eee
where we also used that $r>2$, $k\ne 0$ and 
$$
\int \chi_{k,k,\sigma(k)}\rho_T^2\frac{|\Psi^{(k)}|^2}{\la Z\ra^{2r-\frac 12}}\lesssim \int \chi_{k,k-1,\sigma(k)}\rho_T^2\frac{|\nabla \Psit_{k-1}|^2}{\la Z\ra^{2r-\frac 12-2}}\le \|\rhot,\Psit\|^2_{k,\sigma(k)+\frac 12}
$$
Then using \eqref{venovneoienneonen}:
$$
\left|\int 2\chi\rho_T\nabla \Psit\cdot\nabla \Psi^{(k)}\left(2\nabla \rho_T+\frac{\nabla\chi}\chi\right)\cdot\nabla \Psi^{(k)}\right|\lesssim \int \chi\frac{\rho_T^2|\nabla \Psi^{(k)}|^2}{\la Z\ra}\lesssim e^{-c_{k_m}\tau}$$
 and from \eqref{pohozaevbispouet}, \eqref{venovneoienneonen}:
\bee
\left|\int 2\chi\rho_T\nabla \Psit\cdot\nabla \Psi^{(k)}(\rho_T\Delta \Psi^{(k)})\right|&\lesssim& \int \chi|\nabla \Psi^{(k)}|^2|\left(|\pa(\rho_T^2\nabla \Psit)|+\frac{|\rho_T^2\nabla \Psit|}{\la Z\ra}\right)\lesssim \int \chi\frac{\rho_T^2|\nabla \Psi^{(k)}|^2}{\la Z\ra^2}\\
&\lesssim &e^{-c_{k_m}\tau}
\eee
We now carefully compute from \eqref{pohozaevbispouet} again:
\bea
\label{shaprpohoazev}
\nonumber &&\int \chi\rho_T H_2\Lambda \Psi^{(k)}\left(2\nabla \rho_T\cdot\nabla \Psi^{(k)}+\rho_T\Delta \Psi^{(k)}\right)+\int H_2\Lambda \Psi^{(k)}\rho_T^2\nabla \chi\cdot\nabla \Psi^{(k)}\\
\nonumber & = & 2\sum_{i,j}\int \chi\rho_T H_2Z_j\pa_j\Psi^{(k)}\pa_i\rho_T\pa_i\Psi^{(k)}-\sum_{i,j}\int \pa_i(\chi Z_j H_2\rho_T^2)\pa_i\Psi^{(k)}\pa_j\Psi^{(k)}\\
\nonumber &+& \frac 12\int\nabla \cdot(\chi Z H_2\rho_T^2)|\nabla \Psi^{(k)}|^2+  \sum_{i,j} H_2\rho_T^2Z_j\pa_j\Psi^{(k)}\pa_i\chi\pa_i\Psi^{(k)}\\
\nonumber &= & \sum_{i,j}H_2\pa_j\Psi^{(k)}\pa_i\Psi^{(k)}\left[2\chi\rho_T\pa_i\rho_TZ_j-\pa_i\chi Z_j\rho_T^2-\delta_{ij}\rho_T^2-2Z_j\rho_T\pa_i\rho_T+Z_j\rho_T^2\pa_i\chi\right]\\
\nonumber & + & \frac 12\int\chi H_2\rho_T^2|\nabla \Psi^{(k)}|^2\left[d+\frac{\Lambda \chi}{\chi}+\frac{\Lambda H_2}{H_2}+2\frac{\Lambda \rho_T}{\rho_T}\right]\\
& = & \frac 12\mu\int \chi\rho_T^2|\nabla \Psi^{(k)}|^2\left[d-2+\frac{\Lambda \chi}{\chi}+2\frac{\Lambda \rho_T}{\rho_T}+O\left(\frac{1}{\la Z\ra^r}\right)\right].
\eea
Hence the final formula recalling \eqref{globalbeahvoirubis}:
\bee
&&\int \chi\rho_T H_2\Lambda \Psi^{(k)}\left(2\nabla \rho_T\cdot\nabla \Psi^{(k)}+\rho_T\Delta \Psi^{(k)}\right)+\int H_2\Lamdba \Psi^{(k)}\rho_T^2\nabla \chi\cdot\nabla \Psi^{(k)}+\int\chi\pa_\tau\rho_T\rho_T|\nabla \Psi^{(k)}|^2\\
& = &  \int \mu \chi\rho_T^2|\nabla \Psi^{(k)}|^2\left[\frac{d-2}{2}+\frac 12\frac{\Lambda\chi}{\chi}+\frac{\Lambda \rho_T}{\rho_T}+\frac{1}{\mu}\frac{\pa_\tau\rho_T}{\rho_T}+O\left(\frac{1}{\la Z\ra^r}\right)\right]\\
& = &  \int \mu \chi\rho_T^2|\nabla \Psi^{(k)}|^2\left[\frac{d-2}{2}+\frac 12\frac{\Lambda\chi}{\chi}-\frac{2(r-1)}{p-1}+O\left(\frac{1}{\la Z\ra^r}\right)\right]\\
& = & \int \mu \chi\rho_T^2|\nabla \Psi^{(k)}|^2\left[\frac{d-2}{2}+\frac 12\frac{\Lambda\chi}{\chi}-\frac{2(r-1)}{p-1}\right]+O(e^{-c_{k_m}\tau}).
\eee
\noindent\underline{Loss of derivatives terms}. We integrate by parts the non linear term which must loose derivatives:
\bea
\label{;lossfoofneinoin}
\nonumber &&b^2\left|\int \rho_T\nabla \chi\cdot\nabla \Psi^{(k)}\Delta \rhot^{(k)}\right|\lesssim b^2\left|\int \rho_T\Delta\Psi^{(k)}\nabla \chi\cdot\nabla \rhot^{(k)}\right|+b^2\int \chi \frac{\rho_T}{\la Z\ra^2}|\nabla \Psi^{(k)}||\nabla \rhot^{(k)}|\\&&
\qquad\qquad   \lesssim  b^3\int\chi|\nabla \rhot^{(k)}|^2+b\left[\int \chi\frac{\rho_T^2|\Delta \Psi^{(k)}|^2}{\la Z\ra^2}+\int \chi\frac{\rho_T^2|\nabla \Psi^{(k)}|^2}{\la Z\ra^4}\right]\nonumber\\ &&\qquad\qquad\lesssim e^{-c_{k_m}\tau} + b\int \chi\frac{\rho_T^2|\Delta \Psi^{(k)}|^2}{\la Z\ra^2}.
\eea
We now use \eqref{eq:chik} for $0\le k\le k_m-1$ which
implies
$$
\int \chi_{k,k,\sigma(k)}\frac{\rho_T^2|\Delta \Psi^{(k)}|^2}{\la Z\ra^2}\le\int \chi_{k+1,k+1,\sigma(k+1)}\rho_T^2|\Delta \Psi^{(k)}|^2
\lesssim \|\rhot,\Psi\|^2_{k+1,\sigma(k+1)}\lesssim 1.
$$
Hence
$$b^2\left|\int \rho_T\nabla \chi\cdot\nabla \Psi^{(k)} \Delta \rhot^{(k)}\right|+b^2\left|\int \rho_T\Delta\Psi^{(k)}\nabla \chi\cdot\nabla \rhot^{(k)}\right|+b^2\left|\int\nabla \chi\cdot\nabla \rhot^{(k)}\nabla \rho_T\cdot\nabla \Psi^{(k)}\right|\lesssim e^{-c_{k_m}\tau}.$$

\noindent\underline{Conclusion for linear terms}. The collection of above bounds yields:
\bea
\label{firstestimate}
\nonumber &&\frac 12\frac{d}{d\tau}\left\{\int b^2\chi|\nabla\rhot^{(k)}|^2+(p-1)\int \chi\rho_D^{p-2}\rho_T(\rhot^{(k)})^2+\int\chi\rho_T^2|\nabla \Psi^{(k)}|^2\right\}\leq e^{-c_{k_m\tau}}\\
\nonumber & + & \mu\int\chi\left[-k+\frac d2-(r-1)-\frac{2(r-1)}{p-1}+\frac12\frac{\mu^{-1}\pa_\tau\chi+\Lambda \chi}{\chi}\right]\\
&\times& \left[ b^2|\nabla\rhot^{(k)}|^2+(p-1)\rho_D^{p-2}\rho_T(\rhot^{(k)})^2+\rho_T^2|\nabla \Psi^{(k)}|^2\right]\\
\nonumber & + & \int F_1\chi(p-1)\rho_D^{p-2}\rho_T\rhot^{(k)}+b^2\int \chi\nabla F_1\cdot\nabla \rhot^{(k)}+\int \chi\rho^2_T\nabla F_2\cdot\nabla \Psi^{(k)}.
\eea

\noindent{\bf step 4} $F_1$ terms. We recall \eqref{formluafone} and claim the bound:
\be
\label{boundfone}
(p-1)\int \chi F_1^2\rho_D^{p-2}\rho_T+b^2\int\chi|\nabla F_1|^2\lesssim e^{-c_{k_m}\tau}\left[1+\|\rhot,\Psit\|^2_{k,\sigma}\right]
\ee

\noindent\underline{\em Source term induced by localization}.
 Recall \eqref{profileequationtilde}
 \bee
\tilde{\mathcal E}_{P,\rho}&=&\pa_\tau \rho_D+\rho_D\left[\Delta \Psit_P+\frac{\ell(r-1)}{2}+\left(2\pa_Z\Psit_P+ Z\right)\frac{\pa_Z\rho_D}{\rho_D}\right]\\
& = & \pa_\tau \rho_D+\Lambda \rho_D+\frac{\ell(r-1)}{2}\rho_D+\rho_D\Delta\Psi_P+2\pa_Z\Psi_P\pa_Z\rho_D.
\eee
From the proof of  \eqref{globalbeahvoiru} 
$$
\rho_D=\zeta\left(\frac Z{Z^*}\right) \rho_P,\qquad \pa_\tau \rho_D+\Lambda \rho_D= \zeta\left(\frac Z{Z^*}\right) 
\Lambda\rho_P
$$
Therefore, using the profile equation for $\rho_P$, we obtain
$$
\tilde{\mathcal E}_{P,\rho}= 2\frac{\Psi'_P}{Z} \frac Z{Z^*}\zeta'\rho_P
$$
From \eqref{relationsprofileemden} and  \eqref{limitprofilesbsi} we then conclude that
\be
\label{neineinneonev}
|\pa^k\tilde{\mathcal E}_{P,\rho}|\lesssim \frac{\rho_D}{\la Z\ra^{k+r}}{\bf 1}_{Z\ge Z^*}
\ee
Hence, recalling \eqref{estiamtionnprofile} and \eqref{eq:Dsigma}:
\bee
&&\int \chi \rhot_{D}^{p-2}\rho_T|\pa^k\tilde{\mathcal E}_{P,\rho}|^2\lesssim \int_{Z\ge Z^*}\frac{Z^{d-1}dZ}{Z^{-2k+2\left(\frac d2-(r-1)+\nut-\frac{2(r-1)}{p-1}\right)}}\frac{1}{Z^{2k+2r}}\frac{(Z)^{-\frac{2(r-1)(p+1)}{p-1}}}{\left(\frac{Z}{Z^*}\right)^{(p-1)n_P-2(r-1)+2(r-2)}}\\
& \lesssim & \int_{Z\ge Z^*}\frac{dZ}{Z^{2r+2\nut+1}\left(\frac{Z}{Z^*}\right)^{(p-1)n_P-2}}\lesssim (Z^*)^{-2r-2\nut}\lesssim e^{-c\tau}
\eee
Similarly, from \eqref{eq:Dsigma}:
\bee
b^2\int\chi|\nabla\pa^k \tilde{\mathcal E}_{P,\rho}|^2\lesssim \int_{Z\ge Z^*}\frac{Z^{d-1}dZ}{Z^{2\left(\frac d2-1+\nut\right)+2+2r}}
 \lesssim   (Z^*)^{-2r-2\nut}\lesssim e^{-c\tau}.
\eee

\noindent\underline{\em $[\pa^k,H_1]$ term}. We use \eqref{eq:esterrorpotentials} to estimate:
\be
\label{neoneneonoev}
\left|\begin{array}{ll}
|[\pa^k,H_1]\rhot|\lesssim \sum_{j=0}^{k-1}|\pa^j\rhot\pa^{k-j}H_1|\lesssim \sum_{j=0}^{k-1}\frac{|\pa^j\rhot|}{\la Z\ra^{r+k-j}}\\
|\nabla[\pa^k,H_1]\rhot|\lesssim \sum_{j=0}^{k}\frac{|\pa^j\rhot|}{\la Z\ra^{1+r+k-j}}
\end{array}\right.
\ee
Hence:
\bee
(p-1)\int \chi\rho_D^{p-2}\rho_T([\pa^k,H_1]\rhot)^2&\lesssim &\sum_{j=0}^{k-1}\int \chi\rho_D^{p-1} \frac{|\pa^j\rhot|^2}{\la Z\ra^{2(r+k-j)}}\\
&\lesssim &\|\rhot,\Psit\|^2_{k,\sigma(k)+r}\lesssim  e^{-c_{k_m}\tau},
\eee
and 
\bee
&&b^2\int\chi|\nabla ([\pa^k,H_1]\rhot)|^2\lesssim b^2 \sum_{j=0}^{k}\int \chi\frac{|\pa^j\rhot|^2}{\la Z\ra^{2(1+r+k-j)}}\\
&\lesssim& b^2\int \chi \frac{\rhot^2}{\la Z\ra^{2(1+r+k)}}+b^2 \sum_{j=0}^{k}\int \chi \frac{|\pa^j\nabla \rhot|^2}{\la Z\ra^{2(r+k-j)}}\\
&\lesssim&  b^2\int \chi\frac{\rho_D^2}{\la Z\ra^{2(1+r+k)}}+ e^{-c_{k_m}\tau}\lesssim e^{-c_{k_m}\tau}
\eee
where we used the bootstrap bound \eqref{smallglobalboot}, the decay of $b^2$ and \eqref{eq:Dsigma}.

\noindent\underline{\em $[\pa^k,H_2]$ term}. Similarly, from  \eqref{eq:esterrorpotentials}:
\be 
\label{cneoneoneone}
\left|\begin{array}{ll}
|[\pa^k,H_2]\Lambda\rhot|\lesssim \sum_{j=0}^{k-1}| \pa^j(\Lambda \rhot)\pa^{k-j}H_2|\lesssim \sum_{j=1}^{k}\frac{|\pa^j\rhot|}{\la Z\ra^{r-1+k-j}}\\
|\nabla[\pa^k,H_2]\Lambda\rhot|\lesssim \sum_{j=1}^{k+1}\frac{|\pa^j\rhot|}{\la Z\ra^{r+k-j}}.
\end{array}\right.
\ee
Hence, using $r>1$:
$$(p-1)\int \chi\rho_D^{p-2}\rho_T([\pa^k,H_2]\Lambda \rhot)^2\lesssim \sum_{j=1}^{k}\int {\chi}\rho_D^{p-1} \frac{|\pa^j\rhot|^2}{\la Z\ra^{2(r-1+k-j)}} \lesssim   e^{-c_{k_m}\tau},
$$
and 
\bee
&&b^2\int\chi|\nabla ([\pa^k,H_2]\Lambda\rhot)|^2\lesssim b^2 \sum_{j=1}^{k+1}\int \chi\frac{|\pa^j\rhot|^2}{\la Z\ra^{2(r+k-j)}}\\
&=& b^2 \sum_{j=0}^{k}\int \chi \frac{|\pa^j\nabla \rhot|^2}{\la Z\ra^{2(r-1+k-j)}}\lesssim   \|\rhot,\Psit\|^2_{k,\sigma+r-1}\lesssim e^{-c_{k_m}\tau}\
\eee
\\

\noindent\underline{\em Nonlinear term}. Changing indices, we need to estimate terms
\be
\label{nenveonoenon}
N_{j_1,j_2}=\pa^{j_1}\rho_T\pa^{j_2}\nabla \Psit, \ \ j_1+j_2=k+1, \ \ 2\leq j_1,j_2\le k-1.
\ee
For the profile term:
$$|\pa^{j_1}\rho_D\pa^{j_2}\nabla \Psit|\lesssim \rho_D\frac{|\pa^{j_2}\nabla \Psit|}{\la Z\ra^{j_1}}= \rho_D\frac{|\pa^{j_2}\nabla \Psit|}{\la Z\ra^{k+1-j_2}}$$ and hence using from \eqref{estiamtionnprofile} the rough global bound:
\be
\label{boudnthod}
\rho_D\lesssim \frac{1}{\la Z\ra^{\frac{2(r-1)}{p-1}}}
\ee
yields:
\bee
&&\int(p-1) \chi N_{j_1,j_2}^2\rho_D^{p-2}\rho_T\lesssim \int \chi\frac{\rho_T^2|\pa^{j_2}\nabla \Psit|^2}{\la Z\ra^{2(k+1-j_2)+2(r-1)}}\lesssim  \int \chi\frac{\rho_T^2|\pa^{j_2}\nabla \Psit|^2}{\la Z\ra^{2(k-j_2)+2r}}\lesssim   e^{-c_{k_m}\tau}.
\eee
Similarly, after taking a derivative:
\bee
b^2\int\chi|\nabla N_{j_1,j_2}|^2\lesssim b^2\int \chi\frac{\rho_T^2|\pa^{j_2}\nabla \Psit|^2}{\la Z\ra^{2(k+2-j_2)}}\lesssim e^{-c_{k_m}\tau}.
\eee
We now turn to the control of the nonlinear term.  If $j_1\leq \frac{4k_m}{9}$, then from \eqref{smallglobalboot}:
\be
\int \chi\rho_D^{p-1}|\pa^{j_1}\rhot\pa^{j_2}\nabla \Psi|^2\lesssim \int \chi\rho_D^2\frac{|\pa^{j_2}\nabla \Psi|^2}{\la Z\ra^{2(k+1-j_2)+2(r-1)}}\lesssim e^{-c_{k_m}\tau}.
\ee
If $j_2\le \frac{4k_m}{9}$, then from \eqref{smallglobalboot} and $b=\frac{1}{(Z^*)^{r-2}}$:
\bee
&&\int \chi\rho_D^{p-1}|\pa^{j_1}\rhot\pa^{j_2}\nabla \Psi|^2\\
& \lesssim&  \int_{Z\leq Z^*} \chi\rho_D^{p-1}\frac{|\pa^{j_1}\rhot|^2}{\la Z\ra^{2(k+1+(r-2)-j_1)}}+b^2\int_{Z\ge Z^*}\chi\rho_D^{p-1}\frac{|\pa^{j_1}\rhot|^2}{\la Z\ra^{2(k+1-j_1)}}\lesssim  e^{-c_{k_m}\tau}.
\eee
We may therefore assume $j_1,j_2\ge m_0=\frac{4k_m}{9}+1$, which implies $k\ge m_0$ and $j_1,j_2\leq \frac{2k_m}{3}$. 
From \eqref{defsimgma}:
\bea
\label{lowerbounsogmak}
\nonumber \sigma(k)&=&-\alpha(k_m-k)\geq -\alpha\left(k_m-\frac{4k_m}{9}\right)+O_{k_m\to +\infty}(1)=-\frac{4}{5}\left(1-\frac{4}{9}\right)k_m+O_{k_m\to +\infty}(1)\\
&\geq& -\frac{4k_m}{9}+O_{k_m\to +\infty}(1)
\eea 
From \eqref{defnm}: 
\be
\label{neeneonenoe}
\sigma(k)+n(j_1)+n(j_2)\geq -\frac{4k_m}{9}+\frac{k_m}{4}+\frac{k_m}{4}+O_{k_m\to +\infty}(1)\ge \frac{k_m}{18}
\ee and hence from \eqref{smallglobalboot} and interpolating on $Z\le Z^*_c$:
\bee
&&\int \chi\rho_D^{p-1}|\pa^{j_1}\rhot\pa^{j_2}\nabla \Psi|^2\lesssim e^{-c_{k_m}\tau}\\ &+&\int_{Z\ge Z^*_c}\frac{Z^{d-1}dZ}{\la Z\ra^{\frac{k_m}{10}}}\left[{\bf 1}_{Z\le Z^*}+\left(\frac{Z}{Z^*}\right)^{-(p-3)n_P-\frac{4(r-1)}{(p-1)} -2(r-2)-2(r-1)+4\nut} {\bf 1}_{Z\ge Z^*}\right]\lesssim e^{-c_{k_m}\tau}
\eee
The $b^2$ derivative term and the other nonlinear term in \eqref{formluafone} are estimated similarly. 
We note that the relation 
$$
k_m\gg n_P\gg 1
$$
ensures that the terms containing $k_m$ are dominant and eliminates the need to track the dependence on $n_P$.\\

\noindent{\bf step 5} $F_2$ terms. We claim:
\be
\label{estfoneessentialftwo}
 \int \chi\rho_T^2|\nabla F_2|^2 \lesssim e^{-c_{k_m}\tau}\left[1+\|\rhot,\Psit\|^2_{k+1,\sigma(k+1)}\right].
\ee
\noindent\underline{\em Source term induced by localization}. Recall \eqref{profileequationtilde}:
$$\tilde{\mathcal E}_{P,\Psi}
=|\nabla \Psit_P|^2+\rho_D^{p-1}+e\Psit_P+\frac{1-\mathcal e}{2}\Lambda \Psit_P-1=\rho_D^{p-1}-\rho_P^{p-1}$$ which yields the rough bound
$$|\nabla \pa^k\tilde{\mathcal E}_{P,\Psi}|\lesssim \frac{1}{\la Z\ra^{k+1+2(r-1)}}{\bf 1}_{Z\ge Z^*}$$
 and hence, from \eqref{eq:Dsigma},
\bee
\int \chi\rho_T^2|\nabla \pa^k\tilde{\mathcal E}_{P,\Psi}|^2\lesssim \int_{Z\ge Z^*}\frac{Z^{d-1}}{b^2 \la Z\ra^{-2k+2\left(\frac d2+\nut-1\right)}}\frac{\left(\frac{Z}{Z^*}\right)^{4\nut}}{\la Z\ra^{2k+2+4(r-1)}}\lesssim {Z^*}^{2(r-2)-4(r-1)+4\nut}\lesssim e^{-c\tau}
\eee

\noindent\underline{\em $ [\pa^k,H_2]\Lambda \Psi$ term}. From \eqref{eq:esterrorpotentials}:
\bee
|\nabla([\pa^k,H_2]\Lambda\Psi)|\lesssim \sum_{j=1}^{k+1}\frac{|\pa^j\Psit|}{\la Z\ra^{r+k-j}}\lesssim  \sum_{j=0}^{k}\frac{|\nabla \pa^j\Psit|}{\la Z\ra^{r+k-j-1}}
\eee
and hence:
$$
\int\chi\rho_T^2|\nabla([\pa^k,H_2]\Lambda\Psi)|^2\lesssim \sum_{j=0}^{k}\int\chi\rho_T^2\frac{|\nabla \pa^j\Psit|^2}{\la Z\ra^{2(r-1+k-j)}}\lesssim e^{-c_{k_m}\tau}.
$$

\noindent\underline{$[\pa^k,\rho_D^{p-2}]\rhot-k(p-2)\rho_D^{p-3}\pa\rho_D\pa^{k-1}\rhot$ term}. By Leibnitz:
$$\left|\left[[\pa^k,\rho_D^{p-2}]\rhot-k(p-2)\rho_D^{p-3}\pa\rho_D\pa^{k-1}\rhot\right]\right|\lesssim \sum_{j=0}^{k-2}\frac{|\pa^j\rhot|}{\la Z\ra^{k-j}}\rho_D^{p-2}$$
and, hence, taking a derivative:
\bee
&&\int \chi\rho_T^2\left|\nabla \left[[\pa^k,\rho_D^{p-2}]\rhot-k(p-2)\rho_D^{p-3}\pa\rho_D\pa^{k-1}\rhot\right]\right|^2\lesssim\sum_{j=0}^{k-1}\int\chi \rho_D^{2(p-2)+2}\frac{|\pa^j\rhot|^2}{\la Z\ra^{2(k-j)+2}}\\
&\lesssim&\sum_{j=0}^{k-1}\int\chi \rho_D^{p-1}\frac{|\pa^j\rhot|^2}{\la Z\ra^{2(k-j)+2}} \lesssim e^{-c_{k_m}\tau}.
\eee

\noindent\underline{Nonlinear $\Psi$ term}. Let $$\pa N_{j_1,j_2}=\pa^{j_1}\nabla\Psi\pa^{j_2}\nabla\Psi, \ \ j_1+j_2=k+1, \ \ j_1,j_2\geq 1.$$ If $j_1\leq \frac{4k_m}{9},$ then from \eqref{smallglobalboot}:
$$
\int \chi\rho_T^2|\nabla N_{j_1,j_2}|^2\lesssim \int \rho_T^2\chi\frac{|\pa^{j_2}\nabla\Psi|^2}{\la Z\ra^{2(k+1-j_2)}}\lesssim \|\rhot,\Psit\|_{k,\sigma+\frac 12}^2\lesssim e^{-c_{k_m}\tau}.
$$
The expression being symmetric in $j_1,j_2$, we may assume $j_1,j_2\geq m_0=\frac{4k_m}{9}+1$, $j_1,j_2\leq \frac{2k_m}{3}$
and $k\ge m_0=\frac{4k_m}{9}+1$. Using \eqref{smallglobalboot}, \eqref{neeneonenoe} and arguing as above ($k_m\gg n_P$):
$$
\int  \chi\rho_T^2|\nabla N_{j_1,j_2}|^2\lesssim e^{-c\tau}+\int_{Z\ge Z^*_c}\frac{dZ}{\la Z\ra^{\frac{k_m}{10}}}
\left[{\bf 1}_{Z\le Z^*}+\left(\frac Z{Z^*}\right)^{2n_P+4\nut}{\bf 1}_{Z\ge Z^*}\right]\lesssim e^{-c_{k_m}\tau}.$$

\noindent\underline{Quantum pressure term}. We estimate from Leibniz:
\bee
b^2\left|\pa^k\left(\frac{\Delta \rho_T}{\rho_T}\right)-\frac{\pa^k\Delta \rho_T}{\rho_T}+\frac{k\pa^{k-1}\Delta \rho_T\pa\rho_T}{\rho_T^2}\right|\lesssim b^2\sum_{j_1+j_2=k,j_2\geq 2}\left|\pa^{j_1}\Delta \rho_T\pa^{j_2}\left(\frac{1}{\rho_T}\right)\right|.
\eee
and using the Faa-di Bruno formula:
$$\left|\pa^{j_2}\left(\frac{1}{\rho_T}\right)\right|\lesssim \frac{1}{\rho_T^{j_2+1}}\sum_{q_1+2q_1+\dots+j_2q_{j_2}=j_2}\Pi_{i=0}^{j_2}|(\pa^i\rho_T)^{q_i}|$$
 where $q_0$ is defined by $$q_0+q_1+\dots+q_{j_2}=j_2.$$
 We decompose $\rho_T=\rho_D+\rhot$ and control the $\rho_D$ term using the bound $$|\pa^i\rho_D|\lesssim \frac{\rho_D}{\la Z\ra^i}$$ which yields
 \be
 \label{onenvenoievnoe}
  \frac{1}{\rho_T^{j_2+1}}\sum_{m_1+2m_1+\dots+j_2m_{j_2}=j_2}\Pi_{i=0}^{j_2}|(\pa^i\rho_D)^{m_i}|
\lesssim \frac{1}{\rho_T\la Z\ra^{j_2}}
\ee
The corresponding contribution to \eqref{estfoneessentialftwo}:
\bee
&&b^4\int\chi\rho_T^2\left\{\sum_{j_1+j_2=k,j_2\geq 2}\frac{|\pa^{j_1+1}\Delta \rho_T|^2}{\rho_T^2\la Z\ra^{2j_2}}+\frac{|\pa^{j_1}\Delta \rho_T|^2}{\rho_T^2\la Z\ra^{2j_2+2}}\right\}\\
& \lesssim & b^4\sum_{j_1+j_2=k,j_2\geq 2}\left[\int \chi\frac{\rho_D^2dZ}{\la Z\ra^{2j_2+2(j_1+3)}}+\int \chi\frac{|\pa^{j_1+3}\rhot|^2}{\la Z\ra^{2j_2}}\right]\\
& \lesssim & b^4\int\chi\frac{\rho_D^2dZ}{\la Z\ra^{2k+6}}+ b^4\sum_{j_1=2}^k\int\chi\frac{|\nabla \pa^{j_1}\rhot|^2}{\la Z\ra^{2(k-j_1)+2}}\lesssim e^{-c_{k_m}\tau}
\eee
where we used \eqref{nivneinveoneinv} in the last step.

We now turn to the control of the nonlinear term and consider 
$$N_{j_1,j_2}=b^2\left(\pa^{j_1+1}\Delta \rho_T\right)\frac{1}{\rho_T^{j_2+1}}\sum_{q_1+2q_w+\dots+j_2q_{j_2}=j_2}\Pi_{i=0}^{j_2}(\pa^i\rhoh)^{m_i},$$
where $\rhoh$ is either $\rho_D$ or $\rhot$. In both cases we will use the weaker estimates \eqref{smallglobalboot}.\\
First assume that $q_i=0$ whenever $i\ge \frac{4k_m}{9}+1$, then from \eqref{smallglobalboot}:
$$|N_{j_1,j_2}|\lesssim b^2|\pa^{j_1+1}\Delta \rho_T|\frac{1}{\rho_T^{j_2+1}}\sum_{q_1+2q_1+\dots+j_2q_{j_2}=j_2}\Pi_{i=0}^{j_2}|(\pa^i\rhoh)^{q_i}|\lesssim b^2 \frac{|\pa^{j_1+1}\Delta \rho_T|}{\rho_T\la Z\ra^{j_2}}
$$
and the conclusion follows verbatim as above. Otherwise, there are at most two value $\frac{4k_m}{9}\le i_1\leq i_2\leq j_2$ with $q_{i_1},q_{i_2}\ne 0$ and $q_{i_1}+q_{i_2}\le 2$.
Hence from \eqref{smallglobalboot}:
\bee
\frac{1}{\rho_T^{j_2+1}}\Pi_{i=0}^{j_2}|(\pa^i\rhoh)^{q_i}|&\lesssim& \frac{1}{\rho_T^{j_2+1}}|\pa^{i_1}\rhoh|^{q_{i_1}}|\pa^{i_2}\rhoh|^{q_{i_2}}\Pi_{0\leq i\le j_2, i\notin \{i_1,i_2\}}\left(\frac{\rho_D}{\la Z\ra^{i}}\right)^{q_i}\\
&\lesssim&  \left(\frac{|\pa^{i_1}\rhoh|}{\rho_D}\right)^{q_{i_1}}\left(\frac{|\pa^{i_2}\rhoh|}{\rho_D}\right)^{q_{i_2}}\frac{1}{\rho_T\la Z\ra^{j_2-(q_{i_1}i_1+q_{i_2}i_2)}}.
\eee
Assume first $i_2\ge \frac{2k_m}3+1$, then $q_{i_1}=0$, $q_{i_2}=1$ and $j_1+3\le \frac{4k_m}{9}$ from which:
\bee
\int\chi\rho_T^2|N_{j_1,j_2}|^2&\lesssim& b^4\int\chi\rho_T^2 |\pa^{j_1+1}\Delta \rho_T|^2\frac{|\pa^{i_2}\rhoh|^2}{\rho^2_T}\frac{1}{\rho^2_T\la Z\ra^{2(j_2-i_2)}}\lesssim b^4\int \chi\frac{|\pa^{i_2}\rhoh|^2}{\la Z\ra^{2(j_2-i_2)+2(j_1+3)}}\\
& \lesssim & b^4\int \chi\frac{|\pa^{i_2}\rhoh|^2}{\la Z\ra^{2(k-i_2)+6}}\lesssim e^{-c_{k_m}\tau}.
\eee
There remains the case $\frac{4k_m}{9}+1\leq i_1\leq i_2\leq \frac{2k_m}3$ which imply $j_1+3\leq\frac{2k_m}{3}$, and we distinguish cases:\\
{\em -- case $(m_{i_1},m_{i_2})=(0,1)$}: if $j_1+3\leq \frac{4k_m}{9}$, we estimate
\bee
&&\int\chi\rho_T^2|N_{j_1,j_2}|^2\lesssim b^4\int\chi\rho_D^2 |\pa^{j_1+1}\Delta \rho_T|^2\frac{|\pa^{i_2}\rhoh|^2}{\rho^2_T}\frac{1}{\rho^2_T\la Z\ra^{2(j_2-i_2)}}\\
&\lesssim &b^4\int \chi\frac{|\pa^{i_2}\rhoh|^2}{\la Z\ra^{2(j_2-i_2)+2(j_1+3)}} \lesssim  b^4\int \chi\frac{|\pa^{i_2}\rhoh|^2}{\la Z\ra^{2(k-i_2)+6}}\lesssim e^{-c_{k_m}\tau}.
\eee
Otherwise,  $\frac{4k_m}{9}+1\leq j_1+3\leq \frac{2k_m}{3}$. since $j_2\ge \frac{4k_m}{9}+1$, then necessarily $j_2\le \frac{2k_m}{3}$. Hence $\frac{4k_m}{9}+1\leq j_1+3\leq \frac{2k_m}{3}$, $\frac{4k_m}{9}+1\leq j_2\leq \frac{2k_m}{3}$ and we estimate from \eqref{smallglobalboot}:
\bee
\int\chi\rho_T^2|N_{j_1,j_2}|^2\lesssim b^4\int \frac{Z^{d-1}dZ}{\la Z\ra^{2\left(\sigma(k)+\frac{k_m}{4}+\frac{k_m}{4}+j_2-i_2\right)}}\left[{\bf 1}_{Z\le Z^*}+\left(\frac Z{Z^*}\right)^{2n_P+4\nut}{\bf 1}_{Z\ge Z^*}\right]\lesssim b^4\lesssim e^{-c_{k_m}\tau},
\eee
where we once again used that in this range of $k$
$$
\sigma(k)+\frac{k_m}2\ge \frac{k_m}{10},\qquad k_m\gg n_P\gg 1
$$
{\em -- case $m_{i_1}+m_{i_2}=2$}: we use \eqref{smallglobalboot} and estimate crudely:
\bee
\int\chi\rho_T^2|N_{j_1,j_2}|^2&\lesssim& b^4\int\chi |\pa^{j_1+1}\Delta \rho_T|^2\left(\frac{1}{\la Z\ra^{\frac{k_m}4}}\right)^4\\ &\lesssim&b^4\int \frac{Z^{d-1}dZ}{\la Z\ra^{2\left(\sigma(k)+\frac{k_m}{2}\right)}}
\left[{\bf 1}_{Z\le Z^*}+\left(\frac Z{Z^*}\right)^{2n_P+4\nut}{\bf 1}_{Z\ge Z^*}\right]\lesssim b^4\lesssim e^{-c_{k_m}\tau}.
\eee

\noindent\underline{$\NL(\rhot)$ term}. We expand, using for the sake of simplicity that the power of the nonlinear term is an integer $p\ge 3$:
$$\NL(\rhot)=(\rho_D+\rhot)^{p-1}-\rho_D^{p-1}-(p-1)\rho_D^{p-2}\rhot=\sum_{q=2}^{p-1}c_{q}\rhot^{q}\rho_D^{p-1-q}$$ and hence by Leibniz:
\bee
\pa^{k}\NL(\rhot)&=&\sum_{q=2}^{p-1}\sum_{j_1+j_2=k}c_{q,j_1j,_2}\pa^{j_1}(\rhot^{q})\pa^{j_2}(\rho_D^{p-1-q})\\
& = & \sum_{q=2}^{p-1}\sum_{j_1+j_2=k}\sum_{\ell_1+\dots+\ell_q=j_1}\pa^{\ell_1}\rhot\dots\pa^{\ell_q}\rhot\pa^{j_2}(\rho_D^{p-1-q}).
\eee
Let $$N_{\ell_1,\dots.\ell_q,j_1,q}=\pa^{\ell_1}\rhot\dots\pa^{\ell_q}\rhot\pa^{j_2}(\rho_D^{p-1-q}),\ \ \ell_1\le\dots\le\ell_q,$$ then $$|\nabla N_{\ell_1,\dots.\ell_q,j_1,q}|\lesssim| N_{\ell_1,\dots.\ell_q,j_1,q}^{(1)}|+|N_{\ell_1,\dots.\ell_q,j_1,q}^{(2)}|$$ with $$| N^{(1)}_{\ell_1,\dots.\ell_q,j_1,q}|\lesssim |\pa^{m_1}\rhot\dots\pa^{m_q}\rhot|\frac{\rho_D^{p-1-q}}{\la Z\ra^{j_2}}, \ \ \left|\begin{array}{ll} 0\leq m_1\leq \dots\le m_q\leq  k+1\\ m_1+\dots +m_q=j_1+1.\end{array}\right.$$ We estimate $N^{(1)}_{\ell_1,\dots.\ell_q,j_1,q}$, the other term being estimated similarly. We distinguish cases.\\
 {\em -- case $m_q\leq \frac{4k_m}{9}$}, then from \eqref{smallglobalboot}:
\bee
| N^{(1)}_{m_1,\dots,m_q,j_1,q}|\lesssim \frac{\rhot^{q}}{\la Z\ra^{j_1+1}}\frac{\rho_D^{p-1-q}}{\la Z\ra^{j_2}}\lesssim \frac{\rho_D^{p-1}}{\la Z\ra^{k+1}}
\eee
and hence, from \eqref{estiamtionnprofile} and \eqref{uppperboundchi}, the contribution of this term 
\bee
\int&& \chi \rho_T^2|N^{(1)}_{m_1,\dots,m_q,j_1,q}|^2\lesssim e^{-c\tau}\\&+&\int_{Z\ge Z^*_c}\frac{Z^{d-1}dZ}{\la Z\ra^{2\sigma(k)+2(k+1)}}
\frac {{\bf 1}_{Z\le Z^*}+\left (\frac Z{Z^*}\right)^{ -2(p-1)n_P-{4(r-1)}-2(r-2)+4\nut}{\bf 1}_{Z\ge Z^*}}{\la Z\ra^{4(r-1)+\frac{4(r-1)}{p-1}}}\\
&\lesssim& e^{-c\tau}+\int_{Z\ge Z^*_c}\frac{dZ}{\la Z\ra^{2(r-1)+3}}\lesssim e^{-c_{k_m}\tau}.\eee
We now assume $m_q\geq \frac{4k_m}{9}+1$ and recall $m_q\le j_1+1\le k+1\le k_m$.\\
\noindent{\em -- case $m_{q-1}\le \frac{4k_m}9$}, then from \eqref{smallglobalboot}:
\bee
|N^{(1)}_{m_1,\dots,m_q,j_1,q}|\lesssim  \frac{\rhot^{q-1}_D}{\la Z\ra^{j_1-m_q+1}}|\pa^{m_q}\rhot|\frac{\rho_D^{p-1-q}}{\la Z\ra^{j_2}}\lesssim |\pa^{m_q}\rhot|\rho_D^{p-2}\frac{1}{\la Z\ra^{k+1-m_q}}.
\eee
If $m_q\le k$ then
\bee
&&\int\chi\rho_T^2|N^{(1)}_{m_1,\dots,m_q,j_1,q}|^2\lesssim \int\chi \rho_D^2\frac{|\pa^{m_q}\rhot|^2}{\la Z\ra^{2(k-m_q+1)+\frac{4(r-1)(p-2)}{p-1}}} \lesssim  e^{-c_{k_m}\tau}
\eee
On the other hand, if $m_q=k+1$, then, using \eqref{eq:chik}
\bee
&&\int \chi_{k,k,\sigma(k)} \rho_T^2|N^{(1)}_{m_1,\dots,m_q,j_1,q}|^2\lesssim \int\chi_{k,k,\sigma(k)}  \rho_D^2\rho_D^{2(p-3)}{\rhot^2\,|\pa^{k+1}\rhot|^2}\\ &&\lesssim
 \int_{Z<Z_c^*} \chi_{k+1,k+1,\sigma(k+1)} \rho_D^{2(p-2)}\la Z\ra^2{\rhot^2\,|\pa^{k+1}\rhot|^2}\\ &&+ \int_{Z>Z_c^*}\chi_{k+1,k+1,\sigma(k+1)} \rho_D^{p-1}\frac{\,|\pa^{k+1}\rhot|^2}{\la Z\ra^{-2+2(r-1)}}
 \lesssim  e^{-c_{k_m}\tau} \|\rhot,\Psit\|^2_{k+1,\sigma(k+1)}
\eee
where we used the interpolation bound 
$$\|\rhot\|_{L^\infty(Z\le Z^*_c)}\lesssim e^{-c\tau},$$
estimates
$$
\rhot\lesssim \rho_D\lesssim \la Z\ra^{-\frac{2(r-1)}{p-1}}
$$
and the condition
$$
-2+2(r-1)>0,
$$
which follows from $r>2$.\\
\noindent{\em -- case $m_{q-1}\ge \frac{4k_m}9+1$}, then necessarily $m_{q-2}\le \frac{4k_m}9<\frac{4k_m}9+1\le m_{q-1}\le m_{q}<\frac{2k_m}{3}$ and $k\ge \frac{4k_m}{9}+1$.
Hence:
$$
|N^{(1)}_{m_1,\dots,m_q,j_1,q}|\lesssim \frac{\rho^{q}_D}{\la Z\ra^{\frac{k_m}{4}+\frac{k_m}4}}\rho_D^{p-1-q}\lesssim \frac{\rho_D^{p-1}}{\la Z\ra^{\frac{k_m}{2}}}.$$ 
The integral for $Z<Z^*_c$ is estimated as above, and we further estimate from \eqref{uppperboundchi} and \eqref{neeneonenoe}, using that $k_m\gg n_P\gg 1$,
\bee
\int_{Z\ge Z^*_c}\chi \rho_D^2 |N^{(1)}_{m_1,\dots,m_q,j_1,q}|^2\lesssim \int_{Z\ge Z_c^*}\frac{dZ}{\la Z\ra^{\frac{k_m}{10}}}\lesssim e^{-c_{k_m}\tau}.
\eee

\noindent{\bf step 6} Conclusion. Injecting the collection of above bounds into \eqref{firstestimate} yields:
\bee
&&\frac 12\frac{d}{d\tau}\left\{\int b^2\chi|\nabla\rhot^{(k)}|^2+(p-1)\int \chi\rho_D^{p-2}\rho_T(\rhot^{(k)})^2+\int\chi\rho_T^2|\nabla \Psi^{(k)}|^2\right\}\\
& \leq & \mu\int\chi\left[-k+\frac d2-(r-1)-\frac{2(r-1)}{p-1}+\frac12\frac{\mu^{-1}\pa_\tau\chi+\Lambda \chi}{\chi}\right]\\
&\times& \left[ b^2|\nabla\rhot^{(k)}|^2+(p-1)\rho_D^{p-2}\rho_T(\rhot^{(k)})^2+\rho_T^2|\nabla \Psi^{(k)}|^2\right]\\
& + & e^{-c_{k_m\tau}}.
\eee
{We now compute, noting that $\pa_\tau Z^*=\mu Z^*$ and that $\xi_k$ only depends on $\tau$ through $Z^*$:
\bee
&&\pa_\tau\chi+\mu \Lambda \chi\\
&=& \frac 1{\la Z\ra^{2\sigma(k)}}\left[\pa_\tau \xi_k\left(\frac Z{Z^*}\right) +\mu \Lambda \xi_k\left(\frac Z{Z^*}\right)\right]+ \xi_k\left(\frac Z{Z^*}\right) \mu \Lambda \left( \frac 1{\la Z\ra^{2\sigma(k)}}\right) \\
\\ &=&\xi_k\left(\frac Z{Z^*}\right) \mu \Lambda \left( \frac 1{\la Z\ra^{2\sigma(k)}}\right).
\eee}
Hence recalling \eqref{defsigmnavrelations}:
\bee
&&k-\frac d2+r-1+\frac{2(r-1)}{p-1}-{\frac12\frac{\mu^{-1}\pa_\tau\chi+\Lambda \chi}{\chi}}=k+\sigma(k)-\frac d2+r-1+\frac{2(r-1)}{p-1}+O\left(\frac{1}{\la Z\ra}\right)\\
&\ge& \sigma_\nu -\frac d2+r-1+\frac{2(r-1)}{p-1}+O\left(\frac{1}{\la Z\ra}\right)\ge \nut+O\left(\frac{1}{\la Z\ra}\right).
\eee
Using that $k\le k_m-1$ and the interpolation bound  \eqref{interpolatedbound} we may absorb the $O\left(\frac{1}{\la Z\ra}\right)$
term 
and \eqref{estnienonneoinduction} is proved.


\section{Pointwise bounds}
\label{sec:point}


We are now in position to close the control of the pointwise bounds \eqref{smallglobalboot}. We start with inner bounds $|x|\lesssim 1$:

\begin{lemma}[Interior pointwise bounds]
For all $0\le k\le\frac{2k_m}{3}$:
\be
\label{smallnessoutsideinitbootfinal}
\left|\begin{array}{l}
\forall 0\le k\le\frac{2k_m}{3}, \ \ \|\frac{\la Z\ra ^{n(k)}\pa_Z^k\rho}{\rho_P}\|_{L^\infty(Z\leq Z^*)}\le {\mathcal d}_0\\
\forall 1\le k\le \frac{2k_m}{3}, \ \ \|\la Z\ra^{n(k)}\la Z\ra^{r-2}\pa_Z^k\Psi\|_{L^\infty(Z \le Z^*)}\leq {\mathcal d}_0
\end{array}\right.
\ee
where ${\mathcal d}_0$ is a smallness constant depending on data.
\end{lemma}

\begin{proof} We integrate \eqref{estnienonneoinduction} in time and obtain, by choosing  $0<\nut\ll c+c_{k_m}$, 
$\forall 0\le m\le k_m-1$:
\be
\label{esifnfewwonnw}
I_m(\tau)\leq e^{-2\mu\nut(\tau-\tau_0)}I_m(0)+e^{-c_{k_m}\tau}\le e^{-2\mu\nut(\tau-\tau_0)} e^{-c\tau_0}+e^{-c_{k_m}\tau}\le {\mathcal d}_0 e^{-2\mu\nut\tau}
\ee 
for some small constant ${\mathcal d}_0$, which can be chosen to be arbitrarily small by increasing $\tau_0$. 
Below, we will adjust ${\mathcal d}_0$ to remain small while absorbing any other universal constant.

Recalling \eqref{eq:km}:
\be
\label{cenvnvnevneoen}\forall 0\le m\le k_m-1, \ \ \|\rhot,\Psi\|_{m,\sigma(m)}\le {\mathcal d}_0 e^{-\mu\nut\tau}.
\ee
This, in particular, already implies bounds on the Sobolev and pointwise norms of $(\rhot,\Psi)$ on compact sets:
for any $Z_K<\infty$ and any $k\le k_m-d$
\be\label{eq:compact}
\|(\rhot,\Psi)\|_{H^k(Z\le Z_K)}\le {\mathcal d}_0  e^{-\mu\nut\tau},\qquad \|(\pa^k\rhot,\pa^k\Psi)\|_{L^\infty(Z\le Z_K)}\le
{\mathcal d}_0  e^{-\mu\nut\tau}
\ee

\noindent\underline{case $m\le \frac{4k_m}{9}+1=m_0$}. Recall \eqref{defsimgma}, then \eqref{esifnfewwonnw} implies: $\forall 0\le m\le m_0,$
\bea
\label{nnvivneoneniven}
\nonumber&& \left\|\la Z\ra ^{m-\frac d2+\frac{2(r-1)}{p-1}-\nut}\pa^m_Z\rho\right\|^2_{L^2(Z\le Z^*)}+\left\|\la Z\ra ^{m-\frac d2+(r-1)-\nut}\pa^{m+1}_Z\Psi\right\|^2_{L^2(Z\le Z^*)}\\
&\le& {\mathcal d}_0e^{-2\mu\nut\tau}.
\eea
We now write for any spherically symmetric function $u$ and $\gamma>\frac d2-1$:
\bea
\label{iobebeibbneobis}
\nonumber |u(Z)|&\lesssim & |u(1)|+\int_1^Z|\pa_Zu|d\sigma\lesssim |u(1)|+\left(\int_{1\leq \sigma\leq Z} \frac{|\pa_Zu|^2}{\tau^{2\gamma}}\tau^{d-1}d\tau\right)^{\frac 12}\left(\int_{1\leq \sigma\leq Z}\frac{\tau^{2\gamma}}{\tau^{d-1}}d\tau\right)^{\frac 12}\\
&\lesssim & |u(1)|+\la Z\ra^{\gamma+1-\frac d2}\left\|\frac{\pa_Zu}{\la Z\ra^\gamma}\right\|_{L^2(1\leq \sigma \leq Z)}
\eea
We pick $1\le m\le m_0$ and apply this to $u=Z^{\frac{2(r-1)}{p-1}}Z^{m-1}\pa_Z^{m-1}\rho$, $\gamma+1=\frac d2+\tilde{\nu}$ and obtain for $Z\leq Z^*$ from  \eqref{eq:compact} and \eqref{nnvivneoneniven}:
\bee
&&|Z^{m-1+\frac{2(r-1)}{p-1}}\pa_Z^{m-1}\rho| \lesssim e^{-c\tau}+ \la Z\ra^{\tilde{\nu}}\left\|\frac{\pa_Z(Z^{\frac{2(r-1)}{p-1}+m-1}\pa_Z^{m-1}\rho)}{\la Z\ra^{\frac d2+\tilde{\nu}-1}}\right\|_{L^2(2Z_2\leq Z\leq Z^*)}\\
& \lesssim & e^{-c\tau}+\la Z\ra^{\tilde{\nu}}\left[\left\|\frac{\la Z\ra ^{m+\frac{2(r-1)}{p-1}}\pa_Z^m\rho}{\la Z\ra^{\frac d2+\nut}}\right\|_{L^2(Z\leq Z^*)}+\left\|\frac{\la Z\ra ^{m-1+\frac{2(r-1)}{p-1}}\pa_Z^{m-1}\rho}{\la Z\ra^{\frac d2+\nut}}\right\|_{L^2(Z\leq Z^*)}\right]\\
& \lesssim &  e^{-c\tau}+{\mathcal d}_0\la Z\ra^{\tilde{\nu}}e^{-\mu\tilde{\nu} \tau}\le  e^{-c\tau}+{\mathcal d}_0\left(\frac{\la Z\ra}{Z^*}\right)^{\tilde{\nu}}\lesssim {\mathcal d}_0
\eee
and hence $$\forall 0\le m\le \frac{4k_m}{9}, \ \ \left\|\frac{Z^m\pa_Z^m\rho}{\rho_P}\right\|_{L^\infty(Z\le Z^*)}\leq {\mathcal d}_0.$$
We similarly pick $1\le m\le m_0$, apply \eqref{iobebeibbneobis} to $u=\la Z\ra^{r-2+m}\pa_Z^m\Psi$, $\gamma+1=\frac d2+\tilde{\nu}$, and obtain for $Z\leq Z^*$ from \eqref{nnvivneoneniven}:
\bee
&&|\la Z\ra^{r-2+m}\pa_Z^m\Psi|\lesssim  e^{-c\tau}+\la Z\ra^{\tilde{\nu}}\left\|\frac{\pa_Z(\la Z\ra^{r-2+m}\pa_Z^m\Psi)}{\la Z\ra^{\frac d2+\tilde{\nu}-1}}\right\|_{L^2}\\
&+& \la Z\ra^{\tilde{\nu}} 
\left[\left\|\frac{\la Z\ra^{r-3+m}\pa_Z^m\Psi)}{\la Z\ra^{\frac d2+\tilde{\nu}-1}}\right\|_{L^2(2Z_2\leq Z\leq Z^*)}+\left\|\frac{\la Z\ra^{r-2+m}\pa_Z^{m+1}\Psi)}{\la Z\ra^{\frac d2+\nut-1}}\right\|_{L^2(2Z_2\leq Z\leq Z^*)}\right]\\
& \lesssim & e^{-c\tau}+{\mathcal d}_0\la Z\ra^{\tilde{\nu}}e^{-\mu\tilde{\nu} \tau}\le  e^{-c\tau}+{\mathcal d}_0\left(\frac{\la Z\ra}{Z^*}\right)^{\tilde{\nu}}\le {\mathcal d}_0
\eee

and hence $$\forall 1\le m\le m_0=\frac{4k_m}9+1, \ \ \|\la Z\ra^{r-2+m}\pa_Z^m\Psi\|_{L^\infty(Z\le Z^*)}\le {\mathcal d}_0.$$
\noindent\underline{case $m_0\leq m\leq \frac{2k_m}3+1$}. Recall \eqref{defsigmnavrelations}:
$$\sigma(m)+m=-\alpha(k_m-m)+m=(\alpha+1)(m-m_0)+\sigma_\nu$$ and rewrite the norm:
\bee
\|\rhot,\Psit\|_{m,\sigma}^2&\ge &\sum_{k=0}^m\int_{Z\le Z^*} \frac{1}{\la Z\ra^{2(m-k+\sigma(m))}}\left[b^2|\nabla\rhot^{(k)}|^2+(p-1)\rho_D^{p-2}\rho_T(\rhot^{(k)})^2+\rho_T^2|\nabla \Psi^{(k)}|^2\right]\\
& = & \sum_{k=0}^m\int_{Z\le Z^*} \frac{\la Z\ra^{2k}}{\la Z\ra^{2(\alpha+1)(m-m_0)+2\sigma_\nu}}\left[b^2|\nabla\rhot^{(k)}|^2+(p-1)\rho_D^{p-2}\rho_T(\rhot^{(k)})^2+\rho_T^2|\nabla \Psi^{(k)}|^2\right]
\eee
We use spherical symmetry to infer from \eqref{esifnfewwonnw}:
\bea
\label{evvnkebboebe}
\nonumber &&\int_{2Z_2\leq Z\leq Z^*} \left|\frac{Z^{m-(\alpha+1)(m-m_0)}\pa_Z^{m}\rho}{\la Z\ra^{\frac d2-\frac{2(r-1)}{p-1}+\nut}}\right|^2+\int\left|\frac{Z^{m-(\alpha+1)(m-m_0)}\la Z\ra^{r-1}\pa^{m+1}_Z\Psi}{\la Z\ra^{\frac d2+\nut}}\right|^2\\
\nonumber &\lesssim& \sum_{j=0}^m\int_{2Z_2\leq Z\leq Z^*}\frac{Q|Z^j\pa^j\rho|^2+\rho_T^2|Z^j\pa_Z^j\pa_Z\Psi|^2}{\la Z\ra^{2\sigma_\nu+(\alpha+1)(m-m_0)}}\lesssim \|\rho,\Psi\|_{m,\sigma(m)}^2\\
&\le& {\mathcal d}_0e^{-2\mu\nut \tau}.
\eea
Observe that for $m_0\le m\le\frac{2k_m}{3}+1$, from \eqref{defalphacomnot}:
\bea
\label{loweroundffd}
\nonumber 
&&m-(\alpha+1)(m-m_0)=m_0(1+\alpha)-\alpha m\geq m_0(1+\alpha)-\alpha\left(2\frac{k_m}{3}+1\right)\\
\nonumber &=& k_m\left[\frac{4}{9}\left(1+\frac{4}{5}\right)-\frac 23\frac 45\right]+O_{k_m\to +\infty}(1)=\frac{4k_m}{15}+O_{k_m\to +\infty}(1)\\
&>& \frac{k_m}{4}+10.
\eea
We now apply \eqref{iobebeibbneobis}, \eqref{evvnkebboebe} to $m_0+1\le m\le \frac{2k_m}{3}+1$, $u=\la Z\ra^{m+\frac{2(r-1)}{p-1}-(\alpha+1)(m-m_0)-1}\pa_Z^{m-1}\rho$, $\gamma+1=\frac d2+\nut$ and obtain for $Z\leq Z^*$:
\bee
&&|\la Z\ra^{m+\frac{2(r-1)}{p-1}-(\alpha+1)(m-m_0)-1}\pa_Z^{m-1}\rho|\\
& \lesssim& e^{-c\tau}+ \la Z\ra^{\nut}\left\|\frac{\pa_Z(\la Z\ra^{m+\frac{2(r-1)}{p-1}-(\alpha+1)(m-m_0)-1}\pa_Z^{m-1}\rho)}{\la Z\ra^{\frac d2+\nut-1}}\right\|_{L^2(Z\leq Z^*)}\\
& \lesssim & e^{-c\tau}+\la Z\ra^{\nut}\left\|\frac{\la Z\ra ^{m+\frac{2(r-1)}{p-1}-(\alpha+1)(m-m_0)}\pa_Z^m\rho}{\la Z\ra^{\frac d2+\nut}}\right\|_{L^2(\leq Z^*)}\\
&+& \la Z\ra^{\nut}\left\|\frac{\la Z\ra ^{m+\frac{2(r-1)}{p-1}-(\alpha+1)(m-m_0)-1}\pa_Z^{m-1}\rho}{\la Z\ra^{\frac d2+\nut}}\right\|_{L^2(Z\leq Z^*)}\lesssim {\mathcal d}_0\left[1+\la Z\ra^{\nut}e^{-\mu\nut \tau}\right]
\eee
and hence using \eqref{loweroundffd} for $Z\le Z^*$:
$$\left|\frac{Z^{\frac{k_m}{4}}\pa_Z^{m-1}\rho}{\rho_D}\right|\lesssim \left|Z^{m+\frac{2(r-1)}{p-1}-(\alpha+1)(m-m_0)-1}\pa_Z^{m-1}\rho\right|\lesssim {\mathcal d}_0\left[1+ \left(\frac{Z}{Z^*}\right)^{\nut}\right]\lesssim {\mathcal d}_0$$ and hence
$$\forall \frac{4k_m}{9}+1\le m\le \frac{2k_m}{3}, \ \ \left\|\frac{Z^{\frac{k_m}{4}}\pa_Z^{m}\rho}{\rho_D}\right\|_{L^\infty(Z\le Z^*)}\le {\mathcal d}_0.$$
For the phase, we apply  \eqref{iobebeibbneobis}, \eqref{evvnkebboebe} to $m_0+1\le m\le \frac{2k_m}{3}+1$, $\gamma+1=\frac d2+\nut$ , $u=\la Z\ra^{r-1+m-(\alpha+1)(m-m_0)}\pa_Z^m\Psi$ and obtain:
\bee
&&\la Z\ra^{r-1+m-(\alpha+1)(m-m_0)}|\pa_Z^m\Psi|\lesssim  e^{-c\tau}+ \la Z\ra^{\nut}\left\|\frac{\pa_Z(\la Z\ra^{r-1+m-(\alpha+1)(m-m_0)}\pa_Z^m\Psi)}{\la Z\ra^{\frac d2+\nut-1}}\right\|_{L^2(Z\leq Z^*)}\\
& \lesssim &e^{-c\tau}+\la Z\ra^{\nut}\left\|\frac{\la Z\ra^{r-1+m-1-(\alpha+1)(m-m_0)}\pa_Z^m\Psi}{\la Z\ra^{\frac d2+\nut-1}}\right\|_{L^2(Z\leq Z^*)}\\
&+& \la Z\ra^{\nut}\left\|\frac{\la Z\ra^{r-1+m-(\alpha+1)(m-m_0)}\pa_Z^{m+1}\Psi}{\la Z\ra^{\frac d2+\nut-1}}\right\|_{L^2(Z\leq Z^*)}\le {\mathcal d}_0\left[1+\la Z\ra^{\nut}e^{-\mu\nut \tau}\right]
 \eee
 and hence for $m_0+1\le m\le \frac{2k_m}{3}$ from \eqref{loweroundffd} for $Z\le Z^*$:
 \bee
 \la Z\ra^{r-2+\frac{k_m}{4}}|\pa_Z^{m}\Psi|\lesssim \la Z\ra^{r-1+m-(\alpha+1)(m-m_0)}|\pa_Z^m\Psi|\le {\mathcal d}_0
 \eee
 which concludes the proof of \eqref{smallnessoutsideinitbootfinal}.

\end{proof}
Similar to the above, we also have the following exterior bounds for $|x|\ge 1$: 

\begin{lemma}[Exterior pointwise bounds]
There holds:
\be
\label{smallnessoutsideinitbootfinal-ext}
\left|\begin{array}{l}
\forall 0\le k\le\frac{2k_m}{3}, \ \ \|\frac{\la Z\ra ^{n(k)}\pa_Z^k\rho}{\rho_D}\|_{L^\infty(Z\ge Z^*)}\le {\mathcal d}_0\\
\forall 1\le k\le \frac{2k_m}{3}, \ \ \|\frac{\la Z\ra^{n(k)}\pa_Z^k\Psi}b\|_{L^\infty(Z \ge Z^*)}\leq {\mathcal d}_0
\end{array}\right.
\ee
where ${\mathcal d}_0$ is a smallness constant depending on data.
\end{lemma}
\begin{proof}
We recall \eqref{defsigmnavrelations} and \eqref{defsigmnavrelations-sigma}  and, in the 
\noindent\underline{case $0\leq k\leq \frac{2k_m}3$}, bound
\bee
I_{k,\sigma(k)}\ge\int_{Z\ge Z^*} \la Z\ra^{2k-2\sigma_\nu} \left(\frac Z{Z^*}\right)^{2n_P-\frac{4(r-1)}{p-1}-2(r-2)+4\nut}\left[b^2|\nabla\rhot^{(k)}|^2+(p-1)\rho_D^{p-1}(\rhot^{(k)})^2+\rho_D^2|\nabla \Psi^{(k)}|^2\right]
\eee
We observe from \eqref{fromularhod}, \eqref{defsigmnu}, \eqref{conditionsigma} and $b={Z^*}^{2-r}$ that for $Z\ge Z^*$
\bee
Z^{d-1}\la Z\ra^{2k-2\sigma_\nu} \left(\frac Z{Z^*}\right)^{2n_P-\frac{4(r-1)}{p-1}-2(r-2)+4\nut}b^2|\nabla\rhot^{(k)}|^2&\approx &
\la Z\ra^{d-1+2k-2\sigma_\nu-\frac{4(r-1)}{p-1}-2(r-2)+4\nut} \frac{|\nabla\rhot^{(k)}|^2}{(Z^*)^{4\nut}\rho_D^2}\\
&=&\la Z\ra^{1+2\nut} \left(\frac{\la Z\ra^{k} |\nabla\rhot^{(k)}|}{(Z^*)^{2\nut}\rho_D}\right)^2
\eee
Similarly,
\bee
Z^{d-1}\la Z\ra^{2k-2\sigma_\nu} \left(\frac Z{Z^*}\right)^{2n_P-\frac{4(r-1)}{p-1}-2(r-2)+4\nut}\rho_D^2|\nabla \Psi^{(k)}|^2
&\approx& \la Z\ra^{d-1+2k-2\sigma_\nu-\frac{4(r-1)}{p-1}-2(r-2)+4\nut} \frac{|\nabla\Psi^{(k)}|^2}{(Z^*)^{4\nut}b^2}\\ &=&
\la Z\ra^{1+2\nut} \left(\frac{\la Z\ra^{k} |\nabla\Psi^{(k)}|}{(Z^*)^{2\nut}b}\right)^2
\eee
Now, for a spherically symmetric function $u$, $Z\ge Z^*$ and an arbitrary $\l>0$
\bee
|u(Z)|&=&|u(Z^*)+\int_{Z^*}^Z \pa_Z u|\le |u(Z^*)|+ \left(\int_{Z^*}^Z  \tau^{1+2\l}|\pa_Z u|^2 d\tau\right)^{\frac 12}\left(\int_{Z^*}^Z \tau^{-1-2\l}d\tau\right)^{\frac 12} |\\ &\le& |u(Z^*)|+ (Z^*)^{-\l}\left(\int_{Z^*}^Z  \tau^{1+2\l}|\pa_Z u|^2d\tau\right)^{\frac 12}
\eee
We apply this to $u=\left(\frac{\la Z\ra^{k} \rhot^{(k)}}{(Z^*)^{2\nut}\rho_D}\right)$ for $k\ge 1$ and $\l=\nut$
\bee
&&\left(\frac{\la Z\ra^{k} \rhot^{(k)}}{(Z^*)^{2\nut}\rho_D}\right)(Z)\\
&\le &\left(\frac{\la Z\ra^{k} \rhot^{(k)}}{(Z^*)^{2\nut}\rho_D}\right)(Z^*)+(Z^*)^{-\nut}\left(\int_{Z^*}^Z  \tau^{1+2\nut}\left[\left(\frac{\la \tau\ra^{k} |\nabla\rhot^{(k)}|}{(Z^*)^{2\nut}\rho_D}\right)^2+\left(\frac{\la \tau\ra^{k-1}\rhot^{(k)}}{(Z^*)^{2\nut}\rho_D}\right)^2\right] d\tau\right)^{\frac 12}\\ &\lesssim &(Z^*)^{-2\nut} {\mathcal d}_0+(Z^*)^{-\nut} (I_{k,\sigma(k)}+I_{k,\sigma(k-1)})^{\frac 12},
\eee
where we used the already proved interior bounds \eqref{smallnessoutsideinitbootfinal}. This, together with \eqref{esifnfewwonnw}, immediately implies the exterior bound for $\pa_Z^k\rho$ and $1\le k\le \frac {4k_m}9+1$.
The corresponding bound for $\pa_Z^k\Psi$ is obtained similarly using $u=\left(\frac{\la Z\ra^{k} \Psi^{(k)}}{(Z^*)^{2\nut}b}\right)$ and $\l=\nut$. To prove the result for $\rho$ in the case of  $k=0$ we note that we could run the above argument for 
$u=\left(\frac{\la Z\ra^{k+\frac \nut2} \rhot^{(k)}}{(Z^*)^{2\nut}\rho_D}\right)$ and $\l=\frac \nut2$, which would imply
$$
\left(\frac{\la Z\ra^{k+\frac\nut2} \rhot^{(k)}}{\rho_D}\right)(Z)\lesssim (Z^*)^{\frac \nut 2} {\mathcal d}_0.
$$
It is a stronger estimate and, more crucially in the case of $k=1$, an estimate which can be integrated in $Z$ to produce 
the desired bound for $k=0$. Similar argument applies to $\Psi$. Finally, the regime $\frac {4k_m}9+1\le k\le \frac{2k_m}3$
can be treated by a combination of the argument above and the corresponding interior one. We note that the weight 
function $\xi_k$, which determines the $\frac Z{Z^*}$ dependence and thus relevant in the exterior, remains exactly the 
same in the whole range $0\le k\le \frac {2k_m}3+1$. We omit the details.
\end{proof}


\section{Highest Sobolev norm}
\label{sec:high}


In this section we improve the bootstrap bound \eqref{sobolevinitboot} on the highest unweighted Sobolev norm of $(\rhot,\Psi)$. Specifically,
for (see \eqref{globalsobolevnorm}) 
\be
\label{globalsobolevnorm'}
\|\rhot,\Psit\|_{{k_m}}^2=\sum_{j=0}^{k_m}\sum_{|\alpha|=j}\int \frac{b^2|\nabla\nabla^\alpha\rhot|^2+(p-1)\rho_D^{p-2}\rho_T(\nabla^\alpha\rhot)^2+\rho_T^2|\nabla \nabla^\alpha\Psi|^2}{\la Z\ra^{2(k_m-j)}}
\ee
we will establish the following 
\begin{proposition}[Control of the highest Sobolev norm]
\label{proproopr}
For some small constant ${\mathcal d}$ dependent on the data,
\be
\label{cneioneoneon}
\|\rhot,\Psit\|_{{k_m}}\leq {\mathcal d}.
\ee
\end{proposition}

\begin{proof}[Proof of Proposition \ref{proproopr}] This follows from the global {\it unweighted} quasilinear energy identity. We let $$k_m=2K_m, \ \ K_m\in \Bbb N$$ and denote in this section $$k=k_m, \ \ \rhot^{(k)}=\Delta^{K_m}\rhot, \ \ \Psi^{(k)}=\Delta^{K_m}\Psi.$$
We recall the notation \eqref{defimsigma}
\be
\label{vebioenobvebveb}
I_{k_m}=\int b^2|\nabla \pa^{k_m}\rhot|^2+(p-1)\int \rho_D^{p-2}\rho_T |\pa^{k_m}\rhot|^2+\int\rho_T^2|\nabla \pa^{k_m}\Psit|^2.
\ee

\noindent{\bf step 1} Control of lower order terms. We recall the notation:
$$\left|\begin{array}{ll}
\|\rhot,\Psi\|_{k_m,\sigma(k_m)}^2=\sum_{j=0}^{k_m}\int \chi_{j,k_m,\sigma(k_m)}b^2|\nabla\pa^j\rhot|^2+(p-1)\int \chi_{j,k_m}\rho_D^{p-2}\rho_T(\pa^j\rhot)^2+\int\chi_{j,k_m}\rho_T^2|\nabla \pa^j\Psi|^2\\
\chi_{j,k_m,\sigma(k_m)}(Z)=\frac{1}{\la Z\ra^{k_m-j}}.
\end{array}\right.
$$ 
Observe from \eqref{defsimgma} that for $m_0\le m\le k_m-1$: $$\sigma(m)+m=-\alpha(k_m-m)+m=(1+\alpha)m-\alpha k_m\le k_m$$ and the same holds for $0\le m\le m_0$ for $k_m$ large enough.
 Hence,
$$ \frac{\la Z\ra^{m}|\pa^m_Z\rho|}{\la Z\ra^{k_m}}\lesssim \frac{\la Z\ra^{m}|\pa^m_Z\rho|}{\la Z\ra^{\sigma(m)+m}}$$ and  a similar estimate for $\Psi$, imply, using \eqref{esifnfewwonnw}:
\be
\label{veniovnineonelweroidre}
\|\rhot,\Psi\|_{k_m,\sigma(k_m)}^2\leq I_{k_m}+\|\rhot,\Psi\|^2_{k_m-1,\sigma(k_m-1)}\leq I_{k_m}+{\mathcal d}_{0}
\ee 
By Remark \ref{delta} we can replace (up to the lower order terms controlled as above) $I_{k_m}$ with 
\be
\label{eq:vebioenobvebveb}
J_{k_m}:=\int b^2|\nabla \Delta^{K_m}\rhot|^2+(p-1)\int \rho_D^{p-2}\rho_T |\Delta^{K_m}\rhot|^2+\int\rho_T^2|\nabla \Delta^{K_m}\Psit|^2.
\ee
We claim: there exist $k_m^*(d,r,p),c_{d,r,p}>0$ such that for all $k_m>k_m^*(d,r,p)$, there holds:
\be
\label{estnienonneo}
\frac{d}{d\tau}\left\{J_{k_m}\left[1+O(\delta)\right]\right\}+c_{d,r}k_mJ_{k_m}\leq {\mathcal d}.
\ee
Integrating the above in time, using \eqref{sobolevinit}, \eqref{veniovnineonelweroidre}, yields \eqref{cneioneoneon}.\\

\noindent{\bf step 2} Energy identity. We revisit the computation of \eqref{estqthohrk}, \eqref{formluafone}, \eqref{nkenononenon}, \eqref{exactliearizedflowtilde} in order to extract all the coupling terms at the highest level of derivatives. Recall \eqref{exactliearizedflowtilde}:
$$
\left|\begin{array}{ll} \pa_\tau \rhot=-\rho_T\Delta \Psit-2\nabla\rho_T\cdot\nabla \Psit+H_1\rhot-H_2\Lambda \rhot-\tilde{\mathcal E}_{P,\rho}\\
\pa_\tau \Psit=b^2\frac{\Delta \rho_T}{\rho_T}-\left\{H_2\Lambda \Psit+\mu(r-2)\Psit+|\nabla \Psit|^2+(p-1)\rho_D^{p-2}\rhot +\NL(\rhot)\right\}-\Et_{P,\Psi}.
\end{array}\right.
$$
We use $$[\Delta^{K_m},\Lambda]=k_m\Delta^{K_m}$$ and recall \eqref{estimatecommutatorvlkeveln}:
$$[\Delta^k,V]\Phi-2k\nabla V\cdot\nabla \Delta^{k-1}\Phi=\sum_{|\alpha|+|\beta|=2k,|\beta|\le 2k-2}c_{k,\alpha,\beta}\pa^\alpha V\pa^\beta\Phi.
$$ 
which gives:
$$
\Delta^{K_m}(H_2\Lambda \rhot)=k_m(H_2+\Lambda H_2)\rho_k+H_2\Lambda \rho_k+\mathcal A_k(\rhot) $$
with 
\be
\label{estimatepenttniialvone}
\left|\begin{array}{l}
|\mathcal A_k(\rhot)|\lesssim c_k\sum_{j=1}^{k-1}\frac{|\nabla^j\rhot|}{\la Z\ra^{k_m+r-j}}\\
|\nabla \mathcal A_k(\rhot)|\lesssim c_k\sum_{j=1}^{k}\frac{|\nabla^j\rhot|}{\la Z\ra^{k_m+r+1-j}}
\end{array}\right.
\ee
where $\nabla^j=\pa^{\alpha_1}_1\dots\pa^{\alpha_d}_d$, $j=\alpha_1+\dots+\alpha_d$ denotes a generic derivative of order $j$. Using \eqref{estimatecommutatorvlkeveln} again:
\bea
\label{estqthohrkbisbis}
\nonumber \pa_\tau \rhot^{(k)}&=&(H_1-k(H_2+\Lambda H_2)\rhot_k-H_2\Lambda \rhot_k-(\Delta^{K_m}\rho_T)\Delta \Psit-k\nabla\rho_T\cdot\nabla \Psi^{(k)}-\rho_T\Delta\Psit_k\\
 &-& 2\nabla(\Delta^{K_m}\rho_T)\cdot\nabla \Psit-2\nabla \rho_T\cdot\nabla \Psit_k+  F_1
\eea
with
\bea
\label{formluafonebis}
F_1&=&-\Delta^{K_m}\tilde{\mathcal E}_{P,\rho}+[\Delta^{K_m},H_1]\rhot-A_k(\rhot)\\
\nonumber &-& \sum_{\left|\begin{array}{ll} j_1+j_2=k\\ j_1\geq 2, j_2\geq 1\end{array}\right.}c_{j_1,j_2}\nabla^{j_1}\rho_T\nabla^{j_2}\Delta\Psit-\sum_{\left|\begin{array}{ll}j_1+j_2=k\\
j_1,j_2\geq 1\end{array}\right.}c_{j_1,j_2}\nabla^{j_1}\nabla\rho_T\cdot\nabla^{j_2}\nabla\Psit.
\eea
For the second equation, we have similarly:
\bea
\label{nkenononenonbis}
 \pa_\tau\Psit_k&=&b^2\left(\frac{\Delta^{K_m+1} \rho_T}{\rho_T}-\frac{k\nabla\Delta^{K_m}\rho_T\cdot\nabla\rho_T}{\rho_T^2}\right)\\
\nonumber &-& k(H_2+\Lambda H_2)\Psit_k-H_2\Lambda \Psit_k-\mu(r-2)\Psit_k-2\nabla \Psit\cdot\nabla \Psit_k\\
\nonumber&-& \left[(p-1)\rho_P^{p-2}\rhot_k+k(p-1)(p-2)\rho_D^{p-3}\nabla\rho_D\cdot\nabla\Delta^{K_m-1}\rhot\right]+F_2
\eea
with
\bea
\label{estqthohrkbisbisbisbisbibfebjbifebfji}
\nonumber F_2&=& -\pa^k\tilde{\mathcal E}_{P,\Psi}+b^2\left[\Delta^{K_m}\left(\frac{\Delta \rho_T}{\rho_T}\right)-\frac{\Delta^{K_m+1} \rho_T}{\rho_T}+\frac{k\nabla\Delta^{K_m}\rho_T\cdot\nabla\rho_T}{\rho_T^2}\right]\\
\nonumber &-&(p-1)\left([\Delta^{K_m},\rho_D^{p-2}]\rhot-k(p-2)\rho_D^{p-3}\nabla\rho_D\cdot\nabla\Delta^{K_m-1}\rhot\right)\\&-&\mathcal A_k(\Psit)
- \sum_{j_1+j_2=k,j_1,j_2\geq 1}\nabla^{j_1}\nabla\Psit\cdot\nabla^{j_2}\nabla\Psit-\Delta^{K_m}\NL(\rhot).
\eea
and
\be
\label{estimatepenttniialvone'}
\left|\begin{array}{l}
|{\mathcal A}_k(\Psit)|\lesssim \sum_{j=1}^{k-1}\frac{|\nabla^j\Psit|}{\la Z\ra^{k_m+r-j}}\\
|\nabla \mathcal A_k(\Psit)|\lesssim \sum_{j=1}^{k}\frac{|\nabla^j\Psit|}{\la Z\ra^{k_m+r+1-j}}.
\end{array}\right.
\ee
We then run the global quasilinear energy identity similar to \eqref{algebracienergyidnentiy} with $\chi=1$ and obtain:
\bea
\label{algebracienergyidnentiybis}
&&\frac 12\frac{d}{d\tau}\left\{\int b^2|\nabla\rhot_k|^2+(p-1)\int \rho_D^{p-2}\rho_T\rhot_k^2+\int\rho_T^2|\nabla \Psit_k|^2\right\}\\
\nonumber &=& -\mu(r-2)b^2\int|\nabla\rhot_k|^2+\int\pa_\tau\rho_T\left[\frac{p-1}{2}\rho_D^{p-2}\rhot_k^2+\rho_T|\nabla \Psit_k|^2\right]+ \frac{p-1}{2}\int (p-2)\pa_\tau\rho_D\rho_D^{p-3}\rho_T\rhot_k^2\\
\nonumber & + & \int F_1(p-1)\rho_D^{p-2}\rho_T\rhot_k+b^2\int \nabla F_1\cdot\nabla \rhot_k+\int \rho^2_T\nabla F_2\cdot\nabla \Psit_k\\
\nonumber &-& \int k\nabla\rho_T\cdot\nabla\Psi_k(-b^2\Delta \rhot_k+(p-1)\rho_D^{p-2}\rho_T\rhot_k)+\int b^2\frac{k\nabla \Delta^{K_m}\rho_T\cdot\nabla\rho_T}{\rho_T}(\rho_T\Delta \Psit_k+2\nabla\rho_T\cdot\nabla\Psit_k)\\
\nonumber & + & \int \left[(H_1-k(H_2+\Lambda H_2))\rhot_k-H_2\Lambda \rhot_k-(\Delta^{K_m}\rho_T)\Delta \Psit-2\nabla(\Delta^{K_m}\rho_T)\cdot\nabla \Psit\right]\\
\nonumber&\times& \left[-b^2\Delta \rhot_k+(p-1)\rho_D^{p-2}\rho_T\rhot_k\right]\\
\nonumber & - & \int\left[b^2\Delta^{K_m+1} \rho_D-k\rho_T(H_2+\Lambda H_2)\Psit_k-\rho_T H_2\Lambda \Psit_k-\mu(r-2)\rho_T\Psit_k-2\rho_T\nabla \Psit\cdot\nabla \Psit_k\right]\\
\nonumber&\times & \left[2\nabla\rho_T\cdot\nabla \Psit_k+\rho_T\Delta \Psit_k\right]\\
\nonumber & + & \int k(p-1)(p-2)\rho_T\rho_D^{p-3}\nabla\rho_D\cdot\nabla \Delta^{K_m-1}\rhot\left[2\nabla\rho_T\cdot\nabla \Psit_k+\rho_T\Delta \Psit_k\right].
\eea
We now estimate all terms in \eqref{algebracienergyidnentiybis}. The proof is similar to that one of Proposition \eqref{propinduction} with two main differences:  the absence of a cut-off function $\chi$, and a priori control of lower order derivatives  from \eqref{veniovnineonelweroidre}.  The challenge here is to avoid any loss of derivatives and to compute {\em exactly} the quadratic form at the highest level of derivatives.
The latter will be shown to be positive on a compact set in $Z$ provided $k_m>k_m^*(d,r,p)\gg 1$ has been chosen large enough.\\

In what follows, below, we will use $\delta>0$ as a small universal constant and will assume that the pointwise bounds \eqref{smallnessoutsideinitbootfinal} obtained on the lower order derivatives of $\rhot$ and $\Psit$ are dominated by $\delta$.
On the set $Z\le Z^*$, this will often be a source of smallness, while for $Z\ge Z^*$, we may use the bootstrap bounds \eqref{smallglobalboot} and the $\delta$-smallness will be generated by
extra powers of $Z$. We also note that from \eqref{estnienonneo} the quadratic form is expected to be proportionate to 
$k_m I_{k_m}$. Choosing $k_m$ large will allow us to dominate other quadratic terms without smallness but with the uniform 
dependence on $k_m$. The notation $\lesssim$ will allow dependence on $k_m$, while $O$ will indicate a bound independent of $k_m$. As before, ${\mathcal d}_0$  (as well as ${\mathcal d}$) 
will denote small constants, dependent on the data (or, more precisely, on $\tau_0$), that can be made arbitrarily small. In particular, we will use 
\be\label{eq:smaf}
\|\rhot,\Psi\|_{k_m-1,\sigma(k_m-1)}\le {\mathcal d}_0.
\ee
The constants $\delta\gg {\mathcal d}_0$ will be assumed to be smaller than any power of $k_m$, so that our calculations will be unaffected 
by combinatorics generated by taking $k_m$ derivatives of the equations.
\\

\noindent{\bf step 3} Leading order terms.\\

\noindent\underline{\em Cross term}. Recall \eqref{pohozaevbispouet}:
\bee
\nonumber \int\Delta g F\cdot\nabla gdx&=&\sum_{i,j=1}^d \int \pa_i^2 g F_j\pa_jgdx=-\sum_{i,j=1}^d \int\pa_ig(\pa_iF_j\pa_jg+F_j\pa^2_{i,j}g)\\
&=&-\sum_{i,j=1}^d \int\pa_iF_j\pa_ig\pa_jg+\frac 12\int |\nabla g|^2\nabla \cdot F.
\eee
Letting $g=g_1+g_2$ yields a bilinear off-diagonal Pohozhaev identity:
\bee
\int\left[\Delta g_1 F\cdot\nabla g_2+\Delta g_2 F\cdot\nabla g_1\right]dx=-\sum_{i,j=1}^d \int\pa_iF_j(\pa_ig_1\pa_jg_2+\pa_ig_2\pa_jg_1)+\int \nabla g_1\cdot\nabla g_2(\nabla \cdot F).
\eee
We may therefore integrate by parts the one term in \eqref{algebracienergyidnentiy} which has too many derivatives:
\bee
\nonumber &&b^2k\left|\int \left[\nabla\rho_T\cdot\nabla\Psi_k\Delta \rhot_k+\nabla\Delta^{K_m}\rho_T\cdot\nabla\rho_T\Delta \Psit_k\right]\right|\\
& =& b^2k\left|\int \nabla \rho_T\cdot\nabla \Psi_k\Delta\rhot_k+\nabla \rho_T\cdot\nabla \rhot_k\Delta \Psi_k+\nabla \rho_T\cdot\nabla \Delta^{K_m}\rho_D\Delta \Psi_k\right|\\
& = & b^2k\left|-\int\sum_{i,j=1}^d\pa^2_{i,j}\rho_T(\pa_i\rhot_k\pa_j\Psi_k+\pa_i\Psi_k\pa_j\rhot_k)+\int\nabla \rhot_k\cdot\nabla \Psi_k\Delta \rho_T+\int\nabla\Psi_k\cdot\nabla(\nabla \rho_T\cdot\nabla \Delta^{K_m}\rho_D)\right|\\
& \lesssim & b^2k\int\rho_T|\nabla \Psi_k|\left[\frac{1}{\la Z\ra^{k+2}}+\frac{|\nabla \rhot_k|}{\la Z\ra^2}\right] \leq  \delta \int\rho_T^2|\nabla \Psit_k|^2+C_\delta b^4+C_\delta b^4\int |\nabla \rhot_k|^2\\
&\lesssim& \delta J_{k_m}+C_\delta b^4.
\eee
We estimate similarly:
\bee
\left|kb^2\int\frac{\nabla \Delta^{K_m}\rho_T\cdot\nabla\rho_T}{\rho_T}\nabla \rho_T\cdot\nabla \Psit_k\right|\leq \delta\left[b^2\int|\nabla\rhot_k|^2+\int\rho_T^2|\nabla\Psit_k|^2\right]+c_\delta b^4\lesssim \delta J_{k_m}+C_\delta b^4.
\eee
We use 
\be
\label{vnenvenveneneonv}
\frac{|\rhot|}{\rho_T}+\frac{|\Lambda\rhot|}{\la Z\ra^c\rho_T}\lesssim \delta, \  \ 0<c\ll1
\ee
 to  compute the first coupling term:
\bee
&&-k(p-1)\int\nabla\rho_T\cdot\nabla\Psit_k\rho_D^{p-2}\rho_T\rhot_k= -k\int\rho_D\nabla \rho^{p-1}_D\cdot\nabla \Psit_k\rhot_k+ O\left(\delta\int \frac{|\nabla \Psit_k|\rho_D^{p-1}|\rho_T\rhot_k|}{\la Z\ra}\right)\\
& = & -k\int\rho_D\nabla \rho_D^{p-1}\cdot\nabla \Psit_k\rhot_k+O\left(\delta J_{k_m}\right) \eee
The second coupling term is computed after an integration by parts using \eqref{vnenvenveneneonv}, the control of lower order terms \eqref{veniovnineonelweroidre} and the spherically symmetric assumption:
\bee
&&\int (\rho_T\Delta \Psit_k+2\nabla \rho_T\cdot\nabla \Psit_k)k(p-1)(p-2)\rho_T\rho_D^{p-3}\nabla\rho_D\cdot\nabla\Delta^{K_m-1}\rhot\\
& = & k(p-1)(p-2)\int \nabla \cdot(\rho_T^2\nabla \Psi_k)\rho_D^{p-3}\nabla\rho_D\cdot\nabla\Delta^{K_m-1}\rhot\\
& = & -k(p-1)(p-2)\int\rho_T^2\nabla \Psi_k\cdot\nabla\left(\rho_D^{p-3}\nabla\rho_D\cdot\nabla\Delta^{K_m-1}\rhot\right)\\
& = &  -k(p-1)(p-2)\int\rho_T^2\pa_Z\Psi_k\pa_Z\left(\rho_D^{p-3}\pa_Z\rho_D\pa_Z\Delta^{K_m-1}\rhot\right)\\
& = &   -k(p-1)(p-2)^d\int\rho_D^{p-3}\pa_Z\rho_D\rho_T^2\pa_Z\Psi_k\pa^2_Z\Delta^{K_m-1}\rhot+O\left(\int \rho_T|\nabla \Psi_k|\rho_T^{p-1}\frac{|\nabla^{k_m-1}\rhot|}{\la Z\ra^2}\right)\\
& = &-\int k(p-2)\rho_D\pa_Z(\rho_D^{p-1})\pa_Z\Psit_k \rhot_k+\int k(p-2)(d-1)\rho_D\pa_Z(\rho_D^{p-1})\pa_Z\Psit_k\frac{\pa_Z\Delta^{K_m-1}\rhot}{|Z|}\\ &+&O\left(\int \rho_T|\nabla \Psi_k|\rho_T^{p-1}\frac{|\nabla^{k_m-1}\rhot|}{\la Z\ra^2}\right)\\
& = & -k(p-2)\int\rho_D\nabla \rho^{p-1}_D\cdot\nabla \Psit_k\rhot_k+O\left({\mathcal d}_0+\delta J_{k_m}\right),
\eee
where in the last step we used that
$
\frac {|\pa_Z \rho_D|}{|Z|\rho_D} \lesssim \frac 1{\la Z\ra^2}.
$

\noindent\underline{\em $\rho_k$ terms}. We compute:
\bee
&&\int(H_1-k(H_2+\Lambda H_2)\rhot_k)(-b^2\Delta \rhot_k+(p-1)\rho_D^{p-2}\rho_T\rhot_k)\\
&=& \int(H_1-k(H_2+\Lambda H_2))\left[b^2|\nabla \rhot_k|^2+(p-1)\rho_D^{p-2}\rho_T\rhot^2_k\right]-\frac{b^2}{2}\int \rhot_k^2\Delta(H_1-k(H_2+\Lambda H_2)).
\eee
We now use the global lower bound, see properties \eqref{coercivityquadrcouplinginside} and \eqref{P} of the the profile 
$(w,\sigma)$, 
$$H_2+\Lambda H_2=\mu(1-w-\Lambda w)\ge c_{p,d,r}, \ \ c_{p,d,r}>0$$ 
to estimate using \eqref{esterrorpotentials}, \eqref{veniovnineonelweroidre}:
\bee
&&\int(H_1-k(H_2+\Lambda H_2))\rhot_k(-b^2\Delta \rhot_k+(p-1)\rho_D^{p-2}\rho_T\rhot_k)\\
& \leq& -k\int\left[1+O_{k_m\to +\infty}\left(\frac{1}{k_m}\right)\right](H_2+\Lambda H_2)\left[b^2|\nabla \rhot_k|^2+(p-1)\rho_D^{p-2}\rho_T\rhot^2_k\right]+Cb^2\int\frac{\rho_k^2}{\la Z\ra^{2+r}}\\
& \leq & -k\int (H_2+\Lambda H_2)\left[b^2|\nabla \rhot_k|^2+(p-1)\rho_D^{p-2}\rho_T\rhot^2_k\right]+\delta J_{k_m}.
\eee
Next, using $$|\pa^k\rho_D|\lesssim \frac{\rho_D}{\la Z\ra^{k}},$$
we estimate from \eqref{smallglobalboot}:
\bee
&&b^2\left|\int \Delta \rhot_k\left[(\Delta^{K_m}\rho_D)\Delta \Psit+2\nabla(\Delta^{K_m}\rho_D)\cdot\nabla \Psit\right]\right|\\
&\lesssim&  b^2\int |\nabla \rhot_k|\left[|\nabla(\Delta^{K_m}\rho_D\Delta \Psit)|+|\nabla(\nabla\Delta^{K_m}\rho_D\cdot\nabla \Psit)|\right]\\
& \leq & b^2\delta\int |\nabla \rhot_k|^2+\frac{b^2}{\delta}\sum_{j=1}^3\int \frac{\rho_D^2|\pa^j\Psit|^2}{\la Z\ra^{2\left(k+3-j\right)}}\leq  \delta J_{k_m}+{\mathcal d}_0 b^2.
\eee
For the nonlinear term, we use \eqref{smallglobalboot}, \eqref{pohozaevbispouet}, \eqref{veniovnineonelweroidre} to estimate 
\bee
&&b^2\left|\int \Delta \rhot_k\left[\rhot_k\Delta \Psit+2\nabla\rhot_k\cdot\nabla \Psit\right]\right|\lesssim b^2\left[\int |\pa^2\Psit||\nabla \rhot_k|^2+\int \frac{\rhot_k^2}{\la Z\ra^{2}}\right]\leq   \delta J_{k_m}+{\mathcal d}_0 b^2.
\eee
Next:
\bee
&&\left|\int \left[(\Delta^{K_m}\rho_D)\Delta \Psit-2\nabla(\Delta^{K_m}\rho_D)\cdot\nabla \Psit\right](p-1)\rho_D^{p-2}\rho_T\rhot_k\right|\\
& \leq & \delta\int \rho_D^{p-2}\rho_T\rhot_k^2+\frac{C}{\delta}\int\rho_D^{p-2}\rho_T^2\left[\frac{|\pa^2\Psit|^2}{\la Z\ra^{{2 k}}}+\frac{|\pa\Psit|^2}{\la Z\ra^{2(k+1)}}\right]\\
& \leq & \delta J_{k_m}+{\mathcal d}_0,
\eee
since we are assuming that ${\mathcal d}_0\ll \delta$,
and for the nonlinear term after an integration by parts:
$$\left|\int \left[\rhot_k\Delta \Psit-2\nabla\rhot_k\cdot\nabla \Psit\right](p-1)\rho_D^{p-2}\rho_T\rhot_k\right|\lesssim \delta\int \rho_D^{p-2}\rho_T\rhot_k^2.
$$
From Pohozhaev \eqref{pohozaevbispouet}:
$$
-\int H_2\Lambda \rhot_k(-b^2\Delta \rhot_k)=b^2\int \Delta \rhot_k(ZH_2)\cdot\nabla \rhot_k=  O\left(b^2\int |\nabla \rhot_k|^2.\right)
$$
Integrating by parts and using \eqref{esterrorpotentials}, \eqref{globalbeahvoiru}:
\bee
&&-\int H_2\Lambda \rhot_k\left[(p-1)\rho_D^{p-2}\rho_T\rhot_k\right]+\frac{p-1}{2}\int (p-2)\pa_\tau\rho_D\rho_D^{p-3}\rho_T\rhot_k^2+\frac{p-1}{2}\int \pa_\tau\rho_T\rhot^{p-2}\rhot_k^2\\
&=& \frac {p-1}2\int \rhot_k^2\left[\nabla \cdot(ZH_2\rho_D^{p-2}\rho_T)+\pa_\tau(\rho_D^{p-2})+ \pa_\tau\rho_D\rhot^{p-2}\right]=  O\left(\int \rho_D^{p-2}\rho_T\rhot_k^2\right)
\eee
Note that the above two bounds, even though dependent on the highest order derivatives, contain no $k$ dependence.

\noindent\underline{\em $\Psi_k$ terms}. After an integration by parts:
\bee
&&\left|\int b^2\Delta^{K_m+1}\rho_D\left[2\nabla\rho_T\cdot\nabla \Psit_k+\rho_T\Delta \Psit_k\right]\right|\\
&\lesssim&  b^2\int \rho_T\frac{|\nabla\Psit_k|}{\la Z\ra^{k+3}}+b^2\leq\delta\int \rho_T^2|\nabla \Psi_k|^2+{\mathcal d}_0.
\eee
Then
\bee
&&\mu(r-2)\int\rho_T\Psit_k\left[2\nabla\rho_T\cdot\nabla \Psit_k+\rho_T\Delta \Psit_k\right] \\
&=&  -\mu(r-2)\int \Psit_k^2\nabla\cdot(\rho_T\nabla \rho_T)-\mu(r-2)\int \nabla \Psit_k\cdot\nabla(\rho_T^2\Psit_k)\\
& = & -\mu(r-2)\int \rho_T^2|\nabla \Psit_k|^2
\eee
and similarly, using \eqref{esterrorpotentials}, \eqref{veniovnineonelweroidre}:
\bee
&&k\int\rho_T(H_2+\Lambda H_2)\Psit_k\left[2\nabla\rho_T\cdot\nabla \Psit_k+\rho_T\Delta \Psit_k\right]=k\int (H_2+\Lambda H_2)\Psi_k\nabla\cdot(\rho_T^2\nabla \Psi_k)\\
& = & -k\left[\int (H_2+\Lambda H_2)\rho_T^2|\nabla \Psit_k|^2+\int \rho_T^2\Psit_k^2\left(\frac{\nabla \cdot(\rho_T^2\nabla(H_2+\Lambda H_2))}{2\rho^2_T}\right)\right]\\
&=& -k\int (H_2+\Lambda H_2)\rho_T^2|\nabla \Psit_k|^2+{\mathcal d}_0,
\eee
where the $\Psi_k^2$ term is controlled, with the help of the bound
$$
\left|\frac{\nabla \cdot(\rho_T^2\nabla(H_2+\Lambda H_2))}{2\rho^2_T}\right|\lesssim 
\la Z\ra^{-2-r},
$$
 by using the already bounded $\|(\rhot,\Psi)\|_{k_m-1,\sigma(k_m-1)}$-norm.

Then, from  \eqref{smallnessoutsideinitbootfinal} and \eqref{smallnessoutsideinitbootfinal-ext}:
$$
\left|\int 2\rho_T\nabla \Psit\cdot\nabla \Psit_k(2\nabla \rho_T\cdot\nabla \Psit_k)\right|\lesssim {\mathcal d}_0\int \rho_T^2|\nabla \Psit_k|^2$$
 and using \eqref{pohozaevbispouet}:
\bee
\left|\int 2\rho_T\nabla \Psit\cdot\nabla \Psit_k(\rho_T\Delta \Psit_k)\right|\lesssim \int |\nabla \Psit_k|^2||\pa(\rho_T^2\nabla \Psit)|\lesssim {\mathcal d}_0\int \rho_T^2|\nabla\Psit_k|^2.
\eee
Arguing verbatim as in the proof of \eqref{shaprpohoazev} produces the bound
$$\left|\int \rho_TH_2\Lambda \Psit_k\left(2\nabla \rho_T\cdot\nabla \Psit_k+\rho_T\Delta \Psit_k\right)\right|=O\left(\int\rho_T^2|\nabla \Psi_k|^2\right).
$$

\noindent{\bf step 4} $F_1$ terms. We claim the bound:
\be
\label{estfoneessentialbis}
 b^2\int |\nabla F_1|^2+(p-1)\int {\rho}_{D}^{p-1}F_1^2\lesssim J_{k_m}+  {\mathcal d}_0.
\ee
\noindent\underline{\em Source term induced by localization}. From \eqref{neineinneonev}, for $k_m$ large enough:
$$
\int \rho_D^{p-2}\rho_T|\Delta^{K_m}\tilde{\mathcal E}_{P,\rho}|^2+b^2\int|\nabla \Delta ^{K_m}\tilde{\mathcal E}_{P,\rho}|^2\lesssim {\mathcal d}_0.
$$

\noindent\underline{\em $[\Delta^{K_m},H_1]$ term}. We estimate from \eqref{neoneneonoev}, \eqref{eq:smaf}
 $$(p-1)\int \rho_D ^{p-1}([\Delta^{K_m},H_1]\rhot)^2\lesssim \sum_{j=0}^{k-1}\int  \rho_D^{p-1}\frac{|\pa^j\rhot|^2}{\la Z\ra^{2(r+k-j)}}\le {\mathcal d}_0
$$
and 
\bee
&&b^2\int|\nabla ([\Delta^{K_m},H_1]\rhot)|^2\lesssim b^2 \sum_{j=0}^{k}\int\frac{|\pa_j\rhot|^2}{\la Z\ra^{2(1+r+k-j)}}\\
&=& b^2\int \frac{\rhot^2dZ}{\la Z\ra^{2(1+r+k)}}+b^2 \sum_{j=0}^{k-1}\int \frac{|\pa^j\nabla \rhot|^2}{\la Z\ra^{2(r+k+1-j)+2}}\lesssim b^2+\|\rhot,\Psi\|_{k_m-1,\sigma(k_m-1)}\le {\mathcal d}_0.
\eee

\noindent\underline{\em $\mathcal A_k(\rhot)$ term}. From \eqref{estimatepenttniialvone}, \eqref{eq:smaf}:
$$(p-1)\int \rho_D^{p-1}(\mathcal A_k(\rhot))^2\lesssim \sum_{j=1}^{k-1}\int \rho_D^{p-1} \frac{|\nabla^j\rhot|^2}{\la Z\ra^{2(r+k-j)}}\le {\mathcal d}_0
$$
and 
$$b^2\int|\nabla (\mathcal A_k(\rhot))|^2\lesssim b^2 \sum_{j=0}^{k-1}\int \frac{|\nabla\nabla^j\rhot|^2}{\la Z\ra^{2(r+k-j)}}\le {\mathcal d}_0
$$
and \eqref{estfoneessentialbis} is proved for this term.\\

\noindent\underline{\em Nonlinear term}. After changing indices, we need to estimate
$$
N_{j_1,j_2}=\nabla^{j_1}\rho_T\nabla^{j_2}\nabla \Psit, \ \ j_1+j_2=k+1, \ \ 2\leq j_1,j_2\le k-1.
$$
For the profile term:
$$|\pa^{j_1}\rho_D\nabla^{j_2}\nabla \Psit|\lesssim \rho_{D}\frac{|\nabla^{j_2}\nabla \Psit|}{\la Z\ra^{j_1}}= \rho_D\frac{|\nabla^{j_2}\nabla \Psit|}{\la Z\ra^{k+1-j_2}}$$ and therefore, recalling \eqref{boudnthod}, \eqref{eq:smaf}:
$$\int(p-1) N_{j_1,j_2}^2\rho_D^{p-1}\lesssim \int \frac{\rho_T^2|\nabla^{j_2}\nabla \Psit|^2}{\la Z\ra^{2(k+1-j_2)+2(r-1)}}\le {\mathcal d}_0.
$$
Similarly, after taking a derivative:
$$
b^2\int|\nabla N_{j_1,j_2}|^2\lesssim b^2\int \frac{\rho_T^2|\nabla^{j_2}\nabla \Psit|^2}{\la Z\ra^{2(k+2-j_2)}}
 +b^2\int \frac{\rho_T^2|\nabla^{j_2+1}\nabla \Psit|^2}{\la Z\ra^{2(k+1-j_2)}}\le  {\mathcal d}_0+\delta J_{k_m}.
$$
The $\delta J_{k_m}$ term above controls the case $j_2=k-1$.

We now turn to the control of the nonlinear term.  If $j_1\leq \frac{4k_m}{9}$, then from \eqref{smallglobalboot}, \eqref{veniovnineonelweroidre}:
\bee
&&\int \rho_D^{p-1}|\nabla^{j_1}\rhot\nabla^{j_2}\nabla \Psi|^2\lesssim \int\rho_D
^2\frac{|\nabla^{j_2}\nabla \Psi|^2}{\la Z\ra^{2(k+1-j_2)+(p-1)\frac{2(r-1)}{p-1}}}\le {\mathcal d}_0.
\eee
If $j_2\le \frac{4k_m}{9}$, then from \eqref{smallglobalboot} with $b=\frac{1}{(Z^*)^{r-2}}$:
\bee
&&\int\rho_D^{p-1}|\nabla^{j_1}\rhot\nabla^{j_2}\nabla \Psi|^2 \lesssim \int_{Z\leq Z^*} \rho_D^{p-1}\frac{|\nabla^{j_1}\rhot|^2}{\la Z\ra^{2(k+1+(r-2)-j_1)}}+b^2\int_{Z\ge Z^*}\rho_D^{p-1}\frac{|\nabla^{j_1}\rhot|^2}{\la Z\ra^{2(k+1-j_1)}}\\
&\le&c_{k}
\eee
We may therefore assume $j_1,j_2\ge m_0=\frac{4k_m}{9}+1$, which implies $k\ge m_0$ and $j_1,j_2\leq \frac{2k_m}{3}$ and hence from \eqref{smallglobalboot} and \eqref{smallnessoutsideinitbootfinal}:
\bee
\int \rho_D^{p-1}|\pa^{j_1}\rhot\nabla^{j_2}\nabla \Psi|^2\lesssim {\mathcal d}_0+\int_{Z\ge Z^*}\frac{dZ}{\la Z\ra^{\frac{k_m}{10}}}\le {\mathcal d}_0.
\eee
The $b^2$ derivative contribution of the nonlinear term is estimated similarly.\\

\noindent{\bf step 5} $F_2$ terms. We claim: 
\be
\label{estfoneessentialftwobisbis}
\int\rho_T^2|\nabla (F_2+\Delta^{K_m}\NL(\rhot))|^2 \leq \delta  J_{k_m}+ {\mathcal d}_0.
\ee
The nonlinear term $\Delta^{K_m}\NL(\rhot)$ will be treated in the next step.

\noindent\underline{\em $ \matchal A_k(\Psi)$ term}. From \eqref{estimatepenttniialvone'}
$$
|\nabla\mathcal A_k(\Psi)|\lesssim \sum_{j=1}^{k}\frac{|\nabla^j\Psit|}{\la Z\ra^{r+k-j+1}}$$
and hence:
$$
\int\rho_T^2|\nabla\mathcal A_k(\Psi)|^2\lesssim \sum_{j=0}^{k-1}\int\rho_T^2\frac{|\nabla \nabla^j\Psit|^2}{\la Z\ra^{2(r+k-j)}}\le {\mathcal d}_0.
$$
\noindent\underline{$[\Delta^{K_m},\rho_D^{p-2}]$ term}. From \eqref{estimatecommutatorvlkeveln}:
$$\left|[\Delta^{K_m},\rho_D^{p-2}]\rhot-k(p-2)\rho_D^{p-3}\nabla\rho_D\cdot\nabla\Delta^{K_m-1}\rhot\right|\lesssim \sum_{j=0}^{k-2}\frac{|\nabla^j\rhot|}{\la Z\ra^{k-j}}\rho_D^{p-2}$$
After taking a derivative:
\bee
&&\int \rho_T^2\left|\nabla \left[[\Delta^{K_m},\rho_D^{p-2}]\rhot-k(p-2)\rho_D^{p-3}\nabla\rho_D\cdot\nabla\Delta^{K_m-1}\rhot\right]\right|^2\\
&\lesssim& \sum_{j=0}^{k-1}\int \rho_D^{2(p-2)+2}\frac{|\nabla^j\rhot|^2}{\la Z\ra^{2(k-j)+2}}\le {\mathcal d}_0.
\eee

\noindent\underline{Nonlinear $\Psi$ term}. Let $$\pa N_{j_1,j_2}=\nabla^{j_1}\nabla\Psi\nabla^{j_2}\nabla\Psi, \ \ j_1+j_2={k+1}, \ \ j_1,j_2\geq 1.$$ We first treat the highest derivative term using the $L^\infty$ smallness of small derivatives:
Using \eqref{smallglobalboot} and \eqref{smallnessoutsideinitbootfinal}
$$\int \rho_T^2|\nabla\nabla \Psi|^2|\nabla^{k_m}\nabla \Psi|^2\leq ({\mathcal d}_0+b^2) I_{k_m}.
$$ We now assume $j_1,j_2\le k_m-1$. If $j_1\leq \frac{4k_m}{9},$ then from \eqref{smallglobalboot}, \eqref{veniovnineonelweroidre}:
$$
\int \rho_T^2|\nabla N_{j_1,j_2}|^2\lesssim ({\mathcal d}_0+b^2)\int \rho_T^2\frac{|\nabla^{j_2}\nabla\Psi|^2}{\la Z\ra^{2(k+1-j_2)}}\leq {\mathcal d}_0.
$$
The expression being symmetric in $j_1,j_2$, we may assume $j_1,j_2\geq m_0=\frac{4k_m}{9}+1$, $j_1,j_2\leq \frac{2k_m}{3}$, hence $k\ge m_0=\frac{4k_m}{9}+1$ and using \eqref{smallglobalboot}, \eqref{veniovnineonelweroidre}:
$$
\int  \rho_T^2|\nabla N_{j_1,j_2}|^2\lesssim {\mathcal d}_0 \int_{Z\le Z^*}\frac{dZ}{\la Z\ra^{\frac{k_m}{10}}}+b^4\int_{Z>Z^*}\frac{dZ}{\la Z\ra^{\frac{k_m}{10}}}\leq {\mathcal d}_0.$$

\noindent\underline{Quantum pressure term}. We estimate from Leibniz and \eqref{estimatecommutatorvlkeveln}:
\bee
b^2\left|\Delta^{K_m}\left(\frac{\Delta \rho_T}{\rho_T}\right)-\frac{\Delta^{K_m+1} \rho_T}{\rho_T}+\frac{k\nabla \Delta^{K_m} \rho_T\cdot\nabla \rho_T}{\rho_T^2}\right|\lesssim_k b^2\sum_{j_1+j_2=k,j_2\geq 2}\left|\nabla^{j_1}\Delta \rho_T\pa^{j_2}\left(\frac{1}{\rho_T}\right)\right|.
\eee

We use the Faa di Bruno formula: 
$$N_{j_1,j_2}=b^2\nabla^{j_1+1}\Delta \rho_T\frac{1}{\rho_T^{j_2+1}}\sum_{m_1+2m_1+\dots+j_2m_{j_2}=j_2}\Pi_{i=0}^{j_2}(\nabla^i\rho_T)^{m_i}$$ and  $m_0+m_1+2m_2+\dots+j_2m_{j_2}=j_2$.
We decompose ${\rho_T}=\rho_D+\rhot$ in the $\Sigma$
and estimate the $\rho_D$ contribution:
\bee
&&b^4\int\rho_T^2\left\{\sum_{j_1+j_2=k,j_2\geq 2}\frac{|\nabla^{j_1+1}\Delta \rho_T|^2}{\rho_T^2\la Z\ra^{2j_2}}+\frac{|\nabla^{j_1}\Delta \rho_T|^2}{\rho_T^2\la Z\ra^{2j_2+2}}\right\}\\
& \lesssim & b^4\sum_{j_1+j_2=k,j_2\geq 2}\left[\int \frac{\rho_T^2dZ}{\la Z\ra^{2j_2+2(j_1+3)}}+\int\frac{|\nabla^{j_1+3}\rhot|^2}{\la Z\ra^{2j_2}}\right]\\
& \lesssim & b^4\left(1+\sum_{j_1=2}^k\int\frac{|\nabla \nabla^j\rhot|^2}{\la Z\ra^{2(k-j_1)+2}}\right)\le {\mathcal d}_0+\delta J_{k_m}
\eee
In the general case, we replace $(\nabla^i\rho_T)^{m_i}$ by $(\nabla^i\rhoh)^{m_i}$ where $\rhoh$ is either $\rho_D$
or $\rhot$. In both cases we will use the weaker estimates \eqref{smallglobalboot}.

 First, assume that $m_i=0$ for $i\ge \frac{4k_m}{9}+1$, then from \eqref{smallglobalboot}:
$$|N_{j_1,j_2}|\lesssim b^2|\nabla^{j_1+1}\Delta \rho_T|\frac{1}{\rho_T^{j_2+1}}\sum_{m_1+2m_1+\dots+j_2m_{j_2}=j_2}\Pi_{i=0}^{j_2}|(\nabla^i\rhoh)^{m_i}|\lesssim b^2 \frac{|\nabla^{j_1+1}\Delta \rho_T|}{\rho_T\la Z\ra^{j_2}}
$$
and the conclusion follows as above. Otherwise, there are at most two value $\frac{4k_m}{9}\le i_1\leq i_2\leq j_2$ with $m_{i_1},m_{i_2}\ne 0$ and $m_{i_1}+m_{i_2}\le 2$.
Hence from \eqref{smallglobalboot}:
\bee
\frac{1}{\rho_T^{j_2+1}}\Pi_{i=0}^{j_2}|(\nabla^i\rhot)^{m_i}|&\lesssim& \frac{1}{\rho_D^{j_2+1}}|\nabla^{i_1}\rhot|^{m_{i_1}}|\nabla^{i_2}\rhoh|^{m_{i_2}}\Pi_{0\leq i\le j_2, i\notin \{i_1,i_2\}}\left(\frac{\rho_D}{\la Z\ra^{i}}\right)^{m_i}\\
&\lesssim&  \left(\frac{|\nabla^{i_1}\rhoh|}{\rho_D}\right)^{m_{i_1}}\left(\frac{|\nabla^{i_2}\rhoh|}{\rho_D}\right)^{m_{i_2}}\frac{1}{\rho_D\la Z\ra^{j_2-(m_{i_1}i_1+m_{i_2}i_2)}}.
\eee
Assume first $i_2\ge \frac{2k_m}3+1$, then $m_{i_1}=0$, $m_{i_2}=1$ and $j_1+3\le \frac{4k_m}{9}$ from which:
\bee
\int\rho_T^2|N_{j_1,j_2}|^2&\lesssim& b^4\int\rho_T^2 |\nabla^{j_1+1}\Delta \rho_T|^2\frac{|\nabla^{i_2}\rhoh|^2}{\rho^2_D}\frac{1}{\rho^2_D\la Z\ra^{2(j_2-i_2)}}\lesssim b^4\int \frac{|\nabla^{i_2}\rhoh|^2}{\la Z\ra^{2(j_2-i_2)+2(j_1+3)}}\\
& \lesssim & b^4\int \frac{|\nabla^{i_2}\rhoh|^2}{\la Z\ra^{2(k-i_2)+6}}\le \mathcal d_0
\eee
There remains the case $\frac{4k_m}{9}+1\leq i_1\leq i_2\leq \frac{2k_m}3$ which imply $j_1+3\leq\frac{2k_m}{3}$, and we distinguish cases:\\
{\em -- case $(m_{i_1},m_{i_2})=(0,1)$}: if $j_1+3\leq \frac{4k_m}{9}$, we estimate
\bee
\int\rho_T^2|N_{j_1,j_2}|^2&\lesssim& b^4\int\rho_D^2 |\nabla^{j_1+1}\Delta \rho_T|^2\frac{|\nabla^{i_2}\rhoh|^2}{\rho^2_D}\frac{1}{\rho^2_D\la Z\ra^{2(j_2-i_2)}}\lesssim b^4\int \frac{|\nabla^{i_2}\rhoh|^2}{\la Z\ra^{2(j_2-i_2)+2(j_1+3)}}\\
& \lesssim & b^4\int \frac{|\nabla^{i_2}\rhoh|^2}{\la Z\ra^{2(k-i_2)+6}}\le {\mathcal d}_0
\eee
Otherwise,  $\frac{4k_m}{9}+1\leq j_1+3\leq \frac{2k_m}{3}$. Hence $\frac{4k_m}{9}+1\leq j_1+3\leq \frac{2k_m}{3}$, $\frac{4k_m}{9}+1\leq i_2\leq \frac{2k_m}{3}$ and we estimate from \eqref{smallglobalboot}, using $k_m$ large:
\bee
\int\rho_T^2|N_{j_1,j_2}|^2\lesssim b^4\int \frac{Z^{d-1} dZ}{\la Z\ra^{2\left(\frac{k_m}{4}+\frac{k_m}{4}\right)}}\lesssim b^4\le {\mathcal d}_0.
\eee
{\em -- case $m_{i_1}+m_{i_2}=2$}: we obtain from \eqref{smallglobalboot}, \eqref{neeneonenoe}
and $j_1+3\le \frac{2k_m}{3}$ 
\bee
\int\rho_T^2|N_{j_1,j_2}|^2&\lesssim& b^4\int\rho_D^2 |\pa^{j_1+1}\Delta \rho_T|^2\left(\frac{1}{\la Z\ra^{\frac{k_m}4}}\right)^4\lesssim b^4\int \frac{dZ}{\la Z\ra^{k_m}}\le {\mathcal d}_0.
\eee

\noindent{\bf step 6} $\NL(\rhot)$ term. We need to estimate $$\int\rho_T^2\nabla \Delta^{K_m}\NL(\rhot)\cdot\nabla \Psit_k$$ which requires an integration by part in time for the highest order term.
We expand using for the sake of simplicity that the nonlinearity is an integer:
$$\NL(\rhot)=(\rho_D+\rhot)^{p-1}-\rho_D^{p-1}-(p-1)\rho_D^{p-2}\rhot=\sum_{q=2}^{p-1}c_{q}\rhot^{q}\rho_D^{p-1-q}$$ and hence by Leibniz:
\bee
\Delta^{K_m}\NL(\rhot)&=&\sum_{q=2}^{p-1}c_{q}\rhot^{q-1}\left(\Delta^{K_m}\rhot\right) \rho_D
^{p-1-q}\\
& + & \sum_{q=2}^{p-1}\sum_{j_1+j_2=k}\sum_{\ell_1+\dots+\ell_q=j_1, \ell_1\leq \dots\ell_q\le k-1}\nabla^{\ell_1}\rhot\dots\nabla^{\ell_q}\rhot\nabla^{j_2}(\rho_D^{p-1-q})\\
\eee
Let $$N_{\ell_1,\dots.\ell_q,j_1,q}=\nabla^{\ell_1}\rhot\dots\nabla^{\ell_q}\rhot\nabla^{j_2}(\rho_D^{p-1-q}),\ \ \ell_1\le\dots\le\ell_q.$$\\

\noindent\underline{\em case $\ell_q\le k_m-2$}: we estimate 
$$|\nabla N_{\ell_1,\dots.\ell_q,j_1,q}|\lesssim |\nabla^{m_1}\rhot\dots\nabla^{m_q}\rhot|\frac{\rho_D^{p-1-q}}{\la Z\ra^{j_2}}, \ \ \left|\begin{array}{ll} 0\leq m_i\leq  k_m-1\\ m_1+\dots m_q=j_1+1.\end{array}\right.$$ We may reorder $m_1\leq \dots \leq m_q$. If $m_q\leq \frac{4k_m}{9}$, then:
\bee
|\nabla N_{\ell_1,\dots.\ell_q,j_1,q}|\lesssim \frac{\rhot^{q}}{\la Z\ra^{j_1+1}}\frac{\rho_D^{p-1-q}}{\la Z\ra^{j_2}}\lesssim \frac{{\mathcal d}_0}{\la Z\ra^{\frac{k_m}{2}}}
\eee
and hence the contribution of this term $$\int \rho_T^2|\nabla N_{\ell_1,\dots.\ell_q,j_1,q}|^2\leq {\mathcal d}_0.$$ If $\frac{4k_m}{9}\leq m_q \leq \frac{2k_m}{3}$, then similarly, combining \eqref{smallnessoutsideinitbootfinal-ext}, \eqref{smallnessoutsideinitbootfinal}:
$$|\nabla N_{\ell_1,\dots.\ell_q,j_1,q}|\lesssim \frac{1}{\la Z\ra^{j_2}}\frac{{\mathcal d}_0}{\la Z\ra^{\frac{k_m}{4}}}$$ and the conclusion follows. If $m_q\geq \frac{2k_m}3$, then $m_{q-1}\leq \frac{4k_m}{9}$ from which:
\bee
|\nabla N_{\ell_1,\dots.\ell_q,j_1,q}|\lesssim\frac{\rho_D^{p-1-q}}{\la Z\ra^{j_2}}\frac{\rho_D^{q-1}}{\la Z\ra^{j_1+1-\ell_q}}|\nabla^{\ell_q}\rhot|\lesssim \frac{\rho_D^{p-2}|\nabla^{\ell_q}\rhot|}{\la Z\ra^{k_m+1-m_q}}
\eee
and hence the bound
$$\int \rho_T^2|\nabla N_{\ell_1,\dots.\ell_q,j_1,q}|^2\lesssim \int \rho_T^{2(p-2)+2}\frac{|\nabla^{m_q}\rhot|^2}{\la Z\ra^{2(k_m-m_q)+2}}\lesssim \|\rhot,\Psit\|_{k_m-1,\sigma(k_m-1)}^2\le {\mathcal d}_0.$$

\noindent\underline{\em case $\ell_q= k_m-1$}: we compute $\nabla N_{\ell_1,\dots \ell_q,j_1,q}$. If the derivative falls on $\ell_j$, $j\leq q-1$, we are back to the previous case, and we are therefore left with estimating
$$|\pa^{\ell_1}\rhot\pa^{\ell_{q-1}}\rhot \pa^{k_m}\rhot|\frac{\rho_D^{p-1-q}}{\la Z\ra^{j_2}}, \ \ \left|\begin{array}{ll} \ell_1+\dots+\ell_{q-1}+k_m=j_1+1\\ j_1+j_2=k_m.\end{array}\right.$$
If $j_1=k_m-1$, then $j_2=1$, $\ell_1=\dots=\ell_{q-1}=0$ and we estimate relying onto the
smallness of $\frac{\rhot}{\rho_T}$ from \eqref{veniovnineonelweroidre} (for $Z\le Z^*$) and using \eqref{smallglobalboot}
together with the smallness of $\la Z\ra^{-1}$ (for $Z\ge Z^*$):
\bee
|\nabla^{\ell_1}\rhot...\nabla^{\ell_{q-1}}\rhot \nabla^{k_m}\rhot|\frac{\rho_D^{p-1-q}}{\la Z\ra^{j_2}}\lesssim |\rhot^{q-1} \nabla^{k_m}\rhot|\frac{\rho_D^{p-1-q}}{\la Z\ra}\lesssim {\mathcal d}_0 \rho_D^{p-{2}}|\nabla^{k_m}\rhot|
\eee
and hence the corresponding contribution ($p\ge 3$)
$${\mathcal d}_0 \int \rho_D^{2(p-2)}|\pa^{k_m}\rhot|^2\le \delta J_{k_m}.$$
Similarly, if $j_1=k_m$ then $\ell_1=...=\ell_{q-2}=j_2=0$ and $\ell_{q-1}=1$. 
\bee
|\nabla^{\ell_1}\rhot...\nabla^{\ell_{q-1}}\rhot \nabla^{k_m}\rhot|\frac{\rho_D^{p-1-q}}{\la Z\ra^{j_2}}\lesssim \rhot^{q-2}
|\nabla \rhot| |\nabla^{k_m}\rhot| \rho_D^{p-1-q}\lesssim {\mathcal d}_0 \rho_D^{p-2}|\nabla^{k_m}\rhot|
\eee

\noindent\underline{\em Highest order term} We are left with estimating the highest order term:
$$N_{\ell_1,\dots.\ell_q,j_1,q}=\rhot^{q-1}\rho_D^{p-1-q}\Delta^{K_m}\rhot.$$
We treat this term by integration by parts in time using \eqref{estqthohrkbisbis}:
\bea
\label{vneneknenenpe}
\nonumber &&-\int\rho_T^2\nabla \left[\rhot^{q-1}\rho_D^{p-1-q}\Delta^{K_m}\rhot\right]\cdot\nabla \Psit_k=\int \rhot^{q-1}\rho_D^{p-1-q}\rhot_k\nabla \cdot(\rho_T^2\nabla \Psit_k)\\
\nonumber & = & -\int \rhot^{q-1}\rho_D^{p-1-q}\rhot_k\rho_T\left[\pa_\tau \rhot_k-(H_1-k(H_2+\Lambda H_2))\rhot_k+H_2\Lambda \rhot_k+(\Delta^{K_m}\rho_T)\Delta \Psit\right.\\
&+& \left. k\nabla\rho_T\cdot\nabla\Psit_k+2\nabla(\Delta^{K_m}\rho_T)\cdot\nabla \Psit-  F_1\right]
\eea
and we treat all terms in \eqref{vneneknenenpe}. We will systematically use the smallness \eqref{nonvavnishig}.
The $\pa_\tau\rhot_k$ term is integrated by parts in time:
\bee
&&-\int \rhot^{q-1}\rho_D^{p-1-q}\rho_T\rhot_k\pa_\tau \rhot_k=-\frac12\frac{d}{d\tau}\left\{\int \rhot^{q-1}\rho_D^{p-1-q}\rho_T\rhot^2_k\right\}+ \frac12\int  \rhot^2_k\pa_\tau\left(\rhot^{q-1}\rho_D^{p-1-q}\rho_T\right)\\
& = & -\frac12\frac{d}{d\tau}\left\{\int \rhot^{q-1}\rho_D^{p-1-q}\rho_T\rhot^2_k\right\}+O(\delta\int \rho_D^{p-1}\rhot_k^2)
\eee
and the boundary term in time is small  $$\int \rhot^{q-1}\rho_D^{p-1-q}\rho_T\rhot^2_k\lesssim \delta\int \rho_D^{p-1}\rho_k^2.$$
We then estimate:
$$
\left|\int \rhot^{q-1}\rho_D^{p-1-q}\rhot_k\rho_T(H_1-k(H_2+\Lambda H_2))\rhot_k\right|\lesssim k\delta \int \rho_T^{p-1}\rhot_k^2\lesssim {\delta}J_{k_m}.$$ Using the extra decay in $Z$ and $\|\Delta\Psi\|_{L^\infty}\le \delta\ll 1$:
$$
\left|\int \rhot^{q-1}\rho_D^{p-1-q}\rhot_k\rho_T(\Delta^{K_m}\rho_T)\Delta \Psit\right|\lesssim \int \frac{dZ}{\la Z\ra^{\frac{k_m}{2}}}+\int\rho_T^{p-1}\rhot^2_k|\Delta \Psit|\le {\mathcal d}_0+\delta J_{k_m}.$$
Similarly, after an integration by parts:
\bee
&&\left|-\int \rhot^{q-1}\rho_D^{p-1-q}\rhot_k\rho_T\nabla(\Delta^{K_m}\rho_T)\cdot\nabla \Psit\right|\lesssim {\mathcal d}_0+\left|\int  \rhot^{q-1}\rho_D^{p-1-q}\rho_T\nabla (\rhot_k^2)\cdot\nabla \Psit\right|\\
& \le & {\mathcal d}_0+\delta J_{k_m}.
\eee
Similarly, after an integration by parts using \eqref{smallglobalboot}:
$$\left|-\int \rhot^{q-1}\rho_D^{p-1-q}\rhot_k\rho_TH_2\Lambda \rhot_k\right|\lesssim {\mathcal d}_0+\left(\left\|\frac{\rhot}{\rho_T}\right\|_{L^\infty}+\left\|\frac{ |\nabla\rhot|}{\rho_T}\right\|_{L^\infty}\right)J_{k_m}\le {\mathcal d}_0+\delta J_{k_m}
$$
and similarly:
$$\left|\int \rhot^{q-1}\rho_D^{p-1-q}\rhot_kk\nabla\rho_T\cdot\nabla \Psi_k\right|\lesssim {\mathcal d}_0+\delta J_{k_m}.$$

\noindent{\bf step 7} Conclusion for $k=k_m(d,r)$ large enough. We now sum the collection of above bounds and obtain the differential inequality with $k=k_m$.
\bee
&&\frac 12\frac{d}{d\tau}\left\{J_{k_m}(1+O(\delta))\right\}\\
&\leq& -k\left[1+O\left(\frac{1}{k}\right)\right]\int( H_2+\Lambda H_2)\left[ b^2|\nabla\rhot_k|^2+(p-1) \rho_D^{p-2}\rho_T\rhot_{k}^2+\rho_T^2|\nabla \Psit_k|^2\right]\\
& - & k\int (p-1)\rho_D\pa_Z (\rho_D^{p-1})\rhot_k\pa_Z \Psit_k+{\mathcal d}.
\eee
We recall from \eqref{coercivityquadrcouplinginside}, \eqref{P}:
\be
\label{vnioevineneovklvknlve}
H_2+\Lambda H_2=\mu(1-w-\Lambda w)\ge c_{d,p}>0
\ee 
and  we now claim the pointwise coercivity of the coupled quadratic form: $\exists c_{d,p}>0$ such that $\forall Z\ge 0$,
\bea
\label{enineinevnveo}
\nonumber 
&&(H_2+\Lambda H_2)\left[ (p-1)\rho_D^{p-2}\rho_T\rhot_{k}^2+\rho_T^2|\nabla \Psit_k|^2\right]+ (p-1)\rho_D\pa_Z (\rho_D^{p-1})\rhot_k\pa_Z \Psit_k\\
& \ge & c_{d,p}\left[ (p-1) \rho_D^{p-2}\rho_T\rhot_{k}^2+\rho_T^2|\nabla \Psit_k|^2\right]
\eea
which, after taking $k>k^*(d,p)$ large enough, concludes the proof of \eqref{estnienonneo}.\\
\noindent{\em Proof of \eqref{enineinevnveo}}. The coupling term is lower order for $Z$ large: $$|(p-1)\rho_D\pa_Z (\rho_D^{p-1})\rhot_k\pa_Z \Psit_k|\lesssim \frac{\rho_T^{p-1}}{\la Z\ra}\rhot_k\rho_T\pa_Z\Psi_k\le \delta\left[ (p-1)\rho_D^{p-2}\rho_T\rhot_{k}^2+\rho_T^2|\nabla \Psit_k|^2\right]
$$ for $Z>Z(\delta)$ large enough. On a compact set using the smallness \eqref{nonvavnishig}, \eqref{enineinevnveo}
 is implied by: 
\bea
\label{enineinevnveobis}
\nonumber 
&&(H_2+\Lambda H_2)\left[ (p-1)Q\rhot_{k}^2+\rho_P^2|\nabla \Psit_k|^2\right]+ (p-1)\rho_P\pa_Z Q\rhot_k\pa_Z \Psit_k\\
& \ge & c_{d,p}\left[ (p-1)Q\rhot_{k}^2+\rho_P^2|\nabla \Psit_k|^2\right]
\eea
We compute the discriminant:
\bee
&&{\rm Discr}=(p-1)^2\rho_P^2(\pa_ZQ)^2-4\mu^2(p-1)\rho_P^2Q(H_2+\Lambda H_2)^2\\
& = & (p-1)\rho_P^2 Q\left[(p-1)\frac{(\pa_ZQ)^2}{Q}-4\mu^2(1-w-\Lambda w)^2\right]
\eee
We compute from \eqref{relationsprofileemden} recalling \eqref{definitionF}:
\bee
(p-1)\frac{(\pa_ZQ)^2}{Q} &=&   (p-1)\left(2\pa_Z\sqrt{Q}\right)^2=(p-1)\left(\frac{1-\mathcal e}{2}\sqrt{\ell}\pa_Z(\sigma_P Z)\right)^2=(1-\mathcal e)^2(\pa_Z(Z\sigma_P))^2\\
& =& \frac{4}{r^2}(\pa_Z(Z\sigma_P))^2=4\mu^2F^2
\eee
and hence from \eqref{coercivityquadrcouplinginside}, \eqref{P} the lower bound:
$$-D=4\mu^2(p-1)\rho_P^2 Q\left[(1-w-\Lambda w)^2-F^2\right]\ge c_{d,r}(p-1)\rho_P^2Q, \ \ c_{d,r}>0$$
which together with \eqref{vnioevineneovklvknlve} concludes the proof of \eqref{enineinevnveo}.
\end{proof}


\section{Control of low Sobolev norms and proof of Theorem \ref{thmmain}}
\label{low}

Our aim in this section is to control weighted low Sobolev norms in the interior $r\le 1$ ($Z\le Z^*$). On our way we will conclude the proof of the bootstrap Proposition \ref{propboot}. Theorem \ref{thmmain} will then follow from a classical topological argument.


\subsection{Exponential decay slightly beyond the light cone}


We use the exponential decay estimate \eqref{stabiliteexpobis} for a linear problem to prove exponential decay 
for the nonlinear evolution in the region slightly past the light cone.
We recall the notations of Section 3, in particular $Z_a$ of Lemma \ref{shiftoight}.
\begin{lemma}[Exponential decay slightly past the light cone]
\label{lemmalightcone}
Let $$\tilde{Z_a}=\frac{Z_2+Z_a}2.$$
\be
\label{firstbound}
 \|\nabla \Phi\|_{H^{2k_0}(Z\leq\tilde{Z_a})}+\|\rho\|_{H^{2k_0}(Z\leq \tilde{Z_a})}\lesssim e^{-\frac{\delta_g}{2}\tau}.
  \ee
\end{lemma}

\begin{proof} The proof relies on the spectral theory beyond the light cone and an elementary finite speed propagation like 
argument in renormalized variables, related to \cite{MZwave}.\\

\noindent{\bf step 1} Semigroup decay in $X$ variables. Recall the definition \eqref{notatinotphi} of $X=(\Phi,T)$ 
 \be\label{formulat}
  \left|\begin{array}{ll}
  \Phi=\rho_P\Psi \\
   T=\pa_\tau\Phi+aH_2\Lambda \Phi=-(p-1)\qx\rhox-H_2\Lambda \Phix+(H_1-\mathcal e)\Phix+G_\Phi+aH_2\Lambda \Phi
   \end{array}\right.
  \ee
with $G_\Phi$ given by \eqref{defgrho}, the scalar product \eqref{defscalarproduct} and the definitions
\eqref{nvknneknengno}, \eqref{eq:decomp}: $$\left|\begin{array}{l}
\Lambda_0=\{\l \in \Bbb C, \ \ \Re(\l)\ge 0\} \cap \{\l\ \ \mbox{is an eigenvalue of}\ \ \mathcal M\}=(\l_i)_{1\le i\le N}\\
V=\cup_{1\leq i\leq N}\mbox{\rm ker} (\mathcal M-\l_i I)^{k_{\l_i}}
\end{array}\right.
$$
the projection $P$ associated with $V$, the decay estimate \eqref{stabiliteexpobis} on the range of $(I-P)$ and 
the results of Lemma \ref{browerset}. Relative to the $X$ variables our equations take the form 
$$
\pa_\tau X=\mathcal M X + G,
$$ 
which are considered on the time interval $\tau\ge \tau_0\gg 1$ and the space interval $Z\in [0,Z_a]$ (no boundary conditions
at $Z_a$.) We consider evolution in the Hilbert space ${\Bbb H_{2k_0}}$ with initial data such that 
\be\label{data}
\|(I-P) X(\tau_0)\|_{\Bbb H_{2k_0}}\le e^{-\frac{\delta_g}2\tau_0},\qquad \|P X(\tau_0)\|_{\Bbb H_{2k_0}}\le e^{-\frac {3\delta_g}5\tau_0}.
\ee
According to the bootstrap assumption \eqref{eq:unstX}
\be\label{eq:unstX'}
\|PX(\tau)\|_{\Bbb H_{2k_0}}\le e^{-\frac {\delta_g}{2}\tau}, \qquad \forall\tau\in [\tau_0,\tau^*]
\ee
Lemma \ref{browerset} shows that as long as 
\be\label{eq:G}
\|G\|_{\Bbb H_{2k_0}}\le e^{-\frac{2\delta_g}3\tau}, \qquad \tau\ge \tau_0
\ee
there exists $\Gamma$, which can
 be made as large as we want with a choice of $\tau_0$, such that 
  \be\label{eq:Xfgrow}
 \|PX(\tau)\|_{\Bbb H_{2k_0}}\lesssim e^{-\frac{\delta_g}{2}\tau},\qquad \tau_0\le \tau\le \tau_0+\Gamma.
 \ee
 This will allow us to show eventually that if we can verify \eqref{eq:G}, the bootstrap time $\tau^*\ge \tau_0+\Gamma$.
 
Moreover,  as long as \eqref{eq:G} holds, the decay estimate \eqref{stabiliteexpobis} implies that
\bea
 \label{kmdommepeo}
 \nonumber &&\|(I-P)X(\tau)\|_{\Bbb H_{2k_0}}\lesssim e^{-\frac{\delta_g}{2}(\tau -\tau_0)}\|X(\tau_0)\|_{\Bbb H_{2k_0}}+\int_{\tau_0}^{\tau}e^{-\frac{\delta_g}{2}(\tau-\sigma)}\|G(\sigma)\|_{\Bbb H_{2k_0}}d\sigma\\
 & \lesssim & e^{\frac{-\delta_g}{2}\tau}\left[e^{\frac{\delta_g}{2}\tau_0}\|X(\tau_0)\|_{\Bbb H_{2k_0}}+\int_{\tau_0}^{+\infty}e^{-\frac{\delta_g}{6}\tau}d\tau\right]\le e^{-\frac{\delta_g}2\tau}.
 \eea
 As a result,
 \be\label{eq:Xf}
 \|X(\tau)\|_{\Bbb H_{2k_0}}\lesssim e^{-\frac{\delta_g}{2}\tau},\qquad \tau_0\le \tau\le \tau^*
 \ee
 Below we will verify \eqref{eq:G} $\forall \tau\in[\tau_0,\tau^*]$ under the assumption \eqref{kmdommepeo}, closing
 both. Once again,  this will allow us to show eventually that the length of the bootstrap interval $\tau^*-\tau_0\ge \Gamma$ is sufficiently large.\\ 
 
 Recall from \eqref{defgt}, \eqref{newlinearflow}, \eqref{defscalarproduct}:
 \be
 \label{vnebonennnevno}
 \|G\|^2_{\Bbb H_{2k_0}}\lesssim \int_{Z\le Z_a} |\nabla\Delta^{k_0}G_T|^2gZ^{d-1}dZ+\int_{Z\le Z_a} G_T^2Z^{d-1}dZ
 \ee
 with
 $$\left|\begin{array}{lll}
G_T=\pa_\tau G_\Phi-\left(H_1+H_2\frac{\Lambda Q}{Q}\right)G_\Phi+H_2\Lambda G_\Phi-(p-1)QG_\rho\\
G_\rho=-\rho\Delta \Psi-2\nabla\rho\cdot\nabla \Psix\\
G_\Phi=-\rho_P(|\nabla \Psi|^2+\NL(\rho))+\frac{b^2\rho_P}{\rho_T}\Delta \rho_T.
\end{array}\right.
$$
\noindent{\bf step 2} Semigroup decay for $(\rho,\Psi)$.  
 We now translate the $X$ bound to the bounds for $\rho$ and $\Psi$ and then verify \eqref{eq:G}.
 We recall \eqref{formulat} and obtain for any $\hat Z>Z_2$
  \bee
  \|T\|_{H^{2k_0}(Z\le \hat{Z})}+\|\Phi\|_{H^{2k_0+1}(Z\le \hat{Z})}&\lesssim&\|\rho\|_{H^{2k_0}(Z\le \hat{Z})}+\|\Psi\|_{H^{2k_0+1}(Z\le \hat{Z})} +\|G_\Phi\|_{H^{2k_0}(Z\le \hat{Z})}\\&\lesssim& \|T\|_{H^{2k_0}(Z\le \hat{Z})}+\|\Phi\|_{H^{2k_0+1}(Z\le \hat{Z})}+\|G_\Phi\|_{H^{2k_0}(Z\le \hat{Z})}
 \eee
  and claim:
\be
\label{boundgphi}
\|G_{\Phi}\|_{H^{2k_0}(Z\le \hat{Z})}\lesssim \|\nabla \Psi\|_{H^{k_0}(Z\le \hat{Z})}^2+\|\rho\|_{H^{k_0}(Z\le \hat{Z})}^2+e^{-{\delta_g}\tau}.
\ee
Indeed, since $H^{2k_0}(Z\le \hat{Z})$ is an algebra for $k_0$ large enough:
  $$\|\rho_P(|\nabla \Psi|^2+\NL(\rho))\|_{H^{2k_0}(Z\le \hat{Z})}\lesssim \|\nabla \Psi\|_{H^{k_0}(Z\le \hat{Z})}^2+\|\rho\|_{H^{k_0}(Z\le \hat{Z})}^2.$$
    The remaining quantum pressure term is treated using the pointwise bound \eqref{smallglobalboot} for small Sobolev norms and the smallness of $b$ which imply: $$\|\frac{b^2\rho_P\Delta\rho_T}{\rho_T}\|_{H^{2k_0}(Z\leq \hat{Z})}\lesssim C_Kb^2\leq  e^{-\delta_g\tau}$$
  provided $\delta_g>0$ has been chosen small enough, and \eqref{boundgphi} is proved. 
  Choosing $\hat Z>Z_2$, this implies from \eqref{formulat} and the initial bound \eqref{improvedsobolevlowinit}:
  \bea
  \label{boundonthedata}
  \nonumber \|X(\tau_0)\|_{\Bbb H^{2k_0}}&\lesssim& \lesssim \|\Psi(\tau_0)\|_{H^{2k_0+1}(Z\le \hat{Z})}+\|\rho(\tau_0)\|_{H^{2k_0}(Z\le \hat{Z})}+e^{-\delta_g\tau_0}\\
  &\lesssim & e^{-\frac{\delta_g\tau_0}2}.
  \eea
  This verifies \eqref{data}. On the other hand, choosing $\hat Z=\tilde Z_a$ with 
$$\tilde{Z_a}=\frac{Z_2+Z_a}2,$$ we also obtain from \eqref{eq:Xf}
  \be\label{eq:rhopsi}
  \|\Psi(\tau)\|_{H^{2k_0+1}(Z\le \tilde{Z}_a)}+\|\rho(\tau)\|_{H^{2k_0}(Z\le \tilde{Z}_a)}
  \lesssim \|X(\tau)\|_{\Bbb H^{2k_0}}+ e^{-\delta_g\tau}\lesssim  e^{-\frac{\delta_g\tau}2}.
  \ee
 The estimate \eqref{firstbound} follows.\\
  
 \noindent{\bf step 3} Estimate for $G$. Proof of \eqref{eq:G}. We recall \eqref{vnebonennnevno}.
 On a fixed compact domain $Z\le Z_0$ with $Z_0>Z_2$,
 we can interpolate the bootstrap bound \eqref{eq:bootdecay} with the global large Sobolev bound \eqref{sobolevinitboot} and obtain for $k_m$ large enough and $b_0<b_0(k_m)$ small enough:
 \be
 \label{cneoineonenvoen}
 \|\rho\|_{H^{2k_0+10}(Z\le Z_0)}+\|\Psi\|_{H^{2k_0+10}(Z\le Z_0)}\leq C_Ke^{-\left[\frac{3}{8}-\frac{1}{100}\right]\delta_g\tau}\le e^{-\left[\frac{3}{8}-\frac{1}{50}\right]\delta_g\tau}
 \ee
 and since $H^{2k_0}$ is an algebra and all terms are either quadratic or with a $b$ term, \eqref{cneoineonenvoen} implies
\bea
 \label{estimatheG}
 \nonumber 
 &&\|G_T\|_{H^{2k_0+5}(Z\leq Z_0)}+\|G_\rho\|_{H^{2k_0+5}(Z\leq Z_0)}+\|G_\Phi\|_{H^{2k_0+5}(Z\leq Z_0)}\\
 &\leq &e^{-\left(\frac 34-\frac1{20}\right)\delta_g\tau}\le e^{-\frac{2\delta_g}{3}\tau}
 \eea
 which in particular using \eqref{vnebonennnevno} implies \eqref{eq:G}.
\end{proof}


\subsection{Weighted decay for $m\le 2k_0$ derivatives}


We recall the notation \eqref{defnewvariablephi}. We now transform the exponential decay \eqref{firstbound} from just past the light cone into weighted decay estimate. It is {\em essential} for this argument that the decay \eqref{firstbound} has been 
shown in the region strictly including the light cone $Z=Z_2$. The estimates in the lemma below close the remaining bootstrap 
bound \eqref{eq:bootdecay}.

\begin{lemma}[Weighted Sobolev bound for $m\leq 2k_0$]
\label{lemmabootlocalnorms}
 Let $m\leq 2k_0$ and $\nu_0=\frac{\delta_g}{2\mu}-\frac {2(r-1)}{p-1}$, recall 
 $$
 \chi_{\nu_0,m}=\frac{1}{\la Z\ra^{d-2(r-1)+2(\nu_0-m)}}\zeta\left(\frac{Z}{Z^*}\right), \ \ \zeta(Z)=\left|\begin{array}{ll}1\ \ \mbox{for}\ \ Z\leq 2\\ 0\ \ \mbox{for}\ \ Z\ge 3,
 \end{array}\right.
 $$
 then:
 \be
 \label{improvedsobolevlowbetter}
 \sum_{m=0}^{2k_0}\sum_{i=1}^d\int (p-1)Q(\pa^m_i\rho)^2\chi_{\nu_0,m}+|\nabla \pa^m_i\Phi|^2\chi_{\nu_0,m}\leq C e^{-\frac{4\delta_g}{5}\tau}.
 \ee
 \end{lemma}

\begin{proof}[Proof of Lemma \ref{lemmabootlocalnorms}]  The proof relies on a sharp energy estimate with time dependent localization of $(\rho,\Phi)$. This is  a renormalized version of the finite speed of propagation.\\

\noindent{\bf step 1} $\dot{H}^{m}$ localized energy identity. Pick a smooth well localized radially symmetric function $\chi(\tau,Z)$ and a coordinate $1\leq i\leq d$ and note for $m$ integer
$$\rho_m=\pa^m_i\rho, \ \ \Phi_m=\pa^m_i\Phi,$$ where we omit the $i$ dependence to simplify notations.
We recall the Emden transform formulas \eqref{defhtwohunbis}:
\be
\label{autroformluehonhtwo}
\left|\begin{array}{ll}
H_2=\mu(1-w)\\
H_1=\frac{\mu\ell}{2}(1-w)\left[1+\frac{\Lambda \sigma}{\sigma}\right]\\
H_3=\frac{\Delta \rho_P}{\rho_P}
\end{array}\right.
\ee
which yield the bounds using \eqref{limitprofilesbsi}, \eqref{decayprofile}:
 \be
\label{esterrorpotentials}
\left|\begin{array}{llll}
H_2=\mu+O\left(\frac{1}{\la Z\ra^r}\right), \ \ H_1=-\frac{2\mu(r-1)}{p-1}+O\left(\frac{1}{\la Z\ra^r}\right)\\
 |\la Z\ra^j\pa_Z^j H_1|+||\la Z\ra^j\pa_Z^jH_2|\lesssim \frac 1{\la Z\ra^{r}}, \ \ j\ge 1\\
 |\la Z\ra^j\pa_Z^j H_3|\lesssim \frac{1}{\la Z\ra^2}\\
  \frac{1}{\la Z\ra^{2(r-1)}}\left[1+O\left(\frac{1}{\la Z\ra^{r}}\right)\right]\lesssim_j |\la Z\ra^j\pa_Z^jQ|\lesssim_j \frac{1}{\la Z\ra^{2(r-1)}}
 \end{array}\right.
 \ee
and the commutator bounds:
\be
\label{esterrorpotentialsbis}
\left|\begin{array}{lllll}
|[\pa_i^m,H_1]\rho|\lesssim  \sum_{j=0}^{m-1}\frac{|\pa_Z^j\rho|}{\la Z\ra^{r+m-j}}\\
|\nabla\left([\pa_i^m,H_1]\rho\right)|\lesssim \sum_{j=0}^{m}\frac{|\pa_Z^j\rho|}{\la Z\ra^{m-j+r+1}}\\
|[\pa_i^m,Q]\rho|\lesssim Q\sum_{j=0}^{m-1}\frac{|\pa_Z^j\rho|}{\la Z\ra^{m-j}}\\
|[\pa_i^m,H_2]\Lambda\rho|\lesssim \sum_{j=1}^{m}\frac{|\pa_Z^j\rho|}{\la Z\ra^{r+m-j}}\\
|\nabla\left([\pa_i^m,H_2]\Lambda\Phi\right)|\lesssim \sum_{j=1}^{m+1}\frac{|\pa^j_Z\Phi|}{\la Z\ra^{r+1+m-j}}.
\end{array}\right.
\ee
Commuting \eqref{nekoneneon} with $\pa_i^m$:
$$
\left|\begin{array}{ll}
\pa_\tau\rho_m=H_1\rho_m-H_2(m+\Lambda)\rho_m-\Delta \Phi_m+\pa_i^mG_\rho+E_{m,\rho}\\
\pa_\tau\Phi_m=-(p-1)Q\rho_m-H_2(m+\Lambda) \Phi_m+(H_1-\mathcal e)\Phi_m+\pa^m_iG_\Phi+E_{m,\Phi}
\end{array}\right.
$$
with the bounds $$\left|\begin{array}{ll} |E_{m,\rho}|\lesssim \sum_{j=0}^{m}\frac{|\pa_Z^j\rho|}{\la Z\ra^{r-1+m-j}}+\sum_{j=0}^{m}\frac{|\pa_Z^j\Phi|}{\la Z\ra^{m-j+2}}\\  |\nabla E_{m,\Phi}|\lesssim Q\sum_{j=0}^{m}\frac{|\pa_Z^j\rho|}{\la Z\ra^{m+1-j}}+ \sum_{j=0}^{m+1}\frac{|\pa^j_Z\Phi|}{\la Z\ra^{r+m-j}}.
\end{array}\right.$$

Let $\chi$ be an arbitrary smooth function. We derive  the corresponding energy identity:
\bee
&&\frac 12\frac{d}{d\tau}\left\{\int (p-1)Q\rho_m^2\chi+|\nabla \Phi_m|^2\chi\right\}=\frac 12\int \pa_\tau\chi\left[(p-1)Q\rho_m^2+|\nabla \Phi_m|^2\right]\\
& +& \int(p-1)Q\rho_m\chi\left[H_1\rho_m-H_2(m+\Lambda)\rho_m-\Delta \Phi_m+\pa_i^mG_\rho+E_{m,\rho}\right]\\
& + & \int \chi\nabla \Phi_m\cdot\nabla\left[-(p-1)Q\rho_m-H_2(m+\Lambda \Phi_m)+(H_1-\mu(r-2))\Phi_m+\pa^i_mG_\Phi+E_{m,\Phi}\right]\\
& = & \frac 12\int \pa_\tau\chi\left[(p-1)Q\rho_m^2+|\nabla \Phi_m|^2\right]\\
& + & \int(p-1)Q\rho_m\chi\left[H_1\rho_m-H_2(m+\Lambda)\rho_m+\pa_i^mG_\rho+E_{m,\rho}\right]+\int(p-1)Q\rho_m\nabla\chi\cdot\nabla \Phi_m\\
&+& \int \chi\nabla \Phi_m\cdot\nabla\left[-H_2(m+\Lambda) \Phi_m+(H_1-\mu(r-2))\Phi_m+\pa^m_iG_\Phi+E_{m,\Phi}\right].\\
\eee
In what follows we will use $\omega>0$ as a small universal constant to denote the power of tails of the error terms. 
In most cases, the power is in fact $r>2$ which we do not need.\\
\noindent\underline{$\rho_m$ terms}. From the asymptotic behavior of $Q$ \eqref{decayprofile} and \eqref{esterrorpotentials}:
\bee
&&-\int(p-1)Q\rho_m\chi H_2\Lambda \rho_m=\frac{p-1}2\int \rho_m^2 \chi QH_2\left[d+\frac{\Lambda Q}{Q}+\frac{\Lambda H_2}{H_2}+\frac{\Lambda \chi}{\chi}\right]\\
&= & \int \rho_m^2 (p-1)\chi Q\mu\left[\frac d2-(r-1)+O\left(\frac{1}{\la Z\ra^\omega}\right)\right]+\frac 12\int (p-1)QH_2\Lambda \chi \rho_m^2
\eee
\noindent\underline{$\Phi_m$ terms}. We first estimate recalling \eqref{esterrorpotentials}:
\bee
&&\int \chi\nabla \Phi_m\cdot\nabla\left[(-mH_2+H_1-\mu(r-2))\Phi_m\right]\\
&=&\int(-mH_2+H_1-\mu(r-2))\chi|\nabla\Phi_m|^2+O\left(\int \frac{\chi}{\la Z\ra^{r}}|\nabla\Phi_m||\Phi_m|\right)\\
& = & -\left[\mu(m+r-2)+\frac{2\mu(r-1)}{p-1}\right]\int \chi|\nabla\Phi_m|^2+O\left(\int \frac{\chi}{\la Z\ra^\omega}\left[|\nabla \Phi_m|^2+\frac{\Phi_m^2}{\la Z\ra^2}\right]\right)
\eee
From Pohozhaev identity \eqref{pohozaevbispouet}
with $F=\chi H_2(Z_1,\dots,Z_d)$:
\bee
&&-\int \chi\nabla \Phi_m\cdot\nabla (H_2\Lambda \Phi_m)=\int H_2\Lambda \Phi_m[ \chi \Delta \Phi_m+\nabla \chi\cdot\nabla \Phi_m]\\
& = & -\sum_{i,j=1}^d \int\pa_iF_j\pa_i\Phi_m\pa_j\Phi_m+\frac 12\int |\nabla \Phi_m|^2\nabla \cdot F+\int H_2\Lambda \Phi_m\nabla \chi\cdot\nabla \Phi_m\\
& = & \sum_{i,j=1}^d\pa_i\Phi_m\pa_j\Phi_m\left[-\pa_i(\chi H_2 Z_j)+H_2Z_j\pa_i\chi\right]+\frac 12\int |\nabla \Phi_m|^2\chi H_2\left[d+\frac{\Lambda\chi}{\chi}+\frac{\Lambda H_2}{H_2}\right]\\
& = &\frac{\mu(d-2)}{2}\int\chi|\nabla \Phi_m|^2+\frac 12 \int H_2\Lambda \chi|\nabla \Phi_m|^2+O\left(\int \frac{\chi}{\la Z\ra^\omega}|\nabla \Phi_m|^2\right)
\eee

The collection of above bounds yields for some universal constant $\omega>0$ the weighted energy identity:
\bea
\label{energyidentittyk}
&&\frac 12\frac{d}{d\tau}\left\{\int (p-1)Q\rho_m^2\chi+|\nabla \Phi_m|^2\chi\right\}\\
\nonumber & = &-\int\chi\left[(p-1)Q\rho_m^2+|\nabla\Phi_m|^2\right]\left[\mu(m-\frac d2+r-1)+\frac{2\mu(r-1)}{p-1}+O\left(\frac{1}{\la Z\ra^\omega}\right)\right] \\
\nonumber& + & \frac 12\int (p-1)Q\rho_m^2\left[\pa_\tau \chi+H_2\Lambda \chi\right]+\frac 12\int |\nabla \Phi_m|^2\left[\pa_\tau \chi+H_2\Lambda \chi\right]+\int(p-1)Q\rho_m\nabla\chi\cdot\nabla \Phi_m\\
\nonumber& + & O\left(\int \chi\left[\sum_{j=0}^{m+1}\frac{|\pa_Z^j\Phi|^2}{\la Z\ra^{2(m+1-j)+\omega}}+\sum_{j=0}^{m}\frac{Q|\pa_Z^j\rho|^2}{\la Z\ra^{2(m-j)+\omega}}\right]\right)\\
\nonumber&+ & O\left(\int \chi|\nabla \Phi_m||\nabla \pa^mG_\Phi|+\int\chi Q|\rho_m||\pa^mG_\rho|\right)
\eea

\noindent{\bf step 2} Nonlinear and source terms. We claim the bound for $\chi=\chi_{\nu_0,m}$:
\bea
\label{cenoeneonoenevne}
\nonumber &&\sum_{m=0}^{2k_0} \sum_{i=1}^d\int \chi_{\nu_0,m}|\nabla \pa^mG_\Phi|^2+\int (p-1)Q\chi_{\nu_0,m}|\pa^mG_\rho|^2\\
& \lesssim & \left(\sum_{m=0}^{2k_0}\sum_{i=1}^d\int Q\rho_m^2\chi_{\nu_0+1,m}+|\nabla \Phi_m|^2\chi_{\nu_0+1,m}\right)+b^2.
\eea

\noindent\underline{$G_\rho$ term}. Recall \eqref {defgrho}
$$G_\rho=-\rho\Delta \Psi-2\nabla\rho\cdot\nabla \Psix, $$ then by Leibniz:
$$
|\pa^mG_\rho|^2\lesssim \sum_{j_1+j_2=m+2, j_2\geq 1}|\pa^{j_1}\rho|^2|\pa^{j_2}\Psi|^2.$$ 
We recall the pointwise bounds \eqref{smallglobalboot} for $Z\leq 3Z^*$,
\bee
|\pa^{j_1}\rho|\le \frac{C_K}{\la Z\ra^{j_1+\frac{2(r-1)}{p-1}}}, \ \ |\pa^{j_2}\Psi|&\leq& \frac{C_K}{\la Z\ra^{j_2+r-2}}.
\eee
This yields, recalling \eqref{beubeibebiev}, for $j_1\le 2k_0$:
 \bee
&&\int \chi_{\nu_0,m}Q|\pa^{j_1}\rho|^2|\pa^{j_2}\Psi|^2\lesssim \int Q\zeta\left(\frac{Z}{Z^*}\right) \frac{|\pa^{j_1}\rho|^2}{Z^{2(j_2-m)+d-2(r-1)+2(r-2)+2\nu_0}}\\
& \lesssim & \int\zeta\left(\frac{Z}{Z^*}\right)Q\frac{|\pa^{j_1}\rho|^2}{\la Z\ra^{d-2(r-1)+2(\nu_0-j_1)+2}}\lesssim  \sum_{j=0}^{j_1}\int \chi_{\nu_0+1,j_1}Q|\pa_Z^{j}\rho|^2\\
& \lesssim & \sum_{m=0}^{2k_0}\sum_{i=1}^d\int Q\rho_m^2\chi_{\nu_0+1,m}+|\nabla \Phi_m|^2\chi_{\nu_0+1,m}.
\eee
For $j_1=m+1$, $j_2=1$, we use the other variable:
\bee
&&\int \chi_{\nu_0,m}Q|\pa^{j_1}\rho|^2|\pa^{j_2}\Psi|^2\lesssim \int Q\zeta\left(\frac{Z}{Z^*}\right)\frac{|\pa^{j_2}\Psi|^2}{Z^{2(j_1-m)+d-2(r-1)+\frac{4(r-1)}{p-1}+2\nu_0}}\\
& \lesssim & \int\zeta\left(\frac{Z}{Z^*}\right)\frac{\rho_P^2|\pa^{j_2}\Psi|^2}{\la Z\ra^{d-2(r-1)+2(\nu_0-j_2)+2}}\lesssim \sum_{j=0}^{j_2}\int\zeta\left(\frac{Z}{Z^*}\right)\frac{|\pa_Z^{j}\Phi|^2}{\la Z\ra^{d-2(r-1)+2(\nu_0-j)+2}}\\
&\lesssim &  \sum_{j=0}^{j_2}\int \chi_{\nu_0+1,j}|\pa_Z^{j}\Phi|^2\lesssim  \sum_{m=0}^{2k_0}\sum_{i=1}^d\int Q\rho_m^2\chi_{\nu_0+1,m}+|\nabla \Phi_m|^2\chi_{\nu_0+1,m}
\eee
and \eqref{cenoeneonoenevne} follows for $G_\rho$ by summation on $0\leq m\leq 2k_0$ .

\noindent\underline{$G_\Phi$ term}. Recall \eqref{defgrho}
$$G_\Phi=-\rho_P(|\nabla \Psi|^2+\NL(\rho))+\frac{b^2\rho_P}{\rho_T}\Delta \rho_T.$$
We estimate using the pointwise bounds \eqref{smallglobalboot} for $j_3\le 2k_0$:
\bee
&&|\nabla \pa^m(\rho_P|\nabla \Psi|^2)|\lesssim \sum_{j_1+j_2+j_3=m+1,j_2\le j_3}\frac{\rho_P}{\la Z\ra^{j_1}}|\pa^{j_2+1}\Psi\pa^{j_3+1}\Psi|\\
&\lesssim&  \sum_{j_1+j_2+j_3=m+1,j_2\le j_3}\frac{1}{\la Z\ra^{\frac{2(r-1)}{p-1}+j_1+r-2+j_2+1}}|\pa^{j_3+1}\Psi|\lesssim \sum_{j_3=0}^{2k_0}\frac{|\pa^{j_3+1}\Phi|}{\la Z\ra^{r+m-j_3}}
\eee
and  since $r>1$:
\bee
\sum_{j_3=0}^{2k_0}\int\chi_{\nu_0,m}\frac{|\pa^{j_3+1}\Phi|^2}{\la Z\ra^{2(r+m-j_3)}}\lesssim \sum_{j_3=0}^{2k_0}\int\chi_{\nu_0+1,j_3}|\nabla\Phi_{j_3}|^2.
\eee
For $j_3=2k_0+1$, we use the other variable and the conclusion follows similarly.

The quantum pressure term is estimated using the pointwise bounds \eqref{smallglobalboot}:
\bee
&&\int \chi_{\nu_0,m}\left|\nabla \pa^m\left(\frac{b^2\rho_P}{\rho_T}\Delta \rho_T\right)\right|^2\lesssim C_Kb^4\int_{Z\le 3Z^*}\frac{\chi_{\nu_0,m}}{\la Z\ra^{\frac{4(r-1)}{p-1}+2(m+3)}}\\
& \lesssim & C_K b^4 \int_{Z\le 3Z^*}\frac{Z^{d-1}dZ}{\la Z\ra^{d-2(r-1)+2(\nu_0+\frac{2(r-1)}{p-1}-m)+2(m+3)}}\leq b^2
\eee
\\

\noindent{\bf step 2} Initialization and lower bound on the bootstrap time $\tau^*$.\\
 Fix a large enough $Z_0$ and pick a small enough universal constant $\omega_0$ such that 
\be
 \label{positiiviiefecorrection}
 \forall Z\ge 0, \ \ -\omega_0+H_2\ge \frac{\omega_0}{2}>0
 \ee
 and let $\Gamma=\Gamma(Z_0)$ such that  
\be
\label{definitionna}
\frac{Z_0}{2\hat{Z_a}}e^{-\omega_0\Gamma}=1.
\ee
We claim that provided ${\tau_0}$ has been chosen sufficiently large, the bootstrap time $\tau^*$ of Proposition \ref{propboot} satisfies the lower bound
\be
\label{vneiovnevneonvein}
\tau^*\ge \tau_0+\Gamma.
\ee
Indeed, in view of sections \ref{sectionsobolev}, \ref{sec:point}, \ref{sec:high} there remains to control the bound \eqref{eq:bootdecay} on $[\tau_0,\tau_0+\Gamma]$. By \eqref{eq:Xfgrow}, the desired bounds already hold for $Z\le \tilde Z_a$ on $[\tau_0,\tau_0+\Gamma]$.\\

We now run the energy estimate \eqref{energyidentittyk} with $\chi=\chi_{\nu_0,m}$ and obtain from \eqref{energyidentittyk}, \eqref{cenoeneonoenevne} the rough bound on $[\tau_0,\tau^*]$:
$$\frac{d}{d\tau}\left\{\int (p-1)Q\rho_m^2\chi_{\nu_0,m}+|\nabla \Phi_m|^2\chi_{\nu_0,m}\right\}\leq C\int (p-1)Q\rho_m^2\chi_{\nu_0,m}+|\nabla \Phi_m|^2\chi_{\nu_0,m}+b^2.
$$
which yields using \eqref{improvedsobolevlowinit}:
\bee
&&\int (p-1)Q\rho_m^2\chi_{\nu_0,m}+|\nabla \Phi_m|^2\chi_{\nu_0,m}\le e^{C(\tau-\tau_0)}\int (p-1)Q(\rho_m(0))^2\chi_{\nu_0,m}+|\nabla \Phi_m(0)|^2\chi_{\nu_0,m}\\
& + & e^{C\tau}\int_{\tau_0}^{\tau}e^{-(C+2\delta_g)\sigma}d\sigma\leq e^{C\Gamma}\left[C_0e^{-\delta_g\tau_0}+e^{-2\delta_g\tau_0}\right]\leq 2e^{C\Gamma}C_0e^{-\delta_g\tau_0}
\eee
and hence
\bee
&&e^{\frac{4\delta_g}{5}\tau}\left[\int (p-1)Q\rho_m^2\chi_{\nu_0,m}+|\nabla \Phi_m|^2\chi_{\nu_0,m}\right]\\
&\le&  e^{\frac{4\delta_g}{5}\tau_0}e^{\frac{4\delta_g}{5}\Gamma}\left[\int (p-1)Q\rho_m^2\chi_{\nu_0,m}+|\nabla \Phi_m|^2\chi_{\nu_0,m}\int (p-1)Q\rho_m^2\chi_{\nu_0,m}+|\nabla \Phi_m|^2\chi_{\nu_0,m}\right]\\
& = & 2e^{C\Gamma}C_0e^{-\delta_g\tau_0}e^{\frac{4\delta_g}{5}\tau_0}\le e^{2C\Gamma}e^{-\frac{\delta_g}{10}\tau_0}\le 1
\eee
which concludes the proof of \eqref{vneiovnevneonvein} and \eqref{improvedsobolevlowbetter} for $\tau\in[\tau_0,\tau_0+\Gamma]$.\\

\noindent{\bf step 3} Finite speed of propagation.
We now pick a time $\tau_f\in[\tau_0+\Gamma,\tau^*]$ and propagate the bound \eqref{firstbound} to the compact set $Z\leq {Z_0}$ using a finite speed of propagation argument. We claim:
\be
\label{lowsobolev}
\|\rho\|^2_{H^{2k_0}(Z\leq \frac{Z_0}{2})}+\|\nabla \Psi\|^2_{H^{2k_0}(Z\le \frac{Z_0}{2})}\le Ce^{-{\delta_g}\tau}.
\ee
 Here the key is that \eqref{firstbound} controls a norm on the set {\em strictly including} the light cone $Z\le Z_2$. 
 Let $$\hat{Z_a}=\frac{\tilde{Z_a}+Z_2}{2}$$ and note that we may, without loss of generality by taking $a>0$ small enough, assume:
\be
\label{estnisneione}
\frac{\tilde{Z}_a}{\hat{Z_a}}\leq 2.
\ee
Recall that $\Gamma=\Gamma(Z_0)$ is parametrized by \eqref{definitionna}. We define $$\chi(\tau,Z)=\zeta\left(\frac{Z}{\nu(\tau)}\right), \ \ \nu(\tau)=\frac{Z_0}{2\hat{Z}_a}e^{-\omega_0(\tau_f-\tau)}$$ with $\omega_0>0$ defined 
in \eqref{positiiviiefecorrection}, \eqref{definitionna} and a fixed spherically symmetric non-increasing cut off function 
\be
\label{defzetavneneov}
\zeta(Z)=\left|\begin{array}{ll} 1\ \ \mbox{for}\ \ 0\leq Z\leq 2\hat{Z_a}\\ 0\ \ \mbox{for}\ \ Z\geq 2\hat{Z_a}.
\end{array}\right., \ \ \zeta'\leq 0
\ee 
We define $$\tau_\Gamma=\tau_f-\Gamma$$ so that from \eqref{definitionna}:
\be
\label{defstopingtime}
\left|\begin{array}{l}
 \tau_0\le \tau_\Gamma\le \tau^*\\
\nu(\tau_\Gamma)=\frac{Z_0}{2\hat{Z}_a}e^{-\omega_0(\tau_f-\tau_\Gamma)}=\frac{Z_0}{2\hat{Z_a}}e^{-\omega_0\Gamma}=1.
\end{array}\right.
\ee
We pick $$0\leq m\leq 2k_0$$ 
then  \eqref{defzetavneneov}, \eqref{defstopingtime} ensure ${\rm Supp}(\chi(\tau_\Gamma,\cdot))\subset \{Z\le \hat{Z}_a\}$ with $\hat{Z_a}<\tilde{Z}_a$ and hence from \eqref{firstbound}:
\be
\label{initiazliationtauagtronwall}
\left(\int (p-1)Q\rho_m^2\chi+|\nabla \Phi_m|^2\chi\right)(\tau_\Gamma)\lesssim e^{-\delta_g \tau_\Gamma}.
\ee
This estimate implies that we can integrate energy identity \eqref{energyidentittyk} {\it only} on the interval 
$[\tau_\Gamma,\tau_f]$. 
We now estimate all terms in  \eqref{energyidentittyk}.\\

\noindent\underline{Boundary terms}. We compute the quadratic terms involving $\Lambda \chi$ which should be thought of as boundary terms. First $$\pa_\tau\chi(\tau,Z)=-\frac{\pa_\tau \nu}{\nu} \frac{Z}{\nu}\pa_Z\zeta\left(\frac{Z}{\nu}\right)=-\omega_0\Lambda \chi.$$ We now assume, recalling \eqref{autroformluehonhtwo}, that $\omega_0$ has been chosen small enough so that \eqref{positiiviiefecorrection} holds, and hence the lower bound on the full boundary quadratic form using $\Lambda \chi\le 0$:

\bee
 &&\frac 12\int (p-1)Q\rho_m^2\left[\pa_\tau \chi+H_2\Lambda \chi\right]+\frac 12\int |\nabla\Phi_m|^2\left[\pa_\tau \chi+H_2\Lambda \chi\right]+\int(p-1)Q\rho_m\nabla\chi\cdot\nabla \Phi_m\\
 & = & \int\left\{\frac 12  (p-1)Q\rho_m^2\left[-\omega_0+H_2\right]+\frac 12|\nabla\Phi_m|^2\left[-\delta_0+H_2\right]\Lambda\chi+\int(p-1)\frac{Q}{Z}\pa_Z \Phi_m\rho_m\right\}\Lambda \chi.\\
 & \leq & \int\left\{\frac 12  (p-1)Q\rho_m^2\left[-\omega_0+H_2\right]+\frac 12|\pa_Z\Phi_m|^2\left[-\omega_0+H_2\right]\Lambda\chi+\int(p-1)\frac{Q}{Z}\pa_Z \Phi_m\rho_m\right\}\Lambda \chi.
 \eee

 From \eqref{enoenoenvonev}, the discriminant of the above quadratic form is given by
 \bee
&& \left[(p-1)\frac{Q}Z\right]^2-(-\omega_0+H_2)^2(p-1)Q=(p-1)Q\left[\frac{(p-1)Q}{Z^2}-(-\omega_0+H_2)^2\right]\\
 &= & (p-1)\mu^2Q\left[\sigma^2-\left(-\frac{\omega_0}{\mu}+1-w\right)^2\right]= (p-1)\mu^2Q\left[-D(Z)+O(\omega_0)\right].
 \eee
 We then observe by definition of $\chi$ that for $\tau\geq \tau_A$: $$Z\in {\rm Supp}\Lambda \chi\Leftrightarrow \hat{Z}_a\leq \frac{Z}{\nu(\tau)}\leq \tilde{Z}_a \Rightarrow  Z\geq \nu(\tau) \hat{Z}_a\geq \nu(\tau_\Gamma) \hat{Z}_a=\hat{Z}_a$$ from which since $\tilde{Z}_a>Z_2$: $$Z\in {\rm Supp}\Lambda \chi \Rightarrow -D(Z)+O(\omega_0)<0$$ provided $0<\omega_0\ll 1$ has been chosen small enough. 

 Together with \eqref{positiiviiefecorrection} and $\Lambda\chi<0$, this ensures: $\forall \tau\in [\tau_\Gamma,\tau^*]$, 
\bea
\label{signfnonon}
\nonumber &&\frac 12\int (p-1)Q\rho_m^2\left[\pa_\tau \chi+H_2\Lambda \chi\right]+\frac 12\int |\nabla\Phi_m|^2\left[\pa_\tau \chi+H_2\Lambda \chi\right]+\int(p-1)Q\rho_m\nabla\chi\cdot\nabla \Phi_m\\
&<&0
\eea

\noindent\underline{Nonlinear terms}. From \eqref{defzetavneneov}, \eqref{estnisneione} for $\tau\leq \tau_f$:
$$\mbox{Supp}\chi\subset\{Z\leq \nu(\tau)\tilde{Z}_a\}\subset \{Z\leq \nu(\tau_f)\tilde{Z}_a\}=\subset \{Z\leq \frac{Z_0}{2}\frac{\tilde{Z}_a}{\hat{Z}_a}\}\subset\{Z\leq Z_0\},$$ and hence from \eqref{estimatheG}:
$$\int \chi|\nabla \pa^mG_\Phi|^2+\int (p-1)Q\chi|\pa^mG_\rho|^2\lesssim  \|\nabla G_\Phi\|^2_{H^{2k_0}(Z\leq Z_0)}+\|\pa^mG_\rho\|^2_{H^{2k_0}(Z\leq Z_0)}\leq  e^{-\frac{4\delta_g}{3}\tau}.
$$
\noindent\underline{Conclusion}. Injecting the collection of above bounds into \eqref{energyidentittyk} and summing over $m\in[0,2k_0]$ yields the crude bound: $\forall \tau\in [\tau_\Gamma,\tau_f]$,
$$
\frac{d}{d\tau}\left\{\sum_{m=0}^{2k_0}\int (p-1)Q\rho_m^2\chi+|\nabla \Phi_m|^2\chi\right\}\leq C \sum_{m=0}^{2k_0}\int (p-1)Q\rho_m^2\chi+|\nabla \Phi_m|^2\chi+e^{-\frac{4\delta_g}{3}\tau}.
$$
We integrate the above  on $[\tau_\Gamma,\tau_f]$

and conclude using $$\chi(\tau_f,Z)=\zeta\left(\frac{Z}{\nu(\tau_f)}\right)=\zeta\left(\frac{Z}{\frac{Z_0}{2\hat{Z}_a}}\right)=1\ \ \mbox{for}\ \ Z\leq {Z_0}$$ and the initialization \eqref{initiazliationtauagtronwall}:
\bee
&&\left[\|\rho\|^2_{H^{2k_0}(Z\leq {Z_0})}+\|\nabla \Psi\|^2_{H^{2k_0}(Z\leq {Z_0})}\right](\tau_f)\lesssim \sum_{m=0}^{2k_0} \left[\int (p-1)Q\rho_m^2\chi+|\nabla \Phi_m|^2\chi\right](\tau_\Gamma)\\
& \lesssim & e^{C(\tau_f-\tau_\Gamma)}e^{-\delta_g\tau_\Gamma}+\int_{\tau_\Gamma}^{\tau_f}e^{C(\tau_f-\sigma)}e^{-\frac{4\delta_g}{3}\sigma}d\sigma\lesssim C(\Gamma)e^{-\delta_g\tau_f}=C(Z_0)e^{-\delta_g\tau_f}.
\eee
Since the time $\tau_f$ is arbitrary in $[\tau_0+\Gamma,\tau^*]$, the bound \eqref{lowsobolev} follows.\\

\noindent{\bf step 4} Proof of \eqref{improvedsobolevlowbetter}. We run the energy identity \eqref{energyidentittyk} with $\chi_{\nu_0,m}$  and estimate each term.\\

 \noindent\underline{\em terms $\frac{Z_0}{3}\leq Z\leq \frac{Z_0}{2}$}. In this zone, we have by construction $$\rho=\rhot$$ and hence the bootstrap bounds \eqref{sobolevinitboot} imply $$\|\rho\|_{H^{k_m}(Z\leq \frac{Z_0}{2})}+\|\nabla \Psi\|_{H^{k_m}(Z\leq \frac{Z_0}{2})}\lesssim 1$$ and hence interpolating with \eqref{lowsobolev} for $k_m $ large enough:
 \bea
 \label{inteprolsaiotnnormone}
 \nonumber \|\rho\|_{H^{m}(\frac{Z_0}{3}\leq Z\leq \frac{Z_0}{2})}&\lesssim& \|\rho\|^{\frac{m}{k_m}}_{H^{k_m}(\frac{Z_0}{3}\leq Z\leq \frac{Z_0}{2})}\|\rho\|^{1-\frac{m}{k_m}}_{L^2(\frac{Z_0}{3}\leq Z\leq \frac{Z_0}{2})}\lesssim e^{-\frac{\delta_g}{2}\left(1-\frac{m}{k_m}\right)}\\
 &\leq& e^{-\frac{4\delta_g}{10}}
 \eea
  and similarly for the phase
  \be
 \label{inteprolsaiotnnormbisone}
 \|\nabla\Psi\|_{H^{m}(\frac{Z_0}{3}\leq Z\leq \frac{Z_0}{2})}\lesssim e^{-\frac{\delta_g}{2}\left(1-\frac{m}{k_m}\right)}\leq e^{-\frac{4\delta_g}{10}}.
 \ee
 
 \noindent\underline{\em Linear term}. We observe the cancellation using \eqref{esterrorpotentials}, \eqref{renormalization}:
  \bea
 \label{vnnenoenoene}
 \nonumber &&\pa_\tau\chi_{\nu_0,m}+H_2\Lambda \chi_{\nu_0,m}=\frac{1}{\la Z\ra^{d-2(r-1)+2(\nu_0-m)}}\left[-\mu\Lambda \zeta\left(\frac{Z}{Z^*}\right)\right]\\
 \nonumber &+& \mu(1-w)\left[\frac{1}{\la Z\ra^{d-2(r-1)+2(\nu_0-m)}}\Lambda \zeta\left(\frac{Z}{Z^*}\right)+\Lambda \left(\frac{1}{\la Z\ra^{d-2(r-1)+2(\nu_0-m)}}\right)\zeta\left(\frac{Z}{Z^*}\right)\right]\\
 & = &-\mu\left[d-2(r-1)+2(\nu_0-m)\right]\chi_{\nu_0,m}+O\left(\frac{1}{\la Z\ra^{d-2(r-1)+2(\nu_0-m)+\omega}}\right)
\eea
for some universal constant $\omega>0$. We now estimate the norm for $2Z^*\le Z\le 3Z^*$.
Using spherical symmetry for $Z\ge 1$ and $m\ge 1$:
\be
\label{beubeibebiev}
|Z^{m}\pa^m\rho|\lesssim \sum_{j=1}^mZ^{m}\frac{|\pa_Z^j\rho|}{Z^{m-j}}\lesssim  \sum_{j=1}^m Z^{j}|\pa_Z^j\rho|
\ee
and hence using the outer $L^\infty$ bound \eqref{smallglobalboot}:
\bea
\label{beibeibeibiebebv}
\nonumber &&\int_{2Z^*\leq Z\leq 3Z^*}\frac{(p-1)Q|\pa^m\rho|^2+|\pa^m\nabla \Phi|^2}{\la Z\ra^{d-2(r-1)+2(\nu_0-m)+\omega}}\\
\nonumber&\lesssim&\int_{2Z^*\leq Z\leq 3Z^*}\left[\sum_{j=0}^m \left|\frac{Z^{j}\pa_Z^j\rho}{\la Z\ra^{\frac d2+\nu_0+\frac{\omega}2}}\right|^2+\sum_{j=1}^{m+1}\left|\frac{Z^{j}\pa_Z^j\Phi}{\la Z\ra^{\nu_0+\frac d2-(r-1)+1+\frac{\omega}{2}}}\right|^2\right]\\
\nonumber& \lesssim & \int_{2Z^*\leq Z\leq 3Z^*}\left[\sum_{j=0}^m \left|\frac{Z^{j}\pa_Z^j\rho}{\rho_P\la Z\ra^{\frac d2+\nu_0+\frac{2(r-1)}{p-1}+\frac{\omega}2}}\right|^2+\sum_{j=1}^{m+1}\left|\la Z\ra^{r-2}\frac{Z^{j}\pa_Z^j\Psi}{\la Z\ra^{\nu_0+\frac{2(r-1)}{p-1}+\frac d2+\frac{\omega}{2}}}\right|^2\right]\\
& \lesssim & \frac{1}{(Z^*)^{\omega+2\left[
\nu_0+\frac{2(r-1)}{p-1}\right]}}\left(1+b^2(Z^*)^{2(r-2)}\right)\leq e^{-\delta_g\tau}
\eea
using 
$$b(Z^*)^{r-2}=e^{\tau\left[-\mathcal e+\mu(r-2)\right]}=e^{\tau\left[-\mathcal e+1-2\mu\right]}=1$$
and the explicit choice from \eqref{venovnoenneneo}: $$2\mu\left(\nu_0+\frac{2(r-1)}{p-1}\right)=\delta_g$$
\noindent\underline{Conclusion} Injecting the above bounds into \eqref{energyidentittyk} yields:
\bee
&&\frac 12\frac{d}{d\tau}\left\{\int (p-1)Q\rho_m^2\chi_{\nu_0,m}+|\nabla \Phi_m|^2\chi_{\nu_0,m}\right\}\\
\nonumber & = &-\int\chi_{\nu_0,m}\left[(p-1)Q\rho_m^2+|\nabla\Phi_m|^2\right]\left[\mu\nu_0+\frac{2\mu(r-1)}{p-1}\right] \\
\nonumber& + & O\left(\int_{Z_0\leq Z\leq 2Z^*} \chi_{\nu_0,m}\left[\sum_{m=0}^{m+1}\frac{|\pa_Z^j\Phi|^2}{\la Z\ra^{2(m+1-j)+2\omega}}+\sum_{j=0}^{m}\frac{Q|\pa_Z^j\rho|^2}{\la Z\ra^{2(m-j)+2\omega}}\right]+e^{-\frac{4\delta_g}{5}\tau}\right)\\
\nonumber&+ &  O\left(\int \chi_{\nu_0,m}|\nabla \Phi_m||\nabla \pa^mG_\Phi|+\int\chi_{\nu_0,m}Q|\rho_m||\pa^mG_\rho|\right)\eee
and hence after summing over $m$:
\bee
&&\frac 12\frac{d}{d\tau}\left\{\sum_{m=0}^{2k_0}\int (p-1)Q\rho_m^2\chi_{\nu_0,m}+|\nabla \Phi_m|^2\chi_{\nu_0,m}\right\}\\
\nonumber& = & -\mu\left[\nu_0+\frac{2(r-1)}{p-1}\right]\sum_{m=0}^{2k_0} \int\chi_{\nu_0,m}\left[(p-1)Q\rho_m^2+|\nabla\Phi_m|^2\right]\\
& + & O\left(e^{-\frac{4\delta_g}{5}\tau}+\sum_{m=0}^{2k_0}\int (p-1)Q\rho_m^2\chi_{\nu_0+\omega,m}+|\nabla \Phi_m|^2\chi_{\nu_0+\omega,m}\right)\\
\nonumber&+ &  \sum_{m=0}^{2k_0}O\left(\int \chi_{\nu_0,m}|\nabla \Phi_m||\nabla \pa^mG_\Phi|+\int\chi_{\nu_0,m}Q|\rho_m||\pa^mG_\rho|\right)
\eee
We now inject \eqref{cenoeneonoenevne} and obtain:
\bee
&&\frac 12\frac{d}{d\tau}\left\{\sum_{m=0}^{2k_0}\int (p-1)Q\rho_m^2\chi_{\nu_0,m}+|\nabla \Phi_m|^2\chi_{\nu,m}\right\}\\
\nonumber & \leq &-\sum_{m=0}^{2k_0}\int\chi_{\nu,m}\left[(p-1)Q\rho_m^2+|\nabla\Phi_m|^2\right]\mu\left[\nu_0+\frac{2(r-1)}{p-1}\right]+e^{-\frac{4\delta_g}{5}\tau} \\
\nonumber& + &C\sum_{m=0}^{2k_0}\left(\int_{Z_0\leq Z\leq 2Z^*} \chi_{\nu,m}\left[\sum_{j=0}^{m+1}\frac{|\pa_Z^j\Phi|^2}{\la Z\ra^{2(m+1-j)+2\omega}}+\sum_{j=0}^{m}\frac{Q|\pa_Z^j\rho|^2}{\la Z\ra^{2(m-j)+2\omega}}\right]\right)
\eee
Using \eqref{lowsobolev} we conclude 
\bea
\label{estigneononeo}
&&\frac 12\frac{d}{d\tau}\left\{\sum_{m=0}^{2k_0}\int (p-1)Q\rho_m^2\chi_{\nu_0,m}+|\nabla \Phi_m|^2\chi_{\nu_0,m}\right\}\\
\nonumber& = & -\mu\left[\nu_0+\frac{2(r-1)}{p-1}+O\left(\frac{1}{Z_0^C}\right)\right]\sum_{m=0}^{2k_0} \int\chi_{\nu_0,m}\left[(p-1)Q\rho_m^2+|\nabla\Phi_m|^2\right]\\
\nonumber&+ & O\left(e^{-\frac{4\delta_g}{5}}+\sum_{m=0}^{2k_0} \int \chi_{\nu_0,m}|\nabla \pa^mG_\Phi|^2+\int (p-1)Q\chi_{\nu_0,m}|\pa^mG_\rho|^2\right).
\eea
Therefore, using also \eqref{cenoeneonoenevne}, for $Z_0$ large enough and universal and $$2\mu\left(\nu_0+\frac{2(r-1)}{p-1}\right)=\delta_g,$$ there holds
 \bee
 &&\frac{d}{d\tau}\left\{\sum_{m=0}^{2k_0}\int (p-1)Q\rho_m^2\chi_{\nu_0,m}+|\nabla \Phi_m|^2\chi_{\nu_0,m}\right\}\\
\nonumber& \le & -\delta_g\sum_{m=0}^{2k_0} \int\chi_{\nu_0,m}\left[(p-1)Q\rho_m^2+|\nabla\Phi_m|^2\right]+C e^{-\frac{4\delta_g\tau}{5}}.
\eee
Integrating in time and using \eqref{improvedsobolevlowinit} yields \eqref{improvedsobolevlowbetter}.
 \end{proof}
 

\subsection{Closing the bootstrap and proof of Theorem \ref{thmmain}}
\label{proofthmmamin}


We are now in position to prove the bootstrap Proposition \ref{propboot} which immediately implies Theorem \ref{thmmain}.

 \begin{proof}[Proof of Proposition \ref{propboot} and Theorem \ref{thmmain}] {Recall that the non vanishing of the solution is ensured by \eqref{nonvavnishig}. It remains to close the bound \eqref{bootnonvanishing}. Indeed, from \eqref{renorlianoineo}, \eqref{renormalization}, \eqref{definitionprofilewithtailchange} for $Z\ge Z^*$:
 $$\frac{|\Delta u|}{\rho_D}\lesssim \frac{(Z^*)^2}{\rho_D}\left[|\Delta \rho_T|+\frac{|\pa_Z\rho_T||\pa_Z\Psi_T|}{b}+\frac{|\rho_T\Delta \Psi_T|}{b}\right]\lesssim 1$$ where we used \eqref{smallglobalboot} in the last step. The $|u|^p$ term is handled similarily, and \eqref{bootnonvanishing} is improved for $b_0$ small enough\footnote{The smallness of $b_0$ is responsible for the size of the time length between initial data and formation of a singularity.}. Note also that the bounds \eqref{smallglobalboot} imply $$\|u(t)\|_{H^{k_c}}\le C(t)$$ for $\frac d2\ll k_c\ll k_m$ for times in the bootstrap interval and hence the bootstrap time is strictly smaller than the life time provided by standard Cauchy theory.}\\
We now conclude from a classical topological argument \`a la Brouwer. The bounds of sections 5,6,7,8 have been  shown to hold for all initial data on the time 
 interval $[\tau_0, \tau_0+\Gamma]$ with $\Gamma$ large. Moreover, as explained in the proof of Lemma \ref{lemmalightcone},
 they can be immediately propagated to any time $\tau^*$ after a choice of projection of initial data on the subspace of unstable 
 modes $PX(\tau_0)$. This choice is dictated by Lemma \ref{browerset}. A continuity argument implies $\tau^*=\infty$ for this data, and the conclusions of Theorem \ref{thmmain} follow.

 \end{proof}


\begin{appendix}


\section{Comparison with compressible Euler dynamics}
We consider the compressible Euler equations with a polytropic equation of state: 
\be
\label{NScomp}
\left|\begin{array}{lll}\pa_t\rho+\nabla\cdot(\rho u)=0\\
\rho\pa_tu+\rho u\cdot\nabla u+\nabla P=0\\
P=\frac{\gamma-1}{\gamma}\rho^\gamma\end{array}\right., \ \ x\in \Bbb R^d.
\ee 
for $\gamma>1.$\\

\noindent{\bf step 1} Scaling and renormalization. The scaling symmetry\footnote{We choose a 1-parameter of scaling transformation, which is compatible with the Navier-Stokes equations, out of a larger 2-parameter family of possible transformations.} is
$$\l^{\frac{2}{\gamma+1}} \rho(\l^{\frac{2\gamma}{\gamma+1}}t,\l x), \ \ \l^{\frac{\gamma-1}{\gamma+1}} u(\l^{\frac{2\gamma}{\gamma+1}}t,\l x).$$ We renormalize self-similarly $$\frac{d\tau}{dt}=\frac{1}{\l^{\frac{2\gamma}{\gamma+1}}}, \ \ -\frac{\l_\tau}{\l}=\frac 12$$ and obtain:
\be
\label{selfsimilarflow}
\left|\begin{array}{ll}\pa_\tau \rho+\frac 12\left(\frac{2}{\gamma+1}\rho+y\cdot\nabla\rho\right)+\nabla\cdot(\rho u)=0\\
\rho\pa_\tau u+\frac 12\rho\left(\frac{\gamma-1}{\gamma+1}u+y\cdot\nabla u\right)+\rho u\cdot\nabla u+\nabla P=0.
\end{array}\right.
\ee
As above, we proceed with a front renormalization 
$$\rho\mapsto \frac{1}{b^{\frac 1{\gamma-1}}}\rho(y\sqrt{b}), \ \ u\mapsto \frac{1}{b^{\frac 12}}u(y\sqrt{b})$$ 
with 
$$\frac{b_\tau}b=-{\mathcal e}$$ and consider a potential
 spherically symmetric flow with 
 $u=\nabla \Psi=\Psi'$. A direct computation in which we also integrate the second equation leads to 
$$\left|\begin{array}{ll} -\pa_\tau\rho=\Delta \Psi+\left(\frac{{\mathcal e}}{\gamma-1}+\frac{1}{\gamma+1}\right)+\frac{\nabla \rho}{\rho}\cdot\left[\left(\frac{1-\mathcal e}{2}\right)Z+\nabla \Psi\right]\\
-\pa_\tau\Psi=\frac12|\nabla \Psi|^2+\left({\mathcal e}-\frac{1}{\gamma+1}\right)\Psi-1+\left(\frac{1-{\mathcal e}}{2}\right)Z\cdot\nabla \Psi+\rho^{\gamma-1}
\end{array}\right.
$$
A stationary solution of the above equation satisfies 
\be
\label{generalcequation'}
\left|\begin{array}{ll} \Delta \Psi+\left(\frac{{\mathcal e}}{\gamma-1}+\frac{1}{\gamma+1}\right)+\frac{\nabla \rho}{\rho}\cdot\left[\left(\frac{1-{\mathcal e}}{2}\right)Z+\nabla \Psi\right]=0\\
\frac12|\nabla \Psi|^2+\left({\mathcal e}-\frac{1}{\gamma+1}\right)\Psi+\left(\frac{1-{\mathcal e}}{2}\right)Z\cdot\nabla \Psi+\rho^{\gamma-1}=1
\end{array}\right.
\ee

\noindent{\bf step 2} Emden transform.
We introduce the variables $$V=\Psi', \ \ S= \sqrt{2}\rho^{\frac{\gamma-1}{2}}$$ where $S$ is the space dependent sound speed, so that equivalently taking the derivative of the second equation:

$$
\left|\begin{array}{ll} V'+\frac{d-1}{Z}V+\left(\frac{{\mathcal e}}{\gamma-1}+\frac{1}{\gamma+1}\right)+\frac{2}{\gamma-1}\frac{S'}{S}\left[\left(\frac{1-{\mathcal e}}{2}\right)Z+V\right]=0\\
VV'+\left({\mathcal e}-\frac{1}{\gamma+1}\right)V+\left(\frac{1-{\mathcal e}}{2}\right)(ZV'+V)+SS'=0.
\end{array}\right.
$$

Let $$x=\log Z,\ \ V(Z)=v(x), \ \ S(Z)=s(x), \ \ Z\frac{d}{dZ}=\frac{d}{dX}.$$

\noindent\underline{First equation}. 
$$\frac{v'}{Z}+\frac{(d-1)v}{Z}+\left(\frac{e}{\gamma-1}+\frac{1}{\gamma+1}\right)+\frac{2}{\gamma-1}\frac{s'}{s}\frac{1}{Z}\left[\left(\frac{1-{\mathcal e}}{2}\right)Z+V\right]$$ and hence letting $$v(x)=e^x w, \ \ s(x)=e^x\sigma$$ yields
$$(w'+w)+(d-1)w+\left(\frac{{\mathcal e}}{\gamma-1}+\frac{1}{\gamma+1}\right)+\frac{2}{\gamma-1}\left(\frac{\sigma'}{\sigma}+1\right)\left(\frac{1-{\mathcal e}}{2}+w\right)=0$$ 
i.e., 
$$\sigma w'+\frac{2}{\gamma-1}\left(\frac{1-{\mathcal e}}{2}+w\right)\sigma'+\sigma\left[\left(d+\frac{2}{\gamma-1}\right)w+\frac{2\gamma}{\gamma^2-1}\right]=0$$
\noindent\underline{Second equation}. We get
$$\frac{vv'}{Z}+\left({\mathcal e}-\frac{1}{\gamma+1}\right)v+\left(\frac{1-{\mathcal e}}{2}\right)(v'+v)+\frac{ss'}{Z}=0$$
and hence
$$w(w'+w)+\left({\mathcal e}-\frac{1}{\gamma+1}\right)w+\left(\frac{1-{\mathcal e}}{2}\right)(w'+2w)+\sigma(\sigma'+\sigma)=0$$
or equivalently
$$\left(w+\frac{1-\mathcal e}{2}\right)w'+\sigma\sigma'+\left(w^2+\frac{\gamma}{\gamma+1}w+\sigma^2\right)=0$$
We have obtained:

\begin{lemma}[Emden transform] Let $$x=\log Z, \ \ \Psi'(Z)=e^xw(x), \ \ S(Z)=e^x\sigma(x), \ \ S=\sqrt{2}\rho^{\frac{\gamma-1}{2}},$$ then 
\be
\label{newsystem}
\left|\begin{array}{ll}
\left(w+\frac{1-\mathcal e}{2}\right)w'+\sigma\sigma'+\left(w^2+\frac{\gamma}{\gamma+1}w+\sigma^2\right)=0\\
\sigma w'+\frac{2}{\gamma-1}\left(\frac{1-{\mathcal e}}{2}+w\right)\sigma'+\sigma\left[\left(d+\frac{2}{\gamma-1}\right)w+\frac{2\gamma}{\gamma^2-1}\right]=0\\
\end{array}\right.
\ee
\end{lemma}

\noindent{\bf step 3} Renormalized form. We define 
$$
\ell=\frac{2}{\gamma-1}, \ \  r=\frac{2\gamma}{(1-{\mathcal e})(\gamma+1)}, \ \  \phi^2\left(\frac{2}{{\mathcal e}-1}\right)^2=\ell
$$
and the renormalized unknowns 
\be
\label{newunkonwns}
U=\frac{2}{c-1}w, \ \ \Sigma=\frac{\sigma}{\phi}.
\ee
The second equation becomes:
$$\Sigma \frac{e-1}{2}U'+\ell\frac{e-1}{2}(U-1)\Sigma'+\Sigma\ell\left[\left(1+\frac{d}{\ell}\right)\frac{{\mathcal e}-1}{2}U+\frac{2\gamma}{(\gamma^2-1)\ell}\right]=0$$ i.e., 
$$\frac{\sigma}{\ell}U'+(U-1)\sigma'+\sigma\left[\left(1+\frac{d}{\ell}\right)U-\frac{2\gamma}{(1-{\mathcal e})(\gamma+1)}\right]=0.$$
The first equation becomes
$$\left(\frac{e-1}{2}\right)^2UU'+\phi^2\Sigma\Sigma'+\left[\left(\frac{e-1}{2}\right)^2U^2+\frac{\gamma}{\gamma+1}\frac{{\mathcal e}-1}{2}U+\phi^2\Sigma^2\right]=0$$
and hence \eqref{newunkonwns} yields:
$$(U-1)U'+\ell \Sigma\Sigma'+\left[U^2-\frac{2\gamma}{(1-{\mathcal e})(\gamma+1)}U+\ell \Sigma^2\right]=0.$$
We arrive at the renormalized system
$$\left|\begin{array}{ll}
(U-1)U'+\ell \Sigma\Sigma'+(U^2-rU+\ell \Sigma^2)=0\\
\frac{\Sigma}{\ell}U'+(U-1)\Sigma'+\Sigma\left[U\left(\frac{d}{\ell}+1\right)-r\right]=0
\end{array}
\right.
$$
which is identical to the system \eqref{systemedefoc} for the defocusing NLS but with the parameters 
$$
\ell=\frac{2}{\gamma-1}, \ \  r=\frac{2\gamma}{(1-{\mathcal e})(\gamma+1)}
$$
in place of 
$$
\ell=\frac 4{p-1}, \ \  r=\frac{2}{(1-{\mathcal e})}
$$


\section{Hardy inequality}


\begin{lemma}
Assume $2\gamma\notin \Bbb Z$. Then, for all $u\in \matchal C_{{\rm rad}}^\infty(r\ge 1)$ and $j\ge 1$:
\be
\label{radialhardy}
\int_{r\ge 1}r^{2\gamma}u^2dr\lesssim_{j,\gamma} \|u\|^2_{H^{j}(1\le r\le 2)}+\int_{r\ge 1}r^{2(\gamma+j)}|\pa^{j}_ru|^2dr
\ee
\end{lemma}

\begin{proof} Assume $2\gamma\ne -1$, We integrate by parts
\bee
&&\int r^{2\gamma} u^2dr=\frac{1}{2\gamma+1}[r^{2\gamma+1}u^2]_1^{+\infty}-\frac{2}{2\gamma+1}\int_{r\ge 1}r^{2\gamma+1}u\pa_rudr\\
& \leq & \|u\|^2_{H^{1}(1\le r\le 2)}+C\left(\int_{r\ge 1}r^{2\gamma}u^2dr\right)^{\frac 12}\left(\int_{r\ge 1}r^{2\gamma+2}(\pa_ru)^2dr\right)^{\frac 12}
\eee
where we used the one dimensional Sobolev embedding, and \eqref{radialhardy} for $j=1$ follows by H\"older. 
For higher values of $j$, \eqref{radialhardy} now follows by induction.
\end{proof}

\section{Commutator for $\Delta^k$}

\begin{lemma}[Commutator for $\Delta^k$]
Let $k\ge 1$, then for any two smooth function $V,\Phi$, there holds:
\be
\label{estimatecommutatorvlkeveln}
[\Delta^k,V]\Phi-2k\nabla V\cdot\nabla \Delta^{k-1}\Phi=\sum_{|\alpha|+|\beta|=2k,|\beta|\le 2k-2}c_{k,\alpha,\beta}\pa^\alpha V\pa^\beta\Phi.
\ee
where $\pa^\alpha=\pa_1^{\alpha_1}\dots\pa^{\alpha_d}_d$, $|\alpha|=\alpha_1+\dots+\alpha_d.$
\end{lemma}

\begin{proof} We argue by induction on $k$. For $k=1$:
$$\Delta(V\Phi)-V\Delta \Phi=2\nabla V\cdot\nabla \Phi.$$
We assume \eqref{estimatecommutatorvlkeveln} for $k$ and prove $k+1$. Indeed,
\bee
&&\Delta^{k+1}(V\Phi)=\Delta([\Delta^k,V]\Phi+V\Delta^k\Phi)\\
&=&\Delta\left(2k\nabla V\cdot\nabla \Phi+\sum_{|\alpha|+|\beta|=2k, |\beta|\le 2k-2}c_{k,\alpha,\beta}\pa^\alpha V\pa^\beta\Phi+V\Delta^k\Phi\right)=  2k\nabla V\cdot\nabla \Delta^{k}\Phi\\
&+&\sum_{|\alpha|+|\beta|=2k+2, |\alpha|\ge 1}\tilde{c}_{k,\alpha,\beta}\pa^\alpha V\pa^\beta\Phi+2k\nabla V\cdot\nabla \Delta^k\Phi+V\Delta^{k+1}\Phi+2\nabla V\cdot\nabla \Delta^k\Phi\\
& = & 2(k+1)\nabla V\cdot\nabla \Delta^{k}\Phi+\sum_{|\alpha|+|\beta|=2k+2, |\alpha|\ge 1}c_{k+1,\alpha,\beta}\pa^\alpha V\pa^\beta\Phi
\eee
and \eqref{estimatecommutatorvlkeveln} is proved.
\end{proof}


\section{Behaviour of Sobolev norms}
\label{formal}


We compute Sobolev norms assuming that the leading part of the solution is given by \eqref{nvivnineoneinov}. Computations below are formal but could be justified as a consequence of the bootstrap estimates.\\
\noindent\underline{Dirichlet energy of the profile}. We recall \eqref{vneoenvneovnie}, \eqref{nvivnineoneinov} and compute:
$$
\|\nabla u\|_{L^2}^2\sim1+\|\nabla u\|_{L^2(|x|\le 1)}^2=1+ \int_{|x|\leq 1}|\nabla\rho|^2dx+\int_{|x|\le 1}\rho^2|\nabla\phi|^2.$$
We compute for the first term:
\bee
&&\int_{|x|\le 1}|\nabla\rho|^2dx\sim\frac{1}{(T-t)^{\frac{4(r-1)}{r(p-1)}}} \int_{|Z|\le \frac{1}{(T-t)^{\frac 1r}}}\frac{(T-t)^{\frac dr}}{(T-t)^{\frac{2}{r}}}\frac{Z^{d-1}dZ}{\la Z\ra^{\frac{4(r-1)}{p-1}+2}}\\
& = & \frac{1}{(T-t)^{\frac2r\left(\frac{2(r-2)}{p-1}+\frac{2}{p-1}-\frac d2+1\right)}}\int_{|Z|\le \frac1{(T-t)^{\frac 1r}}}\frac{dZ}{Z^{1+2\left(1+\frac{2(r-2)}{p-1}+\frac{2}{p-1}-\frac d2\right)}}\\
& = & \frac{1}{(T-t)^{\frac 2r(1-\sigma)}}\int_{|Z|\le \frac{1}{(T-t)^{\frac 1r}}}\frac{dZ}{\la Z\ra^{1+2(1-\sigma)}}
 \eee
 with 
 \bee
&& \sigma=s_c-\frac{2(r-2)}{p-1}>1\Leftrightarrow \frac d2-\frac{\ell}{2}-\frac{\ell}{2}(r-2)>1\Leftrightarrow d-\ell(r-1)>2\\
& \Leftrightarrow& d-2>\ell\left(\frac{\ell+d}{\ell+\sqrt{d}}-1\right)=\frac{\ell(d-\sqrt{d})}{\ell+\sqrt{d}}\Leftrightarrow (d-2)\sqrt{d}+\ell(\sqrt{d}-2)>0
\eee
which holds and hence $$\int_{|x|\le 1}|\nabla\rho|^2dx\lesssim 1.$$ Similarily:
  \bee
 &&\int_{|x|\le 1}\rho^2|\nabla\phi|^2=\frac{1}{(T-t)^{\frac{4(r-1)}{r(p-1)}+\frac{2(r-2)}{r}}} \int_{|Z|\le \frac{1}{(T-t)^{\frac 1r}}}\frac{(T-t)^{\frac dr}}{(T-t)^{\frac{2}{r}}}\frac{Z^{d-1}dZ}{\la Z\ra^{2(r-2+1)+\frac{4(r-1)}{p-1}}}\\
 & = & \frac{1}{(T-t)^{\frac 2r\left(r-2+\frac{2(r-2)}{p-1}+1+\frac{2}{p-1}-\frac d2\right)}}\int_{|Z|\le \frac{1}{(T-t)^{\frac 1r}}}\frac{dZ}{\la Z\ra^{1+2\left(1+r-2+\frac{2(r-2)}{p-1}+\frac{2}{p-1}-\frac d2\right)}}
 \eee
 and at $r^*(\ell)$:
 \bee
 &&r-2+\frac{2(r-2)}{p-1}+1+\frac{2}{p-1}-\frac d2<0\Leftrightarrow (r-1)\left(1+\frac{\ell}{2}\right)<\frac d2\\
 & \Leftrightarrow& (2+\ell)(d-\sqrt{d})<d(\ell+\sqrt{d})\Leftrightarrow d(\sqrt{d}-2)+(\ell+2)\sqrt{d}>0
 \eee
 which holds and hence 
$$\int_{|x|\le 1}\rho^2|\nabla\phi|^2\lesssim 1.$$
\noindent\underline{Blow up of large enough Sobolev norms below the scaling}. We now unfold the change of variables 
$$\left|\begin{array}{l}
u(t,x)=\frac{1}{\l(t)^{\frac 2{p-1}}}v(s,y)e^{i\gamma}, \ \ y=\frac{x}{\l}\\
v(s,y)=\frac{1}{(\sqrt{b})^{\frac 2{p-1}}}\left(\rho_Te^{\frac{i}{b}\Psi_T}\right)(\tau,Z), \ \ Z=y\sqrt{b}
\end{array}\right.
$$
which yields
\bee
&&\|\nabla^su\|_{L^2}=\frac{1}{\l^{s-s_c}}\|\nabla^sv\|_{L^2}= \frac{1}{\l^{s-s_c}(\sqrt{b})^{\frac{2}{p-1}}}(\sqrt{b})^{s-\frac d2}\|\nabla^s(\rho_Te^{\frac{\Psi_T}{b}})\|_{L^2}\\
&\gtrsim & e^{\frac{s-s_c}{\tau}}(\sqrt{b})^{s-\frac d2-\frac{2}{p-1}}\|\nabla^s(\rho_Te^{\frac{\Psi_T}{b}})\|_{L^2(|Z|\le 1)}\gtrsim  e^{\frac{s-s_c}{\tau}}(\sqrt{b})^{s-\frac d2-\frac{2}{p-1}}\frac{1}{b^s}\\
& \gtrsim& e^{\tau\left[\frac{s-s_c}{2}+\frac e2\left(s+\frac d2+\frac 2{p-1}\right)\right]}
\eee
which blows up as soon as $$s>\sigma=\frac{1}{1+e}\left[s_c-\mathcal e\left(\frac d2+\frac{2}{p-1}\right)\right].$$ We can check that at $r^*(\ell)$:
\bee
&&\sigma>1\Leftrightarrow \frac d2-\frac{2}{p-1}>1+e+e\left(\frac d2+\frac{2}{p-1}\right)\Leftrightarrow \frac d2(1-\mathcal e)>(1+e)(1+\frac{\ell}{2})\\
&\Leftrightarrow& \frac d2\frac 2r>(r-1)(\ell+2)\Leftrightarrow d>(\ell+2)\left(\frac{d+\ell}{\ell+\sqrt{d}}-1\right)\Leftrightarrow d(\sqrt{d}-2)+(\ell+2)\sqrt{d}>0
\eee
The last inequality holds for our assumptions on $d\ge 5$ and $\ell>0$.
\end{appendix}

\end{document}